\documentclass[oneside,final]{ucr}

\usepackage{etex}
\usepackage[english]{babel}
\usepackage{latexsym,amsfonts,amssymb,exscale,enumerate,amsthm}
\usepackage{tikz,tikz-cd}
\usepackage{flafter}
\usepackage{float}
\usetikzlibrary{intersections}
\usepackage[centertags,sumlimits,intlimits,namelimits,reqno]{amsmath}
\usepackage{amsfonts}
\usepackage{bondgraph}
\usepackage[all,2cell]{xy} \UseAllTwocells
\usepackage[hyphens,spaces,obeyspaces]{url}

\usepackage{tikz,tikz-cd}

\usetikzlibrary{intersections,decorations.markings}
\usetikzlibrary{arrows,positioning,fit,matrix,shapes.geometric,external}
\usetikzlibrary{backgrounds,circuits,circuits.ee.IEC,shapes,fit,matrix}

\pgfdeclarelayer{edgelayer}
\pgfdeclarelayer{nodelayer}
\pgfsetlayers{edgelayer,nodelayer,main}

\tikzset{font=\footnotesize}
\tikzset{->-/.style={decoration={
  markings,
  mark=at position #1 with {\arrow{>}}},postaction={decorate}}}

\tikzstyle{none}=[inner sep=0pt]
\tikzstyle{circ}=[circle,fill=black,draw,inner sep=3pt]

\usepackage{enumitem}

\usepackage[pagebackref=true,
colorlinks=true,linkcolor=blue,citecolor=blue,urlcolor=blue,
pdfauthor={Brandon Hector Coya},
pdftitle={Circuits, Bond Graphs, and Signal-Flow Diagrams: A Categorical Perspective},
pdfsubject={category theory, network theory},
pdfkeywords={category theory, network theory}]{hyperref}

\usepackage{bm} 
\usepackage{color}

\usepackage[mathscr]{euscript}
\SelectTips{eu}{11}

\newtheorem{theorem}{Theorem}[section]
\newtheorem{proposition}[theorem]{Proposition}   
\newtheorem{conjecture}[theorem]{Conjecture}

\newtheorem{lemma}[theorem]{Lemma}   
\newtheorem{corollary}[theorem]{Corollary}  
\theoremstyle{definition}
\newtheorem{definition}[theorem]{Definition}   
\newtheorem{example}[theorem]{Example}

\newcommand{\op}{\mathrm{op}}

\newcommand{\id}{\mathrm{id}}
\newcommand{\maps}{\colon}

\newcommand{\To}{\Rightarrow}
\newcommand\im{{\mathrm{im}}}

\newcommand{\asrelto}{\nrightarrow}

\newcommand{\Rel}{\mathrm{Rel}}

\newcommand{\BondGraph}{\mathrm{BondGraph}}

\newcommand\Span{\mathrm{Span}}
\newcommand{\Corel}{\mathrm{Corel}}
\newcommand{\Cospan}{\mathrm{Cospan}}
\newcommand{\Fin}{\mathrm{Fin}}
\newcommand{\Set}{\mathrm{Set}}

\newcommand{\SigFlow}{\mathrm{SigFlow}}

\newcommand{\Lag}{\mathrm{Lag}}
\newcommand{\Aff}{\mathrm{Aff}}
\newcommand{\PROP}{\mathrm{PROP}}
\newcommand{\Symm}{\mathrm{Symm}}
\newcommand{\Mon}{\mathrm{Mon}}
\newcommand{\Cat}{\mathrm{Cat}}

\newcommand{\Circ}{{\mathrm{Circ}}}

\newcommand{\Ccirc}{\widetilde{\mathrm{Circ}}}
\newcommand{\RLCCirc}{\mathrm{Circ}_{RLC}}

\newcommand{\Mod}{\mathrm{Mod}}

\newcommand{\A}{A}
\newcommand{\C}{C}
\newcommand{\D}{D}
\newcommand{\T}{T}
\newcommand{\NN}{N}
\newcommand{\N}{\mathbb{N}}
\newcommand{\R}{\mathbb{R}}
\newcommand{\Z}{\mathbb{Z}}

\usetikzlibrary{arrows,positioning,fit,matrix,shapes.geometric,external}

\newcommand*\pgfdeclareanchoralias[3]{%
  \expandafter\def\csname pgf@anchor@#1@#3\expandafter\endcsname
     \expandafter{\csname pgf@anchor@#1@#2\endcsname}}

\tikzset{every picture/.style={line width=1.5pt}}

\newcommand{\define}[1]{{\bf \boldmath #1}}

\usetikzlibrary{intersections,decorations.markings}
\usetikzlibrary{arrows,positioning,fit,matrix,shapes.geometric,external}
\usetikzlibrary{backgrounds,circuits,circuits.ee.IEC,shapes,fit,matrix}

\tikzset{
circnode/.style={
  circle, draw=red, very thin, outer sep=0.025em, minimum size=2em,
  fill=red, text centered},
integral/.style={
  regular polygon, regular polygon sides=3, shape border rotate=180, draw=black, very thick,
  outer sep=0.025em, inner sep=0, minimum size=2em, fill=blue!5, text centered},
multiply/.style={
  regular polygon, regular polygon sides=3, shape border rotate=180, draw=black, very thick,
  outer sep=0.025em, inner sep=0, minimum size=2em, fill=blue!5, text centered},
upmultiply/.style={
  regular polygon, regular polygon sides=3, draw=black, very thick,
  outer sep=0.025em, inner sep=0, minimum size=2em, fill=blue!5, text centered},
zero/.style={
  circle, draw=black, very thick, minimum size=0.15cm, fill=black,
  inner sep=0, outer sep=0},
hole/.style={
  circle, draw=white, very thick, minimum size=0.25cm, fill=white,
  inner sep=0, outer sep=0},
bang/.style={
  circle, draw=black, very thick, minimum size=0.15cm, fill=green!10,
  inner sep=0, outer sep=0},
delta/.style={
  regular polygon, regular polygon sides=3, minimum size=0.4cm, inner
  sep=0, outer sep=0.025em, draw=black, very thick, fill=green!10},
codelta/.style={
  regular polygon, regular polygon sides=3, shape border rotate=180, minimum size=0.4cm,
  inner sep=0, outer sep=0.025em, draw=black, very thick, fill=green!10},
plus/.style={
  regular polygon, regular polygon sides=3, shape border rotate=180, minimum size=0.4cm,
  inner sep = 0, outer sep=0.025em, draw=black, very thick, fill=black},
coplus/.style={
  regular polygon, regular polygon sides=3, minimum size=0.4cm,
  inner sep = 0, outer sep=0.025em, draw=black, very thick, fill=black},
sqnode/.style={
  regular polygon, regular polygon sides=4, minimum size=2.6em,
  draw=black, very thick, inner sep=0.2em, outer sep=0.025em,
  fill=yellow!10, text centered},
blackbox/.style={
  regular polygon, regular polygon sides=4, minimum size=2.6em,
  draw=black, very thick, inner sep=0.2em, outer sep=0.025em, fill=black},
bigcirc/.style={
  circle, draw=black, very thick, text width=1.6em, outer sep=0.025em,
  minimum height=1.6em, fill=blue!5, text centered}
 }
\pgfdeclareanchoralias{regular polygon}{corner 1}{io}
\pgfdeclareanchoralias{regular polygon}{corner 2}{left out}
\pgfdeclareanchoralias{regular polygon}{corner 3}{right out}
\pgfdeclareanchoralias{regular polygon}{corner 2}{left copy}
\pgfdeclareanchoralias{regular polygon}{corner 3}{right copy}
\pgfdeclareanchoralias{regular polygon}{corner 3}{addend}
\pgfdeclareanchoralias{regular polygon}{corner 2}{summand}
\pgfdeclareanchoralias{regular polygon}{corner 3}{left in}
\pgfdeclareanchoralias{regular polygon}{corner 2}{right in}

\usepackage{color}

\renewcommand{\arraystretch}{1.3}
\newcommand{\corelation}[1]
{
\begin{aligned}
\begin{tikzpicture}
[circuit ee IEC, set resistor graphic=var resistor IEC
      graphic,scale=.5]
\scalebox{1}{
		\node [style=none] (0) at (-1, -0) {};
		\node  [circle,draw,inner sep=1pt,fill]   (1) at (-0.125, -0) {};
		\node [style=none] (2) at (1, 0.5) {};
		\node [style=none] (3) at (1, -0.5) {};
		\node [style=none] (4) at (-1, -1.5) {};
		\node [style=none] (5) at (-1, -2.5) {};
		\node [style=none] (6) at (1, -1.5) {};
		\node [style=none] (7) at (1, -2.5) {};
		\node [style=none] (9) at (-1.5, -1) {3};
		\node [style=none] (10) at (1.5, -1) {4};
		\node  [circle,draw,inner sep=1pt,fill]   (8) at (0, -2) {};
		\draw[line width = 1.5pt] (0.center) to (1.center);
		\draw[line width = 1.5pt] [in=180, out=60, looseness=1.20] (1.center) to (2.center);
		\draw[line width = 1.5pt] [in=180, out=-60, looseness=1.20] (1.center) to (3.center);
		\draw[line width = 1.5pt] [in=180, out=60, looseness=1.20] (8.center) to (6.center);
		\draw[line width = 1.5pt] [in=180, out=-60, looseness=1.20] (8.center) to (7.center);
		\draw[line width = 1.5pt] [in=0, out=120, looseness=1.20] (8.center) to (4.center);
		\draw[line width = 1.5pt] [in=0, out=-120, looseness=1.20] (8.center) to (5.center);
}
      \end{tikzpicture}
\end{aligned}
}

\newcommand{\mult}[1]
{
\begin{aligned}
\begin{tikzpicture}
[circuit ee IEC, set resistor graphic=var resistor IEC
      graphic,scale=.5]
\scalebox{1}{
		\node [style=none] (0) at (1, -0) {};
		\node  [circle,draw,inner sep=1pt,fill]   (1) at (0.125, -0) {};
		\node [style=none] (2) at (-1, 0.5) {};
		\node [style=none] (3) at (-1, -0.5) {};
		\draw[line width = 1.5pt] (0.center) to (1.center);
		\draw[line width = 1.5pt] [in=0, out=120, looseness=1.20] (1.center) to (2.center);
		\draw[line width = 1.5pt] [in=0, out=-120, looseness=1.20] (1.center) to (3.center);
}
      \end{tikzpicture}
\end{aligned}
}

\newcommand{\unit}[1]
{
  \begin{aligned}
\begin{tikzpicture}
[circuit ee IEC, set resistor graphic=var resistor IEC
      graphic,scale=.5]
\scalebox{1}{
		\node [style=none] (0) at (1, -0) {};
		\node [style=none] (1) at (-1, -0) {};
		\node  [circle,draw,inner sep=1pt,fill]   (2) at (0, -0) {};
		\draw[line width = 1.5pt] (0.center) to (2);
}
      \end{tikzpicture}
  \end{aligned}
}

\newcommand{\comult}[1]
{
\begin{aligned}
\begin{tikzpicture}
[circuit ee IEC, set resistor graphic=var resistor IEC
      graphic,scale=.5]
\scalebox{1}{
		\node [style=none] (0) at (-1, -0) {};
		\node  [circle,draw,inner sep=1pt,fill]   (1) at (-0.125, -0) {};
		\node [style=none] (2) at (1, 0.5) {};
		\node [style=none] (3) at (1, -0.5) {};
		\draw[line width = 1.5pt] (0.center) to (1.center);
		\draw[line width = 1.5pt] [in=180, out=60, looseness=1.20] (1.center) to (2.center);
		\draw[line width = 1.5pt] [in=180, out=-60, looseness=1.20] (1.center) to (3.center);
}
      \end{tikzpicture}
\end{aligned}
}

\newcommand{\counit}[1]
{
\begin{aligned}
\begin{tikzpicture}
[circuit ee IEC, set resistor graphic=var resistor IEC
      graphic,scale=.5]
\scalebox{1}{
		\node [style=none] (0) at (-1, -0) {};
		\node [style=none] (1) at (1, -0) {};
		\node  [circle,draw,inner sep=1pt,fill]   (2) at (0, -0) {};
		\draw[line width = 1.5pt] (0.center) to (2);
}
      \end{tikzpicture}
  \end{aligned}
}

\newcommand{\edge}[1]
{
  \begin{aligned}
\begin{tikzpicture}
[circuit ee IEC, set resistor graphic=var resistor IEC
      graphic,scale=.5]
\scalebox{1}{
		\node [style=none] (0) at (-1, -0) {};
		\node [style=none] (1) at (1, -0) {};
		\draw[line width = 1.5pt] (0.center) to (1)  node [above, pos=0.25] {\Huge $\ell$};
}
      \end{tikzpicture}
  \end{aligned}
}

\newcommand{\idone}[1]
{
  \begin{aligned}
\begin{tikzpicture}
[circuit ee IEC, set resistor graphic=var resistor IEC
      graphic,scale=.5]
\scalebox{1}{
		\node [style=none] (0) at (-1, -0) {};
		\node [style=none] (1) at (1, -0) {};
		\node [style=none] (2) at (0, 0.5) {};
		\node [style=none] (3) at (0, -0.5) {};
		\draw[line width = 1.5pt] (1.center) to (0.center);
}
\end{tikzpicture} 
  \end{aligned}
}

\newcommand{\idonezero}[1]
{
  \begin{aligned}
\begin{tikzpicture}
[circuit ee IEC, set resistor graphic=var resistor IEC
      graphic,scale=.5]
\scalebox{1}{
		\node [style=none] (0) at (-1, -0) {};
		\node [style=none] (1) at (1, -0) {};
		\node [style=none] (2) at (0, 0.5) {};
		\node [style=none] (3) at (0, -0.5) {};
}
\end{tikzpicture} 
  \end{aligned}
}

\newcommand{\identitytwo}[1]
{
  \begin{aligned}
\begin{tikzpicture}
[circuit ee IEC, set resistor graphic=var resistor IEC
      graphic,scale=.5]
\scalebox{1}{
		\node [style=none] (0) at (-.75, -.25) {};
		\node [style=none] (1) at (.75, -.25) {};
		\node [style=none] (2) at (-.75, 0.25) {};
		\node [style=none] (3) at (.75, 0.25) {};
		\draw[line width=1.5pt] (1.center) to (0.center);
		\draw[line width=1.5pt] (2.center) to (3.center);
}
\end{tikzpicture} 
  \end{aligned}
}

\newcommand{\two}[1]
{
  \begin{aligned}
\begin{tikzpicture}
[circuit ee IEC, set resistor graphic=var resistor IEC
      graphic,scale=.5]
\scalebox{1}{
		\node  [circle,draw,inner sep=1pt,fill]   (0) at (0,-.5) {};
		\node  [circle,draw,inner sep=1pt,fill]   (1) at (0, .5) {};
}
\end{tikzpicture}
  \end{aligned}
}

\newcommand{\singlegen}[1]
{
  \begin{aligned}
\begin{tikzpicture}
[circuit ee IEC, set resistor graphic=var resistor IEC
      graphic,scale=.5]
\scalebox{1}{
		\node [style=none] (0) at (-1, -0) {};
		\node [style=none] (1) at (1, -0) {};
		\node [style=none] (2) at (0, 0.5) {};
		\node [style=none] (3) at (0, -0.5) {};
		\draw[line width = 1.5pt] (1.center) to (0.center) node[midway,above] {\Huge $\ell$};
}
\end{tikzpicture} 
  \end{aligned}
}

\newcommand{\swap}[1]
{
  \begin{aligned}
    \resizebox{#1}{!}{
\begin{tikzpicture}
[circuit ee IEC, set resistor graphic=var resistor IEC
      graphic,scale=.5]
	\begin{pgfonlayer}{nodelayer}
		\node [style=none] (2) at (-0.5, -0.5) {};
		\node [style=none] (3) at (-2, 0.5) {};
		\node [style=none] (4) at (-0.5, 0.5) {};
		\node [style=none] (5) at (-2, -0.5) {};
	\end{pgfonlayer}
	\begin{pgfonlayer}{edgelayer}
		\draw[line width = 1.5pt] [in=180, out=0, looseness=1.00] (3.center) to (2.center);
		\draw[line width = 1.5pt] [in=0, out=180, looseness=1.00] (4.center) to (5.center);
	\end{pgfonlayer}
\end{tikzpicture}
    }
  \end{aligned}
}

\newcommand{\parmult}[1]
{
\begin{aligned}
\begin{tikzpicture}
[circuit ee IEC, set resistor graphic=var resistor IEC
      graphic,scale=.5]
\scalebox{1}{
		\node [style=none] (0) at (1, -0) {};
		\node  [circle,draw,inner sep=1pt,fill]   (1) at (0.125, -0) {};
		\node [style=none] (2) at (-1, 0.5) {};
		\node [style=none] (3) at (-1, -0.5) {};
		\node [style=none] (4) at (1, -.5) {};
		\node  [circle,draw,inner sep=1pt,fill]   (5) at (0.125, -.5) {};
		\node [style=none] (6) at (-1, -0) {};
		\node [style=none] (7) at (-1, -1) {};
		\draw[line width = 1.5pt] (0.center) to (1.center);
		\draw[line width = 1.5pt] [in=0, out=120, looseness=1.20] (1.center) to (2.center);
		\draw[line width = 1.5pt] [in=0, out=-120, looseness=1.20] (1.center) to (3.center);
		\draw[line width = 1.5pt] (4.center) to (5.center);
		\draw[line width = 1.5pt] [in=0, out=120, looseness=1.20] (5.center) to (6.center);
		\draw[line width = 1.5pt] [in=0, out=-120, looseness=1.20] (5.center) to (7.center);
}
      \end{tikzpicture}
\end{aligned}
}

\newcommand{\weirdmult}[1]
{
\begin{aligned}
\begin{tikzpicture}
[circuit ee IEC, set resistor graphic=var resistor IEC
      graphic,scale=.5]
\scalebox{1}{
		\node [style=none] (0) at (1, -0) {};
		\node  [circle,draw,inner sep=1pt,fill]   (1) at (0.125, -0) {};
		\node [style=none] (2) at (-1, 0.5) {};
		\node [style=none] (3) at (-1, -0.5) {};
		\node [style=none] (4) at (1, -.75) {};
		\node  [circle,draw,inner sep=1pt,fill]   (5) at (0.125, -.75) {};
		\node [style=none] (6) at (-1, -0) {};
		\node [style=none] (7) at (-1, -1.25) {};
		\draw[line width = 1.5pt] (0.center) to (1.center);
		\draw[line width = 1.5pt] [in=0, out=120, looseness=1.20] (1.center) to (2.center);
		\draw[line width = 1.5pt] [in=0, out=-120, looseness=1.20] (1.center) to (3.center);
		\draw[line width = 1.5pt] (4.center) to (5.center);
		\draw[line width = 1.5pt] [in=0, out=120, looseness=1.50] (5.center) to (7.center);
		\draw[line width = 1.5pt] [in=0, out=-120, looseness=1.50] (5.center) to (6.center);
}
      \end{tikzpicture}
\end{aligned}
}

\newcommand{\parcomult}[1]
{
\begin{aligned}
\begin{tikzpicture}[circuit ee IEC, set resistor graphic=var resistor IEC
      graphic,scale=.5]
\scalebox{1}{
\begin{scope}[rotate=180]
		\node [style=none] (0) at (1, -0) {};
		\node  [circle,draw,inner sep=1pt,fill]   (1) at (0.125, -0) {};
		\node [style=none] (2) at (-1, 0.5) {};
		\node [style=none] (3) at (-1, -0.5) {};
		\node [style=none] (4) at (1, -.5) {};
		\node  [circle,draw,inner sep=1pt,fill]   (5) at (0.125, -.5) {};
		\node [style=none] (6) at (-1, -0) {};
		\node [style=none] (7) at (-1, -1) {};
		\draw[line width = 1.5pt] (0.center) to (1.center);
		\draw[line width = 1.5pt] [in=0, out=120, looseness=1.20] (1.center) to (2.center);
		\draw[line width = 1.5pt] [in=0, out=-120, looseness=1.20] (1.center) to (3.center);
		\draw[line width = 1.5pt] (4.center) to (5.center);
		\draw[line width = 1.5pt] [in=0, out=120, looseness=1.20] (5.center) to (6.center);
		\draw[line width = 1.5pt] [in=0, out=-120, looseness=1.20] (5.center) to (7.center);
\end{scope}
}
      \end{tikzpicture}
\end{aligned}
}

\newcommand{\parunit}[1]
{
  \begin{aligned}
\begin{tikzpicture}[circuit ee IEC, set resistor graphic=var resistor IEC
      graphic,scale=.5]
\scalebox{1}{
		\node [style=none] (0) at (1, -0) {};
		\node [style=none] (1) at (-1, -0) {};
		\node  [circle,draw,inner sep=1pt,fill]   (2) at (0, -0) {};
		\node [style=none] (3) at (1, -.5) {};
		\node [style=none] (4) at (-1, -.5) {};
		\node  [circle,draw,inner sep=1pt,fill]   (5) at (0, -.5) {};
		\draw[line width = 1.5pt] (0.center) to (2);
		\draw[line width = 1.5pt] (3.center) to (5);
}
      \end{tikzpicture}
  \end{aligned}
}

\newcommand{\parcounit}[1]
{
  \begin{aligned}
\begin{tikzpicture}[circuit ee IEC, set resistor graphic=var resistor IEC
      graphic,scale=.5]
\scalebox{1}{
\begin{scope}[rotate=180]
		\node [style=none] (0) at (1, -0) {};
		\node [style=none] (1) at (-1, -0) {};
		\node  [circle,draw,inner sep=1pt,fill]   (2) at (0, -0) {};
		\node [style=none] (3) at (1, -.5) {};
		\node [style=none] (4) at (-1, -.5) {};
		\node  [circle,draw,inner sep=1pt,fill]   (5) at (0, -.5) {};
		\draw[line width = 1.5pt] (0.center) to (2);
		\draw[line width = 1.5pt] (3.center) to (5);
\end{scope}
}
      \end{tikzpicture}
  \end{aligned}
}

\newcommand{\assocl}[1]
{
  \begin{aligned}
\begin{tikzpicture}[circuit ee IEC, set resistor graphic=var resistor IEC
      graphic,scale=.5]
\scalebox{1}{
		\node  [circle,draw,inner sep=1pt,fill]   (0) at (0.125, -0) {};
		\node [style=none] (1) at (-1, 0.5) {};
		\node [style=none] (2) at (-1, -0.5) {};
		\node [style=none] (3) at (0, -1) {};
		\node [style=none] (4) at (2.25, -0.5) {};
		\node [style=none] (5) at (0.25, -0) {};
		\node  [circle,draw,inner sep=1pt,fill]   (6) at (1.25, -0.5) {};
		\node [style=none] (7) at (-1, -1) {};
		\draw[line width = 1.5pt] [in=0, out=120, looseness=1.20] (0.center) to (1.center);
		\draw[line width = 1.5pt] [in=0, out=-120, looseness=1.20] (0.center) to (2.center);
		\draw[line width = 1.5pt] (4.center) to (6);
		\draw[line width = 1.5pt] [in=0, out=120, looseness=1.20] (6) to (5.center);
		\draw[line width = 1.5pt] [in=0, out=-120, looseness=1.20] (6) to (3.center);
		\draw[line width = 1.5pt] (3.center) to (7.center);
}
      \end{tikzpicture}
  \end{aligned}
}

\newcommand{\assocr}[1]
{
  \begin{aligned}
\begin{tikzpicture}[circuit ee IEC, set resistor graphic=var resistor IEC
      graphic,scale=.5]
\scalebox{1}{
		\node  [circle,draw,inner sep=1pt,fill]   (0) at (0.125, -0.5) {};
		\node [style=none] (1) at (-1, -1) {};
		\node [style=none] (2) at (-1, 0) {};
		\node [style=none] (3) at (0, 0.5) {};
		\node [style=none] (4) at (2.25, 0) {};
		\node [style=none] (5) at (0.25, -0.5) {};
		\node  [circle,draw,inner sep=1pt,fill]   (6) at (1.25, 0) {};
		\node [style=none] (7) at (-1, 0.5) {};
		\draw[line width = 1.5pt] [in=0, out=-120, looseness=1.20] (0.center) to (1.center);
		\draw[line width = 1.5pt] [in=0, out=120, looseness=1.20] (0.center) to (2.center);
		\draw[line width = 1.5pt] (4.center) to (6);
		\draw[line width = 1.5pt] [in=0, out=-120, looseness=1.20] (6) to (5.center);
		\draw[line width = 1.5pt] [in=0, out=120, looseness=1.20] (6) to (3.center);
		\draw[line width = 1.5pt] (3.center) to (7.center);
}
      \end{tikzpicture}
  \end{aligned}
}

\newcommand{\coassocl}[1]
{
  \begin{aligned}
\begin{tikzpicture}[circuit ee IEC, set resistor graphic=var resistor IEC
      graphic,scale=.5]
\scalebox{1}{
		\node  [circle,draw,inner sep=1pt,fill]   (0) at (1.125, -0.5) {};
		\node [style=none] (1) at (2.25, -1) {};
		\node [style=none] (2) at (2.25, 0) {};
		\node [style=none] (3) at (1.25, 0.5) {};
		\node [style=none] (4) at (-1, 0) {};
		\node [style=none] (5) at (1, -0.5) {};
		\node  [circle,draw,inner sep=1pt,fill]   (6) at (0, 0) {};
		\node [style=none] (7) at (2.25, 0.5) {};
		\draw[line width = 1.5pt] [in=180, out=-60, looseness=1.20] (0.center) to (1.center);
		\draw[line width = 1.5pt] [in=180, out=60, looseness=1.20] (0.center) to (2.center);
		\draw[line width = 1.5pt] (4.center) to (6);
		\draw[line width = 1.5pt] [in=180, out=-60, looseness=1.20] (6) to (5.center);
		\draw[line width = 1.5pt] [in=180, out=60, looseness=1.20] (6) to (3.center);
		\draw[line width = 1.5pt] (3.center) to (7.center);
}
      \end{tikzpicture}
  \end{aligned}
}

\newcommand{\parassocl}[1]
{
\begin{aligned}
\begin{tikzpicture}[circuit ee IEC, set resistor graphic=var resistor IEC
      graphic,scale=.5]
\scalebox{1}{
		\node [style=none] (0) at (1, -0) {};
		\node  [circle,draw,inner sep=1pt,fill]   (1) at (0.125, -0) {};
		\node [style=none] (2) at (-1, 0.5) {}; %
		\node [style=none] (3) at (-1, -0.5) {};
		\node [style=none] (4) at (1, -.5) {};
		\node  [circle,draw,inner sep=1pt,fill]   (5) at (0.125, -.5) {};
		\node [style=none] (6) at (-1, -0) {}; %
		\node [style=none] (7) at (-1, -1) {};
		\node [style=none] (8) at (-3, .5) {}; 
		\node [style=none] (9) at (-3, 1) {};
		\draw[line width = 1.5pt] (0.center) to (1.center);
		\draw[line width = 1.5pt] [in=0, out=120, looseness=1.20] (1.center) to (2.west);
		\draw[line width = 1.5pt] [in=0, out=-120, looseness=1.20] (1.center) to (3.west);
		\draw[line width = 1.5pt] (4.center) to (5.center);
		\draw[line width = 1.5pt] [in=0, out=120, looseness=1.20] (5.center) to (6.west);
		\draw[line width = 1.5pt] [in=0, out=-120, looseness=1.20] (5.center) to (7.west);
		\draw[line width = 1.5pt] [in=0, out=180, looseness=1.5] (2.center) to (9.center);
		\draw[line width = 1.5pt] [in=0, out=180, looseness=1.5] (6.center) to (8.center);
\begin{scope}[shift={(-2,-.5)}]
		\node [style=none] (0) at (1, -0) {};
		\node  [circle,draw,inner sep=1pt,fill]   (1) at (0.125, -0) {};
		\node [style=none] (2) at (-1, 0.5) {};
		\node [style=none] (3) at (-1, -0.5) {};
		\node [style=none] (4) at (1, -.5) {};
		\node  [circle,draw,inner sep=1pt,fill]   (5) at (0.125, -.5) {};
		\node [style=none] (6) at (-1, -0) {};
		\node [style=none] (7) at (-1, -1) {};
		\draw[line width = 1.5pt] (0.center) to (1.center);
		\draw[line width = 1.5pt] [in=0, out=120, looseness=1.20] (1.center) to (2.center);
		\draw[line width = 1.5pt] [in=0, out=-120, looseness=1.20] (1.center) to (3.center);
		\draw[line width = 1.5pt] (4.center) to (5.center);
		\draw[line width = 1.5pt] [in=0, out=120, looseness=1.20] (5.center) to (6.center);
		\draw[line width = 1.5pt] [in=0, out=-120, looseness=1.20] (5.center) to (7.center);
\end{scope}
}
      \end{tikzpicture}
\end{aligned}
}

\newcommand{\parcoassocl}[1]
{
\begin{aligned}
\begin{tikzpicture}[circuit ee IEC, set resistor graphic=var resistor IEC
      graphic,scale=.5]
\scalebox{1}{
\begin{scope}[rotate=180]
		\node [style=none] (0) at (1, -0) {};
		\node  [circle,draw,inner sep=1pt,fill]   (1) at (0.125, -0) {};
		\node [style=none] (2) at (-1, 0.5) {}; %
		\node [style=none] (3) at (-1, -0.5) {};
		\node [style=none] (4) at (1, -.5) {};
		\node  [circle,draw,inner sep=1pt,fill]   (5) at (0.125, -.5) {};
		\node [style=none] (6) at (-1, -0) {}; %
		\node [style=none] (7) at (-1, -1) {};
		\node [style=none] (8) at (-3, .5) {}; 
		\node [style=none] (9) at (-3, 1) {};
		\draw[line width = 1.5pt] (0.center) to (1.center);
		\draw[line width = 1.5pt] [in=0, out=120, looseness=1.20] (1.center) to (2.east);
		\draw[line width = 1.5pt] [in=0, out=-120, looseness=1.20] (1.center) to (3.east);
		\draw[line width = 1.5pt] (4.center) to (5.center);
		\draw[line width = 1.5pt] [in=0, out=120, looseness=1.20] (5.center) to (6.east);
		\draw[line width = 1.5pt] [in=0, out=-120, looseness=1.20] (5.center) to (7.east);
		\draw[line width = 1.5pt] [in=0, out=180, looseness=1.5] (2.center) to (9.center);
		\draw[line width = 1.5pt] [in=0, out=180, looseness=1.5] (6.center) to (8.center);
\begin{scope}[shift={(-2,-.5)}]
		\node [style=none] (0) at (1, -0) {};
		\node  [circle,draw,inner sep=1pt,fill]   (1) at (0.125, -0) {};
		\node [style=none] (2) at (-1, 0.5) {};
		\node [style=none] (3) at (-1, -0.5) {};
		\node [style=none] (4) at (1, -.5) {};
		\node  [circle,draw,inner sep=1pt,fill]   (5) at (0.125, -.5) {};
		\node [style=none] (6) at (-1, -0) {};
		\node [style=none] (7) at (-1, -1) {};
		\draw[line width = 1.5pt] (0.center) to (1.center);
		\draw[line width = 1.5pt] [in=0, out=120, looseness=1.20] (1.center) to (2.center);
		\draw[line width = 1.5pt] [in=0, out=-120, looseness=1.20] (1.center) to (3.center);
		\draw[line width = 1.5pt] (4.center) to (5.center);
		\draw[line width = 1.5pt] [in=0, out=120, looseness=1.20] (5.center) to (6.center);
		\draw[line width = 1.5pt] [in=0, out=-120, looseness=1.20] (5.center) to (7.center);
\end{scope}
\end{scope}
}
      \end{tikzpicture}
\end{aligned}
}

\newcommand{\parassocr}[1]
{
\begin{aligned}
\begin{tikzpicture}[circuit ee IEC, set resistor graphic=var resistor IEC
      graphic,scale=.5]
\scalebox{1}{
\begin{scope}[yscale=-1]
		\node [style=none] (0) at (1, -0) {};
		\node  [circle,draw,inner sep=1pt,fill]   (1) at (0.125, -0) {};
		\node [style=none] (2) at (-1, 0.5) {}; %
		\node [style=none] (3) at (-1, -0.5) {};
		\node [style=none] (4) at (1, -.5) {};
		\node  [circle,draw,inner sep=1pt,fill]   (5) at (0.125, -.5) {};
		\node [style=none] (6) at (-1, -0) {}; %
		\node [style=none] (7) at (-1, -1) {};
		\node [style=none] (8) at (-3, .5) {}; 
		\node [style=none] (9) at (-3, 1) {};
		\draw[line width = 1.5pt] (0.center) to (1.center);
		\draw[line width = 1.5pt] [in=0, out=120, looseness=1.20] (1.center) to (2.west);
		\draw[line width = 1.5pt] [in=0, out=-120, looseness=1.20] (1.center) to (3.west);
		\draw[line width = 1.5pt] (4.center) to (5.center);
		\draw[line width = 1.5pt] [in=0, out=120, looseness=1.20] (5.center) to (6.west);
		\draw[line width = 1.5pt] [in=0, out=-120, looseness=1.20] (5.center) to (7.west);
		\draw[line width = 1.5pt] [in=0, out=180, looseness=1.5] (2.center) to (9.center);
		\draw[line width = 1.5pt] [in=0, out=180, looseness=1.5] (6.center) to (8.center);
\end{scope}
\begin{scope}[shift={(-2,.5)}, yscale=-1]
		\node [style=none] (0) at (1, -0) {};
		\node  [circle,draw,inner sep=1pt,fill]   (1) at (0.125, -0) {};
		\node [style=none] (2) at (-1, 0.5) {};
		\node [style=none] (3) at (-1, -0.5) {};
		\node [style=none] (4) at (1, -.5) {};
		\node  [circle,draw,inner sep=1pt,fill]   (5) at (0.125, -.5) {};
		\node [style=none] (6) at (-1, -0) {};
		\node [style=none] (7) at (-1, -1) {};
		\draw[line width = 1.5pt] (0.center) to (1.center);
		\draw[line width = 1.5pt] [in=0, out=120, looseness=1.20] (1.center) to (2.center);
		\draw[line width = 1.5pt] [in=0, out=-120, looseness=1.20] (1.center) to (3.center);
		\draw[line width = 1.5pt] (4.center) to (5.center);
		\draw[line width = 1.5pt] [in=0, out=120, looseness=1.20] (5.center) to (6.center);
		\draw[line width = 1.5pt] [in=0, out=-120, looseness=1.20] (5.center) to (7.center);
\end{scope}
}
      \end{tikzpicture}
\end{aligned}
}

\newcommand{\parcoassocr}[1]
{
\begin{aligned}
\begin{tikzpicture}[circuit ee IEC, set resistor graphic=var resistor IEC
      graphic,scale=.5]
\scalebox{1}{
\begin{scope}[rotate=180]
\begin{scope}[yscale=-1]
		\node [style=none] (0) at (1, -0) {};
		\node  [circle,draw,inner sep=1pt,fill]   (1) at (0.125, -0) {};
		\node [style=none] (2) at (-1, 0.5) {}; %
		\node [style=none] (3) at (-1, -0.5) {};
		\node [style=none] (4) at (1, -.5) {};
		\node  [circle,draw,inner sep=1pt,fill]   (5) at (0.125, -.5) {};
		\node [style=none] (6) at (-1, -0) {}; %
		\node [style=none] (7) at (-1, -1) {};
		\node [style=none] (8) at (-3, .5) {}; 
		\node [style=none] (9) at (-3, 1) {};
		\draw[line width = 1.5pt] (0.center) to (1.center);
		\draw[line width = 1.5pt] [in=0, out=120, looseness=1.20] (1.center) to (2.east);
		\draw[line width = 1.5pt] [in=0, out=-120, looseness=1.20] (1.center) to (3.east);
		\draw[line width = 1.5pt] (4.center) to (5.center);
		\draw[line width = 1.5pt] [in=0, out=120, looseness=1.20] (5.center) to (6.east);
		\draw[line width = 1.5pt] [in=0, out=-120, looseness=1.20] (5.center) to (7.east);
		\draw[line width = 1.5pt] [in=0, out=180, looseness=1.5] (2.center) to (9.center);
		\draw[line width = 1.5pt] [in=0, out=180, looseness=1.5] (6.center) to (8.center);
\end{scope}
\begin{scope}[shift={(-2,.5)}, yscale=-1]
		\node [style=none] (0) at (1, -0) {};
		\node  [circle,draw,inner sep=1pt,fill]   (1) at (0.125, -0) {};
		\node [style=none] (2) at (-1, 0.5) {};
		\node [style=none] (3) at (-1, -0.5) {};
		\node [style=none] (4) at (1, -.5) {};
		\node  [circle,draw,inner sep=1pt,fill]   (5) at (0.125, -.5) {};
		\node [style=none] (6) at (-1, -0) {};
		\node [style=none] (7) at (-1, -1) {};
		\draw[line width = 1.5pt] (0.center) to (1.center);
		\draw[line width = 1.5pt] [in=0, out=120, looseness=1.20] (1.center) to (2.center);
		\draw[line width = 1.5pt] [in=0, out=-120, looseness=1.20] (1.center) to (3.center);
		\draw[line width = 1.5pt] (4.center) to (5.center);
		\draw[line width = 1.5pt] [in=0, out=120, looseness=1.20] (5.center) to (6.center);
		\draw[line width = 1.5pt] [in=0, out=-120, looseness=1.20] (5.center) to (7.center);
\end{scope}
\end{scope}
}
      \end{tikzpicture}
\end{aligned}
}

\newcommand{\monadassocl}[1]
{
\begin{aligned}
\begin{tikzpicture}[circuit ee IEC, set resistor graphic=var resistor IEC
      graphic,scale=.5]
\scalebox{1}{
		\node [style=none] (0) at (1, -0) {};
		\node [style=none] (1) at (0.125, -0) {};
		\node [style=none] (2) at (-1, 0.5) {};
		\node [style=none] (3) at (-1, -0.5) {};
		\node [style=none] (4) at (1, -0) {};
		\node [style=none] (5) at (-1, 1) {};
		\node [style=none] (6) at (1, .25) {};
		\node [style=none] (7) at (-1, -1) {};
		\node [style=none] (8) at (1, -.25) {};
		\node [style=none] (9) at (1, 1.25) {};
		\node [style=none] (10) at (3, .5) {};
		\node [style=none] (11) at (1, 1.75) {};
		\node [style=none] (12) at (3, 1) {};
		\node [style=none] (13) at (-1, 2) {};
		\node [style=none] (14) at (-1, 2.5) {};
	\draw[line width = 1.5pt] (2) to [in=0, out=0,looseness = 3] (3);
	\draw[line width = 1.5pt] (6) to [in=0, out=0,looseness = 3] (9);
	\draw[line width = 1.5pt] (5) to [in=180, out=0,looseness = 1] (6.east);
	\draw[line width = 1.5pt] (7) to [in=180, out=0, looseness = 1] (8.east);
	\draw[line width = 1.5pt] (8) to [in=180, out=0, looseness = 1] (10);
	\draw[line width = 1.5pt] (11) to [in=180, out=0, looseness = 1] (12);
	\draw[line width = 1.5pt] (13) to [in=180, out=0, looseness = 1] (9.east);
	\draw[line width = 1.5pt] (14) to [in=180, out=0, looseness = 1] (11.east);
}
      \end{tikzpicture}
\end{aligned}
}

\newcommand{\monadcoassocl}[1]
{
\begin{aligned}
\begin{tikzpicture}[circuit ee IEC, set resistor graphic=var resistor IEC
      graphic,scale=.5]
\scalebox{1}{
\begin{scope}[rotate=180]
		\node [style=none] (0) at (1, -0) {};
		\node [style=none] (1) at (0.125, -0) {};
		\node [style=none] (2) at (-1, 0.5) {};
		\node [style=none] (3) at (-1, -0.5) {};
		\node [style=none] (4) at (1, -0) {};
		\node [style=none] (5) at (-1, 1) {};
		\node [style=none] (6) at (1, .25) {};
		\node [style=none] (7) at (-1, -1) {};
		\node [style=none] (8) at (1, -.25) {};
		\node [style=none] (9) at (1, 1.25) {};
		\node [style=none] (10) at (3, .5) {};
		\node [style=none] (11) at (1, 1.75) {};
		\node [style=none] (12) at (3, 1) {};
		\node [style=none] (13) at (-1, 2) {};
		\node [style=none] (14) at (-1, 2.5) {};
	\draw[line width = 1.5pt] (2) to [in=0, out=0,looseness = 3] (3.west);
	\draw[line width = 1.5pt] (6) to [in=0, out=0,looseness = 3] (9.west);
	\draw[line width = 1.5pt] (5) to [in=180, out=0,looseness = 1] (6.west);
	\draw[line width = 1.5pt] (7) to [in=180, out=0, looseness = 1] (8.west);
	\draw[line width = 1.5pt] (8) to [in=180, out=0, looseness = 1] (10);
	\draw[line width = 1.5pt] (11) to [in=180, out=0, looseness = 1] (12);
	\draw[line width = 1.5pt] (13) to [in=180, out=0, looseness = 1] (9.west);
	\draw[line width = 1.5pt] (14) to [in=180, out=0, looseness = 1] (11.west);
\end{scope}
}
      \end{tikzpicture}
\end{aligned}
}

\newcommand{\monadassocm}[1]
{
\begin{aligned}
\begin{tikzpicture}[circuit ee IEC, set resistor graphic=var resistor IEC
      graphic,scale=.5]
\scalebox{1}{
		\node [style=none] (0) at (1, -0) {};
		\node [style=none] (1) at (0.125, -0) {};
		\node [style=none] (2) at (-1, 0.5) {};
		\node [style=none] (3) at (-1, -0.5) {};
		\node [style=none] (4) at (1, -0) {};
		\node [style=none] (5) at (-1, 1) {};
		\node [style=none] (6) at (1, .25) {};
		\node [style=none] (7) at (-1, -1) {};
		\node [style=none] (8) at (1, -.25) {};
		\node [style=none] (9) at (1, 1.25) {};
		\node [style=none] (10) at (3, .5) {};
		\node [style=none] (11) at (1, 1.75) {};
		\node [style=none] (12) at (3, 1) {};
		\node [style=none] (13) at (-1, 2) {};
		\node [style=none] (14) at (-1, 2.5) {};
	\draw[line width = 1.5pt] (2) to [in=0, out=0,looseness = 3] (3);
	\draw[line width = 1.5pt] (5) to [in=0, out=0,looseness = 3] (13);
	\draw[line width = 1.5pt] (7) to [in=180, out=0, looseness = 1] (10);
	\draw[line width = 1.5pt] (14) to [in=180, out=0, looseness = 1] (12);
}
      \end{tikzpicture}
\end{aligned}
}

\newcommand{\monadcoassocm}[1]
{
\begin{aligned}
\begin{tikzpicture}[circuit ee IEC, set resistor graphic=var resistor IEC
      graphic,scale=.5]
\scalebox{1}{
\begin{scope}[rotate=180]
		\node [style=none] (0) at (1, -0) {};
		\node [style=none] (1) at (0.125, -0) {};
		\node [style=none] (2) at (-1, 0.5) {};
		\node [style=none] (3) at (-1, -0.5) {};
		\node [style=none] (4) at (1, -0) {};
		\node [style=none] (5) at (-1, 1) {};
		\node [style=none] (6) at (1, .25) {};
		\node [style=none] (7) at (-1, -1) {};
		\node [style=none] (8) at (1, -.25) {};
		\node [style=none] (9) at (1, 1.25) {};
		\node [style=none] (10) at (3, .5) {};
		\node [style=none] (11) at (1, 1.75) {};
		\node [style=none] (12) at (3, 1) {};
		\node [style=none] (13) at (-1, 2) {};
		\node [style=none] (14) at (-1, 2.5) {};
	\draw[line width = 1.5pt] (2) to [in=0, out=0,looseness = 3] (3.west);
	\draw[line width = 1.5pt] (5) to [in=0, out=0,looseness = 3] (13.west);
	\draw[line width = 1.5pt] (7) to [in=180, out=0, looseness = 1] (10);
	\draw[line width = 1.5pt] (14) to [in=180, out=0, looseness = 1] (12);
\end{scope}
}
      \end{tikzpicture}
\end{aligned}
}

\newcommand{\monadassocr}[1]
{
\begin{aligned}
\begin{tikzpicture}[circuit ee IEC, set resistor graphic=var resistor IEC
      graphic,scale=.5]
\scalebox{1}{
\begin{scope}[yscale=-1]
		\node [style=none] (0) at (1, -0) {};
		\node [style=none] (1) at (0.125, -0) {};
		\node [style=none] (2) at (-1, 0.5) {};
		\node [style=none] (3) at (-1, -0.5) {};
		\node [style=none] (4) at (1, -0) {};
		\node [style=none] (5) at (-1, 1) {};
		\node [style=none] (6) at (1, .25) {};
		\node [style=none] (7) at (-1, -1) {};
		\node [style=none] (8) at (1, -.25) {};
		\node [style=none] (9) at (1, 1.25) {};
		\node [style=none] (10) at (3, .5) {};
		\node [style=none] (11) at (1, 1.75) {};
		\node [style=none] (12) at (3, 1) {};
		\node [style=none] (13) at (-1, 2) {};
		\node [style=none] (14) at (-1, 2.5) {};
	\draw[line width = 1.5pt] (2) to [in=0, out=0,looseness = 3] (3);
	\draw[line width = 1.5pt] (6) to [in=0, out=0,looseness = 3] (9);
	\draw[line width = 1.5pt] (5) to [in=180, out=0,looseness = 1] (6.east);
	\draw[line width = 1.5pt] (7) to [in=180, out=0, looseness = 1] (8.east);
	\draw[line width = 1.5pt] (8) to [in=180, out=0, looseness = 1] (10);
	\draw[line width = 1.5pt] (11) to [in=180, out=0, looseness = 1] (12);
	\draw[line width = 1.5pt] (13) to [in=180, out=0, looseness = 1] (9.east);
	\draw[line width = 1.5pt] (14) to [in=180, out=0, looseness = 1] (11.east);
\end{scope}
}
      \end{tikzpicture}
\end{aligned}
}

\newcommand{\monadcoassocr}[1]
{
\begin{aligned}
\begin{tikzpicture}[circuit ee IEC, set resistor graphic=var resistor IEC
      graphic,scale=.5]
\scalebox{1}{
\begin{scope}[yscale=-1,rotate = 180]
		\node [style=none] (0) at (1, -0) {};
		\node [style=none] (1) at (0.125, -0) {};
		\node [style=none] (2) at (-1, 0.5) {};
		\node [style=none] (3) at (-1, -0.5) {};
		\node [style=none] (4) at (1, -0) {};
		\node [style=none] (5) at (-1, 1) {};
		\node [style=none] (6) at (1, .25) {};
		\node [style=none] (7) at (-1, -1) {};
		\node [style=none] (8) at (1, -.25) {};
		\node [style=none] (9) at (1, 1.25) {};
		\node [style=none] (10) at (3, .5) {};
		\node [style=none] (11) at (1, 1.75) {};
		\node [style=none] (12) at (3, 1) {};
		\node [style=none] (13) at (-1, 2) {};
		\node [style=none] (14) at (-1, 2.5) {};
	\draw[line width = 1.5pt] (2) to [in=0, out=0,looseness = 3] (3.west);
	\draw[line width = 1.5pt] (6) to [in=0, out=0,looseness = 3] (9.west);
	\draw[line width = 1.5pt] (5) to [in=180, out=0,looseness = 1] (6.west);
	\draw[line width = 1.5pt] (7) to [in=180, out=0, looseness = 1] (8.west);
	\draw[line width = 1.5pt] (8) to [in=180, out=0, looseness = 1] (10);
	\draw[line width = 1.5pt] (11) to [in=180, out=0, looseness = 1] (12);
	\draw[line width = 1.5pt] (13) to [in=180, out=0, looseness = 1] (9.west);
	\draw[line width = 1.5pt] (14) to [in=180, out=0, looseness = 1] (11.west);
\end{scope}
}
      \end{tikzpicture}
\end{aligned}
}

\newcommand{\coassocr}[1]
{
  \begin{aligned}
\begin{tikzpicture}[circuit ee IEC, set resistor graphic=var resistor IEC
      graphic,scale=.5]
\scalebox{1}{
		\node  [circle,draw,inner sep=1pt,fill]   (0) at (1.125, 0) {};
		\node [style=none] (1) at (2.25, 0.5) {};
		\node [style=none] (2) at (2.25, -0.5) {};
		\node [style=none] (3) at (1.25, -1) {};
		\node [style=none] (4) at (-1, -0.5) {};
		\node [style=none] (5) at (1, 0) {};
		\node  [circle,draw,inner sep=1pt,fill]   (6) at (0, -0.5) {};
		\node [style=none] (7) at (2.25, -1) {};
		\draw[line width = 1.5pt] [in=180, out=60, looseness=1.20] (0.center) to (1.center);
		\draw[line width = 1.5pt] [in=180, out=-60, looseness=1.20] (0.center) to (2.center);
		\draw[line width = 1.5pt] (4.center) to (6);
		\draw[line width = 1.5pt] [in=180, out=60, looseness=1.20] (6) to (5.center);
		\draw[line width = 1.5pt] [in=180, out=-60, looseness=1.20] (6) to (3.center);
		\draw[line width = 1.5pt] (3.center) to (7.center);
}
      \end{tikzpicture}
  \end{aligned}
}

\newcommand{\unitl}[1]
{
  \begin{aligned}
\begin{tikzpicture}[circuit ee IEC, set resistor graphic=var resistor IEC
      graphic,scale=.5]
\scalebox{1}{
		\node [style=none] (0) at (1, -0) {};
		\node  [circle,draw,inner sep=1pt,fill]   (1) at (0.125, -0) {};
		\node  [circle,draw,inner sep=1pt,fill]   (2) at (-1, 0.5) {};
		\node [style=none] (3) at (-1, -0.5) {};
		\node [style=none] (4) at (-2, -0.5) {};
		\draw[line width = 1.5pt] (0.center) to (1.center);
		\draw[line width = 1.5pt] [in=0, out=120, looseness=1.20] (1.center) to (2.center);
		\draw[line width = 1.5pt] [in=0, out=-120, looseness=1.20] (1.center) to (3.center);
		\draw[line width = 1.5pt] (4.center) to (3.center);
}
\end{tikzpicture}
  \end{aligned}
}

\newcommand{\unitr}[1]
{
  \begin{aligned}
\begin{tikzpicture}[circuit ee IEC, set resistor graphic=var resistor IEC
      graphic,scale=.5]
\scalebox{1}{
		\node [style=none] (0) at (1, -0) {};
		\node  [circle,draw,inner sep=1pt,fill]   (1) at (0.125, -0) {};
		\node  [circle,draw,inner sep=1pt,fill]   (2) at (-1, -0.5) {};
		\node [style=none] (3) at (-1, 0.5) {};
		\node [style=none] (4) at (-2, 0.5) {};
		\draw[line width = 1.5pt] (0.center) to (1.center);
		\draw[line width = 1.5pt] [in=0, out=120, looseness=1.20] (1.center) to (3.center);
		\draw[line width = 1.5pt] [in=0, out=-120, looseness=1.20] (1.center) to (2.center);
		\draw[line width = 1.5pt] (4.center) to (3.center);
}
\end{tikzpicture}
  \end{aligned}
}

\newcommand{\parunitl}[1]
{
\begin{aligned}
\begin{tikzpicture}[circuit ee IEC, set resistor graphic=var resistor IEC
      graphic,scale=.5]
\scalebox{1}{
		\node [style=none] (0) at (1, -0) {};
		\node  [circle,draw,inner sep=1pt,fill]   (1) at (0.125, -0) {};
		\node [style=none] (2) at (-1, 0.5) {};
		\node [style=none] (3) at (-1, -0.5) {}; 
		\node [style=none] (4) at (1, -.5) {};
		\node  [circle,draw,inner sep=1pt,fill]   (5) at (0.125, -.5) {};
		\node [style=none] (6) at (-1, -0) {};
		\node [style=none] (7) at (-1, -1) {}; 
		\node  [circle,draw,inner sep=1pt,fill]   (8) at (-1.5, -1) {}; 
		\node  [circle,draw,inner sep=1pt,fill]   (9) at (-1.5, -0.5) {}; 
		\node [style=none] (10) at (-2, .5) {};
		\node [style=none] (11) at (-2, 0) {}; 
		\draw[line width = 1.5pt] (0.center) to (1.center);
		\draw[line width = 1.5pt] [in=0, out=120, looseness=1.20] (1.center) to (2.center);
		\draw[line width = 1.5pt] [in=0, out=-120, looseness=1.20] (1.center) to (3.center);
		\draw[line width = 1.5pt] (4.center) to (5.center);
		\draw[line width = 1.5pt] [in=0, out=120, looseness=1.20] (5.center) to (6.center);
		\draw[line width = 1.5pt] [in=0, out=-120, looseness=1.20] (5.center) to (7.center);
		\draw[line width = 1.5pt] (7.east) to (8);
		\draw[line width = 1.5pt] (3.east) to (9);
		\draw[line width = 1.5pt] (2.east) to (10);
		\draw[line width = 1.5pt] (6.east) to (11);
}
      \end{tikzpicture}
\end{aligned}
}

\newcommand{\parcounitl}[1]
{
\begin{aligned}
\begin{tikzpicture}[circuit ee IEC, set resistor graphic=var resistor IEC
      graphic,scale=.5]
\scalebox{1}{
\begin{scope}[rotate=180]
		\node [style=none] (0) at (1, -0) {};
		\node  [circle,draw,inner sep=1pt,fill]   (1) at (0.125, -0) {};
		\node [style=none] (2) at (-1, 0.5) {};
		\node [style=none] (3) at (-1, -0.5) {}; 
		\node [style=none] (4) at (1, -.5) {};
		\node  [circle,draw,inner sep=1pt,fill]   (5) at (0.125, -.5) {};
		\node [style=none] (6) at (-1, -0) {};
		\node [style=none] (7) at (-1, -1) {}; 
		\node  [circle,draw,inner sep=1pt,fill]   (8) at (-1.5, -1) {}; 
		\node  [circle,draw,inner sep=1pt,fill]   (9) at (-1.5, -0.5) {}; 
		\node [style=none] (10) at (-2, .5) {};
		\node [style=none] (11) at (-2, 0) {}; 
		\draw[line width = 1.5pt] (0.center) to (1.center);
		\draw[line width = 1.5pt] [in=0, out=120, looseness=1.20] (1.center) to (2.center);
		\draw[line width = 1.5pt] [in=0, out=-120, looseness=1.20] (1.center) to (3.center);
		\draw[line width = 1.5pt] (4.center) to (5.center);
		\draw[line width = 1.5pt] [in=0, out=120, looseness=1.20] (5.center) to (6.center);
		\draw[line width = 1.5pt] [in=0, out=-120, looseness=1.20] (5.center) to (7.center);
		\draw[line width = 1.5pt] (7.west) to (8);
		\draw[line width = 1.5pt] (3.west) to (9);
		\draw[line width = 1.5pt] (2.west) to (10);
		\draw[line width = 1.5pt] (6.west) to (11);
\end{scope}
}
      \end{tikzpicture}
\end{aligned}
}

\newcommand{\parunitr}[1]
{
\begin{aligned}
\begin{tikzpicture}[circuit ee IEC, set resistor graphic=var resistor IEC
      graphic,scale=.5]
\scalebox{1}{
\begin{scope}[yscale=-1]
		\node [style=none] (0) at (1, -0) {};
		\node  [circle,draw,inner sep=1pt,fill]   (1) at (0.125, -0) {};
		\node [style=none] (2) at (-1, 0.5) {};
		\node [style=none] (3) at (-1, -0.5) {}; 
		\node [style=none] (4) at (1, -.5) {};
		\node  [circle,draw,inner sep=1pt,fill]   (5) at (0.125, -.5) {};
		\node [style=none] (6) at (-1, -0) {};
		\node [style=none] (7) at (-1, -1) {}; 
		\node  [circle,draw,inner sep=1pt,fill]   (8) at (-1.5, -1) {}; 
		\node  [circle,draw,inner sep=1pt,fill]   (9) at (-1.5, -0.5) {}; 
		\node [style=none] (10) at (-2, .5) {};
		\node [style=none] (11) at (-2, 0) {}; 
		\draw[line width = 1.5pt] (0.center) to (1.center);
		\draw[line width = 1.5pt] [in=0, out=120, looseness=1.20] (1.center) to (2.center);
		\draw[line width = 1.5pt] [in=0, out=-120, looseness=1.20] (1.center) to (3.center);
		\draw[line width = 1.5pt] (4.center) to (5.center);
		\draw[line width = 1.5pt] [in=0, out=120, looseness=1.20] (5.center) to (6.center);
		\draw[line width = 1.5pt] [in=0, out=-120, looseness=1.20] (5.center) to (7.center);
		\draw[line width = 1.5pt] (7.east) to (8);
		\draw[line width = 1.5pt] (3.east) to (9);
		\draw[line width = 1.5pt] (2.east) to (10);
		\draw[line width = 1.5pt] (6.east) to (11);
\end{scope}
}
      \end{tikzpicture}
\end{aligned}
}

\newcommand{\parcounitr}[1]
{
\begin{aligned}
\begin{tikzpicture}[circuit ee IEC, set resistor graphic=var resistor IEC
      graphic,scale=.5]
\scalebox{1}{
\begin{scope}[rotate=180]
\begin{scope}[yscale=-1]
		\node [style=none] (0) at (1, -0) {};
		\node  [circle,draw,inner sep=1pt,fill]   (1) at (0.125, -0) {};
		\node [style=none] (2) at (-1, 0.5) {};
		\node [style=none] (3) at (-1, -0.5) {}; 
		\node [style=none] (4) at (1, -.5) {};
		\node  [circle,draw,inner sep=1pt,fill]   (5) at (0.125, -.5) {};
		\node [style=none] (6) at (-1, -0) {};
		\node [style=none] (7) at (-1, -1) {}; 
		\node  [circle,draw,inner sep=1pt,fill]   (8) at (-1.5, -1) {}; 
		\node  [circle,draw,inner sep=1pt,fill]   (9) at (-1.5, -0.5) {}; 
		\node [style=none] (10) at (-2, .5) {};
		\node [style=none] (11) at (-2, 0) {}; 
		\draw[line width = 1.5pt] (0.center) to (1.center);
		\draw[line width = 1.5pt] [in=0, out=120, looseness=1.20] (1.center) to (2.center);
		\draw[line width = 1.5pt] [in=0, out=-120, looseness=1.20] (1.center) to (3.center);
		\draw[line width = 1.5pt] (4.center) to (5.center);
		\draw[line width = 1.5pt] [in=0, out=120, looseness=1.20] (5.center) to (6.center);
		\draw[line width = 1.5pt] [in=0, out=-120, looseness=1.20] (5.center) to (7.center);
		\draw[line width = 1.5pt] (7.west) to (8);
		\draw[line width = 1.5pt] (3.west) to (9);
		\draw[line width = 1.5pt] (2.west) to (10);
		\draw[line width = 1.5pt] (6.west) to (11);
\end{scope}
\end{scope}
}
      \end{tikzpicture}
\end{aligned}
}

\newcommand{\monadunitl}[1]
{
\begin{aligned}
\begin{tikzpicture}[circuit ee IEC, set resistor graphic=var resistor IEC
      graphic,scale=.5]
\scalebox{1}{
		\node [style=none] (0) at (1, -0) {};
		\node [style=none] (1) at (0.125, -0) {};
		\node [style=none] (2) at (-1, 0.5) {};
		\node [style=none] (3) at (-1, -0.5) {};
		\node [style=none] (4) at (1, -0) {};
		\node [style=none] (5) at (-1, 1) {};
		\node [style=none] (6) at (1, .25) {};
		\node [style=none] (7) at (-1, -1) {};
		\node [style=none] (8) at (1, -.25) {};
		\node [style=none] (9) at (-2, 1) {};
		\node [style=none] (10) at (-2, .5) {};
	\draw[line width = 1.5pt] (2) to [in=0, out=0,looseness = 3] (3);
	\draw[line width = 1.5pt] (3.east) to [in=180, out=180,looseness = 4] (7.east);
	\draw[line width = 1.5pt] (5) to [in=180, out=0,looseness = 1] (6);
	\draw[line width = 1.5pt] (7) to [in=180, out=0, looseness = 1] (8);
	\draw[line width = 1.5pt] (5.east) to  (9);
	\draw[line width = 1.5pt] (2.east) to  (10);
}
      \end{tikzpicture}
\end{aligned}
}

\newcommand{\monadcounitl}[1]
{
\begin{aligned}
\begin{tikzpicture}[circuit ee IEC, set resistor graphic=var resistor IEC
      graphic,scale=.5]
\scalebox{1}{
\begin{scope}[rotate=180]
		\node [style=none] (0) at (1, -0) {};
		\node [style=none] (1) at (0.125, -0) {};
		\node [style=none] (2) at (-1, 0.5) {};
		\node [style=none] (3) at (-1, -0.5) {};
		\node [style=none] (4) at (1, -0) {};
		\node [style=none] (5) at (-1, 1) {};
		\node [style=none] (6) at (1, .25) {};
		\node [style=none] (7) at (-1, -1) {};
		\node [style=none] (8) at (1, -.25) {};
		\node [style=none] (9) at (-2, 1) {};
		\node [style=none] (10) at (-2, .5) {};
	\draw[line width = 1.5pt] (2) to [in=0, out=0,looseness = 3] (3);
	\draw[line width = 1.5pt] (3.west) to [in=180, out=180,looseness = 4] (7.west);
	\draw[line width = 1.5pt] (5) to [in=180, out=0,looseness = 1] (6);
	\draw[line width = 1.5pt] (7) to [in=180, out=0, looseness = 1] (8);
	\draw[line width = 1.5pt] (5.west) to  (9);
	\draw[line width = 1.5pt] (2.west) to  (10);
\end{scope}
}
      \end{tikzpicture}
\end{aligned}
}

\newcommand{\monadunitr}[1]
{
\begin{aligned}
\begin{tikzpicture}[circuit ee IEC, set resistor graphic=var resistor IEC
      graphic,scale=.5]
\scalebox{1}{
\begin{scope}[yscale=-1]
		\node [style=none] (0) at (1, -0) {};
		\node [style=none] (1) at (0.125, -0) {};
		\node [style=none] (2) at (-1, 0.5) {};
		\node [style=none] (3) at (-1, -0.5) {};
		\node [style=none] (4) at (1, -0) {};
		\node [style=none] (5) at (-1, 1) {};
		\node [style=none] (6) at (1, .25) {};
		\node [style=none] (7) at (-1, -1) {};
		\node [style=none] (8) at (1, -.25) {};
		\node [style=none] (9) at (-2, 1) {};
		\node [style=none] (10) at (-2, .5) {};
	\draw[line width = 1.5pt] (2) to [in=0, out=0,looseness = 3] (3);
	\draw[line width = 1.5pt] (3.east) to [in=180, out=180,looseness = 4] (7.east);
	\draw[line width = 1.5pt] (5) to [in=180, out=0,looseness = 1] (6);
	\draw[line width = 1.5pt] (7) to [in=180, out=0, looseness = 1] (8);
	\draw[line width = 1.5pt] (5.east) to  (9);
	\draw[line width = 1.5pt] (2.east) to  (10);
\end{scope}
}
      \end{tikzpicture}
\end{aligned}
}

\newcommand{\monadcounitr}[1]
{
\begin{aligned}
\begin{tikzpicture}[circuit ee IEC, set resistor graphic=var resistor IEC
      graphic,scale=.5]
\scalebox{1}{
\begin{scope}[yscale=-1, rotate=180]
		\node [style=none] (0) at (1, -0) {};
		\node [style=none] (1) at (0.125, -0) {};
		\node [style=none] (2) at (-1, 0.5) {};
		\node [style=none] (3) at (-1, -0.5) {};
		\node [style=none] (4) at (1, -0) {};
		\node [style=none] (5) at (-1, 1) {};
		\node [style=none] (6) at (1, .25) {};
		\node [style=none] (7) at (-1, -1) {};
		\node [style=none] (8) at (1, -.25) {};
		\node [style=none] (9) at (-2, 1) {};
		\node [style=none] (10) at (-2, .5) {};
	\draw[line width = 1.5pt] (2) to [in=0, out=0,looseness = 3] (3);
	\draw[line width = 1.5pt] (3.west) to [in=180, out=180,looseness = 4] (7.west);
	\draw[line width = 1.5pt] (5) to [in=180, out=0,looseness = 1] (6);
	\draw[line width = 1.5pt] (7) to [in=180, out=0, looseness = 1] (8);
	\draw[line width = 1.5pt] (5.west) to  (9);
	\draw[line width = 1.5pt] (2.west) to  (10);
\end{scope}
}
      \end{tikzpicture}
\end{aligned}
}

\newcommand{\counitl}[1]
{
  \begin{aligned}
\begin{tikzpicture}[circuit ee IEC, set resistor graphic=var resistor IEC
      graphic,scale=.5]
\scalebox{1}{
		\node [style=none] (0) at (-2, -0) {};
		\node  [circle,draw,inner sep=1pt,fill]   (1) at (-1.125, -0) {};
		\node  [circle,draw,inner sep=1pt,fill]   (2) at (0, 0.5) {};
		\node [style=none] (3) at (0, -0.5) {};
		\node [style=none] (4) at (1, -0.5) {};
		\draw[line width = 1.5pt] (0.center) to (1.center);
		\draw[line width = 1.5pt] [in=180, out=60, looseness=1.20] (1.center) to (2.center);
		\draw[line width = 1.5pt] [in=180, out=-60, looseness=1.20] (1.center) to (3.center);
		\draw[line width = 1.5pt] (4.center) to (3.center);
}
\end{tikzpicture}
  \end{aligned}
}

\newcommand{\counitr}[1]
{
  \begin{aligned}
\begin{tikzpicture}[circuit ee IEC, set resistor graphic=var resistor IEC
      graphic,scale=.5]
\scalebox{1}{
		\node [style=none] (0) at (-2, -0) {};
		\node  [circle,draw,inner sep=1pt,fill]   (1) at (-1.125, -0) {};
		\node  [circle,draw,inner sep=1pt,fill]   (2) at (0, -0.5) {};
		\node [style=none] (3) at (0, 0.5) {};
		\node [style=none] (4) at (1, 0.5) {};
		\draw[line width = 1.5pt] (0.center) to (1.center);
		\draw[line width = 1.5pt] [in=180, out=60, looseness=1.20] (1.center) to (3.center);
		\draw[line width = 1.5pt] [in=180, out=-60, looseness=1.20] (1.center) to (2.center);
		\draw[line width = 1.5pt] (4.center) to (3.center);
}
\end{tikzpicture}
  \end{aligned}
}

\newcommand{\commute}[1]
{
  \begin{aligned}
\begin{tikzpicture}[circuit ee IEC, set resistor graphic=var resistor IEC
      graphic,scale=.5]
\scalebox{1}{
		\node [style=none] (0) at (1.25, -0) {};
		\node  [circle,draw,inner sep=1pt,fill]   (1) at (0.375, -0) {};
		\node [style=none] (2) at (-0.5, -0.5) {};
		\node [style=none] (3) at (-2, 0.5) {};
		\node [style=none] (4) at (-0.5, 0.5) {};
		\node [style=none] (5) at (-2, -0.5) {};
		\draw[line width = 1.5pt] (0.center) to (1.center);
		\draw[line width = 1.5pt] [in=0, out=-120, looseness=1.20] (1.center) to (2.center);
		\draw[line width = 1.5pt] [in=180, out=0, looseness=1.00] (3.center) to (2.center);
		\draw[line width = 1.5pt] [in=0, out=120, looseness=1.20] (1.center) to (4.center);
		\draw[line width = 1.5pt] [in=0, out=180, looseness=1.00] (4.center) to (5.center);
}
\end{tikzpicture}
  \end{aligned}
}

\newcommand{\cocommute}[1]
{
  \begin{aligned}
\begin{tikzpicture}[circuit ee IEC, set resistor graphic=var resistor IEC
      graphic,scale=.5]
\scalebox{1}{
		\node [style=none] (0) at (-2, -0) {};
		\node  [circle,draw,inner sep=1pt,fill]   (1) at (-1.125, -0) {};
		\node [style=none] (2) at (-0.25, -0.5) {};
		\node [style=none] (3) at (1.25, 0.5) {};
		\node [style=none] (4) at (-0.25, 0.5) {};
		\node [style=none] (5) at (1.25, -0.5) {};
		\draw[line width = 1.5pt] (0.center) to (1.center);
		\draw[line width = 1.5pt] [in=180, out=-60, looseness=1.20] (1.center) to (2.center);
		\draw[line width = 1.5pt] [in=0, out=180, looseness=1.00] (3.center) to (2.center);
		\draw[line width = 1.5pt] [in=180, out=60, looseness=1.20] (1.center) to (4.center);
		\draw[line width = 1.5pt] [in=180, out=0, looseness=1.00] (4.center) to (5.center);
}
\end{tikzpicture}
  \end{aligned}
}

\newcommand{\symmetric}[1]
{
  \begin{aligned}
\begin{tikzpicture}[circuit ee IEC, set resistor graphic=var resistor IEC
      graphic,scale=.5]
\scalebox{1}{
		\node [style=none] (0) at (1.25, -0) {};
		\node  [circle,draw,inner sep=1pt,fill]   (1) at (0.375, -0) {};
		\node [style=none] (2) at (-0.5, -0.5) {};
		\node [style=none] (3) at (-2, 0.5) {};
		\node [style=none] (4) at (-0.5, 0.5) {};
		\node [style=none] (5) at (-2, -0.5) {};
		\node  [circle,draw,inner sep=1pt,fill]   (6) at (1.25, -0) {};
		\draw[line width = 1.5pt] (0.center) to (1.center);
		\draw[line width = 1.5pt] [in=0, out=-120, looseness=1.20] (1.center) to (2.center);
		\draw[line width = 1.5pt] [in=180, out=0, looseness=1.00] (3.center) to (2.center);
		\draw[line width = 1.5pt] [in=0, out=120, looseness=1.20] (1.center) to (4.center);
		\draw[line width = 1.5pt] [in=0, out=180, looseness=1.00] (4.center) to (5.center);
}
\end{tikzpicture}
  \end{aligned}
}

\newcommand{\monadsymml}[1]
{
\begin{aligned}
\begin{tikzpicture}[circuit ee IEC, set resistor graphic=var resistor IEC
      graphic,scale=.5]
\scalebox{1}{
		\node [style=none] (0) at (1, -0) {};
		\node [style=none] (1) at (0.125, -0) {};
		\node [style=none] (2) at (-1, 0.5) {};
		\node [style=none] (3) at (-1, -0.5) {};
		\node [style=none] (4) at (1, -0) {};
		\node [style=none] (5) at (-1, 1) {};
		\node [style=none] (6) at (1, .25) {};
		\node [style=none] (7) at (-1, -1) {};
		\node [style=none] (8) at (1, -.25) {};
		\node [style=none] (10) at (-3, 1) {};
		\node [style=none] (11) at (-3, .5) {};
		\node [style=none] (12) at (-3, -.5) {};
		\node [style=none] (13) at (-3, -1) {};
	\draw[line width = 1.5pt] (2) to [in=0, out=0,looseness = 3] (3);
	\draw[line width = 1.5pt] (6.west) to [in=0, out=0,looseness = 3] (8.west);
	\draw[line width = 1.5pt] (5) to [in=180, out=0,looseness = 1] (6);
	\draw[line width = 1.5pt] (7) to [in=180, out=0, looseness = 1] (8);
	\draw[line width = 1.5pt] (5.east) to  [in=0, out=180, looseness = .5]   (12);
	\draw[line width = 1.5pt] (2.east) to [in=0, out=180, looseness = .5]  (13);
	\draw[line width = 1.5pt] (7.east) to  [in=0, out=180, looseness = .5] (11);
	\draw[line width = 1.5pt] (3.east) to [in=0, out=180, looseness = .5]  (10);
}
      \end{tikzpicture}
\end{aligned}
}

\newcommand{\monadsymmm}[1]
{
\begin{aligned}
\begin{tikzpicture}[circuit ee IEC, set resistor graphic=var resistor IEC
      graphic,scale=.5]
\scalebox{1}{
		\node [style=none] (0) at (1, -0) {};
		\node [style=none] (1) at (0.125, -0) {};
		\node [style=none] (2) at (-1, 0.5) {};
		\node [style=none] (3) at (-1, -0.5) {};
		\node [style=none] (4) at (1, -0) {};
		\node [style=none] (5) at (-1, 1) {};
		\node [style=none] (6) at (1, .25) {};
		\node [style=none] (7) at (-1, -1) {};
		\node [style=none] (8) at (1, -.25) {};
		\node [style=none] (10) at (-3, 1) {};
		\node [style=none] (11) at (-3, .5) {};
		\node [style=none] (12) at (-3, -.5) {};
		\node [style=none] (13) at (-3, -1) {};
	\draw[line width = 1.5pt] (2) to [in=0, out=0,looseness = 3] (3);
	\draw[line width = 1.5pt] (6.west) to [in=0, out=0,looseness = 3] (8.west);
	\draw[line width = 1.5pt] (5) to [in=180, out=0,looseness = 1] (6);
	\draw[line width = 1.5pt] (7) to [in=180, out=0, looseness = 1] (8);
	\draw[line width = 1.5pt] (5.east) to  [in=0, out=180, looseness = .5]   (12);
	\draw[line width = 1.5pt] (2.east) to [in=0, out=180, looseness = 1]  (10);
	\draw[line width = 1.5pt] (7.east) to  [in=0, out=180, looseness = .5] (11);
	\draw[line width = 1.5pt] (3.east) to [in=0, out=180, looseness = 1]  (13);
}
      \end{tikzpicture}
\end{aligned}
}

\newcommand{\monadsymmr}[1]
{
\begin{aligned}
\begin{tikzpicture}[circuit ee IEC, set resistor graphic=var resistor IEC
      graphic,scale=.5]
\scalebox{1}{
		\node [style=none] (0) at (1, -0) {};
		\node [style=none] (1) at (0.125, -0) {};
		\node [style=none] (2) at (-1, 0.5) {};
		\node [style=none] (3) at (-1, -0.5) {};
		\node [style=none] (4) at (1, -0) {};
		\node [style=none] (5) at (-1, 1) {};
		\node [style=none] (6) at (1, .25) {};
		\node [style=none] (7) at (-1, -1) {};
		\node [style=none] (8) at (1, -.25) {};
		\node [style=none] (10) at (-3, 1) {};
		\node [style=none] (11) at (-3, .5) {};
		\node [style=none] (12) at (-3, -.5) {};
		\node [style=none] (13) at (-3, -1) {};
	\draw[line width = 1.5pt] (2) to [in=0, out=0,looseness = 3] (3);
	\draw[line width = 1.5pt] (6.west) to [in=0, out=0,looseness = 3] (8.west);
	\draw[line width = 1.5pt] (5) to [in=180, out=0,looseness = 1] (6);
	\draw[line width = 1.5pt] (7) to [in=180, out=0, looseness = 1] (8);
	\draw[line width = 1.5pt] (5.east) to  [in=0, out=180, looseness = 1]   (11);
	\draw[line width = 1.5pt] (2.east) to [in=0, out=180, looseness = 1]  (10);
	\draw[line width = 1.5pt] (7.east) to  [in=0, out=180, looseness = 1] (12);
	\draw[line width = 1.5pt] (3.east) to [in=0, out=180, looseness = 1]  (13);
}
      \end{tikzpicture}
\end{aligned}
}

\newcommand{\monadsymmrr}[1]
{
\begin{aligned}
\begin{tikzpicture}[circuit ee IEC, set resistor graphic=var resistor IEC
      graphic,scale=.5]
\scalebox{1}{
		\node [style=none] (0) at (1, -0) {};
		\node [style=none] (1) at (0.125, -0) {};
		\node [style=none] (2) at (-1, 0.5) {};
		\node [style=none] (3) at (-1, -0.5) {};
		\node [style=none] (4) at (1, -0) {};
		\node [style=none] (5) at (-1, 1) {};
		\node [style=none] (6) at (1, .25) {};
		\node [style=none] (7) at (-1, -1) {};
		\node [style=none] (8) at (1, -.25) {};
		\node [style=none] (10) at (-3, 1) {};
		\node [style=none] (11) at (-3, .5) {};
		\node [style=none] (12) at (-3, -.5) {};
		\node [style=none] (13) at (-3, -1) {};
	\draw[line width = 1.5pt] (10) to [in=0, out=0,looseness = 4] (13);
	\draw[line width = 1.5pt] (5) to [in=0, out=0,looseness = 2] (7);
	\draw[line width = 1.5pt] (5.east) to  [in=0, out=180, looseness = 1.5]   (11);
	\draw[line width = 1.5pt] (7.east) to  [in=0, out=180, looseness = 1.5] (12);
}
      \end{tikzpicture}
\end{aligned}
}

\newcommand{\monadsymmcups}[1]
{
\begin{aligned}
\begin{tikzpicture}[circuit ee IEC, set resistor graphic=var resistor IEC
      graphic,scale=.5]
\scalebox{1}{
		\node [style=none] (-1) at (-4.5, -0) {};
		\node [style=none] (0) at (1, -0) {};
		\node [style=none] (1) at (0.125, -0) {};
		\node [style=none] (2) at (-1, 0.5) {};
		\node [style=none] (3) at (-1, -0.5) {};
		\node [style=none] (4) at (1, -0) {};
		\node [style=none] (5) at (-1, 1) {};
		\node [style=none] (6) at (1, .25) {};
		\node [style=none] (7) at (-1, -1) {};
		\node [style=none] (8) at (1, -.25) {};
		\node [style=none] (10) at (-3, 1) {};
		\node [style=none] (11) at (-3, .5) {};
		\node [style=none] (12) at (-3, -.5) {};
		\node [style=none] (13) at (-3, -1) {};
	\draw[line width = 1.5pt] (10) to [in=0, out=0,looseness = 4] (13);
	\draw[line width = 1.5pt] (11) to [in=0, out=0,looseness = 5] (12);
}
      \end{tikzpicture}
\end{aligned}
}

\newcommand{\monadsymmend}[1]
{
\begin{aligned}
\begin{tikzpicture}[circuit ee IEC, set resistor graphic=var resistor IEC
      graphic,scale=.5]
\scalebox{1}{
\begin{scope}[shift={(-1,0)}]
		\node [style=none] (-1) at (2.5, -0) {};
		\node [style=none] (0) at (1, -0) {};
		\node [style=none] (1) at (0.125, -0) {};
		\node [style=none] (2) at (-1, 0.5) {};
		\node [style=none] (3) at (-1, -0.5) {};
		\node [style=none] (4) at (1, -0) {};
		\node [style=none] (5) at (-1, 1) {};
		\node [style=none] (6) at (1, .25) {};
		\node [style=none] (7) at (-1, -1) {};
		\node [style=none] (8) at (1, -.25) {};
		\node [style=none] (10) at (-3, 1) {};
		\node [style=none] (11) at (-3, .5) {};
	\draw[line width = 1.5pt] (2) to [in=0, out=0,looseness = 3] (3);
	\draw[line width = 1.5pt] (6.west) to [in=0, out=0,looseness = 3] (8.west);
	\draw[line width = 1.5pt] (5) to [in=180, out=0,looseness = 1] (6);
	\draw[line width = 1.5pt] (7) to [in=180, out=0, looseness = 1] (8);
\end{scope}
}
      \end{tikzpicture}
\end{aligned}
}

\newcommand{\weirdthing}[1]
{
\begin{aligned}
\begin{tikzpicture}[circuit ee IEC, set resistor graphic=var resistor IEC
      graphic,scale=.5]
\scalebox{1}{
\begin{scope}[shift={(-1,0)}]
		\node [style=none] (-1) at (2.5, -0) {};
		\node [style=none] (0) at (1, -0) {};
		\node [style=none] (1) at (0.125, -0) {};
		\node [style=none] (2) at (-1, 0.5) {};
		\node [style=none] (3) at (-1, -0.5) {};
		\node [style=none] (4) at (1, -0) {};
		\node [style=none] (5) at (-1, 1) {};
		\node [style=none] (6) at (1, .25) {};
		\node [style=none] (7) at (-1, -1) {};
		\node [style=none] (8) at (1, -.25) {};
		\node [style=none] (10) at (-3, 1) {};
		\node [style=none] (11) at (-3, .5) {};
	\draw[line width = 1.5pt] (2) to [in=0, out=0,looseness = 3] (3);
	\draw[line width = 1.5pt] (6.west) to [in=0, out=0,looseness = 3] (8.west);
	\draw[line width = 1.5pt] (5) to [in=180, out=0,looseness = 1] (6);
	\draw[line width = 1.5pt] (7) to [in=180, out=0, looseness = 1] (8);
\end{scope}
\begin{scope}[rotate=180]
\begin{scope}[shift={(-3,0)}]
		\node [style=none] (-1) at (2.5, -0) {};
		\node [style=none] (0) at (1, -0) {};
		\node [style=none] (1) at (0.125, -0) {};
		\node [style=none] (2) at (-1, 0.5) {};
		\node [style=none] (3) at (-1, -0.5) {};
		\node [style=none] (4) at (1, -0) {};
		\node [style=none] (5) at (-1, 1) {};
		\node [style=none] (6) at (1, .25) {};
		\node [style=none] (7) at (-1, -1) {};
		\node [style=none] (8) at (1, -.25) {};
		\node [style=none] (10) at (-3, 1) {};
		\node [style=none] (11) at (-3, .5) {};
	\draw[line width = 1.5pt] (2) to [in=0, out=0,looseness = 3] (3);
	\draw[line width = 1.5pt] (6.west) to [in=0, out=0,looseness = 3] (8.west);
	\draw[line width = 1.5pt] (5) to [in=180, out=0,looseness = 1] (6);
	\draw[line width = 1.5pt] (7) to [in=180, out=0, looseness = 1] (8);
\end{scope}
\end{scope}
}
      \end{tikzpicture}
\end{aligned}
}

\newcommand{\multunit}[1]
{
\begin{aligned}
\begin{tikzpicture}[circuit ee IEC, set resistor graphic=var resistor IEC
      graphic,scale=.5]
\scalebox{1}{
		\node [style=none] (0) at (1, -0) {};
		\node  [circle,draw,inner sep=1pt,fill]   (1) at (0.125, -0) {};
		\node [style=none] (2) at (-1, 0.5) {};
		\node [style=none] (3) at (-1, -0.5) {};
		\node  [circle,draw,inner sep=1pt,fill]   (4) at (1, -0) {};
		\draw[line width = 1.5pt] (0.center) to (1.center);
		\draw[line width = 1.5pt] [in=0, out=120, looseness=1.20] (1.center) to (2.center);
		\draw[line width = 1.5pt] [in=0, out=-120, looseness=1.20] (1.center) to (3.center);
}
      \end{tikzpicture}
\end{aligned}
}

\newcommand{\cuptwo}[1]
{
\begin{aligned}
\begin{tikzpicture}[circuit ee IEC, set resistor graphic=var resistor IEC
      graphic,scale=.5]
\scalebox{1}{
		\node [style=none] (0) at (1, -0) {};
		\node [style=none] (1) at (0.125, -0) {};
		\node [style=none] (2) at (-1, 0.5) {};
		\node [style=none] (3) at (-1, -0.5) {};
		\node [style=none] (4) at (1, -0) {};
	\draw[line width = 1.5pt] (2) to [in=0, out=0,looseness = 4] (3);
}
      \end{tikzpicture}
\end{aligned}
}

\newcommand{\monadmult}[1]
{
\begin{aligned}
\begin{tikzpicture}[circuit ee IEC, set resistor graphic=var resistor IEC
      graphic,scale=.5]
\scalebox{1}{
		\node [style=none] (0) at (1, -0) {};
		\node [style=none] (1) at (0.125, -0) {};
		\node [style=none] (2) at (-1, 0.5) {};
		\node [style=none] (3) at (-1, -0.5) {};
		\node [style=none] (4) at (1, -0) {};
		\node [style=none] (5) at (-1, 1) {};
		\node [style=none] (6) at (1, .25) {};
		\node [style=none] (7) at (-1, -1) {};
		\node [style=none] (8) at (1, -.25) {};
	\draw[line width = 1.5pt] (2) to [in=0, out=0,looseness = 3] (3);
	\draw[line width = 1.5pt] (5) to [in=180, out=0,looseness = 1] (6);
	\draw[line width = 1.5pt] (7) to [in=180, out=0, looseness = 1] (8);
}
      \end{tikzpicture}
\end{aligned}
}

\newcommand{\monadcomult}[1]
{
\begin{aligned}
\begin{tikzpicture}[circuit ee IEC, set resistor graphic=var resistor IEC
      graphic,scale=.5]
\scalebox{1}{
\begin{scope}[rotate=180]
		\node [style=none] (0) at (1, -0) {};
		\node [style=none] (1) at (0.125, -0) {};
		\node [style=none] (2) at (-1, 0.5) {};
		\node [style=none] (3) at (-1, -0.5) {};
		\node [style=none] (4) at (1, -0) {};
		\node [style=none] (5) at (-1, 1) {};
		\node [style=none] (6) at (1, .25) {};
		\node [style=none] (7) at (-1, -1) {};
		\node [style=none] (8) at (1, -.25) {};
	\draw[line width = 1.5pt] (2) to [in=0, out=0,looseness = 3] (3.west);
	\draw[line width = 1.5pt] (5) to [in=180, out=0,looseness = 1] (6);
	\draw[line width = 1.5pt] (7) to [in=180, out=0, looseness = 1] (8);
\end{scope}
}
      \end{tikzpicture}
\end{aligned}
}

\newcommand{\cosymmetric}[1]
{
  \begin{aligned}
\begin{tikzpicture}[circuit ee IEC, set resistor graphic=var resistor IEC
      graphic,scale=.5]
\scalebox{1}{
		\node [style=none] (0) at (-2, -0) {};
		\node  [circle,draw,inner sep=1pt,fill]   (1) at (-1.125, -0) {};
		\node [style=none] (2) at (-0.25, -0.5) {};
		\node [style=none] (3) at (1.25, 0.5) {};
		\node [style=none] (4) at (-0.25, 0.5) {};
		\node [style=none] (5) at (1.25, -0.5) {};
		\node  [circle,draw,inner sep=1pt,fill]   (6) at (-2, -0) {};
		\draw[line width = 1.5pt] (0.center) to (1.center);
		\draw[line width = 1.5pt] [in=180, out=-60, looseness=1.20] (1.center) to (2.center);
		\draw[line width = 1.5pt] [in=0, out=180, looseness=1.00] (3.center) to (2.center);
		\draw[line width = 1.5pt] [in=180, out=60, looseness=1.20] (1.center) to (4.center);
		\draw[line width = 1.5pt] [in=180, out=0, looseness=1.00] (4.center) to (5.center);
}
\end{tikzpicture}
  \end{aligned}
}

\newcommand{\comultcounit}[1]
{
\begin{aligned}
\begin{tikzpicture}
[circuit ee IEC, set resistor graphic=var resistor IEC
      graphic,scale=.5]
\scalebox{1}{
		\node [style=none] (0) at (-1, -0) {};
		\node  [circle,draw,inner sep=1pt,fill]   (1) at (-0.125, -0) {};
		\node [style=none] (2) at (1, 0.5) {};
		\node [style=none] (3) at (1, -0.5) {};
		\node  [circle,draw,inner sep=1pt,fill]   (4) at (-1, -0) {};
		\draw[line width = 1.5pt] (0.center) to (1.center);
		\draw[line width = 1.5pt] [in=180, out=60, looseness=1.20] (1.center) to (2.center);
		\draw[line width = 1.5pt] [in=180, out=-60, looseness=1.20] (1.center) to (3.center);
}
      \end{tikzpicture}
\end{aligned}
}

\newcommand{\captwo}[1]
{
\begin{aligned}
\begin{tikzpicture}
[circuit ee IEC, set resistor graphic=var resistor IEC
      graphic,scale=.5]
\scalebox{1}{
		\node [style=none] (0) at (-.75, -0) {};
		\node [style=none] (1) at (-0.125, -0) {};
		\node [style=none] (2) at (1, 0.5) {};
		\node [style=none] (3) at (1, -0.5) {};
		\node [style=none] (4) at (-1, -0) {};
		\node [style=none] (5) at (-0.5, -.5) {};
		\node [style=none] (6) at (-0.5, .5) {};
	\draw[line width = 1.5pt] (2) to [in=180, out=180,looseness = 4] (3);
}
      \end{tikzpicture}
\end{aligned}
}

\newcommand{\capn}[1]
{
\begin{aligned}
\begin{tikzpicture}
[circuit ee IEC, set resistor graphic=var resistor IEC
      graphic,scale=.5]
\scalebox{1}{
		\node [style=none] (0) at (-.75, -0) {};
		\node [style=none] (1) at (-0.125, -0) {};
		\node [style=none] (2) at (1, 0.5) {};
		\node [style=none] (3) at (1, -0.5) {};
		\node [style=none] (4) at (-1, -0) {};
		\node [style=none] (5) at (-0.5, -.5) {};
		\node [style=none] (6) at (-0.5, .5) {};
		\node [style=none] (7) at (1, 1) {};
		\node [style=none] (8) at (1, -1) {};
	\draw[line width = 1.5pt] (2) to [in=180, out=180,looseness = 4] (3);
	\draw[line width = 1.5pt] (7) to [in=180, out=180,looseness = 3] (8);
}
      \end{tikzpicture}
\end{aligned}
}

\newcommand{\zigzag}[1]
{
\begin{aligned}
\begin{tikzpicture}[circuit ee IEC, set resistor graphic=var resistor IEC
      graphic,scale=.5]
\scalebox{1}{
		\node [style=none] (0) at (-1, -0) {};
		\node  [circle,draw,inner sep=1pt,fill]   (1) at (-0.125, -0) {};
		\node [style=none] (2) at (1, 0.5) {};
		\node [style=none] (3) at (1, -0.5) {};
		\node [style=circ, draw,inner sep=1pt,fill] (4) at (-1, -0) {};
		\node [style=none] (5) at (1, -0.5) {};
		\node [style=none] (6) at (1, -1.5) {};
		\node  [circle,draw,inner sep=1pt,fill]   (7) at (2.125, -1) {};
		\node  [circle,draw,inner sep=1pt,fill]   (8) at (3, -1) {};
		\node [style=none] (9) at (3, -1) {};
		\node [style=none] (10) at (3, .5) {};
		\node [style=none] (11) at (-1, -1.5) {};
		\draw[line width=1pt] (0.center) to (1.center);
		\draw[line width=1pt] [in=180, out=60, looseness=1.20] (1.center) to (2.center);
		\draw[line width=1pt] [in=180, out=-60, looseness=1.20] (1.center) to (3.center);
		\draw[line width=1pt] [in=0, out=120, looseness=1.20] (7.center) to (5.center);
		\draw[line width=1pt] [in=0, out=-120, looseness=1.20] (7.center) to (6.center);
		\draw[line width=1pt] (9.center) to (7.center);
		\draw[line width=1pt] (10.center) to (2.center);
		\draw[line width=1pt] (11.center) to (6.center);
}
      \end{tikzpicture}
\end{aligned}
}

\newcommand{\zigzaglaw}[1]
{
\begin{aligned}
\begin{tikzpicture}
[circuit ee IEC, set resistor graphic=var resistor IEC
      graphic,scale=.5]
\scalebox{1}{
		\node [style=none] (2) at (1, 0.5) {};
		\node [style=none] (3) at (1, -0.5) {};
		\node [style=none] (4) at (1, -1.5) {};
		\node [style=none] (5) at (2.25, 0.5) {};
		\node [style=none] (6) at (-.25, -1.5) {};
	\draw[line width = 1.5pt] (2) to [in=180, out=180,looseness = 4] (3.east);
	\draw[line width = 1.5pt] (3) to [in=0, out=0,looseness = 4] (4);
	\draw[line width = 1.5pt] (2.west) to (5.west);
	\draw[line width = 1.5pt] (4.east) to  (6.east);
}
      \end{tikzpicture}
\end{aligned}
}

\newcommand{\zigzaglawother}[1]
{
\begin{aligned}
\begin{tikzpicture}
[circuit ee IEC, set resistor graphic=var resistor IEC
      graphic,scale=.5]
\scalebox{1}{
		\node [style=none] (2) at (1, 0.5) {};
		\node [style=none] (3) at (1, -0.5) {};
		\node [style=none] (4) at (1, -1.5) {};
		\node [style=none] (5) at (-.25, 0.5) {};
		\node [style=none] (6) at (2.25, -1.5) {};
	\draw[line width = 1.5pt] (2) to [in=0, out=0,looseness = 4] (3.west);
	\draw[line width = 1.5pt] (3) to [in=180, out=180,looseness = 4] (4);
	\draw[line width = 1.5pt] (2.east) to (5.east);
	\draw[line width = 1.5pt] (4.west) to  (6.west);
}
      \end{tikzpicture}
\end{aligned}
}

\newcommand{\frobs}[1]
{
  \begin{aligned}
\begin{tikzpicture}[circuit ee IEC, set resistor graphic=var resistor IEC
      graphic,scale=.5]
\scalebox{1}{
		\node [style=none] (0) at (-1.5, 0.5) {};
		\node  [circle,draw,inner sep=1pt,fill]   (1) at (-0.75, 0.5) {};
		\node [style=none] (2) at (0.25, -0) {};
		\node [style=none] (3) at (0.25, 1) {};
		\node  [circle,draw,inner sep=1pt,fill]   (4) at (1, -0.5) {};
		\node [style=none] (5) at (0, -0) {};
		\node [style=none] (6) at (1.75, -0.5) {};
		\node [style=none] (7) at (0, -1) {};
		\node [style=none] (8) at (1.75, 1) {};
		\node [style=none] (9) at (-1.5, -1) {};
		\draw[line width = 1.5pt] [in=180, out=-60, looseness=1.20] (1) to (2.center);
		\draw[line width = 1.5pt] [in=180, out=60, looseness=1.20] (1) to (3.center);
		\draw[line width = 1.5pt] (0.center) to (1);
		\draw[line width = 1.5pt] (6.center) to (4);
		\draw[line width = 1.5pt] [in=0, out=120, looseness=1.20] (4) to (5.center);
		\draw[line width = 1.5pt] [in=0, out=-120, looseness=1.20] (4) to (7.center);
		\draw[line width = 1.5pt] (3.center) to (8.center);
		\draw[line width = 1.5pt] (7.center) to (9.center);
}
\end{tikzpicture}
  \end{aligned}
}

\newcommand{\frobx}[1]
{
  \begin{aligned}
\begin{tikzpicture}[circuit ee IEC, set resistor graphic=var resistor IEC
      graphic,scale=.5]
\scalebox{1}{
		\node  [circle,draw,inner sep=1pt,fill]   (0) at (-0.5, -0) {};
		\node [style=none] (1) at (-1.5, -0.5) {};
		\node [style=none] (2) at (-1.5, 0.5) {};
		\node  [circle,draw,inner sep=1pt,fill]   (3) at (0.5, -0) {};
		\node [style=none] (4) at (1.5, -0.5) {};
		\node [style=none] (5) at (1.5, 0.5) {};
		\draw[line width = 1.5pt] [in=0, out=-120, looseness=1.20] (0.center) to (1.center);
		\draw[line width = 1.5pt] [in=0, out=120, looseness=1.20] (0.center) to (2.center);
		\draw[line width = 1.5pt] [in=180, out=-60, looseness=1.20] (3) to (4.center);
		\draw[line width = 1.5pt] [in=180, out=60, looseness=1.20] (3) to (5.center);
		\draw[line width = 1.5pt] (0) to (3);
}
\end{tikzpicture}
  \end{aligned}
}

\newcommand{\frobz}[1]
{
  \begin{aligned}
\begin{tikzpicture}[circuit ee IEC, set resistor graphic=var resistor IEC
      graphic,scale=.5]
\scalebox{1}{
		\node [style=none] (0) at (1.75, 0.5) {};
		\node  [circle,draw,inner sep=1pt,fill]   (1) at (1, 0.5) {};
		\node [style=none] (2) at (0, -0) {};
		\node [style=none] (3) at (0, 1) {};
		\node  [circle,draw,inner sep=1pt,fill]   (4) at (-0.75, -0.5) {};
		\node [style=none] (5) at (0.25, -0) {};
		\node [style=none] (6) at (-1.5, -0.5) {};
		\node [style=none] (7) at (0.25, -1) {};
		\node [style=none] (8) at (-1.5, 1) {};
		\node [style=none] (9) at (1.75, -1) {};
		\draw[line width = 1.5pt] [in=0, out=-120, looseness=1.20] (1) to (2.center);
		\draw[line width = 1.5pt] [in=0, out=120, looseness=1.20] (1) to (3.center);
		\draw[line width = 1.5pt] (0.center) to (1);
		\draw[line width = 1.5pt] (6.center) to (4);
		\draw[line width = 1.5pt] [in=180, out=60, looseness=1.20] (4) to (5.center);
		\draw[line width = 1.5pt] [in=180, out=-60, looseness=1.20] (4) to (7.center);
		\draw[line width = 1.5pt] (3.center) to (8.center);
		\draw[line width = 1.5pt] (7.center) to (9.center);
}
\end{tikzpicture}
  \end{aligned}
}

\newcommand{\parfrobl}[1]
{
\begin{aligned}
\begin{tikzpicture}[circuit ee IEC, set resistor graphic=var resistor IEC
      graphic,scale=.5]
\scalebox{1}{
		\node [style=none] (0) at (1, -0) {};
		\node  [circle,draw,inner sep=1pt,fill]   (1) at (0.125, -0) {};
		\node [style=none] (2) at (-1, 0.5) {}; %
		\node [style=none] (3) at (-1, -0.5) {};
		\node [style=none] (4) at (1, -.5) {};
		\node  [circle,draw,inner sep=1pt,fill]   (5) at (0.125, -.5) {};
		\node [style=none] (6) at (-1, -0) {}; %
		\node [style=none] (7) at (-1, -1) {};
		\node [style=none] (8) at (-3, 0) {}; 
		\node [style=none] (9) at (-3, .5) {};
		\draw[line width = 1.5pt] (0.center) to (1.center);
		\draw[line width = 1.5pt] [in=0, out=120, looseness=1.20] (1.center) to (2.west);
		\draw[line width = 1.5pt] [in=0, out=-120, looseness=1.20] (1.center) to (3.west);
		\draw[line width = 1.5pt] (4.center) to (5.center);
		\draw[line width = 1.5pt] [in=0, out=120, looseness=1.20] (5.center) to (6.west);
		\draw[line width = 1.5pt] [in=0, out=-120, looseness=1.20] (5.center) to (7.west);
		\draw[line width = 1.5pt] [in=0, out=180] (2.center) to (9.center);
		\draw[line width = 1.5pt] [in=0, out=180] (6.center) to (8.center);
\begin{scope}[shift={(-2,-1.5)}, rotate=180]
		\node [style=none] (0) at (1, -0) {};
		\node  [circle,draw,inner sep=1pt,fill]   (1) at (0.125, -0) {};
		\node [style=none] (2) at (-1, 0.5) {}; %
		\node [style=none] (3) at (-1, -0.5) {};
		\node [style=none] (4) at (1, -.5) {};
		\node  [circle,draw,inner sep=1pt,fill]   (5) at (0.125, -.5) {};
		\node [style=none] (6) at (-1, -0) {}; %
		\node [style=none] (7) at (-1, -1) {};
		\node [style=none] (8) at (-3, 0) {}; 
		\node [style=none] (9) at (-3, .5) {};
		\draw[line width = 1.5pt] (0.center) to (1.center);
		\draw[line width = 1.5pt] [in=0, out=120, looseness=1.20] (1.center) to (2.east);
		\draw[line width = 1.5pt] [in=0, out=-120, looseness=1.20] (1.center) to (3.east);
		\draw[line width = 1.5pt] (4.center) to (5.center);
		\draw[line width = 1.5pt] [in=0, out=120, looseness=1.20] (5.center) to (6.east);
		\draw[line width = 1.5pt] [in=0, out=-120, looseness=1.20] (5.center) to (7.east);
		\draw[line width = 1.5pt] [in=0, out=180] (2.center) to (9.center);
		\draw[line width = 1.5pt] [in=0, out=180] (6.center) to (8.center);
\end{scope}
}
      \end{tikzpicture}
\end{aligned}
}

\newcommand{\parfrobm}[1]
{
\begin{aligned}
\begin{tikzpicture}[circuit ee IEC, set resistor graphic=var resistor IEC
      graphic,scale=.5]
\scalebox{1}{
		\node [style=none] (0) at (1, -0) {};
		\node  [circle,draw,inner sep=1pt,fill]   (1) at (0.125, -0) {};
		\node [style=none] (2) at (-1, 0.5) {};
		\node [style=none] (3) at (-1, -0.5) {};
		\node [style=none] (4) at (1, -.5) {};
		\node  [circle,draw,inner sep=1pt,fill]   (5) at (0.125, -.5) {};
		\node [style=none] (6) at (-1, -0) {};
		\node [style=none] (7) at (-1, -1) {};
		\draw[line width = 1.5pt] (0.center) to (1.center);
		\draw[line width = 1.5pt] [in=0, out=120, looseness=1.20] (1.center) to (2.center);
		\draw[line width = 1.5pt] [in=0, out=-120, looseness=1.20] (1.center) to (3.center);
		\draw[line width = 1.5pt] (4.center) to (5.center);
		\draw[line width = 1.5pt] [in=0, out=120, looseness=1.20] (5.center) to (6.center);
		\draw[line width = 1.5pt] [in=0, out=-120, looseness=1.20] (5.center) to (7.center);
\begin{scope}[shift={(2,-.5)}, rotate =180]
		\node [style=none] (0) at (1, -0) {};
		\node  [circle,draw,inner sep=1pt,fill]   (1) at (0.125, -0) {};
		\node [style=none] (2) at (-1, 0.5) {};
		\node [style=none] (3) at (-1, -0.5) {};
		\node [style=none] (4) at (1, -.5) {};
		\node  [circle,draw,inner sep=1pt,fill]   (5) at (0.125, -.5) {};
		\node [style=none] (6) at (-1, -0) {};
		\node [style=none] (7) at (-1, -1) {};
		\draw[line width = 1.5pt] (0.center) to (1.center);
		\draw[line width = 1.5pt] [in=0, out=120, looseness=1.20] (1.center) to (2.center);
		\draw[line width = 1.5pt] [in=0, out=-120, looseness=1.20] (1.center) to (3.center);
		\draw[line width = 1.5pt] (4.center) to (5.center);
		\draw[line width = 1.5pt] [in=0, out=120, looseness=1.20] (5.center) to (6.center);
		\draw[line width = 1.5pt] [in=0, out=-120, looseness=1.20] (5.center) to (7.center);
\end{scope}
}
      \end{tikzpicture}
\end{aligned}
}

\newcommand{\parfrobr}[1]
{
\begin{aligned}
\begin{tikzpicture}[circuit ee IEC, set resistor graphic=var resistor IEC
      graphic,scale=.5]
\scalebox{1}{
\begin{scope}[yscale=-1]
		\node [style=none] (0) at (1, -0) {};
		\node  [circle,draw,inner sep=1pt,fill]   (1) at (0.125, -0) {};
		\node [style=none] (2) at (-1, 0.5) {}; %
		\node [style=none] (3) at (-1, -0.5) {};
		\node [style=none] (4) at (1, -.5) {};
		\node  [circle,draw,inner sep=1pt,fill]   (5) at (0.125, -.5) {};
		\node [style=none] (6) at (-1, -0) {}; %
		\node [style=none] (7) at (-1, -1) {};
		\node [style=none] (8) at (-3, 0) {}; 
		\node [style=none] (9) at (-3, .5) {};
		\draw[line width = 1.5pt] (0.center) to (1.center);
		\draw[line width = 1.5pt] [in=0, out=120, looseness=1.20] (1.center) to (2.west);
		\draw[line width = 1.5pt] [in=0, out=-120, looseness=1.20] (1.center) to (3.west);
		\draw[line width = 1.5pt] (4.center) to (5.center);
		\draw[line width = 1.5pt] [in=0, out=120, looseness=1.20] (5.center) to (6.west);
		\draw[line width = 1.5pt] [in=0, out=-120, looseness=1.20] (5.center) to (7.west);
		\draw[line width = 1.5pt] [in=0, out=180] (2.center) to (9.center);
		\draw[line width = 1.5pt] [in=0, out=180] (6.center) to (8.center);
\begin{scope}[shift={(-2,-1.5)}, rotate=180]
		\node [style=none] (0) at (1, -0) {};
		\node  [circle,draw,inner sep=1pt,fill]   (1) at (0.125, -0) {};
		\node [style=none] (2) at (-1, 0.5) {}; %
		\node [style=none] (3) at (-1, -0.5) {};
		\node [style=none] (4) at (1, -.5) {};
		\node  [circle,draw,inner sep=1pt,fill]   (5) at (0.125, -.5) {};
		\node [style=none] (6) at (-1, -0) {}; %
		\node [style=none] (7) at (-1, -1) {};
		\node [style=none] (8) at (-3, 0) {}; 
		\node [style=none] (9) at (-3, .5) {};
		\draw[line width = 1.5pt] (0.center) to (1.center);
		\draw[line width = 1.5pt] [in=0, out=120, looseness=1.20] (1.center) to (2.east);
		\draw[line width = 1.5pt] [in=0, out=-120, looseness=1.20] (1.center) to (3.east);
		\draw[line width = 1.5pt] (4.center) to (5.center);
		\draw[line width = 1.5pt] [in=0, out=120, looseness=1.20] (5.center) to (6.east);
		\draw[line width = 1.5pt] [in=0, out=-120, looseness=1.20] (5.center) to (7.east);
		\draw[line width = 1.5pt] [in=0, out=180] (2.center) to (9.center);
		\draw[line width = 1.5pt] [in=0, out=180] (6.center) to (8.center);
\end{scope}
\end{scope}
}
      \end{tikzpicture}
\end{aligned}
}

\newcommand{\monadfrobl}[1]
{
\begin{aligned}
\begin{tikzpicture}[circuit ee IEC, set resistor graphic=var resistor IEC
      graphic,scale=.5]
\scalebox{1}{
		\node [style=none] (0) at (1, -0) {};
		\node [style=none] (1) at (0.125, -0) {};
		\node [style=none] (2) at (-1, 0.5) {}; 
		\node [style=none] (3) at (-1, -0.5) {}; 
		\node [style=none] (4) at (1, -0) {};
		\node [style=none] (5) at (-1, 1) {};
		\node [style=none] (6) at (1, .25) {};
		\node [style=none] (7) at (-1, -1) {}; 
		\node [style=none] (8) at (1, -.25) {};
		\node [style=none] (9) at (-1, -2) {}; 
		\node [style=none] (10) at (-3, -1.25) {}; 
		\node [style=none] (11) at (-1, -2.5) {}; 
		\node [style=none] (12) at (-3, -1.75) {}; 
		\node [style=none] (13) at (1, -2) {}; 
		\node [style=none] (14) at (1, -2.5) {}; 
		\node [style=none] (15) at (-3, 1) {}; 
		\node [style=none] (16) at (-3, .5) {}; 
	\draw[line width = 1.5pt] (2) to [in=0, out=0,looseness = 3] (3);
	\draw[line width = 1.5pt] (7.east) to [in=180, out=180,looseness = 3] (9.east);
	\draw[line width = 1.5pt] (5) to [in=180, out=0,looseness = 1] (6);
	\draw[line width = 1.5pt] (7) to [in=180, out=0, looseness = 1] (8);
	\draw[line width = 1.5pt] (3.east) to [in=0, out=180,looseness = 1] (10);
	\draw[line width = 1.5pt] (11.east) to [in=0, out=180,looseness = 1] (12);
	\draw[line width = 1.5pt] (9.east) to (13);
	\draw[line width = 1.5pt] (11.east) to (14);
	\draw[line width = 1.5pt] (5.east) to (15);
	\draw[line width = 1.5pt] (2.east) to (16);
}
      \end{tikzpicture}
\end{aligned}
}

\newcommand{\monadfrobm}[1]
{
\begin{aligned}
\begin{tikzpicture}[circuit ee IEC, set resistor graphic=var resistor IEC
      graphic,scale=.5]
\scalebox{1}{
		\node [style=none] (0) at (1, -0) {};
		\node [style=none] (1) at (0.125, -0) {};
		\node [style=none] (2) at (-1, 0.5) {}; 
		\node [style=none] (3) at (-1, -0.5) {}; 
		\node [style=none] (4) at (1, -0) {};
		\node [style=none] (5) at (-1, 1) {};
		\node [style=none] (6) at (1, .25) {};
		\node [style=none] (7) at (-1, -1) {}; 
		\node [style=none] (8) at (1, -1) {};
		\node [style=none] (9) at (-1, -2) {}; 
		\node [style=none] (10) at (-3, -.5) {}; 
		\node [style=none] (11) at (-1, -2.5) {}; 
		\node [style=none] (12) at (-3, -1.75) {}; 
		\node [style=none] (13) at (1, -2) {}; 
		\node [style=none] (14) at (1, -2.5) {}; 
		\node [style=none] (15) at (-3, 1) {}; 
		\node [style=none] (16) at (-3, .5) {}; 
	\draw[line width = 1.5pt] (2) to [in=0, out=0,looseness = 3] (3);
	\draw[line width = 1.5pt] (7.east) to [in=180, out=180,looseness = 3] (9.east);
	\draw[line width = 1.5pt] (5) to [in=180, out=0,looseness = 1] (6);
	\draw[line width = 1.5pt] (7) to [in=180, out=0, looseness = 1] (8);
	\draw[line width = 1.5pt] (3.east) to [in=0, out=180,looseness = 1] (10);
	\draw[line width = 1.5pt] (11.east) to [in=0, out=180,looseness = 1] (12);
	\draw[line width = 1.5pt] (9.east) to (13);
	\draw[line width = 1.5pt] (11.east) to (14);
	\draw[line width = 1.5pt] (5.east) to (15);
	\draw[line width = 1.5pt] (2.east) to (16);
}
      \end{tikzpicture}
\end{aligned}
}

\newcommand{\monadfrob}[1]
{
\begin{aligned}
\begin{tikzpicture}[circuit ee IEC, set resistor graphic=var resistor IEC
      graphic,scale=.5]
\scalebox{1}{
		\node [style=none] (0) at (1, -0) {};
		\node [style=none] (1) at (0.125, -0) {};
		\node [style=none] (2) at (-1, 0.5) {};
		\node [style=none] (3) at (-1, -0.5) {};
		\node [style=none] (4) at (1, -0) {};
		\node [style=none] (5) at (-1, 1) {};
		\node [style=none] (6) at (1, .25) {};
		\node [style=none] (7) at (-1, -1) {};
		\node [style=none] (8) at (1, -.25) {};
		\node [style=none] (9) at (3, 1) {};
		\node [style=none] (10) at (3, -1) {};
		\node [style=none] (11) at (3, .5) {};
		\node [style=none] (12) at (3, -.5) {};
	\draw[line width = 1.5pt] (2) to [in=0, out=0,looseness = 3] (3);
	\draw[line width = 1.5pt] (11) to [in=180, out=180,looseness = 3] (12);
	\draw[line width = 1.5pt] (5) to [in=180, out=0,looseness = 1] (6.east);
	\draw[line width = 1.5pt] (7) to [in=180, out=0, looseness = 1] (8.east);
	\draw[line width = 1.5pt] (6.east) to [in=180, out=0,looseness = 1] (9);
	\draw[line width = 1.5pt] (8.east) to [in=180, out=0, looseness = 1] (10);
}
      \end{tikzpicture}
\end{aligned}
}

\newcommand{\bi}[1]
{
  \begin{aligned}
\begin{tikzpicture}[circuit ee IEC, set resistor graphic=var resistor IEC
      graphic,scale=.5]
\scalebox{1}{
		\node [style=none] (0) at (-2, 0.75) {};
		\node [style=none] (1) at (-2, -0.5) {};
		\node  [circle,draw,inner sep=1pt,fill]   (2) at (-1, 0.75) {};
		\node  [circle,draw,inner sep=1pt,fill]   (3) at (-1, -0.5) {};
		\node [style=none] (4) at (0, 1.25) {};
		\node [style=none] (5) at (-0.25, 0.25) {};
		\node [style=none] (6) at (-0.25, -0) {};
		\node [style=none] (7) at (0, -1) {};
		\node [style=none] (8) at (0, 1.25) {};
		\node [style=none] (9) at (0.25, 0.25) {};
		\node [style=none] (10) at (0.25, -0) {};
		\node [style=none] (11) at (0, -1) {};
		\node  [circle,draw,inner sep=1pt,fill]   (12) at (1, 0.75) {};
		\node  [circle,draw,inner sep=1pt,fill]   (13) at (1, -0.5) {};
		\node [style=none] (14) at (2, 0.75) {};
		\node [style=none] (15) at (2, -0.5) {};
		\draw[line width = 1.5pt] (0.center) to (2);
		\draw[line width = 1.5pt] (1.center) to (3);
		\draw[line width = 1.5pt] (12) to (14.center);
		\draw[line width = 1.5pt] (13) to (15.center);
		\draw[line width = 1.5pt] [in=180, out=60, looseness=1.20] (2) to (4.center);
		\draw[line width = 1.5pt] [in=180, out=-60, looseness=1.20] (2) to (5.center);
		\draw[line width = 1.5pt] [in=180, out=60, looseness=1.20] (3) to (6.center);
		\draw[line width = 1.5pt] [in=180, out=-60, looseness=1.20] (3) to (7.center);
		\draw[line width = 1.5pt] [in=120, out=0, looseness=1.20] (8.center) to (12);
		\draw[line width = 1.5pt] [in=-120, out=0, looseness=1.20] (9.center) to (12);
		\draw[line width = 1.5pt] [in=120, out=0, looseness=1.20] (10.center) to (13);
		\draw[line width = 1.5pt] [in=-120, out=0, looseness=1.20] (11.center) to (13);
		\draw[line width = 1.5pt] (4.center) to (8.center);
		\draw[line width = 1.5pt] [in=150, out=-30, looseness=1.00] (5.center) to (10.center);
		\draw[line width = 1.5pt] [in=-150, out=30, looseness=1.00] (6.center) to (9.center);
		\draw[line width = 1.5pt] (7.center) to (11.center);
}
\end{tikzpicture}
  \end{aligned}
}

\newcommand{\bimultl}[1]
{
\begin{aligned}
\begin{tikzpicture}[circuit ee IEC, set resistor graphic=var resistor IEC
      graphic,scale=.5]
\scalebox{1}{
		\node  [circle,draw,inner sep=1pt,fill]   (0) at (1, -0) {};
		\node  [circle,draw,inner sep=1pt,fill]   (1) at (0.125, -0) {};
		\node [style=none] (2) at (-1, 0.5) {};
		\node [style=none] (3) at (-1, -0.5) {};
		\draw[line width = 1.5pt] (0.center) to (1.center);
		\draw[line width = 1.5pt] [in=0, out=120, looseness=1.20] (1.center) to (2.center);
		\draw[line width = 1.5pt] [in=0, out=-120, looseness=1.20] (1.center) to (3.center);
}
      \end{tikzpicture}
\end{aligned}
}

\newcommand{\bimultr}[1]
{
\begin{aligned}
\begin{tikzpicture}[circuit ee IEC, set resistor graphic=var resistor IEC
      graphic,scale=.5]
\scalebox{1}{
		\node [style=none] (e) at (1, -0) {};
		\node  [circle,draw,inner sep=1pt,fill]   (0) at (0.125, 0.5) {};
		\node  [circle,draw,inner sep=1pt,fill]   (1) at (0.125, -0.5) {};
		\node [style=none] (2) at (-1, 0.5) {};
		\node [style=none] (3) at (-1, -0.5) {};
		\draw[line width = 1.5pt] (2.center) to (0.center);
		\draw[line width = 1.5pt] (3.center) to (1.center);
}
      \end{tikzpicture}
\end{aligned}
}

\newcommand{\bicomultl}[1]
{
\begin{aligned}
\begin{tikzpicture}[circuit ee IEC, set resistor graphic=var resistor IEC
      graphic,scale=.5]
\scalebox{1}{
		\node  [circle,draw,inner sep=1pt,fill]   (0) at (-1, -0) {};
		\node  [circle,draw,inner sep=1pt,fill]   (1) at (-0.125, -0) {};
		\node [style=none] (2) at (1, 0.5) {};
		\node [style=none] (3) at (1, -0.5) {};
		\draw[line width = 1.5pt] (0.center) to (1.center);
		\draw[line width = 1.5pt] [in=180, out=60, looseness=1.20] (1.center) to (2.center);
		\draw[line width = 1.5pt] [in=180, out=-60, looseness=1.20] (1.center) to (3.center);
}
      \end{tikzpicture}
\end{aligned}
}

\newcommand{\bicomultr}[1]
{
\begin{aligned}
\begin{tikzpicture}[circuit ee IEC, set resistor graphic=var resistor IEC
      graphic,scale=.5]
\scalebox{1}{
		\node [style=none] (e) at (-1, -0) {};
		\node  [circle,draw,inner sep=1pt,fill]   (0) at (-0.125, 0.5) {};
		\node  [circle,draw,inner sep=1pt,fill]   (1) at (-0.125, -0.5) {};
		\node [style=none] (2) at (1, 0.5) {};
		\node [style=none] (3) at (1, -0.5) {};
		\draw[line width = 1.5pt] (0.center) to (2.center);
		\draw[line width = 1.5pt] (1.center) to (3.center);
}
      \end{tikzpicture}
\end{aligned}
}

\newcommand{\weakbil}[1]
{
\begin{aligned}
\begin{tikzpicture}[circuit ee IEC, set resistor graphic=var resistor IEC
      graphic,scale=.5]
\scalebox{1}{
		\node [style=none] (0) at (1, -0) {};
		\node  [circle,draw,inner sep=1pt,fill]   (1) at (0.125, -0) {};
		\node [style=none] (2) at (-1, 0.5) {};
		\node [style=none] (3) at (-1, -0.5) {};
		\node  [circle,draw,inner sep=1pt,fill]   (4) at (1, -0) {};
		\node [style=none] (5) at (-1, -0) {};
		\draw[line width = 1.5pt] (0.center) to (1.center);
		\draw[line width = 1.5pt] [in=0, out=120, looseness=1.20] (1.center) to (2.center);
		\draw[line width = 1.5pt] [in=0, out=-120, looseness=1.20] (1.center) to (3.center);
		\draw[line width = 1.5pt] [in=0, out=180] (1.center) to (5.center);
}
      \end{tikzpicture}
\end{aligned}
}

\newcommand{\weakbim}[1]
{
\begin{aligned}
\begin{tikzpicture}[circuit ee IEC, set resistor graphic=var resistor IEC
      graphic,scale=.5]
\scalebox{1}{
		\node [style=none] (0) at (1, -0) {};
		\node  [circle,draw,inner sep=1pt,fill]   (1) at (0.125, -0) {};
		\node [style=none] (2) at (-1, 0.5) {};
		\node [style=none] (3) at (-1, -0.5) {};
		\node  [circle,draw,inner sep=1pt,fill]   (4) at (1, -0) {};
		\node [style=none] (5) at (-3, -0.5) {};
		\draw[line width = 1.5pt] (0.center) to (1.center);
		\draw[line width = 1.5pt] (3.center) to (5.center);
		\draw[line width = 1.5pt] [in=0, out=120, looseness=1.20] (1.center) to (2.center);
		\draw[line width = 1.5pt] [in=0, out=-120, looseness=1.20] (1.center) to (3.center);
\begin{scope}[shift={(0,2)}]
		\node [style=none] (0) at (1, -0) {};
		\node  [circle,draw,inner sep=1pt,fill]   (1) at (0.125, -0) {};
		\node [style=none] (2) at (-1, 0.5) {};
		\node [style=none] (3) at (-1, -0.5) {};
		\node  [circle,draw,inner sep=1pt,fill]   (4) at (1, -0) {};
	     \node [style=none] (5) at (-3, 0.5) {};
		\draw[line width = 1.5pt] (0.center) to (1.center);
		\draw[line width = 1.5pt] (2.center) to (5.center);
		\draw[line width = 1.5pt] [in=0, out=120, looseness=1.20] (1.center) to (2.center);
		\draw[line width = 1.5pt] [in=0, out=-120, looseness=1.20] (1.center) to (3.center);
\end{scope}
\begin{scope}[shift={(-2,1)}, rotate =180 ]
		\node [style=none] (0) at (1, -0) {};
		\node  [circle,draw,inner sep=1pt,fill]   (1) at (0.125, -0) {};
		\node [style=none] (2) at (-1, 0.5) {};
		\node [style=none] (3) at (-1, -0.5) {};
		\draw[line width = 1.5pt] (0.center) to (1.center);
		\draw[line width = 1.5pt] [in=0, out=120, looseness=1.20] (1.center) to (2.center);
		\draw[line width = 1.5pt] [in=0, out=-120, looseness=1.20] (1.center) to (3.center);
\end{scope}
}
      \end{tikzpicture}
\end{aligned}
}

\newcommand{\weakbir}[1]
{
\begin{aligned}
\begin{tikzpicture}[circuit ee IEC, set resistor graphic=var resistor IEC
      graphic,scale=.5]
\scalebox{1}{
		\node [style=none] (0) at (1, -0) {};
		\node  [circle,draw,inner sep=1pt,fill]   (1) at (0.125, -0) {};
		\node [style=none] (2) at (-1, 0.5) {};
		\node [style=none] (3) at (-1, -0.5) {};
		\node  [circle,draw,inner sep=1pt,fill]   (4) at (1, -0) {};
		\node [style=none] (5) at (-4, -0.5) {};
		\draw[line width = 1.5pt] (0.center) to (1.center);
		\draw[line width = 1.5pt] (3.center) to (5.center);
		\draw[line width = 1.5pt] [in=0, out=120, looseness=1.20] (1.center) to (2.center);
		\draw[line width = 1.5pt] [in=0, out=-120, looseness=1.20] (1.center) to (3.center);
\begin{scope}[shift={(0,2)}]
		\node [style=none] (0) at (1, -0) {};
		\node  [circle,draw,inner sep=1pt,fill]   (1) at (0.125, -0) {};
		\node [style=none] (2) at (-1, 0.5) {};
		\node [style=none] (3) at (-1, -0.5) {};
		\node  [circle,draw,inner sep=1pt,fill]   (4) at (1, -0) {};
	     \node [style=none] (5) at (-4, 0.5) {};
		\draw[line width = 1.5pt] (0.center) to (1.center);
		\draw[line width = 1.5pt] (2.center) to (5.center);
		\draw[line width = 1.5pt] [in=0, out=120, looseness=1.20] (1.center) to (2.center);
		\draw[line width = 1.5pt] [in=0, out=-120, looseness=1.20] (1.center) to (3.center);
\end{scope}
\begin{scope}[shift={(-3,1)}, rotate =180 ]
		\node [style=none] (0) at (1.25, -0) {};
		\node  [circle,draw,inner sep=1pt,fill]   (1) at (0.375, -0) {};
		\node [style=none] (2) at (-0.5, -0.5) {};
		\node [style=none] (3) at (-2, 0.5) {};
		\node [style=none] (4) at (-0.5, 0.5) {};
		\node [style=none] (5) at (-2, -0.5) {};
		\draw[line width = 1.5pt] (0.center) to (1.center);
		\draw[line width = 1.5pt] [in=0, out=-120, looseness=1.20] (1.center) to (2.center);
		\draw[line width = 1.5pt] [in=180, out=0, looseness=1.00] (3.center) to (2.center);
		\draw[line width = 1.5pt] [in=0, out=120, looseness=1.20] (1.center) to (4.center);
		\draw[line width = 1.5pt] [in=0, out=180, looseness=1.00] (4.center) to (5.center);
\end{scope}
}
      \end{tikzpicture}
\end{aligned}
}

\newcommand{\weakcobil}[1]
{
\begin{aligned}
\begin{tikzpicture}[circuit ee IEC, set resistor graphic=var resistor IEC
      graphic,scale=.5]
\scalebox{1}{
\begin{scope}[rotate =180]
		\node [style=none] (0) at (1, -0) {};
		\node  [circle,draw,inner sep=1pt,fill]   (1) at (0.125, -0) {};
		\node [style=none] (2) at (-1, 0.5) {};
		\node [style=none] (3) at (-1, -0.5) {};
		\node  [circle,draw,inner sep=1pt,fill]   (4) at (1, -0) {};
		\node [style=none] (5) at (-1, -0) {};
		\draw[line width = 1.5pt] (0.center) to (1.center);
		\draw[line width = 1.5pt] [in=0, out=120, looseness=1.20] (1.center) to (2.center);
		\draw[line width = 1.5pt] [in=0, out=-120, looseness=1.20] (1.center) to (3.center);
		\draw[line width = 1.5pt] [in=0, out=180] (1.center) to (5.center);
\end{scope}
}
      \end{tikzpicture}
\end{aligned}
}

\newcommand{\weakcobim}[1]
{
\begin{aligned}
\begin{tikzpicture}[circuit ee IEC, set resistor graphic=var resistor IEC
      graphic,scale=.5]
\scalebox{1}{
\begin{scope}[rotate=180]
		\node [style=none] (0) at (1, -0) {};
		\node  [circle,draw,inner sep=1pt,fill]   (1) at (0.125, -0) {};
		\node [style=none] (2) at (-1, 0.5) {};
		\node [style=none] (3) at (-1, -0.5) {};
		\node  [circle,draw,inner sep=1pt,fill]   (4) at (1, -0) {};
		\node [style=none] (5) at (-3, -0.5) {};
		\draw[line width = 1.5pt] (0.center) to (1.center);
		\draw[line width = 1.5pt] (3.center) to (5.center);
		\draw[line width = 1.5pt] [in=0, out=120, looseness=1.20] (1.center) to (2.center);
		\draw[line width = 1.5pt] [in=0, out=-120, looseness=1.20] (1.center) to (3.center);
\begin{scope}[shift={(0,2)}]
		\node [style=none] (0) at (1, -0) {};
		\node  [circle,draw,inner sep=1pt,fill]   (1) at (0.125, -0) {};
		\node [style=none] (2) at (-1, 0.5) {};
		\node [style=none] (3) at (-1, -0.5) {};
		\node  [circle,draw,inner sep=1pt,fill]   (4) at (1, -0) {};
	     \node [style=none] (5) at (-3, 0.5) {};
		\draw[line width = 1.5pt] (0.center) to (1.center);
		\draw[line width = 1.5pt] (2.center) to (5.center);
		\draw[line width = 1.5pt] [in=0, out=120, looseness=1.20] (1.center) to (2.center);
		\draw[line width = 1.5pt] [in=0, out=-120, looseness=1.20] (1.center) to (3.center);
\end{scope}
\begin{scope}[shift={(-2,1)}, rotate =180 ]
		\node [style=none] (0) at (1, -0) {};
		\node  [circle,draw,inner sep=1pt,fill]   (1) at (0.125, -0) {};
		\node [style=none] (2) at (-1, 0.5) {};
		\node [style=none] (3) at (-1, -0.5) {};
		\draw[line width = 1.5pt] (0.center) to (1.center);
		\draw[line width = 1.5pt] [in=0, out=120, looseness=1.20] (1.center) to (2.center);
		\draw[line width = 1.5pt] [in=0, out=-120, looseness=1.20] (1.center) to (3.center);
\end{scope}
\end{scope}
}
      \end{tikzpicture}
\end{aligned}
}

\newcommand{\weakcobir}[1]
{
\begin{aligned}
\begin{tikzpicture}[circuit ee IEC, set resistor graphic=var resistor IEC
      graphic,scale=.5]
\scalebox{1}{
\begin{scope}[rotate=180]
		\node [style=none] (0) at (1, -0) {};
		\node  [circle,draw,inner sep=1pt,fill]   (1) at (0.125, -0) {};
		\node [style=none] (2) at (-1, 0.5) {};
		\node [style=none] (3) at (-1, -0.5) {};
		\node  [circle,draw,inner sep=1pt,fill]   (4) at (1, -0) {};
		\node [style=none] (5) at (-4, -0.5) {};
		\draw[line width = 1.5pt] (0.center) to (1.center);
		\draw[line width = 1.5pt] (3.center) to (5.center);
		\draw[line width = 1.5pt] [in=0, out=120, looseness=1.20] (1.center) to (2.center);
		\draw[line width = 1.5pt] [in=0, out=-120, looseness=1.20] (1.center) to (3.center);
\begin{scope}[shift={(0,2)}]
		\node [style=none] (0) at (1, -0) {};
		\node  [circle,draw,inner sep=1pt,fill]   (1) at (0.125, -0) {};
		\node [style=none] (2) at (-1, 0.5) {};
		\node [style=none] (3) at (-1, -0.5) {};
		\node  [circle,draw,inner sep=1pt,fill]   (4) at (1, -0) {};
	     \node [style=none] (5) at (-4, 0.5) {};
		\draw[line width = 1.5pt] (0.center) to (1.center);
		\draw[line width = 1.5pt] (2.center) to (5.center);
		\draw[line width = 1.5pt] [in=0, out=120, looseness=1.20] (1.center) to (2.center);
		\draw[line width = 1.5pt] [in=0, out=-120, looseness=1.20] (1.center) to (3.center);
\end{scope}
\begin{scope}[shift={(-3,1)}, rotate =180 ]
		\node [style=none] (0) at (1.25, -0) {};
		\node  [circle,draw,inner sep=1pt,fill]   (1) at (0.375, -0) {};
		\node [style=none] (2) at (-0.5, -0.5) {};
		\node [style=none] (3) at (-2, 0.5) {};
		\node [style=none] (4) at (-0.5, 0.5) {};
		\node [style=none] (5) at (-2, -0.5) {};
		\draw[line width = 1.5pt] (0.center) to (1.center);
		\draw[line width = 1.5pt] [in=0, out=-120, looseness=1.20] (1.center) to (2.center);
		\draw[line width = 1.5pt] [in=180, out=0, looseness=1.00] (3.center) to (2.center);
		\draw[line width = 1.5pt] [in=0, out=120, looseness=1.20] (1.center) to (4.center);
		\draw[line width = 1.5pt] [in=0, out=180, looseness=1.00] (4.center) to (5.center);
\end{scope}
\end{scope}
}
      \end{tikzpicture}
\end{aligned}
}

\newcommand{\weakbimonoidtwol}[1]
{
\begin{aligned}
\begin{tikzpicture}[circuit ee IEC, set resistor graphic=var resistor IEC
      graphic,scale=.5]
\scalebox{1}{
		\node [style=none] (0) at (1, -0) {};
		\node  [circle,draw,inner sep=1pt,fill]   (1) at (0.125, -0) {};
		\node [style=none] (2) at (-1, 0.5) {}; %
		\node [style=none] (3) at (-1, -0.5) {};
		\node [style=none] (4) at (1, -.5) {};
		\node  [circle,draw,inner sep=1pt,fill]   (5) at (0.125, -.5) {};
		\node [style=none] (6) at (-1, -0) {}; %
		\node [style=none] (7) at (-1, -1) {};
		\node [style=none] (8) at (-3, .5) {}; 
		\node [style=none] (9) at (-3, 1) {};
		\draw[line width = 1.5pt] (0.center) to (1.center);
		\draw[line width = 1.5pt] [in=0, out=120, looseness=1.20] (1.center) to (2.west);
		\draw[line width = 1.5pt] [in=0, out=-120, looseness=1.20] (1.center) to (3.west);
		\draw[line width = 1.5pt] (4.center) to (5.center);
		\draw[line width = 1.5pt] [in=0, out=120, looseness=1.20] (5.center) to (6.west);
		\draw[line width = 1.5pt] [in=0, out=-120, looseness=1.20] (5.center) to (7.west);
		\draw[line width = 1.5pt] [in=0, out=180, looseness=1.5] (2.center) to (9.center);
		\draw[line width = 1.5pt] [in=0, out=180, looseness=1.5] (6.center) to (8.center);
		\draw[line width = 1.5pt] (0.west) to [in=0, out=0,looseness = 3] (4.west);
\begin{scope}[shift={(-2,-.5)}]
		\node [style=none] (0) at (1, -0) {};
		\node  [circle,draw,inner sep=1pt,fill]   (1) at (0.125, -0) {};
		\node [style=none] (2) at (-1, 0.5) {};
		\node [style=none] (3) at (-1, -0.5) {};
		\node [style=none] (4) at (1, -.5) {};
		\node  [circle,draw,inner sep=1pt,fill]   (5) at (0.125, -.5) {};
		\node [style=none] (6) at (-1, -0) {};
		\node [style=none] (7) at (-1, -1) {};
		\draw[line width = 1.5pt] (0.center) to (1.center);
		\draw[line width = 1.5pt] [in=0, out=120, looseness=1.20] (1.center) to (2.center);
		\draw[line width = 1.5pt] [in=0, out=-120, looseness=1.20] (1.center) to (3.center);
		\draw[line width = 1.5pt] (4.center) to (5.center);
		\draw[line width = 1.5pt] [in=0, out=120, looseness=1.20] (5.center) to (6.center);
		\draw[line width = 1.5pt] [in=0, out=-120, looseness=1.20] (5.center) to (7.center);
\end{scope}
}
      \end{tikzpicture}
\end{aligned}
}

\newcommand{\weakbimonoidtwom}[1]
{
\begin{aligned}
\begin{tikzpicture}[circuit ee IEC, set resistor graphic=var resistor IEC
      graphic,scale=.5]
\scalebox{1}{
		\node [style=none] (0) at (1, -0) {};
		\node  [circle,draw,inner sep=1pt,fill]   (1) at (0.125, -0) {};
		\node [style=none] (2) at (-1, 0.5) {}; %
		\node [style=none] (3) at (-1, -0.5) {};
		\node [style=none] (4) at (1, -.5) {};
		\node  [circle,draw,inner sep=1pt,fill]   (5) at (0.125, -.5) {};
		\node [style=none] (6) at (-1, -0) {}; %
		\node [style=none] (7) at (-1, -1) {};
		\node [style=none] (8) at (-3, .5) {}; 
		\node [style=none] (9) at (-3, 1) {};
		\node [style=none] (10) at (-3, -1) {}; 
		\node [style=none] (11) at (-3, -1.5) {};
		\node [style=none] (12) at (-5, .5) {}; 
		\node [style=none] (13) at (-5, 1) {};
		\draw[line width = 1.5pt] (0.center) to (1.center);
		\draw[line width = 1.5pt] [in=0, out=120, looseness=1.20] (1.center) to (2.west);
		\draw[line width = 1.5pt] [in=0, out=-120, looseness=1.20] (1.center) to (3.west);
		\draw[line width = 1.5pt] (4.center) to (5.center);
		\draw[line width = 1.5pt] [in=0, out=120, looseness=1.20] (5.center) to (6.west);
		\draw[line width = 1.5pt] [in=0, out=-120, looseness=1.20] (5.center) to (7.west);
		\draw[line width = 1.5pt] [in=0, out=180, looseness=1.5] (2.center) to (9.center);
		\draw[line width = 1.5pt] [in=0, out=180, looseness=1.5] (6.center) to (8.center);
		\draw[line width = 1.5pt] [in=0, out=180, looseness=1.5] (7.center) to (11.west);
		\draw[line width = 1.5pt] [in=0, out=180, looseness=1.5] (3.center) to (10.west);
		\draw[line width = 1.5pt] [in=0, out=180, looseness=1.5] (8.center) to (12.west);
		\draw[line width = 1.5pt] [in=0, out=180, looseness=1.5] (9.center) to (13.west);
		\draw[line width = 1.5pt] (0.west) to [in=0, out=0,looseness = 3] (4.west);
\begin{scope}[shift={(0,-4)}]
\begin{scope}[yscale=-1]
		\node [style=none] (0) at (1, -0) {};
		\node  [circle,draw,inner sep=1pt,fill]   (1) at (0.125, -0) {};
		\node [style=none] (2) at (-1, 0.5) {}; %
		\node [style=none] (3) at (-1, -0.5) {};
		\node [style=none] (4) at (1, -.5) {};
		\node  [circle,draw,inner sep=1pt,fill]   (5) at (0.125, -.5) {};
		\node [style=none] (6) at (-1, -0) {}; %
		\node [style=none] (7) at (-1, -1) {};
		\node [style=none] (8) at (-3, .5) {}; 
		\node [style=none] (9) at (-3, 1) {};
		\node [style=none] (10) at (-3, -1) {}; 
		\node [style=none] (11) at (-3, -1.5) {};
		\node [style=none] (12) at (-5, .5) {}; 
		\node [style=none] (13) at (-5, 1) {};
		\draw[line width = 1.5pt] (0.center) to (1.center);
		\draw[line width = 1.5pt] [in=0, out=120, looseness=1.20] (1.center) to (2.west);
		\draw[line width = 1.5pt] [in=0, out=-120, looseness=1.20] (1.center) to (3.west);
		\draw[line width = 1.5pt] (4.center) to (5.center);
		\draw[line width = 1.5pt] [in=0, out=120, looseness=1.20] (5.center) to (6.west);
		\draw[line width = 1.5pt] [in=0, out=-120, looseness=1.20] (5.center) to (7.west);
		\draw[line width = 1.5pt] [in=0, out=180, looseness=1.5] (2.center) to (9.center);
		\draw[line width = 1.5pt] [in=0, out=180, looseness=1.5] (6.center) to (8.center);
		\draw[line width = 1.5pt] [in=0, out=180, looseness=1.5] (7.center) to (11.west);
		\draw[line width = 1.5pt] [in=0, out=180, looseness=1.5] (3.center) to (10.west);
		\draw[line width = 1.5pt] [in=0, out=180, looseness=1.5] (8.center) to (12.west);
		\draw[line width = 1.5pt] [in=0, out=180, looseness=1.5] (9.center) to (13.west);
		\draw[line width = 1.5pt] (0.west) to [in=0, out=0,looseness = 3] (4.west);
\end{scope}
\end{scope}
\begin{scope}[rotate=180,shift={(4,2)}]
		\node [style=none] (0) at (1, -0) {};
		\node [style=none] (1) at (0.125, -0) {};
		\node [style=none] (2) at (-1, 0.5) {};
		\node [style=none] (3) at (-1, -0.5) {};
		\node [style=none] (4) at (1, -0) {};
		\node [style=none] (5) at (-1, 1) {};
		\node [style=none] (6) at (1, .25) {};
		\node [style=none] (7) at (-1, -1) {};
		\node [style=none] (8) at (1, -.25) {};
	\draw[line width = 1.5pt] (2) to [in=0, out=0,looseness = 3] (3.west);
	\draw[line width = 1.5pt] (5) to [in=180, out=0,looseness = 1] (6);
	\draw[line width = 1.5pt] (7) to [in=180, out=0, looseness = 1] (8);
\end{scope}
}
      \end{tikzpicture}
\end{aligned}
}

\newcommand{\weakbimonoidtwor}[1]
{
\begin{aligned}
\begin{tikzpicture}[circuit ee IEC, set resistor graphic=var resistor IEC
      graphic,scale=.5]
\scalebox{1}{
		\node [style=none] (0) at (1, -0) {};
		\node  [circle,draw,inner sep=1pt,fill]   (1) at (0.125, -0) {};
		\node [style=none] (2) at (-1, 0.5) {}; %
		\node [style=none] (3) at (-1, -0.5) {};
		\node [style=none] (4) at (1, -.5) {};
		\node  [circle,draw,inner sep=1pt,fill]   (5) at (0.125, -.5) {};
		\node [style=none] (6) at (-1, -0) {}; %
		\node [style=none] (7) at (-1, -1) {};
		\node [style=none] (8) at (-3, .5) {}; 
		\node [style=none] (9) at (-3, 1) {};
		\node [style=none] (10) at (-3, -2.5) {}; 
		\node [style=none] (11) at (-3, -3) {};
		\node [style=none] (12) at (-5, .5) {}; 
		\node [style=none] (13) at (-5, 1) {};
		\draw[line width = 1.5pt] (0.center) to (1.center);
		\draw[line width = 1.5pt] [in=0, out=120, looseness=1.20] (1.center) to (2.west);
		\draw[line width = 1.5pt] [in=0, out=-120, looseness=1.20] (1.center) to (3.west);
		\draw[line width = 1.5pt] (4.center) to (5.center);
		\draw[line width = 1.5pt] [in=0, out=120, looseness=1.20] (5.center) to (6.west);
		\draw[line width = 1.5pt] [in=0, out=-120, looseness=1.20] (5.center) to (7.west);
		\draw[line width = 1.5pt] [in=0, out=180, looseness=1.5] (2.center) to (9.center);
		\draw[line width = 1.5pt] [in=0, out=180, looseness=1.5] (6.center) to (8.center);
		\draw[line width = 1.5pt] [in=0, out=180, looseness=1] (7.center) to (11.west);
		\draw[line width = 1.5pt] [in=0, out=180, looseness=1] (3.center) to (10.west);
		\draw[line width = 1.5pt] [in=0, out=180, looseness=1.5] (8.center) to (12.west);
		\draw[line width = 1.5pt] [in=0, out=180, looseness=1.5] (9.center) to (13.west);
		\draw[line width = 1.5pt] (0.west) to [in=0, out=0,looseness = 3] (4.west);
\begin{scope}[shift={(0,-4)}]
\begin{scope}[yscale=-1]
		\node [style=none] (0) at (1, -0) {};
		\node  [circle,draw,inner sep=1pt,fill]   (1) at (0.125, -0) {};
		\node [style=none] (2) at (-1, 0.5) {}; %
		\node [style=none] (3) at (-1, -0.5) {};
		\node [style=none] (4) at (1, -.5) {};
		\node  [circle,draw,inner sep=1pt,fill]   (5) at (0.125, -.5) {};
		\node [style=none] (6) at (-1, -0) {}; %
		\node [style=none] (7) at (-1, -1) {};
		\node [style=none] (8) at (-3, .5) {}; 
		\node [style=none] (9) at (-3, 1) {};
		\node [style=none] (10) at (-3, -2.5) {}; 
		\node [style=none] (11) at (-3, -3) {};
		\node [style=none] (12) at (-5, .5) {}; 
		\node [style=none] (13) at (-5, 1) {};
		\draw[line width = 1.5pt] (0.center) to (1.center);
		\draw[line width = 1.5pt] [in=0, out=120, looseness=1.20] (1.center) to (2.west);
		\draw[line width = 1.5pt] [in=0, out=-120, looseness=1.20] (1.center) to (3.west);
		\draw[line width = 1.5pt] (4.center) to (5.center);
		\draw[line width = 1.5pt] [in=0, out=120, looseness=1.20] (5.center) to (6.west);
		\draw[line width = 1.5pt] [in=0, out=-120, looseness=1.20] (5.center) to (7.west);
		\draw[line width = 1.5pt] [in=0, out=180, looseness=1.5] (2.center) to (9.center);
		\draw[line width = 1.5pt] [in=0, out=180, looseness=1.5] (6.center) to (8.center);
		\draw[line width = 1.5pt] [in=0, out=180, looseness=1] (7.center) to (11.west);
		\draw[line width = 1.5pt] [in=0, out=180, looseness=1] (3.center) to (10.west);
		\draw[line width = 1.5pt] [in=0, out=180, looseness=1.5] (8.center) to (12.west);
		\draw[line width = 1.5pt] [in=0, out=180, looseness=1.5] (9.center) to (13.west);
		\draw[line width = 1.5pt] (0.west) to [in=0, out=0,looseness = 3] (4.west);
\end{scope}
\end{scope}
\begin{scope}[rotate=180,shift={(4,2)}]
		\node [style=none] (0) at (1, -0) {};
		\node [style=none] (1) at (0.125, -0) {};
		\node [style=none] (2) at (-1, 0.5) {};
		\node [style=none] (3) at (-1, -0.5) {};
		\node [style=none] (4) at (1, -0) {};
		\node [style=none] (5) at (-1, 1) {};
		\node [style=none] (6) at (1, .25) {};
		\node [style=none] (7) at (-1, -1) {};
		\node [style=none] (8) at (1, -.25) {};
	\draw[line width = 1.5pt] (2) to [in=0, out=0,looseness = 3] (3.west);
	\draw[line width = 1.5pt] (5) to [in=180, out=0,looseness = 1] (6);
	\draw[line width = 1.5pt] (7) to [in=180, out=0, looseness = 1] (8);
\end{scope}
}
      \end{tikzpicture}
\end{aligned}
}

\newcommand{\weakbimonoidcotwol}[1]
{
\begin{aligned}
\begin{tikzpicture}[circuit ee IEC, set resistor graphic=var resistor IEC
      graphic,scale=.5]
\scalebox{1}{
\begin{scope}[rotate=180]
		\node [style=none] (0) at (1, -0) {};
		\node [style=none] (1) at (0.125, -0) {};
		\node [style=none] (2) at (-1, 0.5) {};
		\node [style=none] (3) at (-1, -0.5) {};
		\node [style=none] (4) at (1, -0) {};
		\node [style=none] (5) at (-1, 1) {};
		\node [style=none] (6) at (1, .25) {};
		\node [style=none] (7) at (-1, -1) {};
		\node [style=none] (8) at (1, -.25) {};
		\node [style=none] (9) at (1, 1.25) {};
		\node [style=none] (10) at (3, .5) {};
		\node [style=none] (11) at (1, 1.75) {};
		\node [style=none] (12) at (3, 1) {};
		\node [style=none] (13) at (-1, 2) {};
		\node [style=none] (14) at (-1, 2.5) {};
	\draw[line width = 1.5pt] (2) to [in=0, out=0,looseness = 3] (3.west);
	\draw[line width = 1.5pt] (5) to [in=0, out=0,looseness = 3] (13.west);
	\draw[line width = 1.5pt] (7) to [in=180, out=0, looseness = 1] (10);
	\draw[line width = 1.5pt] (14) to [in=180, out=0, looseness = 1] (12);
\end{scope}
\begin{scope}[shift={(-4,-.5)}]
		\node [style=none] (0) at (1, -0) {};
		\node [style=none] (1) at (-1, -0) {};
		\node  [circle,draw,inner sep=1pt,fill]   (2) at (0, -0) {};
		\node [style=none] (3) at (1, -.5) {};
		\node [style=none] (4) at (-1, -.5) {};
		\node  [circle,draw,inner sep=1pt,fill]   (5) at (0, -.5) {};
		\draw[line width = 1.5pt] (0.center) to (2);
		\draw[line width = 1.5pt] (3.center) to (5);
\end{scope}
}
      \end{tikzpicture}
\end{aligned}
}

\newcommand{\weakbimonoidcotwom}[1]
{
\begin{aligned}
\begin{tikzpicture}[circuit ee IEC, set resistor graphic=var resistor IEC
      graphic,scale=.5]
\scalebox{1}{
\begin{scope}[yscale=-1,rotate = 180]
		\node [style=none] (0) at (1, -0) {};
		\node [style=none] (1) at (0.125, -0) {};
		\node [style=none] (2) at (-1, 0.5) {};
		\node [style=none] (3) at (-1, -0.5) {};
		\node [style=none] (4) at (1, -0) {};
		\node [style=none] (5) at (-1, 1) {};
		\node [style=none] (6) at (1, .25) {};
		\node [style=none] (7) at (-1, -1) {};
		\node [style=none] (8) at (1, -.25) {};
		\node [style=none] (9) at (1, 1.25) {};
		\node [style=none] (10) at (3, .5) {};
		\node [style=none] (11) at (1, 1.75) {};
		\node [style=none] (12) at (3, 1) {};
		\node [style=none] (13) at (-2, 1.25) {};
		\node [style=none] (14) at (-2, 1.75) {};
		\node [style=none] (15) at (-0, .25) {};
		\node [style=none] (16) at (-0, -.25) {};
	\draw[line width = 1.5pt] (6) to [in=0, out=0,looseness = 3] (9.west);
	\draw[line width = 1.5pt] (8) to [in=180, out=0, looseness = 1] (10);
	\draw[line width = 1.5pt] (11) to [in=180, out=0, looseness = 1] (12);
	\draw[line width = 1.5pt] (13) to [in=180, out=0, looseness = 1] (9.west);
	\draw[line width = 1.5pt] (14) to [in=180, out=0, looseness = 1] (11.west);
	\draw[line width = 1.5pt] (6.west) to [in=0, out=180,looseness = 3] (15.east);
	\draw[line width = 1.5pt] (8.west) to [in=0, out=180,looseness = 3] (16.east);
\end{scope}
\begin{scope}[rotate = 180, shift={(0,1)}]
		\node [style=none] (0) at (1, -0) {};
		\node [style=none] (1) at (0.125, -0) {};
		\node [style=none] (2) at (-1, 0.5) {};
		\node [style=none] (3) at (-1, -0.5) {};
		\node [style=none] (4) at (1, -0) {};
		\node [style=none] (5) at (-1, 1) {};
		\node [style=none] (6) at (1, .25) {};
		\node [style=none] (7) at (-1, -1) {};
		\node [style=none] (8) at (1, -.25) {};
		\node [style=none] (9) at (1, 1.25) {};
		\node [style=none] (10) at (3, .5) {};
		\node [style=none] (11) at (1, 1.75) {};
		\node [style=none] (12) at (3, 1) {};
		\node [style=none] (13) at (-2, 1.25) {};
		\node [style=none] (14) at (-2, 1.75) {};
		\node [style=none] (15) at (-0, .25) {};
		\node [style=none] (16) at (-0, -.25) {};
	\draw[line width = 1.5pt] (6) to [in=0, out=0,looseness = 3] (9.west);
	\draw[line width = 1.5pt] (8) to [in=180, out=0, looseness = 1] (10);
	\draw[line width = 1.5pt] (11) to [in=180, out=0, looseness = 1] (12);
	\draw[line width = 1.5pt] (13) to [in=180, out=0, looseness = 1] (9.west);
	\draw[line width = 1.5pt] (14) to [in=180, out=0, looseness = 1] (11.west);
	\draw[line width = 1.5pt] (6.west) to [in=0, out=180,looseness = 3] (15.east);
	\draw[line width = 1.5pt] (8.west) to [in=0, out=180,looseness = 3] (16.east);
\end{scope}
\begin{scope}[shift={(1,-.25)}]
	\node [style=none] (0) at (1, -0) {};
		\node  [circle,draw,inner sep=1pt,fill]   (1) at (0.125, -0) {};
		\node [style=none] (2) at (-1, 0.5) {};
		\node [style=none] (3) at (-1, -0.5) {};
		\node [style=none] (4) at (1, -.5) {};
		\node  [circle,draw,inner sep=1pt,fill]   (5) at (0.125, -.5) {};
		\node [style=none] (6) at (-1, -0) {};
		\node [style=none] (7) at (-1, -1) {};
		\draw[line width = 1.5pt] (0.center) to (1.center);
		\draw[line width = 1.5pt] [in=0, out=120, looseness=1.20] (1.center) to (2.center);
		\draw[line width = 1.5pt] [in=0, out=-120, looseness=1.20] (1.center) to (3.center);
		\draw[line width = 1.5pt] (4.center) to (5.center);
		\draw[line width = 1.5pt] [in=0, out=120, looseness=1.20] (5.center) to (6.center);
		\draw[line width = 1.5pt] [in=0, out=-120, looseness=1.20] (5.center) to (7.center);
\end{scope}
\begin{scope}[shift={(-4,1)}]
		\node [style=none] (0) at (1, -0) {};
		\node  [circle,draw,inner sep=1pt,fill]   (2) at (0, -0) {};
		\node [style=none] (3) at (1, -.5) {};
		\node  [circle,draw,inner sep=1pt,fill]   (5) at (0, -.5) {};
		\draw[line width = 1.5pt] (0.center) to (2);
		\draw[line width = 1.5pt] (3.center) to (5);
\end{scope}
\begin{scope}[shift={(-4,-1.5)}]
		\node [style=none] (0) at (1, -0) {};
		\node  [circle,draw,inner sep=1pt,fill]   (2) at (0, -0) {};
		\node [style=none] (3) at (1, -.5) {};
		\node  [circle,draw,inner sep=1pt,fill]   (5) at (0, -.5) {};
		\draw[line width = 1.5pt] (0.center) to (2);
		\draw[line width = 1.5pt] (3.center) to (5);
\end{scope}
}
      \end{tikzpicture}
\end{aligned}
}

\newcommand{\weakbimonoidcotwor}[1]
{
\begin{aligned}
\begin{tikzpicture}[circuit ee IEC, set resistor graphic=var resistor IEC
      graphic,scale=.5]
\scalebox{1}{
\begin{scope}[yscale=-1,rotate = 180]
		\node [style=none] (0) at (1, -0) {};
		\node [style=none] (1) at (0.125, -0) {};
		\node [style=none] (2) at (-1, 0.5) {};
		\node [style=none] (3) at (-1, -0.5) {};
		\node [style=none] (4) at (1, -0) {};
		\node [style=none] (5) at (-1, 1) {};
		\node [style=none] (6) at (1, .25) {};
		\node [style=none] (7) at (-1, -1) {};
		\node [style=none] (8) at (1, -.25) {};
		\node [style=none] (9) at (1, 1.25) {};
		\node [style=none] (10) at (3, .5) {};
		\node [style=none] (11) at (1, 1.75) {};
		\node [style=none] (12) at (3, 1) {};
		\node [style=none] (13) at (-2, 1.25) {};
		\node [style=none] (14) at (-2, 1.75) {};
		\node [style=none] (15) at (-0, -.75) {};
		\node [style=none] (16) at (-0, -1.25) {};
	\draw[line width = 1.5pt] (6) to [in=0, out=0,looseness = 3] (9.west);
	\draw[line width = 1.5pt] (8) to [in=180, out=0, looseness = 1] (10);
	\draw[line width = 1.5pt] (11) to [in=180, out=0, looseness = 1] (12);
	\draw[line width = 1.5pt] (13) to [in=180, out=0, looseness = 1] (9.west);
	\draw[line width = 1.5pt] (14) to [in=180, out=0, looseness = 1] (11.west);
	\draw[line width = 1.5pt] (6.west) to [in=0, out=180,looseness = 1] (15.east);
	\draw[line width = 1.5pt] (8.west) to [in=0, out=180,looseness = 1] (16.east);
\end{scope}
\begin{scope}[rotate = 180, shift={(0,1)}]
		\node [style=none] (0) at (1, -0) {};
		\node [style=none] (1) at (0.125, -0) {};
		\node [style=none] (2) at (-1, 0.5) {};
		\node [style=none] (3) at (-1, -0.5) {};
		\node [style=none] (4) at (1, -0) {};
		\node [style=none] (5) at (-1, 1) {};
		\node [style=none] (6) at (1, .25) {};
		\node [style=none] (7) at (-1, -1) {};
		\node [style=none] (8) at (1, -.25) {};
		\node [style=none] (9) at (1, 1.25) {};
		\node [style=none] (10) at (3, .5) {};
		\node [style=none] (11) at (1, 1.75) {};
		\node [style=none] (12) at (3, 1) {};
		\node [style=none] (13) at (-2, 1.25) {};
		\node [style=none] (14) at (-2, 1.75) {};
		\node [style=none] (15) at (-0, -.75) {};
		\node [style=none] (16) at (-0, -1.25) {};
	\draw[line width = 1.5pt] (6) to [in=0, out=0,looseness = 3] (9.west);
	\draw[line width = 1.5pt] (8) to [in=180, out=0, looseness = 1] (10);
	\draw[line width = 1.5pt] (11) to [in=180, out=0, looseness = 1] (12);
	\draw[line width = 1.5pt] (13) to [in=180, out=0, looseness = 1] (9.west);
	\draw[line width = 1.5pt] (14) to [in=180, out=0, looseness = 1] (11.west);
	\draw[line width = 1.5pt] (6.west) to [in=0, out=180,looseness = 1] (15.east);
	\draw[line width = 1.5pt] (8.west) to [in=0, out=180,looseness = 1] (16.east);
\end{scope}
\begin{scope}[shift={(1,-.25)}]
	\node [style=none] (0) at (1, -0) {};
		\node  [circle,draw,inner sep=1pt,fill]   (1) at (0.125, -0) {};
		\node [style=none] (2) at (-1, 0.5) {};
		\node [style=none] (3) at (-1, -0.5) {};
		\node [style=none] (4) at (1, -.5) {};
		\node  [circle,draw,inner sep=1pt,fill]   (5) at (0.125, -.5) {};
		\node [style=none] (6) at (-1, -0) {};
		\node [style=none] (7) at (-1, -1) {};
		\draw[line width = 1.5pt] (0.center) to (1.center);
		\draw[line width = 1.5pt] [in=0, out=120, looseness=1.20] (1.center) to (2.center);
		\draw[line width = 1.5pt] [in=0, out=-120, looseness=1.20] (1.center) to (3.center);
		\draw[line width = 1.5pt] (4.center) to (5.center);
		\draw[line width = 1.5pt] [in=0, out=120, looseness=1.20] (5.center) to (6.center);
		\draw[line width = 1.5pt] [in=0, out=-120, looseness=1.20] (5.center) to (7.center);
\end{scope}
\begin{scope}[shift={(-4,1)}]
		\node [style=none] (0) at (1, -0) {};
		\node  [circle,draw,inner sep=1pt,fill]   (2) at (0, -0) {};
		\node [style=none] (3) at (1, -.5) {};
		\node  [circle,draw,inner sep=1pt,fill]   (5) at (0, -.5) {};
		\draw[line width = 1.5pt] (0.center) to (2);
		\draw[line width = 1.5pt] (3.center) to (5);
\end{scope}
\begin{scope}[shift={(-4,-1.5)}]
		\node [style=none] (0) at (1, -0) {};
		\node  [circle,draw,inner sep=1pt,fill]   (2) at (0, -0) {};
		\node [style=none] (3) at (1, -.5) {};
		\node  [circle,draw,inner sep=1pt,fill]   (5) at (0, -.5) {};
		\draw[line width = 1.5pt] (0.center) to (2);
		\draw[line width = 1.5pt] (3.center) to (5);
\end{scope}
}
      \end{tikzpicture}
\end{aligned}
}

\newcommand{\bigbimonoidone}[1]
{
\begin{aligned}
\begin{tikzpicture}[circuit ee IEC, set resistor graphic=var resistor IEC
      graphic,scale=.4]
\scalebox{1}{
		\node [style=none] (0) at (1, -0) {};
		\node  [circle,draw,inner sep=1pt,fill]   (1) at (0.125, -0) {};
		\node [style=none] (2) at (-1, 0.5) {};
		\node [style=none] (3) at (-1, -0.5) {};
		\node [style=none] (4) at (1, -.5) {};
		\node  [circle,draw,inner sep=1pt,fill]   (5) at (0.125, -.5) {};
		\node [style=none] (6) at (-1, -0) {};
		\node [style=none] (7) at (-1, -1) {};
		\node [style=none] (8) at (-7, -1) {};
		\node [style=none] (9) at (-7, -.5) {};
		\node [style=none] (10) at (-7, 3) {};
		\node [style=none] (11) at (-7, 3.5) {};
		\draw[line width = 1.5pt] (0.center) to (1.center);
		\draw[line width = 1.5pt] [in=0, out=120, looseness=1.20] (1.center) to (2.center);
		\draw[line width = 1.5pt] [in=0, out=-120, looseness=1.20] (1.center) to (3.center);
		\draw[line width = 1.5pt] (4.center) to (5.center);
		\draw[line width = 1.5pt] [in=0, out=120, looseness=1.20] (5.center) to (6.center);
		\draw[line width = 1.5pt] [in=0, out=-120, looseness=1.20] (5.center) to (7.center);
		\draw[line width = 1.5pt] (7.center) to (8.west);
		\draw[line width = 1.5pt] (3.center) to (9.west);
		\draw[line width = 1.5pt] [in=0, out=180, looseness=.75] (6.center) to (10.west);
		\draw[line width = 1.5pt] [in=0, out=180, looseness=.75] (2.center) to (11.west);
\begin{scope}[shift={(0,4.5)}]
		\node [style=none] (0) at (1, -0) {};
		\node  [circle,draw,inner sep=1pt,fill]   (1) at (0.125, -0) {};
		\node [style=none] (2) at (-1, 0.5) {};
		\node [style=none] (3) at (-1, -0.5) {};
		\node [style=none] (4) at (1, -.5) {};
		\node  [circle,draw,inner sep=1pt,fill]   (5) at (0.125, -.5) {};
		\node [style=none] (6) at (-1, -0) {};
		\node [style=none] (7) at (-1, -1) {};
		\node [style=none] (8) at (-7, 0) {};
		\node [style=none] (9) at (-7, .5) {};
		\node [style=none] (10) at (-7, -4) {};
		\node [style=none] (11) at (-7, -3.5) {};
		\draw[line width = 1.5pt] (0.center) to (1.center);
		\draw[line width = 1.5pt] [in=0, out=120, looseness=1.20] (1.center) to (2.center);
		\draw[line width = 1.5pt] [in=0, out=-120, looseness=1.20] (1.center) to (3.center);
		\draw[line width = 1.5pt] (4.center) to (5.center);
		\draw[line width = 1.5pt] [in=0, out=120, looseness=1.20] (5.center) to (6.center);
		\draw[line width = 1.5pt] [in=0, out=-120, looseness=1.20] (5.center) to (7.center);
		\draw[line width = 1.5pt] (6.center) to (8.west);
		\draw[line width = 1.5pt] (2.center) to (9.west);
		\draw[line width = 1.5pt] [in=0, out=180, looseness=.75] (7.center) to (10.west);
		\draw[line width = 1.5pt] [in=0, out=180, looseness=.75] (3.center) to (11.west);
\end{scope}
\begin{scope}[rotate=180, shift={(8,0)}]
		\node [style=none] (0) at (1, -0) {};
		\node [style=none] (1) at (0.125, -0) {};
		\node [style=none] (2) at (-1, 0.5) {};
		\node [style=none] (3) at (-1, -0.5) {};
		\node [style=none] (4) at (1, -0) {};
		\node [style=none] (5) at (-1, 1) {};
		\node [style=none] (6) at (1, .25) {};
		\node [style=none] (7) at (-1, -1) {};
		\node [style=none] (8) at (1, -.25) {};
	\draw[line width = 1.5pt] (2) to [in=0, out=0,looseness = 3] (3);
	\draw[line width = 1.5pt] (5) to [in=180, out=0,looseness = 1] (6);
	\draw[line width = 1.5pt] (7) to [in=180, out=0, looseness = 1] (8);
\end{scope}
\begin{scope}[rotate=180, shift={(8,-4)}]
		\node [style=none] (0) at (1, -0) {};
		\node [style=none] (1) at (0.125, -0) {};
		\node [style=none] (2) at (-1, 0.5) {};
		\node [style=none] (3) at (-1, -0.5) {};
		\node [style=none] (4) at (1, -0) {};
		\node [style=none] (5) at (-1, 1) {};
		\node [style=none] (6) at (1, .25) {};
		\node [style=none] (7) at (-1, -1) {};
		\node [style=none] (8) at (1, -.25) {};
	\draw[line width = 1.5pt] (2) to [in=0, out=0,looseness = 3] (3);
	\draw[line width = 1.5pt] (5) to [in=180, out=0,looseness = 1] (6);
	\draw[line width = 1.5pt] (7) to [in=180, out=0, looseness = 1] (8);
\end{scope}
}
      \end{tikzpicture}
\end{aligned}
}

\newcommand{\bigbimonoidtwo}[1]
{
\begin{aligned}
\begin{tikzpicture}[circuit ee IEC, set resistor graphic=var resistor IEC
      graphic,scale=.4]
\scalebox{1}{
		\node [style=none] (0) at (1, -0) {};
		\node  [circle,draw,inner sep=1pt,fill]   (1) at (0.125, -0) {};
		\node [style=none] (2) at (-1, 0.5) {};
		\node [style=none] (3) at (-1, -0.5) {};
		\node [style=none] (4) at (1, -.5) {};
		\node  [circle,draw,inner sep=1pt,fill]   (5) at (0.125, -.5) {};
		\node [style=none] (6) at (-1, -0) {};
		\node [style=none] (7) at (-1, -1) {};
		\node [style=none] (8) at (-7, -1) {};
		\node [style=none] (9) at (-3.5, -.5) {};
		\node [style=none] (10) at (-7, 3) {};
		\node [style=none] (11) at (-3.5, 3.5) {};
		\draw[line width = 1.5pt] (0.center) to (1.center);
		\draw[line width = 1.5pt] [in=0, out=120, looseness=1.20] (1.center) to (2.center);
		\draw[line width = 1.5pt] [in=0, out=-120, looseness=1.20] (1.center) to (3.center);
		\draw[line width = 1.5pt] (4.center) to (5.center);
		\draw[line width = 1.5pt] [in=0, out=120, looseness=1.20] (5.center) to (6.center);
		\draw[line width = 1.5pt] [in=0, out=-120, looseness=1.20] (5.center) to (7.center);
		\draw[line width = 1.5pt] (7.center) to (8.west);
		\draw[line width = 1.5pt] (3.center) to (9.west);
		\draw[line width = 1.5pt] [in=0, out=180, looseness=.75] (6.center) to (10.west);
		\draw[line width = 1.5pt] [in=0, out=180, looseness=.75] (2.center) to (11.west);
\begin{scope}[shift={(0,4.5)}]
		\node [style=none] (0) at (1, -0) {};
		\node  [circle,draw,inner sep=1pt,fill]   (1) at (0.125, -0) {};
		\node [style=none] (2) at (-1, 0.5) {};
		\node [style=none] (3) at (-1, -0.5) {};
		\node [style=none] (4) at (1, -.5) {};
		\node  [circle,draw,inner sep=1pt,fill]   (5) at (0.125, -.5) {};
		\node [style=none] (6) at (-1, -0) {};
		\node [style=none] (7) at (-1, -1) {};
		\node [style=none] (8) at (-3.5, 0) {};
		\node [style=none] (9) at (-7, .5) {};
		\node [style=none] (10) at (-3.5, -4) {};
		\node [style=none] (11) at (-7, -3.5) {};
		\draw[line width = 1.5pt] (0.center) to (1.center);
		\draw[line width = 1.5pt] [in=0, out=120, looseness=1.20] (1.center) to (2.center);
		\draw[line width = 1.5pt] [in=0, out=-120, looseness=1.20] (1.center) to (3.center);
		\draw[line width = 1.5pt] (4.center) to (5.center);
		\draw[line width = 1.5pt] [in=0, out=120, looseness=1.20] (5.center) to (6.center);
		\draw[line width = 1.5pt] [in=0, out=-120, looseness=1.20] (5.center) to (7.center);
		\draw[line width = 1.5pt] (6.center) to (8.west);
		\draw[line width = 1.5pt] (2.center) to (9.west);
		\draw[line width = 1.5pt] [in=0, out=180, looseness=.75] (7.center) to (10.center);
		\draw[line width = 1.5pt] [in=0, out=180, looseness=.75] (3.center) to (11.west);
\end{scope}
\begin{scope}[rotate=180, shift={(8,0)}]
		\node [style=none] (0) at (1, -0) {};
		\node [style=none] (1) at (0.125, -0) {};
		\node [style=none] (2) at (-4.5, 0.5) {};
		\node [style=none] (3) at (-4.5, -0.5) {};
		\node [style=none] (4) at (1, -0) {};
		\node [style=none] (5) at (-1, 1) {};
		\node [style=none] (6) at (1, .25) {};
		\node [style=none] (7) at (-1, -1) {};
		\node [style=none] (8) at (1, -.25) {};
	\draw[line width = 1.5pt] (2) to [in=0, out=0,looseness = 3] (3);
	\draw[line width = 1.5pt] (5) to [in=180, out=0,looseness = 1] (6);
	\draw[line width = 1.5pt] (7) to [in=180, out=0, looseness = 1] (8);
\end{scope}
\begin{scope}[rotate=180, shift={(8,-4)}]
		\node [style=none] (0) at (1, -0) {};
		\node [style=none] (1) at (0.125, -0) {};
		\node [style=none] (2) at (-4.5, 0.5) {};
		\node [style=none] (3) at (-4.5, -0.5) {};
		\node [style=none] (4) at (1, -0) {};
		\node [style=none] (5) at (-1, 1) {};
		\node [style=none] (6) at (1, .25) {};
		\node [style=none] (7) at (-1, -1) {};
		\node [style=none] (8) at (1, -.25) {};
	\draw[line width = 1.5pt] (2) to [in=0, out=0,looseness = 3] (3);
	\draw[line width = 1.5pt] (5) to [in=180, out=0,looseness = 1] (6);
	\draw[line width = 1.5pt] (7) to [in=180, out=0, looseness = 1] (8);
\end{scope}
}
      \end{tikzpicture}
\end{aligned}
}

\newcommand{\bigbimonoidthree}[1]
{
\begin{aligned}
\begin{tikzpicture}[circuit ee IEC, set resistor graphic=var resistor IEC
      graphic,scale=.4]
\scalebox{1}{
		\node [style=none] (0) at (1, 1) {};
		\node  [circle,draw,inner sep=1pt,fill]   (1) at (0.125, 1) {};
		\node [style=none] (2) at (-1, 1.5) {};
		\node [style=none] (3) at (-1, .5) {};
		\node [style=none] (4) at (1, -.5) {};
		\node  [circle,draw,inner sep=1pt,fill]   (5) at (0.125, -.5) {};
		\node [style=none] (6) at (-1, -0) {};
		\node [style=none] (7) at (-1, -1) {};
		\node [style=none] (8) at (-7, -1) {};
		\node [style=none] (9) at (-1.5, .5) {};
		\node [style=none] (10) at (-7, 3) {};
		\node [style=none] (11) at (-1.5, 2.5) {};
		\draw[line width = 1.5pt] (0.center) to (1.center);
		\draw[line width = 1.5pt] [in=0, out=120, looseness=1.20] (1.center) to (2.center);
		\draw[line width = 1.5pt] [in=0, out=-120, looseness=1.20] (1.center) to (3.center);
		\draw[line width = 1.5pt] (4.center) to (5.center);
		\draw[line width = 1.5pt] [in=0, out=120, looseness=1.20] (5.center) to (6.center);
		\draw[line width = 1.5pt] [in=0, out=-120, looseness=1.20] (5.center) to (7.center);
		\draw[line width = 1.5pt] (7.center) to (8.west);
		\draw[line width = 1.5pt] (3.center) to (9.west);
		\draw[line width = 1.5pt] [in=0, out=180, looseness=1] (6.center) to (10.west);
		\draw[line width = 1.5pt] [in=0, out=180, looseness=.8] (2.center) to (11.center);
\begin{scope}[shift={(0,4.5)}]
		\node [style=none] (0) at (1, -0) {};
		\node  [circle,draw,inner sep=1pt,fill]   (1) at (0.125, -0) {};
		\node [style=none] (2) at (-1, 0.5) {};
		\node [style=none] (3) at (-1, -0.5) {};
		\node [style=none] (4) at (1, -1.5) {};
		\node  [circle,draw,inner sep=1pt,fill]   (5) at (0.125, -1.5) {};
		\node [style=none] (6) at (-1, -1) {};
		\node [style=none] (7) at (-1, -2) {};
		\node [style=none] (8) at (-1.5, -1) {};
		\node [style=none] (9) at (-7, .5) {};
		\node [style=none] (10) at (-1.5, -3) {};
		\node [style=none] (11) at (-7, -3.5) {};
		\draw[line width = 1.5pt] (0.center) to (1.center);
		\draw[line width = 1.5pt] [in=0, out=120, looseness=1.20] (1.center) to (2.center);
		\draw[line width = 1.5pt] [in=0, out=-120, looseness=1.20] (1.center) to (3.center);
		\draw[line width = 1.5pt] (4.center) to (5.center);
		\draw[line width = 1.5pt] [in=0, out=120, looseness=1.20] (5.center) to (6.center);
		\draw[line width = 1.5pt] [in=0, out=-120, looseness=1.20] (5.center) to (7.center);
		\draw[line width = 1.5pt] (6.center) to (8.west);
		\draw[line width = 1.5pt] (2.center) to (9.west);
		\draw[line width = 1.5pt] [in=180, out=0, looseness=.8] (10.center) to (7.center);
		\draw[line width = 1.5pt] [in=0, out=180, looseness=1] (3.center) to (11.west);
\end{scope}
\begin{scope}[rotate=180, shift={(8,0)}]
		\node [style=none] (0) at (1, -0) {};
		\node [style=none] (1) at (0.125, -0) {};
		\node [style=none] (2) at (-6.5, -.5) {};
		\node [style=none] (3) at (-6.5, -1.5) {};
		\node [style=none] (4) at (1, -0) {};
		\node [style=none] (5) at (-1, 1) {};
		\node [style=none] (6) at (1, .25) {};
		\node [style=none] (7) at (-1, -1) {};
		\node [style=none] (8) at (1, -.25) {};
	\draw[line width = 1.5pt] (2.east) to [in=0, out=0,looseness = 3] (3.east);
	\draw[line width = 1.5pt] (5) to [in=180, out=0,looseness = 1] (6);
	\draw[line width = 1.5pt] (7) to [in=180, out=0, looseness = 1] (8);
\end{scope}
\begin{scope}[rotate=180, shift={(8,-4)}]
		\node [style=none] (0) at (1, -0) {};
		\node [style=none] (1) at (0.125, -0) {};
		\node [style=none] (2) at (-6.5, 1.5) {};
		\node [style=none] (3) at (-6.5, .5) {};
		\node [style=none] (4) at (1, -0) {};
		\node [style=none] (5) at (-1, 1) {};
		\node [style=none] (6) at (1, .25) {};
		\node [style=none] (7) at (-1, -1) {};
		\node [style=none] (8) at (1, -.25) {};
	\draw[line width = 1.5pt] (2.east) to [in=0, out=0,looseness = 3] (3.east);
	\draw[line width = 1.5pt] (5) to [in=180, out=0,looseness = 1] (6);
	\draw[line width = 1.5pt] (7) to [in=180, out=0, looseness = 1] (8);
\end{scope}
}
      \end{tikzpicture}
\end{aligned}
}

\newcommand{\bigbimonoidfour}[1]
{
\begin{aligned}
\begin{tikzpicture}[circuit ee IEC, set resistor graphic=var resistor IEC
      graphic,scale=.4]
\scalebox{1}{
		\node [style=none] (0) at (1, 1) {};
		\node  [circle,draw,inner sep=1pt,fill]   (1) at (0.125, 1) {};
		\node [style=none] (2) at (-1, 1.5) {};
		\node [style=none] (3) at (-1, .5) {};
		\node [style=none] (4) at (1, -.5) {};
		\node  [circle,draw,inner sep=1pt,fill]   (5) at (0.125, -.5) {};
		\node [style=none] (6) at (-1, -0) {};
		\node [style=none] (7) at (-1, -1) {};
		\node [style=none] (8) at (-7, -1) {};
		\node [style=none] (9) at (-1.5, .5) {};
		\node [style=none] (10) at (-7, 3) {};
		\node [style=none] (11) at (-1.5, 2.5) {};
		\node [style=none] (12) at (-1, 3.5) {};
		\node [style=none] (13) at (-1, 2.5) {};
		\draw[line width = 1.5pt] (0.center) to (1.center);
		\draw[line width = 1.5pt] [in=0, out=120, looseness=1.20] (1.center) to (2.center);
		\draw[line width = 1.5pt] [in=0, out=-120, looseness=1.20] (1.center) to (3.east);
		\draw[line width = 1.5pt] (4.center) to (5.center);
		\draw[line width = 1.5pt] [in=0, out=120, looseness=1.20] (5.center) to (6.center);
		\draw[line width = 1.5pt] [in=0, out=-120, looseness=1.20] (5.center) to (7.center);
		\draw[line width = 1.5pt] (7.center) to (8.west);
		\draw[line width = 1.5pt] [in=0, out=180, looseness=.75] (6.center) to (10.west);
		\draw[line width = 1.5pt] (12.east) to [in=180, out=180,looseness = 2.5] (3.east);
		\draw[line width = 1.5pt] (13.east) to [in=180, out=180,looseness = 3] (2.east);
\begin{scope}[shift={(0,4.5)}]
		\node [style=none] (0) at (1, -0) {};
		\node  [circle,draw,inner sep=1pt,fill]   (1) at (0.125, -0) {};
		\node [style=none] (2) at (-1, 0.5) {};
		\node [style=none] (3) at (-1, -0.5) {};
		\node [style=none] (4) at (1, -1.5) {};
		\node  [circle,draw,inner sep=1pt,fill]   (5) at (0.125, -1.5) {};
		\node [style=none] (6) at (-1, -1) {};
		\node [style=none] (7) at (-1, -2) {};
		\node [style=none] (8) at (-1.5, -1) {};
		\node [style=none] (9) at (-7, .5) {};
		\node [style=none] (10) at (-1.5, -3) {};
		\node [style=none] (11) at (-7, -3.5) {};
		\draw[line width = 1.5pt] (0.center) to (1.center);
		\draw[line width = 1.5pt] [in=0, out=120, looseness=1.20] (1.center) to (2.center);
		\draw[line width = 1.5pt] [in=0, out=-120, looseness=1.20] (1.center) to (3.center);
		\draw[line width = 1.5pt] (4.center) to (5.center);
		\draw[line width = 1.5pt] [in=0, out=120, looseness=1.20] (5.center) to (6.center);
		\draw[line width = 1.5pt] [in=0, out=-120, looseness=1.20] (5.center) to (7.center);
		\draw[line width = 1.5pt] (2.center) to (9.west);
		\draw[line width = 1.5pt] [in=0, out=180, looseness=.75] (3.center) to (11.west);
\end{scope}
\begin{scope}[rotate=180, shift={(8,0)}]
		\node [style=none] (0) at (1, -0) {};
		\node [style=none] (1) at (0.125, -0) {};
		\node [style=none] (2) at (-6.5, -.5) {};
		\node [style=none] (3) at (-6.5, -1.5) {};
		\node [style=none] (4) at (1, -0) {};
		\node [style=none] (5) at (-1, 1) {};
		\node [style=none] (6) at (1, .25) {};
		\node [style=none] (7) at (-1, -1) {};
		\node [style=none] (8) at (1, -.25) {};
	\draw[line width = 1.5pt] (5) to [in=180, out=0,looseness = 1] (6);
	\draw[line width = 1.5pt] (7) to [in=180, out=0, looseness = 1] (8);
\end{scope}
\begin{scope}[rotate=180, shift={(8,-4)}]
		\node [style=none] (0) at (1, -0) {};
		\node [style=none] (1) at (0.125, -0) {};
		\node [style=none] (2) at (-6.5, 1.5) {};
		\node [style=none] (3) at (-6.5, .5) {};
		\node [style=none] (4) at (1, -0) {};
		\node [style=none] (5) at (-1, 1) {};
		\node [style=none] (6) at (1, .25) {};
		\node [style=none] (7) at (-1, -1) {};
		\node [style=none] (8) at (1, -.25) {};
	\draw[line width = 1.5pt] (5) to [in=180, out=0,looseness = 1] (6);
	\draw[line width = 1.5pt] (7) to [in=180, out=0, looseness = 1] (8);
\end{scope}
}
      \end{tikzpicture}
\end{aligned}
}

\newcommand{\bigbimonoidfive}[1]
{
\begin{aligned}
\begin{tikzpicture}[circuit ee IEC, set resistor graphic=var resistor IEC
      graphic,scale=.4]
\scalebox{1}{
		\node [style=none] (0) at (1, 1) {};
		\node [style=none] (1) at (1, 3) {};
		\node [style=none] (2) at (-1, 1.5) {};
		\node [style=none] (3) at (-1, .5) {};
		\node [style=none] (4) at (1, -.5) {};
		\node  [circle,draw,inner sep=1pt,fill]   (5) at (0.125, -.5) {};
		\node [style=none] (6) at (-1, -0) {};
		\node [style=none] (7) at (-1, -1) {};
		\node [style=none] (8) at (-7, -1) {};
		\node [style=none] (9) at (-1.5, .5) {};
		\node [style=none] (10) at (-7, 3) {};
		\node [style=none] (11) at (-1.5, 2.5) {};
		\node [style=none] (12) at (-1, 3.5) {};
		\node [style=none] (13) at (-1, 2.5) {};
		\draw[line width = 1.5pt] (4.center) to (5.center);
		\draw[line width = 1.5pt] [in=0, out=120, looseness=1.20] (5.center) to (6.center);
		\draw[line width = 1.5pt] [in=0, out=-120, looseness=1.20] (5.center) to (7.center);
		\draw[line width = 1.5pt] (7.center) to (8.west);
		\draw[line width = 1.5pt] [in=0, out=180, looseness=.75] (6.center) to (10.west);
		\draw[line width = 1.5pt] (0) to [in=180, out=180,looseness = 4] (1);
\begin{scope}[shift={(0,4.5)}]
		\node [style=none] (0) at (1, -0) {};
		\node  [circle,draw,inner sep=1pt,fill]   (1) at (0.125, -0) {};
		\node [style=none] (2) at (-1, 0.5) {};
		\node [style=none] (3) at (-1, -0.5) {};
		\node [style=none] (8) at (-1.5, -1) {};
		\node [style=none] (9) at (-7, .5) {};
		\node [style=none] (10) at (-1.5, -3) {};
		\node [style=none] (11) at (-7, -3.5) {};
		\draw[line width = 1.5pt] (0.center) to (1.center);
		\draw[line width = 1.5pt] [in=0, out=120, looseness=1.20] (1.center) to (2.center);
		\draw[line width = 1.5pt] [in=0, out=-120, looseness=1.20] (1.center) to (3.center);
		\draw[line width = 1.5pt] (2.center) to (9.west);
		\draw[line width = 1.5pt] [in=0, out=180, looseness=.75] (3.center) to (11.west);
\end{scope}
\begin{scope}[rotate=180, shift={(8,0)}]
		\node [style=none] (0) at (1, -0) {};
		\node [style=none] (1) at (0.125, -0) {};
		\node [style=none] (2) at (-6.5, -.5) {};
		\node [style=none] (3) at (-6.5, -1.5) {};
		\node [style=none] (4) at (1, -0) {};
		\node [style=none] (5) at (-1, 1) {};
		\node [style=none] (6) at (1, .25) {};
		\node [style=none] (7) at (-1, -1) {};
		\node [style=none] (8) at (1, -.25) {};
	\draw[line width = 1.5pt] (5) to [in=180, out=0,looseness = 1] (6);
	\draw[line width = 1.5pt] (7) to [in=180, out=0, looseness = 1] (8);
\end{scope}
\begin{scope}[rotate=180, shift={(8,-4)}]
		\node [style=none] (0) at (1, -0) {};
		\node [style=none] (1) at (0.125, -0) {};
		\node [style=none] (2) at (-6.5, 1.5) {};
		\node [style=none] (3) at (-6.5, .5) {};
		\node [style=none] (4) at (1, -0) {};
		\node [style=none] (5) at (-1, 1) {};
		\node [style=none] (6) at (1, .25) {};
		\node [style=none] (7) at (-1, -1) {};
		\node [style=none] (8) at (1, -.25) {};
	\draw[line width = 1.5pt] (5) to [in=180, out=0,looseness = 1] (6);
	\draw[line width = 1.5pt] (7) to [in=180, out=0, looseness = 1] (8);
\end{scope}
}
      \end{tikzpicture}
\end{aligned}
}

\newcommand{\bigbimonoidsix}[1]
{
\begin{aligned}
\begin{tikzpicture}[circuit ee IEC, set resistor graphic=var resistor IEC
      graphic,scale=.4]
\scalebox{1}{
		\node [style=none] (0) at (1, 1) {};
		\node [style=none] (1) at (1, 3) {};
		\node [style=none] (2) at (-1, 1.5) {};
		\node [style=none] (3) at (-1, .5) {};
		\node [style=none] (4) at (1, 0) {};
		\node  [circle,draw,inner sep=1pt,fill]   (5) at (0.125-4, 0) {};
		\node [style=none] (6) at (-5, .5) {};
		\node [style=none] (7) at (-5, -.5) {};
		\node [style=none] (8) at (-7, -.5) {};
		\node [style=none] (9) at (-1.5, .5) {};
		\node [style=none] (10) at (-7, 2.5) {};
		\node [style=none] (11) at (-1.5, 2.5) {};
		\node [style=none] (12) at (-1, 3.5) {};
		\node [style=none] (13) at (-1, 2.5) {};
		\draw[line width = 1.5pt] (4.center) to (5.center);
		\draw[line width = 1.5pt] [in=0, out=120, looseness=1.20] (5.center) to (6.center);
		\draw[line width = 1.5pt] [in=0, out=-120, looseness=1.20] (5.center) to (7.center);
		\draw[line width = 1.5pt] (7.center) to (8.west);
		\draw[line width = 1.5pt] [in=0, out=180, looseness=.75] (6.center) to (10.west);
		\draw[line width = 1.5pt] (0) to [in=180, out=180,looseness = 6] (1);
\begin{scope}[shift={(0,4.5)}]
		\node [style=none] (0) at (1, -.5) {};
		\node  [circle,draw,inner sep=1pt,fill]   (1) at (0.125-4, -.5) {};
		\node [style=none] (2) at (-5, 0) {};
		\node [style=none] (3) at (-5, -1) {};
		\node [style=none] (8) at (-1.5, -1) {};
		\node [style=none] (9) at (-7, 0) {};
		\node [style=none] (10) at (-1.5, -3) {};
		\node [style=none] (11) at (-7, -3) {};
		\draw[line width = 1.5pt] (0.center) to (1.center);
		\draw[line width = 1.5pt] [in=0, out=120, looseness=1.20] (1.center) to (2.center);
		\draw[line width = 1.5pt] [in=0, out=-120, looseness=1.20] (1.center) to (3.center);
		\draw[line width = 1.5pt] (2.center) to (9.west);
		\draw[line width = 1.5pt] [in=0, out=180, looseness=.75] (3.center) to (11.west);
\end{scope}
\begin{scope}[rotate=180, shift={(8,-.5)}]
		\node [style=none] (0) at (1, -0) {};
		\node [style=none] (1) at (0.125, -0) {};
		\node [style=none] (2) at (-6.5, -.5) {};
		\node [style=none] (3) at (-6.5, -1.5) {};
		\node [style=none] (4) at (1, -0) {};
		\node [style=none] (5) at (-1, 1) {};
		\node [style=none] (6) at (1, .25) {};
		\node [style=none] (7) at (-1, -1) {};
		\node [style=none] (8) at (1, -.25) {};
	\draw[line width = 1.5pt] (5) to [in=180, out=0,looseness = 1] (6);
	\draw[line width = 1.5pt] (7) to [in=180, out=0, looseness = 1] (8);
\end{scope}
\begin{scope}[rotate=180, shift={(8,-3.5)}]
		\node [style=none] (0) at (1, -0) {};
		\node [style=none] (1) at (0.125, -0) {};
		\node [style=none] (2) at (-6.5, 1.5) {};
		\node [style=none] (3) at (-6.5, .5) {};
		\node [style=none] (4) at (1, -0) {};
		\node [style=none] (5) at (-1, 1) {};
		\node [style=none] (6) at (1, .25) {};
		\node [style=none] (7) at (-1, -1) {};
		\node [style=none] (8) at (1, -.25) {};
	\draw[line width = 1.5pt] (5) to [in=180, out=0,looseness = 1] (6);
	\draw[line width = 1.5pt] (7) to [in=180, out=0, looseness = 1] (8);
\end{scope}
}
      \end{tikzpicture}
\end{aligned}
}

\newcommand{\bigbimonoidseven}[1]
{
\begin{aligned}
\begin{tikzpicture}[circuit ee IEC, set resistor graphic=var resistor IEC
      graphic,scale=.4]
\scalebox{1}{
\begin{scope}[yscale=2]
		\node [style=none] (0) at (1, -0) {};
		\node  [circle,draw,inner sep=1pt,fill]   (1) at (0.125, -0) {};
		\node [style=none] (2) at (-1, 0.5) {};
		\node [style=none] (3) at (-1, -0.5) {};
		\node [style=none] (4) at (1, -.5) {};
		\node  [circle,draw,inner sep=1pt,fill]   (5) at (0.125, -.5) {};
		\node [style=none] (6) at (-1, -0) {};
		\node [style=none] (7) at (-1, -1) {};
		\node [style=none] (8) at (7, -0) {};
		\node [style=none] (9) at (7, -.5) {};
		\draw[line width = 1.5pt] (0.center) to (1.center);
		\draw[line width = 1.5pt] [in=0, out=120, looseness=1.20] (1.center) to (2.center);
		\draw[line width = 1.5pt] [in=0, out=-120, looseness=1.20] (1.center) to (3.center);
		\draw[line width = 1.5pt] (4.center) to (5.center);
		\draw[line width = 1.5pt] [in=0, out=120, looseness=1.20] (5.center) to (6.center);
		\draw[line width = 1.5pt] [in=0, out=-120, looseness=1.20] (5.center) to (7.center);
		\draw[line width = 1.5pt] (0.center) to (8.center);
		\draw[line width = 1.5pt] (4.center) to (9.center);
\begin{scope}[rotate=180, shift ={(-8,.25)}]
		\node [style=none] (0) at (1, -0) {};
		\node [style=none] (1) at (0.125, -0) {};
		\node [style=none] (2) at (-1, 0.5) {};
		\node [style=none] (3) at (-1, -0.5) {};
		\node [style=none] (4) at (1, -0) {};
		\node [style=none] (5) at (-1, 1) {};
		\node [style=none] (6) at (1, .25) {};
		\node [style=none] (7) at (-1, -1) {};
		\node [style=none] (8) at (1, -.25) {};
	\draw[line width = 1.5pt] (2) to [in=0, out=0,looseness = 3] (3);
	\draw[line width = 1.5pt] (5) to [in=180, out=0,looseness = 1] (6);
	\draw[line width = 1.5pt] (7) to [in=180, out=0, looseness = 1] (8);
\end{scope}
\end{scope}
}
      \end{tikzpicture}
\end{aligned}
}

\newcommand{\bigbimonoidstepthree}[1]
{
\begin{aligned}
\begin{tikzpicture}[circuit ee IEC, set resistor graphic=var resistor IEC
      graphic,scale=.5]
\scalebox{1}{
		\node [style=none] (0) at (1, 1) {};
		\node  [circle,draw,inner sep=1pt,fill]   (1) at (0.125, 1) {};
		\node [style=none] (2) at (-1, 1.5) {};
		\node [style=none] (3) at (-1, .5) {};
		\node [style=none] (9) at (-1.5, .5) {};
		\node [style=none] (11) at (-1.5, 2.5) {};
		\draw[line width = 1.5pt] (0.center) to (1.center);
		\draw[line width = 1.5pt] [in=0, out=120, looseness=1.20] (1.center) to (2.center);
		\draw[line width = 1.5pt] [in=0, out=-120, looseness=1.20] (1.center) to (3.center);
		\draw[line width = 1.5pt] (3.center) to (9.west);
		\draw[line width = 1.5pt] [in=0, out=180, looseness=.8] (2.center) to (11.center);
\begin{scope}[shift={(0,4.5)}]
		\node [style=none] (4) at (1, -1.5) {};
		\node  [circle,draw,inner sep=1pt,fill]   (5) at (0.125, -1.5) {};
		\node [style=none] (6) at (-1, -1) {};
		\node [style=none] (7) at (-1, -2) {};
		\node [style=none] (8) at (-1.5, -1) {};
		\node [style=none] (10) at (-1.5, -3) {};
		\draw[line width = 1.5pt] (4.center) to (5.center);
		\draw[line width = 1.5pt] [in=0, out=120, looseness=1.20] (5.center) to (6.center);
		\draw[line width = 1.5pt] [in=0, out=-120, looseness=1.20] (5.center) to (7.center);
		\draw[line width = 1.5pt] (6.center) to (8.west);
		\draw[line width = 1.5pt] [in=0, out=180, looseness=.8] (7.center) to (10.center);
\end{scope}
\begin{scope}[rotate=180, shift={(8,0)}]
		\node [style=none] (2) at (-6.5, -.5) {};
		\node [style=none] (3) at (-6.5, -1.5) {};
	\draw[line width = 1.5pt] (2.east) to [in=0, out=0,looseness = 3] (3.east);
\end{scope}
\begin{scope}[rotate=180, shift={(8,-4)}]
		\node [style=none] (2) at (-6.5, 1.5) {};
		\node [style=none] (3) at (-6.5, .5) {};
	\draw[line width = 1.5pt] (2.east) to [in=0, out=0,looseness = 3] (3.east);
\end{scope}
}
      \end{tikzpicture}
\end{aligned}
}

\newcommand{\bigbimonoidstepfour}[1]
{
\begin{aligned}
\begin{tikzpicture}[circuit ee IEC, set resistor graphic=var resistor IEC
      graphic,scale=.5]
\scalebox{1}{
		\node [style=none] (0) at (1, 1) {};
		\node  [circle,draw,inner sep=1pt,fill]   (1) at (0.125, 1) {};
		\node [style=none] (2) at (-1, 1.5) {};
		\node [style=none] (3) at (-1, .5) {};
		\node [style=none] (9) at (-1.5, .5) {};
		\node [style=none] (11) at (-1.5, 2.5) {};
		\node [style=none] (12) at (-1, 3.5) {};
		\node [style=none] (13) at (-1, 2.5) {};
		\draw[line width = 1.5pt] (0.center) to (1.center);
		\draw[line width = 1.5pt] (12.east) to [in=180, out=180,looseness = 2.5] (3.east);
		\draw[line width = 1.5pt] (13.east) to [in=180, out=180,looseness = 3] (2.east);
		\draw[line width = 1.5pt] [in=0, out=120, looseness=1.20] (1.center) to (2.center);
		\draw[line width = 1.5pt] [in=0, out=-120, looseness=1.20] (1.center) to (3.center);
\begin{scope}[shift={(0,4.5)}]
		\node [style=none] (4) at (1, -1.5) {};
		\node  [circle,draw,inner sep=1pt,fill]   (5) at (0.125, -1.5) {};
		\node [style=none] (6) at (-1, -1) {};
		\node [style=none] (7) at (-1, -2) {};
		\node [style=none] (8) at (-1.5, -1) {};
		\node [style=none] (10) at (-1.5, -3) {};
		\draw[line width = 1.5pt] (4.center) to (5.center);
		\draw[line width = 1.5pt] [in=0, out=120, looseness=1.20] (5.center) to (6.center);
		\draw[line width = 1.5pt] [in=0, out=-120, looseness=1.20] (5.center) to (7.center);
\end{scope}
}
      \end{tikzpicture}
\end{aligned}
}

\newcommand{\bigbimonoidstepfive}[1]
{
\begin{aligned}
\begin{tikzpicture}[circuit ee IEC, set resistor graphic=var resistor IEC
      graphic,scale=.5]
\scalebox{1}{
		\node [style=none] (0) at (-1, -0) {};
		\node  [circle,draw,inner sep=1pt,fill]   (1) at (-0.125, -0) {};
		\node [style=none] (2) at (1, 0.5) {};
		\node [style=none] (3) at (1, -0.5) {};
		\draw[line width = 1.5pt] (0.center) to (1.center);
		\draw[line width = 1.5pt] [in=180, out=60, looseness=1.20] (1.center) to (2.east);
		\draw[line width = 1.5pt] [in=180, out=-60, looseness=1.20] (1.center) to (3.east);
\begin{scope}[shift={(2,1)}]
		\node [style=none] (0) at (1, -0) {};
		\node [style=none] (1) at (0.125, -0) {};
		\node [style=none] (2) at (-1, 0.5) {};
		\node [style=none] (3) at (-1, -0.5) {};
		\node [style=none] (4) at (1, -0) {};
		\node [style=none] (5) at (-3, 0.5) {};
	\draw[line width = 1.5pt] (2) to [in=0, out=0,looseness = 4] (3);
	\draw[line width = 1.5pt] (2.east) to (5.center);
\end{scope}
}
      \end{tikzpicture}
\end{aligned}
}

\newcommand{\bimonoidextral}[1]
{
  \begin{aligned}
\begin{tikzpicture}[circuit ee IEC, set resistor graphic=var resistor IEC
      graphic,scale=.5]
\scalebox{1}{
		\node [style=none] (0) at (1, -0) {};
		\node [style=none] (1) at (-1, -0) {};
		\node  [circle,draw,inner sep=1pt,fill]   (2) at (0, -0) {};
		\node [style=none] (3) at (1, -.5) {};
		\node [style=none] (4) at (-1, -.5) {};
		\node  [circle,draw,inner sep=1pt,fill]   (5) at (0, -.5) {};
		\draw[line width = 1.5pt] (0.center) to (2);
		\draw[line width = 1.5pt] (3.center) to (5);
		\draw[line width = 1.5pt] (0.west) to [in=0, out=0,looseness = 3] (3.west);
}
      \end{tikzpicture}
  \end{aligned}
}

\newcommand{\bimonoidextram}[1]
{
  \begin{aligned}
\begin{tikzpicture}[circuit ee IEC, set resistor graphic=var resistor IEC
      graphic,scale=.5]
\scalebox{1}{
		\node [style=none] (0) at (1, -0) {};
		\node [style=none] (1) at (-1, -0) {};
		\node  [circle,draw,inner sep=1pt,fill]   (2) at (0, -0) {};
		\node [style=none] (3) at (1, -.5) {};
		\node [style=none] (4) at (-1, -.5) {};
		\node  [circle,draw,inner sep=1pt,fill]   (5) at (0, -.5) {};
		\node  [circle,draw,inner sep=1pt,fill]   (6) at (3, -.25) {};
		\node [style=none] (7) at (1.4, -.25) {};
		\draw[line width = 1.5pt] (0.center) to (2);
		\draw[line width = 1.5pt] (3.center) to (5);
		\draw[line width = 1.5pt] (6.center) to (7.west);
		\draw[line width = 1.5pt] (0.west) to [in=0, out=0,looseness = 3] (3.west);
}
      \end{tikzpicture}
  \end{aligned}
}

\newcommand{\bimonoidextrar}[1]
{
  \begin{aligned}
\begin{tikzpicture}[circuit ee IEC, set resistor graphic=var resistor IEC
      graphic,scale=.5]
\scalebox{1}{
		\node [style=none] (0) at (1, -0) {};
		\node [style=none] (1) at (-1, -0) {};
		\node  [circle,draw,inner sep=1pt,fill]   (2) at (0, -0) {};
		\node [style=none] (3) at (1, -.5) {};
		\node [style=none] (4) at (-1, -.5) {};
		\node  [circle,draw,inner sep=1pt,fill]   (5) at (0, -.5) {};
		\draw[line width = 1.5pt] (0.center) to (2);
		\draw[line width = 1.5pt] (3.center) to (5);
}
      \end{tikzpicture}
  \end{aligned}
}

\newcommand{\parcoml}[1]
{
\begin{aligned}
\begin{tikzpicture}[circuit ee IEC, set resistor graphic=var resistor IEC
      graphic,scale=.5]
\scalebox{1}{
		\node [style=none] (0) at (1, -0) {};
		\node  [circle,draw,inner sep=1pt,fill]   (1) at (0.125, -0) {};
		\node [style=none] (2) at (-1, 0.5) {};
		\node [style=none] (3) at (-1, -0.5) {};
		\node [style=none] (4) at (1, -.5) {};
		\node  [circle,draw,inner sep=1pt,fill]   (5) at (0.125, -.5) {};
		\node [style=none] (6) at (-1, -0) {};
		\node [style=none] (7) at (-1, -1) {};
		\node [style=none] (8) at (-3, .5) {};
		\node [style=none] (9) at (-3, -.5) {};
		\node [style=none] (10) at (-3, -0) {};
		\node [style=none] (11) at (-3, -1) {};
		\draw[line width = 1.5pt] (0.center) to (1.center);
		\draw[line width = 1.5pt] [in=0, out=120, looseness=1.20] (1.center) to (2.center);
		\draw[line width = 1.5pt] [in=0, out=-120, looseness=1.20] (1.center) to (3.center);
		\draw[line width = 1.5pt] (4.center) to (5.center);
		\draw[line width = 1.5pt] [in=0, out=120, looseness=1.20] (5.center) to (6.center);
		\draw[line width = 1.5pt] [in=0, out=-120, looseness=1.20] (5.center) to (7.center);
		\draw[line width = 1.5pt] [in=180, out=0, looseness=1] (9.center) to (2.center);
		\draw[line width = 1.5pt] [in=180, out=0, looseness=1] (8.center) to (3.center);
		\draw[line width = 1.5pt] [in=180, out=0, looseness=1] (11.center) to (6.center);
		\draw[line width = 1.5pt] [in=180, out=0, looseness=1] (10.center) to (7.center);
}
      \end{tikzpicture}
\end{aligned}
}

\newcommand{\parcomm}[1]
{
\begin{aligned}
\begin{tikzpicture}[circuit ee IEC, set resistor graphic=var resistor IEC
      graphic,scale=.5]
\scalebox{1}{
		\node [style=none] (0) at (1, -0) {};
		\node  [circle,draw,inner sep=1pt,fill]   (1) at (0.125, -0) {};
		\node [style=none] (2) at (-1, 0.5) {};
		\node [style=none] (3) at (-1, -0.5) {};
		\node [style=none] (4) at (1, -.5) {};
		\node  [circle,draw,inner sep=1pt,fill]   (5) at (0.125, -.5) {};
		\node [style=none] (6) at (-1, -0) {};
		\node [style=none] (7) at (-1, -1) {};
		\node [style=none] (8) at (-3, .5) {};
		\node [style=none] (9) at (-3, -.5) {};
		\node [style=none] (10) at (-3, -0) {};
		\node [style=none] (11) at (-3, -1) {};
		\draw[line width = 1.5pt] (0.center) to (1.center);
		\draw[line width = 1.5pt] [in=0, out=120, looseness=1.20] (1.center) to (2.center);
		\draw[line width = 1.5pt] [in=0, out=-120, looseness=1.20] (1.center) to (3.center);
		\draw[line width = 1.5pt] (4.center) to (5.center);
		\draw[line width = 1.5pt] [in=0, out=120, looseness=1.20] (5.center) to (6.center);
		\draw[line width = 1.5pt] [in=0, out=-120, looseness=1.20] (5.center) to (7.center);
		\draw[line width = 1.5pt] [in=180, out=0, looseness=1] (8.center) to (2.center);
		\draw[line width = 1.5pt] [in=180, out=0, looseness=1] (9.center) to (3.center);
		\draw[line width = 1.5pt] [in=180, out=0, looseness=1] (11.center) to (6.center);
		\draw[line width = 1.5pt] [in=180, out=0, looseness=1] (10.center) to (7.center);
}
      \end{tikzpicture}
\end{aligned}
}

\newcommand{\parcomr}[1]
{
\begin{aligned}
\begin{tikzpicture}[circuit ee IEC, set resistor graphic=var resistor IEC
      graphic,scale=.5]
\scalebox{1}{
		\node [style=none] (0) at (1, -0) {};
		\node  [circle,draw,inner sep=1pt,fill]   (1) at (0.125, -0) {};
		\node [style=none] (2) at (-1, 0.5) {};
		\node [style=none] (3) at (-1, -0.5) {};
		\node [style=none] (4) at (1, -.5) {};
		\node  [circle,draw,inner sep=1pt,fill]   (5) at (0.125, -.5) {};
		\node [style=none] (6) at (-1, -0) {};
		\node [style=none] (7) at (-1, -1) {};
		\node [style=none] (8) at (-3, .5) {};
		\node [style=none] (9) at (-3, -.5) {};
		\node [style=none] (10) at (-3, -0) {};
		\node [style=none] (11) at (-3, -1) {};
		\draw[line width = 1.5pt] (0.center) to (1.center);
		\draw[line width = 1.5pt] [in=0, out=120, looseness=1.20] (1.center) to (2.center);
		\draw[line width = 1.5pt] [in=0, out=-120, looseness=1.20] (1.center) to (3.center);
		\draw[line width = 1.5pt] (4.center) to (5.center);
		\draw[line width = 1.5pt] [in=0, out=120, looseness=1.20] (5.center) to (6.center);
		\draw[line width = 1.5pt] [in=0, out=-120, looseness=1.20] (5.center) to (7.center);
		\draw[line width = 1.5pt] [in=180, out=0, looseness=1] (8.center) to (2.center);
		\draw[line width = 1.5pt] [in=180, out=0, looseness=1] (9.center) to (3.center);
		\draw[line width = 1.5pt] [in=180, out=0, looseness=1] (10.center) to (6.center);
		\draw[line width = 1.5pt] [in=180, out=0, looseness=1] (11.center) to (7.center);
}
      \end{tikzpicture}
\end{aligned}
}

\newcommand{\spec}[1]
{
  \begin{aligned}
\begin{tikzpicture}[circuit ee IEC, set resistor graphic=var resistor IEC
      graphic,scale=.5]
\scalebox{1}{
		\node [style=none] (0) at (1.75, -0) {};
		\node  [circle,draw,inner sep=1pt,fill]   (1) at (0.75, -0) {};
		\node [style=none] (2) at (0, -0.5) {};
		\node [style=none] (3) at (0, 0.5) {};
		\node  [circle,draw,inner sep=1pt,fill]   (4) at (-0.75, -0) {};
		\node [style=none] (5) at (0, -0.5) {};
		\node [style=none] (6) at (-1.75, -0) {};
		\node [style=none] (7) at (0, 0.5) {};
		\draw[line width = 1.5pt] (0.center) to (1.center);
		\draw[line width = 1.5pt] [in=0, out=-120, looseness=1.20] (1.center) to (2.center);
		\draw[line width = 1.5pt] [in=0, out=120, looseness=1.20] (1.center) to (3.center);
		\draw[line width = 1.5pt] (6.center) to (4);
		\draw[line width = 1.5pt] [in=180, out=-60, looseness=1.20] (4) to (5.center);
		\draw[line width = 1.5pt] [in=180, out=60, looseness=1.20] (4) to (7.center);
}
\end{tikzpicture}
  \end{aligned}
}

\newcommand{\parspec}[1]
{
\begin{aligned}
\begin{tikzpicture}[circuit ee IEC, set resistor graphic=var resistor IEC
      graphic,scale=.5]
\scalebox{1}{
		\node [style=none] (0) at (1, -0) {};
		\node  [circle,draw,inner sep=1pt,fill]   (1) at (0.125, -0) {};
		\node [style=none] (2) at (-1, 0.5) {};
		\node [style=none] (3) at (-1, -0.5) {};
		\node [style=none] (4) at (1, -.5) {};
		\node  [circle,draw,inner sep=1pt,fill]   (5) at (0.125, -.5) {};
		\node [style=none] (6) at (-1, -0) {};
		\node [style=none] (7) at (-1, -1) {};
		\draw[line width = 1.5pt] (0.center) to (1.center);
		\draw[line width = 1.5pt] [in=0, out=120, looseness=1.20] (1.center) to (2.west);
		\draw[line width = 1.5pt] [in=0, out=-120, looseness=1.20] (1.center) to (3.west);
		\draw[line width = 1.5pt] (4.center) to (5.center);
		\draw[line width = 1.5pt] [in=0, out=120, looseness=1.20] (5.center) to (6.west);
		\draw[line width = 1.5pt] [in=0, out=-120, looseness=1.20] (5.center) to (7.west);
\begin{scope}[shift={(-2,-.5)}, rotate = 180]
		\node [style=none] (0) at (1, -0) {};
		\node  [circle,draw,inner sep=1pt,fill]   (1) at (0.125, -0) {};
		\node [style=none] (2) at (-1, 0.5) {};
		\node [style=none] (3) at (-1, -0.5) {};
		\node [style=none] (4) at (1, -.5) {};
		\node  [circle,draw,inner sep=1pt,fill]   (5) at (0.125, -.5) {};
		\node [style=none] (6) at (-1, -0) {};
		\node [style=none] (7) at (-1, -1) {};
		\draw[line width = 1.5pt] (0.center) to (1.center);
		\draw[line width = 1.5pt] [in=0, out=120, looseness=1.20] (1.center) to (2.center);
		\draw[line width = 1.5pt] [in=0, out=-120, looseness=1.20] (1.center) to (3.center);
		\draw[line width = 1.5pt] (4.center) to (5.center);
		\draw[line width = 1.5pt] [in=0, out=120, looseness=1.20] (5.center) to (6.center);
		\draw[line width = 1.5pt] [in=0, out=-120, looseness=1.20] (5.center) to (7.center);
\end{scope}
}
      \end{tikzpicture}
\end{aligned}
}

\newcommand{\parextra}[1]
{
  \begin{aligned}
\begin{tikzpicture}[circuit ee IEC, set resistor graphic=var resistor IEC
      graphic,scale=.5]
\scalebox{1}{
		\node [style=none] (0) at (1, -0) {};
		\node [style=none] (1) at (-1, -0) {};
		\node  [circle,draw,inner sep=1pt,fill]   (2) at (0, -0) {};
		\node [style=none] (3) at (1, -.5) {};
		\node [style=none] (4) at (-1, -.5) {};
		\node  [circle,draw,inner sep=1pt,fill]   (5) at (0, -.5) {};
		\draw[line width = 1.5pt] (0.center) to (2.center);
		\draw[line width = 1.5pt] (3.center) to (5.center);
\begin{scope}[shift={(1.5,-.5)}, rotate = 180]
		\node [style=none] (0) at (1, -0) {};
		\node [style=none] (1) at (-1, -0) {};
		\node  [circle,draw,inner sep=1pt,fill]   (2) at (0, -0) {};
		\node [style=none] (3) at (1, -.5) {};
		\node [style=none] (4) at (-1, -.5) {};
		\node  [circle,draw,inner sep=1pt,fill]   (5) at (0, -.5) {};
		\draw[line width = 1.5pt] (0.center) to (2);
		\draw[line width = 1.5pt] (3.center) to (5);
\end{scope}
}
      \end{tikzpicture}
  \end{aligned}
}

\newcommand{\monadextra}[1]
{
\begin{aligned}
\begin{tikzpicture}[circuit ee IEC, set resistor graphic=var resistor IEC
      graphic,scale=.5]
\scalebox{1}{
		\node [style=none] (0) at (1, -0) {};
		\node [style=none] (1) at (0.125, -0) {};
		\node [style=none] (2) at (-1, 0.5) {};
		\node [style=none] (3) at (-1, -0.5) {};
		\node [style=none] (4) at (1, -0) {};
	\draw[line width = 1.5pt] (2.east) to [in=0, out=0,looseness = 3] (3.west);
	\draw[line width = 1.5pt] (2.east) to [in=180, out=180,looseness = 3] (3.west);
}
      \end{tikzpicture}
\end{aligned}
}

\newcommand{\monadextram}[1]
{
\begin{aligned}
\begin{tikzpicture}[circuit ee IEC, set resistor graphic=var resistor IEC
      graphic,scale=.5]
\scalebox{1}{
		\node [style=none] (0) at (1, -0) {};
		\node  [circle,draw,inner sep=1pt,fill]   (1) at (0.125, -0) {};
		\node [style=none] (2) at (-1, 0.5) {};
		\node [style=none] (3) at (-1, -0.5) {};
		\node  [circle,draw,inner sep=1pt,fill]   (4) at (1, -0) {};
		\draw[line width = 1.5pt] (0.center) to (1.center);
		\draw[line width = 1.5pt] [in=0, out=120, looseness=1.20] (1.center) to (2.center);
		\draw[line width = 1.5pt] [in=0, out=-120, looseness=1.20] (1.center) to (3.center);
\begin{scope}[shift={(-2,0)}, rotate = 180]
		\node [style=none] (0) at (1, -0) {};
		\node  [circle,draw,inner sep=1pt,fill]   (1) at (0.125, -0) {};
		\node [style=none] (2) at (-1, 0.5) {};
		\node [style=none] (3) at (-1, -0.5) {};
		\node  [circle,draw,inner sep=1pt,fill]   (4) at (1, -0) {};
		\draw[line width = 1.5pt] (0.center) to (1.center);
		\draw[line width = 1.5pt] [in=0, out=120, looseness=1.20] (1.center) to (2.center);
		\draw[line width = 1.5pt] [in=0, out=-120, looseness=1.20] (1.center) to (3.center);
\end{scope}
}
      \end{tikzpicture}
\end{aligned}
}

\newcommand{\monadspec}[1]
{
\begin{aligned}
\begin{tikzpicture}[circuit ee IEC, set resistor graphic=var resistor IEC
      graphic,scale=.5]
\scalebox{1}{
		\node [style=none] (0) at (1, -0) {};
		\node [style=none] (1) at (0.125, -0) {};
		\node [style=none] (2) at (-1, 0.5) {};
		\node [style=none] (3) at (-1, -0.5) {};
		\node [style=none] (4) at (1, -0) {};
		\node [style=none] (5) at (-1, 1) {};
		\node [style=none] (6) at (1, .25) {};
		\node [style=none] (7) at (-1, -1) {};
		\node [style=none] (8) at (1, -.25) {};
	\draw[line width = 1.5pt] (2.west) to [in=0, out=0,looseness = 3] (3.west);
	\draw[line width = 1.5pt] (5.west) to [in=180, out=0,looseness = 1] (6);
	\draw[line width = 1.5pt] (7.west) to [in=180, out=0, looseness = 1] (8);
\begin{scope}[shift={(-2,0)}, rotate = 180]
		\node [style=none] (0) at (1, -0) {};
		\node [style=none] (1) at (0.125, -0) {};
		\node [style=none] (2) at (-1, 0.5) {};
		\node [style=none] (3) at (-1, -0.5) {};
		\node [style=none] (4) at (1, -0) {};
		\node [style=none] (5) at (-1, 1) {};
		\node [style=none] (6) at (1, .25) {};
		\node [style=none] (7) at (-1, -1) {};
		\node [style=none] (8) at (1, -.25) {};
	\draw[line width = 1.5pt] (2) to [in=0, out=0,looseness = 3] (3);
	\draw[line width = 1.5pt] (5) to [in=180, out=0,looseness = 1] (6);
	\draw[line width = 1.5pt] (7) to [in=180, out=0, looseness = 1] (8);
\end{scope}
}
      \end{tikzpicture}
\end{aligned}
}

\newcommand{\extral}[1]
{
  \begin{aligned}
\begin{tikzpicture}[circuit ee IEC, set resistor graphic=var resistor IEC
      graphic,scale=.5]
\scalebox{1}{
		\node [style=none] (0) at (1.75, -0) {};
		\node  [circle,draw,inner sep=1pt,fill]   (1) at (0.75, -0) {};
		\node  [circle,draw,inner sep=1pt,fill]   (4) at (-0.75, -0) {};
		\node [style=none] (6) at (-1.75, -0) {};
	  \draw[line width = 1.5pt] (1.center) to (4.center);
}
\end{tikzpicture}
  \end{aligned}
}

\newcommand{\extrar}[1]
{
  \begin{aligned}
\begin{tikzpicture}[circuit ee IEC, set resistor graphic=var resistor IEC
      graphic,scale=.5]
\scalebox{1}{
		\node [style=none] (0) at (1.75, -0) {};
		\node [style=none] (6) at (-1.75, -0) {};
}
\end{tikzpicture}
  \end{aligned}
}

\newcommand{\strangelawl}[1]
{
\begin{aligned}
\begin{tikzpicture}[circuit ee IEC, set resistor graphic=var resistor IEC
      graphic,scale=.5]
\scalebox{1}{
		\node [style=none] (0) at (1, -0) {};
		\node [style=none] (1) at (0.125, -0) {};
		\node [style=none] (2) at (-1, 0.5) {};
		\node [style=none] (3) at (-1, -0.5) {};
		\node [style=none] (4) at (1, -0) {};
		\node [style=none] (5) at (-1, 1) {};
		\node [style=none] (6) at (1, .25) {};
		\node [style=none] (7) at (-1, -1) {};
		\node [style=none] (8) at (1, -.25) {};
		\node [style=none] (9) at (-2, .75) {};
		\node [style=none] (10) at (-2, -.75) {};
		\node [style=none] (11) at (-2, .25) {};
		\node [style=none] (12) at (-2, -.25) {};
	\draw[line width = 1.5pt] (2) to [in=0, out=0,looseness = 3] (3);
	\draw[line width = 1.5pt] (5) to [in=180, out=0,looseness = 1] (6);
	\draw[line width = 1.5pt] (7) to [in=180, out=0, looseness = 1] (8);
	\draw[line width = 1.5pt] (5.east) to [in=0, out=180,looseness = 1] (9.west);
	\draw[line width = 1.5pt] (7.east) to [in=0, out=180, looseness = 1] (10.west);
	\draw[line width = 1.5pt] (2.east) to [in=0, out=180,looseness = 1] (11.west);
	\draw[line width = 1.5pt] (3.east) to [in=0, out=180, looseness = 1] (12.west);
\begin{scope}[shift={(-3,-.25)}]
\begin{scope}[rotate=180]
		\node [style=none] (0) at (1, -0) {};
		\node  [circle,draw,inner sep=1pt,fill]   (1) at (0.125, -0) {};
		\node [style=none] (2) at (-1, 0.5) {};
		\node [style=none] (3) at (-1, -0.5) {};
		\node [style=none] (4) at (1, -.5) {};
		\node  [circle,draw,inner sep=1pt,fill]   (5) at (0.125, -.5) {};
		\node [style=none] (6) at (-1, -0) {};
		\node [style=none] (7) at (-1, -1) {};
		\draw[line width = 1.5pt] (0.center) to (1.center);
		\draw[line width = 1.5pt] [in=0, out=120, looseness=1.20] (1.center) to (2.center);
		\draw[line width = 1.5pt] [in=0, out=-120, looseness=1.20] (1.center) to (3.center);
		\draw[line width = 1.5pt] (4.center) to (5.center);
		\draw[line width = 1.5pt] [in=0, out=120, looseness=1.20] (5.center) to (6.center);
		\draw[line width = 1.5pt] [in=0, out=-120, looseness=1.20] (5.center) to (7.center);
\end{scope}
\end{scope}
}
      \end{tikzpicture}
\end{aligned}
}

\newcommand{\strangelawr}[1]
{
\begin{aligned}
\begin{tikzpicture}[circuit ee IEC, set resistor graphic=var resistor IEC
      graphic,scale=.5]
\scalebox{1}{
\begin{scope}[rotate=180]
		\node [style=none] (0) at (1, -0) {};
		\node [style=none] (1) at (0.125, -0) {};
		\node [style=none] (2) at (-1, 0.5) {};
		\node [style=none] (3) at (-1, -0.5) {};
		\node [style=none] (4) at (1, -0) {};
		\node [style=none] (5) at (-1, 1) {};
		\node [style=none] (6) at (1, .25) {};
		\node [style=none] (7) at (-1, -1) {};
		\node [style=none] (8) at (1, -.25) {};
		\node [style=none] (9) at (-2, .75) {};
		\node [style=none] (10) at (-2, -.75) {};
		\node [style=none] (11) at (-2, .25) {};
		\node [style=none] (12) at (-2, -.25) {};
	\draw[line width = 1.5pt] (2) to [in=0, out=0,looseness = 3] (3);
	\draw[line width = 1.5pt] (5) to [in=180, out=0,looseness = 1] (6);
	\draw[line width = 1.5pt] (7) to [in=180, out=0, looseness = 1] (8);
	\draw[line width = 1.5pt] (5.west) to [in=0, out=180,looseness = 1] (9.east);
	\draw[line width = 1.5pt] (7.west) to [in=0, out=180, looseness = 1] (10.east);
	\draw[line width = 1.5pt] (2.west) to [in=0, out=180,looseness = 1] (11.east);
	\draw[line width = 1.5pt] (3.east) to [in=0, out=180, looseness = 1] (12.east);
\begin{scope}[shift={(-3,-.25)}]
\begin{scope}[rotate=180]
		\node [style=none] (0) at (1, -0) {};
		\node  [circle,draw,inner sep=1pt,fill]   (1) at (0.125, -0) {};
		\node [style=none] (2) at (-1, 0.5) {};
		\node [style=none] (3) at (-1, -0.5) {};
		\node [style=none] (4) at (1, -.5) {};
		\node  [circle,draw,inner sep=1pt,fill]   (5) at (0.125, -.5) {};
		\node [style=none] (6) at (-1, -0) {};
		\node [style=none] (7) at (-1, -1) {};
		\draw[line width = 1.5pt] (0.center) to (1.center);
		\draw[line width = 1.5pt] [in=0, out=120, looseness=1.20] (1.center) to (2.center);
		\draw[line width = 1.5pt] [in=0, out=-120, looseness=1.20] (1.center) to (3.center);
		\draw[line width = 1.5pt] (4.center) to (5.center);
		\draw[line width = 1.5pt] [in=0, out=120, looseness=1.20] (5.center) to (6.center);
		\draw[line width = 1.5pt] [in=0, out=-120, looseness=1.20] (5.center) to (7.center);
\end{scope}
\end{scope}
\end{scope}
}
      \end{tikzpicture}
\end{aligned}
}

\def\sigflowadd{
   \node[plus] (adder) {};
   \node[coordinate] (f) at (-0.5,0.65) {};
   \node[coordinate] (g) at (0.5,0.65) {};
   \node (out) [right of=adder] {};
   \node (pref) [left of=f] {};
   \node (preg) [left of=g] {};
   \draw[rounded corners,thick] (pref) -- (f) -- (adder.addend);
   \draw[rounded corners,thick] (preg) -- (g) -- (adder.summand);
   \draw[thick] (adder) -- (out);
}

\def\sigflowaddsideways{
\begin{scope}[rotate=90]
\sigflowadd
\end{scope}
}

\newcommand{\sigflowpicadd}[1]
{
  \begin{aligned}
    \resizebox{#1}{!}{
\begin{tikzpicture}[thick]
	\begin{pgfonlayer}{nodelayer}
\sigflowaddsideways
	\end{pgfonlayer}
\end{tikzpicture}
    }
  \end{aligned}
}

\def\sigflowcoadd{
   \node[coplus] (adder) {};
   \node[coordinate] (g) at (-0.5,0.65) {};
   \node[coordinate] (f) at (0.5,0.65) {};
   \node (out) [left of=adder] {};
   \node (pref) [right of=f] {};
   \node (preg) [right of=g] {};
   \draw[rounded corners,thick] (pref) -- (f) -- (adder.addend);
   \draw[rounded corners,thick] (preg) -- (g) -- (adder.summand);
   \draw[thick] (adder) -- (out);
}

\def\sigflowcoaddsideways{
\begin{scope}[rotate=90]
\sigflowcoadd
\end{scope}
}

\newcommand{\sigflowpiccoadd}[1]
{
  \begin{aligned}
    \resizebox{#1}{!}{
\begin{tikzpicture}[thick, xscale = -1]
	\begin{pgfonlayer}{nodelayer}
\sigflowcoaddsideways
	\end{pgfonlayer}
\end{tikzpicture}
    }
  \end{aligned}
}

\def\sigflowdup{
   \node[delta] (dupe){};
   \node[coordinate] (o2) at (-0.5,-0.65) {};
   \node[coordinate] (o1) at (0.5,-0.65) {};
   \node (in) [right of=dupe] {};
   \node (posto1) [left of=o1] {};
   \node (posto2) [left of=o2] {};
   \draw[rounded corners,thick] (posto1) -- (o1) -- (dupe.left out);
   \draw[rounded corners,thick] (posto2) -- (o2) -- (dupe.right out);
   \draw[thick] (in) -- (dupe);
}

\def\sigflowdupsideways{
\begin{scope}[rotate=90]
\sigflowdup
\end{scope}
}

\newcommand{\sigflowpiccodup}[1]
{
  \begin{aligned}
    \resizebox{#1}{!}{
\begin{tikzpicture}[thick, xscale = -1]
\begin{pgfonlayer}{nodelayer}
\sigflowdupsideways
	\end{pgfonlayer}
\end{tikzpicture}
    }
  \end{aligned}
}

\def\sigflowcodup{
   \node[codelta] (dupe){};
   \node[coordinate] (o1) at (-0.5,-0.65) {};
   \node[coordinate] (o2) at (0.5,-0.65) {};
   \node (in) [left of=dupe] {};
   \node (posto1) [right of=o1] {};
   \node (posto2) [right of=o2] {};
   \draw[rounded corners,thick] (posto1) -- (o1) -- (dupe.left out);
   \draw[rounded corners,thick] (posto2) -- (o2) -- (dupe.right out);
   \draw[thick] (in) -- (dupe);
}

\def\sigflowcodupsideways{
\begin{scope}[rotate=90]
\sigflowcodup
\end{scope}
}

\newcommand{\sigflowpicdup}[1]
{
  \begin{aligned}
    \resizebox{#1}{!}{
\begin{tikzpicture}[thick]
\begin{pgfonlayer}{nodelayer}
\sigflowcodupsideways
	\end{pgfonlayer}
\end{tikzpicture}
    }
  \end{aligned}
}

\def\SigMult{
\begin{scope}[shift={(0,0)}, xscale=1]
\sigflowaddsideways
\end{scope}
\begin{scope}[shift={(0,1.5)}, xscale=-1]
\sigflowdupsideways
\end{scope}
\begin{scope}[shift={(-1.95,.75)}, scale=.63, rotate=90]
\sigflowswap
\end{scope}
\begin{scope}[rotate=90]
   \node[coordinate] (a) at (-.501175,1.15) {};
   \node[coordinate] (b) at (-.501175,2.35) {};
   \draw[rounded corners,thick] (b) -- (a);
   \node[coordinate] (c) at (2.0001,1.15) {};
   \node[coordinate] (d) at (2.0001,2.35) {};
   \draw[rounded corners,thick] (d) -- (c);
\end{scope}
}

\newcommand{\SigMultpic}[1]
{
  \begin{aligned}
    \resizebox{#1}{!}{
\begin{tikzpicture}[thick]
\begin{pgfonlayer}{nodelayer}
  \SigMult
	\end{pgfonlayer}
\end{tikzpicture}
    }
  \end{aligned}
}

\def\SigCoMult{
\begin{scope}[shift={(0,-1.5)}, xscale=-1]
\sigflowcoaddsideways
\end{scope}
\begin{scope}[shift={(0,0)}, xscale=1]
\sigflowcodupsideways
\end{scope}
\begin{scope}[shift={(1.95,-.75)}, scale=.63, rotate = -90]
\sigflowswap
\end{scope}
\begin{scope}[rotate=-90]
   \node[coordinate] (a) at (-.5,1.15) {};
   \node[coordinate] (b) at (-.5,2.35) {};
   \draw[rounded corners,thick] (b) -- (a);
   \node[coordinate] (c) at (2,1.15) {};
   \node[coordinate] (d) at (2,2.35) {};
   \draw[rounded corners,thick] (d) -- (c);
\end{scope}
}

\newcommand{\SigCoMultpic}[1]
{
  \begin{aligned}
    \resizebox{#1}{!}{
\begin{tikzpicture}[thick]
\begin{pgfonlayer}{nodelayer}
  \SigCoMult
	\end{pgfonlayer}
\end{tikzpicture}
    }
  \end{aligned}
}

\def\SigOtherMult{
\begin{scope}[shift={(0,0)}, xscale=-1]
\sigflowdupsideways
\end{scope}
\begin{scope}[shift={(0,1.5)}, xscale=1]
\sigflowaddsideways
\end{scope}
\begin{scope}[shift={(-1.95,.75)}, scale=.63, rotate=90]
\sigflowswap
\end{scope}
\begin{scope}[rotate=90]
   \node[coordinate] (a) at (-.501175,1.15) {};
   \node[coordinate] (b) at (-.501175,2.35) {};
   \draw[rounded corners,thick] (b) -- (a);
   \node[coordinate] (c) at (2.0001,1.15) {};
   \node[coordinate] (d) at (2.0001,2.35) {};
   \draw[rounded corners,thick] (d) -- (c);
\end{scope}
}

\newcommand{\SigOtherMultpic}[1]
{
  \begin{aligned}
    \resizebox{#1}{!}{
\begin{tikzpicture}[thick]
\begin{pgfonlayer}{nodelayer}
  \SigOtherMult
	\end{pgfonlayer}
\end{tikzpicture}
    }
  \end{aligned}
}

\def\SigOtherCoMult{
\begin{scope}[shift={(0,-1.5)}, xscale=1]
\sigflowcodupsideways
\end{scope}
\begin{scope}[shift={(0,0)}, xscale=-1]
\sigflowcoaddsideways
\end{scope}
\begin{scope}[shift={(1.95,-.75)}, scale=.63, rotate = -90]
\sigflowswap
\end{scope}
\begin{scope}[rotate=-90]
   \node[coordinate] (a) at (-.5,1.15) {};
   \node[coordinate] (b) at (-.5,2.35) {};
   \draw[rounded corners,thick] (b) -- (a);
   \node[coordinate] (c) at (2,1.15) {};
   \node[coordinate] (d) at (2,2.35) {};
   \draw[rounded corners,thick] (d) -- (c);
\end{scope}
}

\newcommand{\SigOtherCoMultpic}[1]
{
  \begin{aligned}
    \resizebox{#1}{!}{
\begin{tikzpicture}[thick]
\begin{pgfonlayer}{nodelayer}
  \SigOtherCoMult
	\end{pgfonlayer}
\end{tikzpicture}
    }
  \end{aligned}
}

\def\sigflowzero{
   \node (out1) {};
   \node [zero] (ins1) [left of=out1, shift={(-.2,0)}] {};
   \draw (out1) -- (ins1);
}

\newcommand{\sigflowpiczero}[1]
{
  \begin{aligned}
    \resizebox{#1}{!}{
\begin{tikzpicture}[thick, node distance=0.85cm]
	\begin{pgfonlayer}{nodelayer}
\sigflowzero
	\end{pgfonlayer}
\end{tikzpicture}
    }
  \end{aligned}
}

\def\sigflowcozero{
   \node (out1) {};
   \node [zero] (ins1) [right of=out1, shift={(.2,0)}] {};
   \draw (out1) -- (ins1);
}

\newcommand{\sigflowpiccozero}[1]
{
  \begin{aligned}
    \resizebox{#1}{!}{
\begin{tikzpicture}[thick, node distance=0.85cm]
	\begin{pgfonlayer}{nodelayer}
\sigflowcozero
	\end{pgfonlayer}
\end{tikzpicture}
    }
  \end{aligned}
}

\def\sigflowdel{
   \node (in1) {};
   \node [bang] (del1) [left of=in1, shift={(-.2,0)}] {};
   \draw (in1) -- (del1);
}

\newcommand{\sigflowpiccodel}[1]
{
  \begin{aligned}
    \resizebox{#1}{!}{
\begin{tikzpicture}[thick, node distance=0.85cm]
	\begin{pgfonlayer}{nodelayer}
\sigflowdel
	\end{pgfonlayer}
\end{tikzpicture}
    }
  \end{aligned}
}

\def\sigflowcodel{
   \node (in1) {};
   \node [bang] (del1) [right of=in1, shift={(.2,0)}] {};
   \draw (in1) -- (del1);
}

\newcommand{\sigflowpicdel}[1]
{
  \begin{aligned}
    \resizebox{#1}{!}{
\begin{tikzpicture}[thick, node distance=0.85cm]
	\begin{pgfonlayer}{nodelayer}
\sigflowcodel
	\end{pgfonlayer}
\end{tikzpicture}
    }
  \end{aligned}
}

\def\SigUnit{
\begin{scope}[shift={(0,0)}, xscale=.3]
\sigflowzero
\end{scope}
\begin{scope}[shift={(0,.5)}, xscale=-.3]
\sigflowdel
\end{scope}
}

\newcommand{\SigUnitpic}[1]
{
  \begin{aligned}
    \resizebox{#1}{!}{
\begin{tikzpicture}[thick, node distance=0.85cm]
	\begin{pgfonlayer}{nodelayer}
\SigUnit
	\end{pgfonlayer}
\end{tikzpicture}
    }
  \end{aligned}
}

\def\SigCoUnit{
\begin{scope}[shift={(0,0)}, xscale=.5]
\sigflowcozero
\end{scope}
\begin{scope}[shift={(0,.5)}, xscale=-.5]
\sigflowcodel
\end{scope}
}

\newcommand{\SigCoUnitpic}[1]
{
  \begin{aligned}
    \resizebox{#1}{!}{
\begin{tikzpicture}[thick, node distance=0.85cm]
	\begin{pgfonlayer}{nodelayer}
\SigCoUnit
	\end{pgfonlayer}
\end{tikzpicture}
    }
  \end{aligned}
}

\def\SigOtherUnit{
\begin{scope}[shift={(0,0)}, xscale=-.3]
\sigflowdel
\end{scope}
\begin{scope}[shift={(0,.5)}, xscale=.3]
\sigflowzero
\end{scope}
}

\newcommand{\SigOtherUnitpic}[1]
{
  \begin{aligned}
    \resizebox{#1}{!}{
\begin{tikzpicture}[thick, node distance=0.85cm]
	\begin{pgfonlayer}{nodelayer}
\SigOtherUnit
	\end{pgfonlayer}
\end{tikzpicture}
    }
  \end{aligned}
}

\def\SigOtherCoUnit{
\begin{scope}[shift={(0,0)}, xscale=-.5]
\sigflowcodel
\end{scope}
\begin{scope}[shift={(0,.5)}, xscale=.5]
\sigflowcozero
\end{scope}
}

\newcommand{\SigOtherCoUnitpic}[1]
{
  \begin{aligned}
    \resizebox{#1}{!}{
\begin{tikzpicture}[thick, node distance=0.85cm]
	\begin{pgfonlayer}{nodelayer}
\SigOtherCoUnit
	\end{pgfonlayer}
\end{tikzpicture}
    }
  \end{aligned}
}

\def\sigflowId{
   \node (in1) {};
   \node [coordinate] (del1) [right =2 of in1, shift={(.4,0)}] {};
   \draw (in1) -- (del1);
}

\def\SigDoubleId{
\begin{scope}[shift={(0,0)}, xscale=1]
\sigflowId
\end{scope}
\begin{scope}[shift={(0,.5)}, xscale=1]
\sigflowId
\end{scope}
}

\newcommand{\SigDoubleIdpic}[1]
{
  \begin{aligned}
    \resizebox{#1}{!}{
\begin{tikzpicture}[thick, node distance=0.85cm]
	\begin{pgfonlayer}{nodelayer}
\SigDoubleId
	\end{pgfonlayer}
\end{tikzpicture}
    }
  \end{aligned}
}

\def\SigLabel{
   \node[coordinate] (in) {};
   \node [multiply] (mult) [right of=in] {$\mathrm{c}$};
   \node[coordinate] (out) [right of=mult] {};
   \node (label) [below of=out, shift={(.1,0)}] {};
   \draw  [thick] (in) -- (mult) -- (out);
}

\newcommand{\SigLabelpic}[1]
{
  \begin{aligned}
    \resizebox{#1}{!}{
\begin{tikzpicture}[thick, node distance=0.85cm]
	\begin{pgfonlayer}{nodelayer}
\SigLabel
	\end{pgfonlayer}
\end{tikzpicture}
    }
  \end{aligned}
}

\def\sigflowswap{
   \node (UpUpLeft) at (-0.4,0.9) {};
   \node [coordinate] (UpLeft) at (-0.4,0.4) {};
   \node (mid) at (0,0) {};
   \node [coordinate] (DownRight) at (0.4,-0.4) {};
   \node (DownDownRight) at (0.4,-0.9) {};
   \node [coordinate] (UpRight) at (0.4,0.4) {};
   \node (UpUpRight) at (0.4,0.9) {};
   \node [coordinate] (DownLeft) at (-0.4,-0.4) {};
   \node (DownDownLeft) at (-0.4,-0.9) {};
   \draw [rounded corners=2mm,thick] (UpUpLeft) -- (UpLeft) -- (mid) --
   (DownRight) -- (DownDownRight) (UpUpRight) -- (UpRight) -- (DownLeft) -- (DownDownLeft);
}

\def\SigLabelEdge{
\begin{scope}[shift={(0,0)}, xscale=-1]
\sigflowcoaddsideways
\end{scope}
\begin{scope}[shift={(4,2)}, xscale=-1]
\sigflowdupsideways
\end{scope}
   \node[bang2] (mult) at (2,1) {$Z$};
   \draw[rounded corners = 2mm,thick] (posto2.east) -- (posto2.west) -- (mult);
   \draw[rounded corners = 2mm,thick] (pref.west) -- (pref.east) -- (mult);
   \draw[rounded corners = 2mm,line width= 1.pt] (preg.west) -- +(3.75,0);
   \draw[rounded corners = 2mm,line width= 1.2pt] (posto1.east) -- +(-3.75,0);
}

\newcommand{\SigLabelEdgepic}[1]
{
  \begin{aligned}
    \resizebox{#1}{!}{
\begin{tikzpicture}[thick]
	\begin{pgfonlayer}{nodelayer}
\SigLabelEdge
	\end{pgfonlayer}
\end{tikzpicture}
    }
  \end{aligned}
}

\def\BondtoSigMult{   
\node[plus] (adder) {};
   \node[coordinate] (f) at (-0.65, -0.5) {};
   \node[coordinate] (g) at (-0.65, 0.5) {};
   \node (out) [coordinate] [right of=adder] {};
   \node (pref) [left of=f] {};
   \node (preg) [left of=g] {};
   \draw[rounded corners,thick] (pref) -- (f) -- (adder.addend);
   \draw[rounded corners,thick] (preg) -- (g) -- (adder.summand);
   \draw[thick] (adder) -- (out);

   \node[coordinate] (fa) [below = .25 of f] {};
   \node[coordinate] (fa2) [below= .25 of fa] {};
   \node (postofa) [left of=fa] {};
   \node (postofa2) [left of=fa2] {};
   \node [coordinate] (addera) [below =.25 of adder] {};
   \node [coordinate] (adderb) [below =.25 of addera] {};
   \node [coordinate] (infa) [right of=addera] {};
   \node [coordinate] (infb) [right of=adderb] {};
   \node (outfa) [right = 1.25 of infa] {};
   \node (outfb) [right = 1.25 of infb] {};
   \draw[rounded corners,thick] (postofa) -- (fa) -- (addera) -- (infa) -- (outfa);
   \draw[rounded corners,thick] (postofa2) -- (fa2) -- (adderb) -- (infb) -- (outfb);

   \node [zero] (ins1) [right =.25 of out] {};
   \draw[rounded corners,thick] (out) -- (ins1);

\begin{scope}[shift={(0,1.5)}]
   \node[delta] (dupe){};
   \node[coordinate] (o2) at (-0.65,-0.5) {};
   \node[coordinate] (o1) at (-0.65,0.5) {};
   \node (in) [coordinate] [right of=dupe] {};
   \node (posto1) [left of=o1] {};
   \node (posto2) [left of=o2] {};
   \draw[rounded corners,thick] (posto1) -- (o1) -- (dupe.left out);
   \draw[rounded corners,thick] (posto2) -- (o2) -- (dupe.right out);
   \draw[thick] (in) -- (dupe);
   \node[coordinate] (o1a) [above = .25 of o1] {};
   \node[coordinate] (o2a) [above= .25 of o1a] {};
   \node (posto1a) [left of=o1a] {};
   \node (posto2a) [left of=o2a] {};
   \node [coordinate] (dupea) [above =.25 of dupe] {};
   \node [coordinate] (dupeb) [above =.25 of dupea] {};
   \node (ina)  [coordinate] [right of=dupea] {};
   \node (inb)  [coordinate] [right of=dupeb] {};
   \node (outa)  [right = 1.25 of ina] {};
   \node (outb)  [right = 1.25 of inb] {};
   \draw[rounded corners,thick] (posto1a) -- (o1a) -- (dupea) -- (ina) -- (outa);
   \draw[rounded corners,thick] (posto2a) -- (o2a) -- (dupeb) -- (inb) -- (outb);

   \node [bang] (del1) [right = .25 of in] {};
   \draw[rounded corners,thick] (in) -- (del1);

\end{scope}

   \node[coordinate] (a) at (-1.15,-.501175) {};
   \node[coordinate] (b) at (-2.35, -.501175) {};
   \draw[rounded corners,thick] (b) -- (a);
   \node[coordinate] (a1) [below= .25 of a] {};
   \node[coordinate] (b1) [below= .25 of b] {};
   \draw[rounded corners,thick] (b1) -- (a1);
   \node[coordinate] (a2) [below= .25 of a1]{};
   \node[coordinate] (b2) [below= .25 of b1] {};
   \draw[rounded corners,thick] (b2) -- (a2);
  \node[coordinate] (c) at (-1.15, 2.0001) {};
   \node[coordinate] (d) at (-2.35, 2.0001) {};
   \draw[rounded corners,thick] (d) -- (c);
   \node[coordinate] (c1) [above= .25 of c] {};
   \node[coordinate] (d1) [above= .25 of d] {};
   \draw[rounded corners,thick] (d1) -- (c1);
   \node[coordinate] (c2) [above= .25 of c1]{};
   \node[coordinate] (d2) [above= .25 of d1] {};
   \draw[rounded corners,thick] (d2) -- (c2);
\begin{scope}[shift={(-1.95,.75)}, scale=.63, rotate=90]
\sigflowswap
\end{scope}
}

\newcommand{\BondtoSigMultpic}[1]
{
  \begin{aligned}
    \resizebox{#1}{!}{
\begin{tikzpicture}[thick]
\begin{pgfonlayer}{nodelayer}
\BondtoSigMult
	\end{pgfonlayer}
\end{tikzpicture}
    }
  \end{aligned}
}

\newcommand{\BondtoSigCoMultpic}[1]
{
  \begin{aligned}
    \resizebox{#1}{!}{
\begin{tikzpicture}[thick, xscale=-1, every node/.style={transform shape}]
\begin{pgfonlayer}{nodelayer}
\BondtoSigMult
	\end{pgfonlayer}
\end{tikzpicture}
    }
  \end{aligned}
}

\def\BondtoSigOtherMult{   
\node[plus] (adder) {};
   \node[coordinate] (f) at (-0.65, -0.5) {};
   \node[coordinate] (g) at (-0.65, 0.5) {};
   \node (out) [coordinate] [right of=adder] {};
   \node (pref) [left of=f] {};
   \node (preg) [coordinate] [left of=g] {};
   \draw[rounded corners,thick] (pref) -- (f) -- (adder.addend);
   \draw[rounded corners,thick] (preg) -- (g) -- (adder.summand);
   \draw[thick] (adder) -- (out);

\begin{scope}[shift={(0,1.5)}]
   \node[delta] (dupe){};
   \node[coordinate] (o2) at (-0.65,-0.5) {};
   \node[coordinate] (o1) at (-0.65,0.5) {};
   \node (in) [coordinate] [right of=dupe] {};
   \node (posto1) [left of=o1] {};
   \node (posto2) [coordinate] [left of=o2] {};
   \draw[rounded corners,thick] (posto1) -- (o1) -- (dupe.left out);
   \draw[rounded corners,thick] (posto2) -- (o2) -- (dupe.right out);
   \draw[thick] (in) -- (dupe);

\end{scope}
   \node[coordinate] (a) at (-1.15,-.501175) {};
   \node[coordinate] (b) at (-2.35, -.501175) {};
   \draw[rounded corners,thick] (b) -- (a);
  \node[coordinate] (c) at (-1.15, 2.0001) {};
   \node[coordinate] (d) at (-2.35, 2.0001) {};
   \draw[rounded corners,thick] (d) -- (c);

    \draw[overdraw, rounded corners = 2mm,line width= 1.2pt] (posto1.east) -- +(-2.5,0);
    \draw[overdraw, rounded corners = 2mm,line width= 1.2pt] (preg.east) -- +(-.5,0) -- +(-1,1.25) -- +(-1.75,1.25) -- +(-2.35,1.25);

    \draw[overdraw, rounded corners = 2mm,line width= 1.2pt] (posto2.east) -- +(-1.25,0) -- +(-1.5,-1.75) -- +(-1.75,-2.5) -- +(-2.35,-2.5);
    \draw[overdraw, rounded corners = 2mm,line width= 1.2pt] (b.east) -- +(-.5,0) -- +(-1,-1.25) -- +(-1.65,-1.25);

\begin{scope}[shift={(0,-3)}]
\node[plus] (adder) {};
   \node[coordinate] (f) at (-0.65, -0.5) {};
   \node[coordinate] (g) at (-0.65, 0.5) {};
   \node (out) [coordinate] [right of=adder] {};
   \node (pref) [left of=f] {};
   \node (preg) [left of=g] {};
   \draw[rounded corners,thick] (pref) -- (f) -- (adder.addend);
   \draw[rounded corners,thick] (preg) -- (g) -- (adder.summand);
   \draw[thick] (adder) -- (out);

\begin{scope}[shift={(0,1.5)}]
   \node[delta] (dupe){};
   \node[coordinate] (o2) at (-0.65,-0.5) {};
   \node[coordinate] (o1) at (-0.65,0.5) {};
   \node (in) [coordinate] [right of=dupe] {};
   \node (posto1) [left of=o1] {};
   \node (posto2) [left of=o2] {};
   \draw[rounded corners,thick] (posto1) -- (o1) -- (dupe.left out);
   \draw[rounded corners,thick] (posto2) -- (o2) -- (dupe.right out);
   \draw[thick] (in) -- (dupe);

\end{scope}
   \node[coordinate] (a) at (-1.15,-.501175) {};
   \node[coordinate] (b) at (-2.35, -.501175) {};
   \draw[rounded corners,thick] (b) -- (a);
  \node[coordinate] (c) at (-1.15, 2.0001) {};
   \node[coordinate] (d) at (-2.35, 2.0001) {};
   \draw[rounded corners,thick] (d) -- (c);

    \draw[overdraw, rounded corners = 2mm,line width= 1.2pt] (pref.east) -- +(-2.5,0);
    \draw[overdraw, rounded corners = 2mm,line width= 1.2pt] (posto2.east) -- +(-.5,0) -- +(-1,-1.25) -- +(-1.75,-1.25) -- +(-2.5,-1.25);

    \draw[overdraw, rounded corners = 2mm,line width= 1.2pt] (d.east) -- +(-.5,0) -- +(-1,1.25) -- +(-1.65,1.25);
    \draw[overdraw, rounded corners = 2mm,line width= 1.2pt] (preg.east) -- +(-1.4,0) -- +(-1.65,1.75) -- +(-1.9,2.5) -- +(-2.50,2.5);

\end{scope}
}

\newcommand{\BondtoSigOtherMultpic}[1]
{
  \begin{aligned}
    \resizebox{#1}{!}{
\begin{tikzpicture}[thick]
\begin{pgfonlayer}{nodelayer}
\BondtoSigOtherMult
	\end{pgfonlayer}
\end{tikzpicture}
    }
  \end{aligned}
}

\newcommand{\BondtoSigOtherCoMultpic}[1]
{
  \begin{aligned}
    \resizebox{#1}{!}{
\begin{tikzpicture}[thick, xscale=-1, every node/.style={transform shape}]
\begin{pgfonlayer}{nodelayer}
\BondtoSigOtherMult
	\end{pgfonlayer}
\end{tikzpicture}
    }
  \end{aligned}
}

\def\BondtoSigCoUnit{   
\node[plus] (adder) {};
   \node[coordinate] (f) at (-0.65, -0.5) {};
   \node[coordinate] (g) at (-0.65, 0.5) {};
   \node (out) [coordinate] [right of=adder] {};
   \node (pref) [left of=f] {};
   \node (preg) [left of=g] {};
   \draw[rounded corners,thick] (pref) -- (f) -- (adder.addend);
   \draw[rounded corners,thick] (preg) -- (g) -- (adder.summand);
   \draw[thick] (adder) -- (out);

   \node [zero] (ins1) [right =.25 of out] {};
   \draw[rounded corners,thick] (out) -- (ins1);

\begin{scope}[shift={(0,1.5)}]
   \node[delta] (dupe){};
   \node[coordinate] (o2) at (-0.65,-0.5) {};
   \node[coordinate] (o1) at (-0.65,0.5) {};
   \node (in) [coordinate] [right of=dupe] {};
   \node (posto1) [left of=o1] {};
   \node (posto2) [left of=o2] {};
   \draw[rounded corners,thick] (posto1) -- (o1) -- (dupe.left out);
   \draw[rounded corners,thick] (posto2) -- (o2) -- (dupe.right out);
   \draw[thick] (in) -- (dupe);

   \node [bang] (del1) [right = .25 of in] {};
   \draw[rounded corners,thick] (in) -- (del1);
\end{scope}

 \node[coordinate] (a) at (-1.15,-.501175) {};
   \node[coordinate] (b) at (-2.35, -.501175) {};
   \draw[rounded corners,thick] (b) -- (a);
  \node[coordinate] (c) at (-1.15, 2.0001) {};
   \node[coordinate] (d) at (-2.35, 2.0001) {};
   \draw[rounded corners,thick] (d) -- (c);

\begin{scope}[shift={(-1.95,.75)}, scale=.63, rotate=90]
\sigflowswap
\end{scope}

}

\newcommand{\BondtoSigCoUnitpic}[1]
{
  \begin{aligned}
    \resizebox{#1}{!}{
\begin{tikzpicture}
\begin{pgfonlayer}{nodelayer}
\BondtoSigCoUnit
	\end{pgfonlayer}
\end{tikzpicture}
    }
  \end{aligned}
}

\newcommand{\BondtoSigUnitpic}[1]
{
  \begin{aligned}
    \resizebox{#1}{!}{
\begin{tikzpicture}[thick, xscale=-1, every node/.style={transform shape}]
\begin{pgfonlayer}{nodelayer}
\BondtoSigCoUnit
	\end{pgfonlayer}
\end{tikzpicture}
    }
  \end{aligned}
}

\def\BondtoSigOtherUnit{
\begin{scope}[shift={(0,0)}, xscale=.3]
\sigflowzero
\end{scope}
\begin{scope}[shift={(0,.25)}, xscale=-.3]
\sigflowdel
\end{scope}
\begin{scope}[shift={(0,.5)}, xscale=.3]
\sigflowzero
\end{scope}
\begin{scope}[shift={(0,.75)}, xscale=-.3]
\sigflowdel
\end{scope}
}

\newcommand{\BondtoSigOtherUnitpic}[1]
{
  \begin{aligned}
    \resizebox{#1}{!}{
\begin{tikzpicture}
\begin{pgfonlayer}{nodelayer}
\BondtoSigOtherUnit
	\end{pgfonlayer}
\end{tikzpicture}
    }
  \end{aligned}
}

\def\BondtoSigOtherCoUnit{
\begin{scope}[shift={(0,0)}, xscale=.5]
\sigflowcozero
\end{scope}
\begin{scope}[shift={(0,.25)}, xscale=-.5]
\sigflowcodel
\end{scope}
\begin{scope}[shift={(0,.5)}, xscale=.5]
\sigflowcozero
\end{scope}
\begin{scope}[shift={(0,.75)}, xscale=-.5]
\sigflowcodel
\end{scope}
}

\newcommand{\BondtoSigOtherCoUnitpic}[1]
{
  \begin{aligned}
    \resizebox{#1}{!}{
\begin{tikzpicture}
\begin{pgfonlayer}{nodelayer}
\BondtoSigOtherCoUnit
	\end{pgfonlayer}
\end{tikzpicture}
    }
  \end{aligned}
}

\def\BondtoCircuitConverter{
\begin{scope}[shift={(0,-1.5)}, xscale=1]
\sigflowcodupsideways
\end{scope}
\begin{scope}[shift={(0,0)}, xscale=-1]
\sigflowcoaddsideways
\end{scope}
\begin{scope}[shift={(1.95,-.75)}, scale=.63, rotate = -90]
\sigflowswap
\end{scope}
\begin{scope}[rotate=-90]
   \node[bang2] (a) at (-.5,1.15) {$-1$};
   \node[coordinate] (b) at (-.5,2.35) {};
   \draw[rounded corners,thick] (b) -- (a);
   \node[bang2] (c) at (2,1.15) {$-1$};
   \node[coordinate] (d) at (2,2.35) {};
   \draw[rounded corners,thick] (d) -- (c);
\end{scope}
}

\newcommand{\BondtoCircuitConverterpic}[1]
{
  \begin{aligned}
    \resizebox{#1}{!}{
\begin{tikzpicture}
\begin{pgfonlayer}{nodelayer}
\BondtoCircuitConverter
	\end{pgfonlayer}
\end{tikzpicture}
    }
  \end{aligned}
}

\newcommand{\lcircpropexample}[1]
{
  \begin{aligned}
\begin{tikzpicture}[circuit ee IEC, set resistor graphic=var resistor IEC
      graphic,scale=.5]
\scalebox{1}{
		\node [style=none] (1) at (1, -1) {};
		\node [style=none] (3) at (-1, -1) {};
		\node [style=none] (2) at (1, 0) {};
		\node [style=none] (4) at (-1, 0) {};
		\node [style=none] (7) at (1, 1) {};
		\node [style=none] (8) at (-1, 1) {};
   		\path (1) edge  (3);
    		\path (2) edge  (4);
 		\path (7) edge  (8);
		\node [style=none] (a) at (0, -.70) {1};
		\node [style=none] (b) at (0, .30) {3};
		\node [style=none] (c) at (0, 1.30) {2};
}
      \end{tikzpicture}
  \end{aligned}
}

\newcommand{\lcircpropexampleother}[1]
{
\begin{aligned}
\begin{tikzpicture}
[circuit ee IEC, set resistor graphic=var resistor IEC
      graphic,scale=.5]
\scalebox{1}{	
		\node [style=none] (0) at (1, -.5) {};
		\node  [circle,draw,inner sep=1pt,fill]   (1) at (0.125, -.5) {};
		\node [style=none] (2) at (-1, 0) {};
		\node [style=none] (3) at (-1, -1) {};
		\draw[line width = 1.5pt] (0.center) to (1.center);
		\draw[line width = 1.5pt] [in=0, out=120, looseness=1.20] (1.center) to (2.center);
		\draw[line width = 1.5pt] [in=0, out=-120, looseness=1.20] (1.center) to (3.center);
		\node[style=none] (5) at (0,-1.5) {};
		\node [style=none] (7) at (1, -2) {};
		\node [style=none] (8) at (-1, -2) {};
 		\path (7) edge  (8);
		\node [style=none] (a) at (0, -1.70) {.9};
}
      \end{tikzpicture}
\end{aligned}
}

\tikzset{
    overdraw/.style={preaction={draw,white,line width=#1}},
    overdraw/.default=5pt
}

\tikzset{
circnode/.style={
  circle, draw=red, very thin, outer sep=0.025em, minimum size=2em,
  fill=red, text centered},
integral/.style={
  regular polygon, regular polygon sides=3, shape border rotate=180, draw=black, very thick,
  outer sep=0.025em, inner sep=0, minimum size=2em, fill=blue!5, text centered},
multiply/.style={
  regular polygon, regular polygon sides=3, shape border rotate=-90, draw=black,  thick,
  outer sep=0.025em, inner sep=0, minimum size=2em, fill=blue!5, text centered},
vsource/.style={
  regular polygon, regular polygon sides=5, shape border rotate=-90, draw=black, thick,
  outer sep=0.025em, inner sep=0, minimum size=2em, fill=blue!5, text centered},
isource/.style={
  regular polygon, regular polygon sides=5, shape border rotate=0, draw=black, very thick,
  outer sep=0.025em, inner sep=0, minimum size=2em, fill=blue!5, text centered},
upmultiply/.style={
  regular polygon, regular polygon sides=3, draw=black, very thick,
  outer sep=0.025em, inner sep=0, minimum size=2em, fill=blue!5, text centered},
zero/.style={
  circle, draw=black, very thick, minimum size=0.15cm, fill=black,
  inner sep=0, outer sep=0},
hole/.style={
  circle, draw=white, very thick, minimum size=0.25cm, fill=white,
  inner sep=0, outer sep=0},
bang/.style={
  circle, draw=black, very thick, minimum size=0.15cm, fill=green!10,
  inner sep=0, outer sep=0},
bang2/.style={
  circle, draw=black, thick, minimum size=0.5cm, fill=green!10,
  inner sep=0, outer sep=0},
delta/.style={
  regular polygon, regular polygon sides=3, minimum size=0.4cm,  shape border rotate=-90, inner
  sep=0, outer sep=0.025em, draw=black, very thick, fill=green!10},
codelta/.style={
  regular polygon, regular polygon sides=3, shape border rotate=90, minimum size=0.4cm,
  inner sep=0, outer sep=0.025em, draw=black, very thick, fill=green!10},
plus/.style={
  regular polygon, regular polygon sides=3, shape border rotate=-90, minimum size=0.4cm,
  inner sep = 0, outer sep=0.025em, draw=black, very thick, fill=black},
coplus/.style={
  regular polygon, regular polygon sides=3, minimum size=0.4cm,shape border rotate=90,
  inner sep = 0, outer sep=0.025em, draw=black, very thick, fill=black},
sqnode/.style={
  regular polygon, regular polygon sides=4, minimum size=2.6em,
  draw=black, very thick, inner sep=0.2em, outer sep=0.025em,
  fill=yellow!10, text centered},
blackbox/.style={
  regular polygon, regular polygon sides=4, minimum size=2.6em,
  draw=black, very thick, inner sep=0.2em, outer sep=0.025em, fill=black},
bigcirc/.style={
  circle, draw=black, very thick, text width=1.6em, outer sep=0.025em,
  minimum height=1.6em, fill=blue!5, text centered}
 }

\begin{document}
\title{Circuits, Bond Graphs, and Signal-Flow Diagrams: A Categorical Perspective}

\author{Brandon~Hector~Coya}

\degreemonth{June}
\degreeyear{2018}
\degree{Doctor of Philosophy}
\chair{Professor John Baez}
\othermembers{Professor Vyjayanthi Chari \newline Professor Wee Liang Gan}
\numberofmembers{3}

\field{Mathematics}
\campus{Riverside}

\maketitle
\copyrightpage{}
\approvalpage{}

\degreesemester{Spring}

\begin{frontmatter}

\begin{acknowledgements}

Most of all I thank my advisor John Baez. His guidance is the main reason that I understand any category theory, have three papers, and now a PhD. The way he explains any topic is enlightening and makes everything seem simple and easy to understand. This meant that he also accidentally taught me how to teach very well. I highly recommend his writing and lectures as resources for learning mathematics. 

My fellow graduate students motivated me to continue finishing a PhD. My first couple of years were made tolerable by Hannah Croft-Seidler and Mike Knap. If not for them I would have given up almost immediately. After this I leaned on Nick Woods; my conversations with him helped me understand the basics of category theory. Around this time Brendan Fong was also largely influential; I thank him for being the coauthor on my first ever paper and for teaching me so much.

 My wonderful girlfriend Evelyn Easdale kept me sane and happy during my struggles as a graduate student. She has infinite patience and kindness.  Andy Walker, Alex Sherbetjian, and Donna Blanton made life outside of mathematics fun. There are too many others to thank individually, but any that read this will know how much they have influenced me. Lastly, I thank all of the students that I have taught in my time at UCR for giving me a passion for teaching.

The material on corelations in Chapters \ref{chap:Important_Categories} and \ref{chap:Chapter2} comes from joint work with Brendan Fong and has been published previously in our paper ``Corelations are the prop for extraspecial commutative Frobenius monoids" \cite{CF}. The material on props in Chapters  \ref{chap:Important_Categories} and \ref{chap:appendix}, and  the material on electrical circuits and signal-flow diagrams in Chapters  \ref{chap:Chapter2} and \ref{chap:Functors_for_Circuits}, comes from joint work with John Baez and Franciscus Rebro in our paper ``Props in network theory" \cite{BC}. The material on bond graphs in Chapter \ref{chap:Bond_Graphs} comes from my paper ``A compositional framework for bond graphs" \cite{C}.
\end{acknowledgements}
 
\begin{dedication}

\vspace*{\fill}

\begin{center}
To my family. They are the ultimate reason that I made it so far.
\end{center}

\vspace*{\fill}

\end{dedication}

\abstract

We use the framework of ``props" to study electrical circuits, signal-flow diagrams, and bond graphs.  A prop is a strict symmetric monoidal category where the objects are natural numbers, with the tensor product of objects given by addition. In this approach, electrical circuits make up the morphisms in a prop, as do signal-flow diagrams, and bond graphs. A network, such as an electrical circuit, with $m$ inputs and $n$ outputs is a morphism from $m$ to $n$, while putting networks together in series is composition, and setting them side by side is tensoring.   Here we work out the details of this approach for various kinds of electrical circuits, then signal-flow diagrams, and then bond graphs. Each kind of network corresponds to a mathematically natural prop.  We also describe the ``behavior" of electrical circuits, bond graphs, and signal-flow diagrams using morphisms between props.  To assign a behavior to a network we ``black box"  the network, which forgets its inner workings and records only the relation it imposes between inputs and outputs.   The process of black-boxing a network then corresponds to a morphism between props. Interestingly, there are two different behaviors for any bond graph, related by a natural transformation. To achieve all of this we first prove some foundational results about props. These results let us describe any prop in terms of generators and equations, and also define morphisms of props by naming where the generators go and checking that relevant equations hold. Technically, the key tools are the Rosebrugh--Sabadini--Walters result relating circuits to special commutative Frobenius monoids, the monadic adjunction between props and signatures, and a result saying which symmetric monoidal categories are equivalent to props.

\endabstract

\tableofcontents

\end{frontmatter}

\chapter{Introduction}

People, and more specifically scientists, are naturally inclined to draw diagrams and pictures when they want to better understand a problem. One example is when Feynman introduced his famous diagrams in 1949; particle physicists have been using them ever since \cite{Fe,Ka}.  In the 1940s Olson \cite{Ol} pointed out analogies in electrical, mechanical, thermodynamic, hydraulic, and chemical systems, which allowed ``circuit diagrams" to be applied to a wide variety of fields. These analogies  are by now well-understood, and have been well explained by Karnopp, Margolis, and Rosenberg \cite{KMR, KR1, KR3} and Brown \cite{Brown}. 

On April 24, 1959, Paynter \cite{HP, HP2} woke up and invented the diagrammatic language of bond graphs, now widely used in engineering \cite{JT, JD}, to study generalized versions of voltage and current, called ``effort" and ``flow," which are implicit in the analogies found by Olson. Carl Petri introduced ``Petri nets" in computer science with his doctoral thesis; however, he claims to have invented them as a teenager to describe chemical reactions \cite{P}. Control theorists use ``signal-flow diagrams" to study linear open dynamical systems \cite{F}.

 Although category theory predates some of these diagrams, it was not until the 1980s that Joyal and Street \cite{JS1} showed string digrams can be used to reason about morphisms in any symmetric monoidal category. This motivates our first goal: view electrical circuits, signal-flow diagrams, and bond graphs as string diagrams for morphisms in symmetric monoidal categories. Doing so lets us study these networks from a \emph{compositional} perspective. That is, we study a big network by describing how it is composed of smaller pieces. Treating networks as morphisms in a symmetric monoidal category lets us build larger ones from smaller ones by composing and tensoring them: this makes the compositional perspective into precise mathematics.  To study a network in this way we must first define a notion of ``input" and ``output" for the network diagram. Then gluing diagrams together, so long as the outputs of one match the inputs of the other, defines the composition for a category.

 Network diagrams are typically assigned data, such as the potential and current associated to a wire in an electrical circuit. Since the relation between the data tells us how a network behaves, we call this relation the ``behavior" of a network. The way in which we assign behavior to a network comes from first treating a network as a ``black box,'' which is a system with inputs and outputs whose internal mechanisms are unknown or ignored.  A simple example is the lock on a doorknob: one can insert a key and try to turn it; it either opens the door or not, and it fulfills this function without us needing to know its inner workings. We can treat a system as a black box through the process called ``black-boxing," which forgets its inner workings and records only the relation it imposes between its inputs and outputs. 

 Since systems with inputs and outputs can be seen as morphisms in a category we expect black-boxing to be a functor out of a category of this sort. Assigning each diagram its behavior in a functorial way is formalized by functorial semantics, first introduced by Lawvere \cite{Law} in his thesis. This consists of using categories with specific extra structure as ``theories" whose ``models" are structure-preserving functors into other such categories. We then think of the diagrams as a syntax, while the behaviors are the semantics. Thus black-boxing is actually an example of functorial semantics. This leads us to another goal: to study the functorial semantics, i.e.\ black-boxing functors, for electrical circuits, signal-flow diagrams, and bond graphs.

Baez and Fong \cite{BF} began this type of work by showing how to describe circuits made of wires, resistors, capacitors, and inductors, as morphisms in a category using ``decorated cospans." Baez and Erbele \cite{BE}, and Bonchi, Soboci\'nski, and Zanasi \cite{BSZ} separately studied signal flow diagrams as morphisms in some category. In other work Baez, Fong, and Pollard \cite{BFP} looked at Markov processes, while Baez and Pollard \cite{BP} studied  reaction networks using decorated cospans. In all of these cases, the functorial semantics of the categories were also studied. 

Our main tool is the framework of ``props," also called ``PROPs," introduced by Mac Lane \cite{Ma65}. The acronym stands for products and permutations, and these operations roughly describe what a prop can do. More precisely, a prop is a strict symmetric monoidal category equipped with a distinguished object $X$ such that every object is a tensor power $X^{\otimes n}$. Props arise in our work because very often we think of a network as going between some set of input nodes and some set of output nodes, where the nodes are indistinguishable from each other. Thus we typically think of a network as simply having some natural number as an input and some natural number as an output, so that the network is actually a morphism in a prop. 

\section{Circuits and bond graphs}

We now briefly introduce circuits and bond graphs. Much more detail is contained in the chapters where we study them, but this introduction serves as a guide for what to picture and how to think whenever we casually mention the terminology found in electrical engineering. 

Here is an electrical circuit made of only perfectly conductive wires:
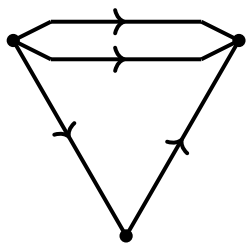
\begin{center}
    \begin{tikzpicture}[circuit ee IEC, set resistor graphic=var resistor IEC graphic]
\scalebox{1}{
      \node[contact]         (A) at (0,0) {};
      \node[contact]         (B) at (3,0) {};
      \node[contact]         (C) at (1.5,-2.6) {};
      \coordinate         (ua) at (.5,.25) {};
      \coordinate         (ub) at (2.5,.25) {};
      \coordinate         (la) at (.5,-.25) {};
      \coordinate         (lb) at (2.5,-.25) {};
      \path (A) edge (ua);
      \path (A) edge (la);
      \path (B) edge (ub);
      \path (B) edge (lb);
      \path (ua) edge  [->-=.5] node[label={[label distance=1pt]90:{}}] {} (ub);
      \path (la) edge  [->-=.5] node[label={[label distance=1pt]270:{}}] {} (lb);
      \path (A) edge  [->-=.5] node[label={[label distance=2pt]180:{}}] {} (C);
      \path (C) edge  [->-=.5] node[label={[label distance=2pt]0:{}}] {} (B);
}
    \end{tikzpicture}
  \end{center}
This is just a \emph{graph}, consisting of a set $N$ of nodes, a set $E$ of edges, and maps $s,t\maps E\to N$ sending each edge to its source and target node. We refer to the edges as perfectly conductive wires and say that wires go between nodes. Then associated to each perfectly conductive wire in an electrical circuit is a pair of real numbers called ``potential," $\phi$, and ``current," $I$.
\begin{figure}[H]
\centering
\begin{tikzpicture}
[circuit ee IEC, set resistor graphic=var resistor IEC
      graphic,scale=.5]
\scalebox{1}{
		\node [style=none] (0) at (-1, -0) {};
		\node [style=none] (1) at (1, -0) {};
		\node [style=none] (2) at (0, 0.5) {};
		\node [style=none] (3) at (0, -0.5) {};
		\draw[line width = 1.5pt] node[above] {$\phi,I$}  (1.center) to (0.center);
}
\end{tikzpicture} 
\end{figure}
\vspace{-2ex}
Typically each \emph{node} gets a potential, but in the above case the potential at either end of a wire would be the same so we may as well associate the potential to the wire. Current and potential in circuits like these obey two laws due to Kirchoff. First, at any node, the sum of currents flowing into that node is equal to the sum of currents flowing out of that node. The other law states that any connected wires must have the same potential.

We say that the above circuit is \emph{closed} as opposed to being \emph{open} because it does not have any inputs or outputs. In order to talk about open circuits and thereby bring the ``compositional persepective" into play we need a notion for inputs and outputs of a circuit. We do this using two maps $i\maps X\to N$ and $o\maps Y \to N$ that specifiy the \emph{inputs} and \emph{outputs} of a circuit. Here is an example:
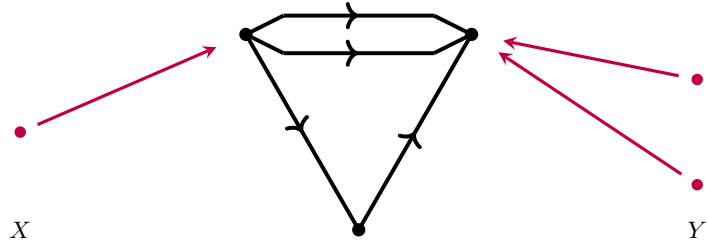
\begin{center}
    \begin{tikzpicture}[circuit ee IEC, set resistor graphic=var resistor IEC graphic]
\scalebox{1}{
      {\node[circle,draw,inner sep=1pt,fill=gray,color=purple]         (x) at
	(-3,-1.3) {};
	\node at (-3,-2.6) {$X$};}
      \node[contact]         (A) at (0,0) {};
      \node[contact]         (B) at (3,0) {};
      \node[contact]         (C) at (1.5,-2.6) {};
      {\node[circle,draw,inner sep=1pt,fill=gray,color=purple]         (y1) at
	(6,-.6) {};
	  \node[circle,draw,inner sep=1pt,fill=gray,color=purple]         (y2) at
	  (6,-2) {};
	  \node at (6,-2.6) {$Y$};}
      \coordinate         (ua) at (.5,.25) {};
      \coordinate         (ub) at (2.5,.25) {};
      \coordinate         (la) at (.5,-.25) {};
      \coordinate         (lb) at (2.5,-.25) {};
      \path (A) edge (ua);
      \path (A) edge (la);
      \path (B) edge (ub);
      \path (B) edge (lb);
      \path (ua) edge  [->-=.5] node[label={[label distance=1pt]90:{}}] {} (ub);
      \path (la) edge  [->-=.5] node[label={[label distance=1pt]270:{}}] {} (lb);
      \path (A) edge  [->-=.5] node[label={[label distance=2pt]180:{}}] {} (C);
      \path (C) edge  [->-=.5] node[label={[label distance=2pt]0:{}}] {} (B);
      {
	\path[color=purple, very thick, shorten >=10pt, shorten <=5pt, ->, >=stealth] (x) edge (A);
	\path[color=purple, very thick, shorten >=10pt, shorten <=5pt, ->, >=stealth] (y1) edge (B);
	\path[color=purple, very thick, shorten >=10pt, shorten <=5pt, ->, >=stealth] (y2)
      edge (B);}
}
    \end{tikzpicture}
  \end{center}
We call the sets $X$, $Y$, and the disjoint union $ X + Y$ the \define{inputs},  \define{outputs}, and \define{terminals} of the circuit, respectively. To each terminal we associate a potential and current. In total this gives a space of allowed potentials and currents on the terminals and we call this space the ``behavior" of the circuit. Since we do this association without knowing the potentials and currents inside the rest of the circuit we call this process ``black-boxing" the circuit. This process ``black-boxes" the internal workings of the circuit and just tells us the relation between inputs and outputs.  In fact this association is functorial, but to understand the functoriality first requires that we say how to compose these kinds of circuits. We save this for later.

There are also electrical circuits that have ``components" such as resistors, inductors, voltage sources, and current sources. These are graphs as above, but with edges now labelled by elements in some set $L$. Here is one for example:
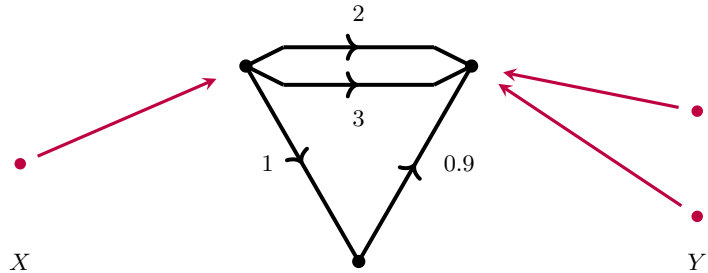
\begin{center}
    \begin{tikzpicture}[circuit ee IEC, set resistor graphic=var resistor IEC graphic]
\scalebox{1}{
      {\node[circle,draw,inner sep=1pt,fill=gray,color=purple]         (x) at
	(-3,-1.3) {};
	\node at (-3,-2.6) {$X$};}
      \node[contact]         (A) at (0,0) {};
      \node[contact]         (B) at (3,0) {};
      \node[contact]         (C) at (1.5,-2.6) {};
      {\node[circle,draw,inner sep=1pt,fill=gray,color=purple]         (y1) at
	(6,-.6) {};
	  \node[circle,draw,inner sep=1pt,fill=gray,color=purple]         (y2) at
	  (6,-2) {};
	  \node at (6,-2.6) {$Y$};}
      \coordinate         (ua) at (.5,.25) {};
      \coordinate         (ub) at (2.5,.25) {};
      \coordinate         (la) at (.5,-.25) {};
      \coordinate         (lb) at (2.5,-.25) {};
      \path (A) edge (ua);
      \path (A) edge (la);
      \path (B) edge (ub);
      \path (B) edge (lb);
      \path (ua) edge  [->-=.5] node[label={[label distance=1pt]90:{$2$}}] {} (ub);
      \path (la) edge  [->-=.5] node[label={[label distance=1pt]270:{$3$}}] {} (lb);
      \path (A) edge  [->-=.5] node[label={[label distance=2pt]180:{$1$}}] {} (C);
      \path (C) edge  [->-=.5] node[label={[label distance=2pt]0:{$0.9$}}] {} (B);
      {
	\path[color=purple, very thick, shorten >=10pt, shorten <=5pt, ->, >=stealth] (x) edge (A);
	\path[color=purple, very thick, shorten >=10pt, shorten <=5pt, ->, >=stealth] (y1) edge (B);
	\path[color=purple, very thick, shorten >=10pt, shorten <=5pt, ->, >=stealth] (y2)
      edge (B);}
}
    \end{tikzpicture}
  \end{center}
We call this an $L$-circuit. We may also black-box an $L$-circuit to get a space of potentials and currents, i.e.\ the behavior of the $L$-circuit, and this process is functorial as well. The components in a circuit determine the possible potential and current pairs because they impose additional relationships. For example, a resistor between two nodes has a resistance $R$ and is drawn as:
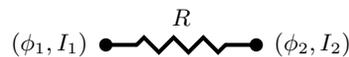
\[
\begin{tikzpicture}[circuit ee IEC, set resistor graphic=var resistor IEC graphic]
\node[contact] (C) at (0,2) {};
\node[contact] (D) at (2,2) {};
\node (A) at (-.75,2) {$(\phi_1,I_1)$};
\node (B) at (2.75,2) {$(\phi_2,I_2)$};
  \draw (0,2) to [resistor={info={$R$}}] ++(2,0);
\end{tikzpicture}
\]
In an $L$-circuit this would be an edge labelled by some positive real number $R$. For a resistor like this we know that $I_1=I_2$ and that Ohm's Law says $\phi_2-\phi_1 =I_1R$. This tells us how to construct the black-boxing functor that extracts the right behavior.  

Engineers often work with wires that come in pairs where the current on one wire is the negative of the current on the other wire. In such a case engineers care about the difference in potential more than each individual potential. For such pairs of perfectly conductive wires:
\begin{figure}[H]
\centering
\begin{tikzpicture}
[circuit ee IEC, set resistor graphic=var resistor IEC
      graphic,scale=.5]
\scalebox{1}{
		\node [style=none] (2) at (-1, -.25) {};
		\node [style=none] (3) at (1, -.25) {};
		\draw[line width = 1.5pt] node[below=.15] {$\phi_2,I_2$}  (2.center) to (3.center);
		\node [style=none] (4) at (-1, .25) {};
		\node [style=none] (5) at (1, .25) {};
		\draw[line width = 1.5pt] node[above=.15] {$\phi_1,I_1$}  (4.center) to (5.center);
}
\end{tikzpicture} 
\end{figure}
\vspace{-2ex}
 \noindent we call $V=\phi_2-\phi_1$ the ``voltage" and $I=I_1=-I_2$ the ``current." Note the word current is used for two different, yet related concepts. We call a pair of wires like this a ``bond" and a pair of nodes like this a ``port." To summarize we say that bonds go between ports, and in a ``bond graph" we draw a bond as follows:  
\begin{figure}[H]
\centering
\begin{tikzpicture}[circuit ee IEC, set resistor graphic=var resistor IEC
      graphic, scale=0.8, every node/.append style={transform shape}]
[
	node distance=1.5cm,
	mynewelement/.style={
		color=blue!50!black!75,
		thick
	},
	mybondmodpoint/.style={
	rectangle,
	minimum size=3mm,
	very thick,
	draw=red!50!black!50, 
	outer sep=2pt
	}
]		
	\node(J11) {};
	\node (R3) [right=1 of J11] {}
	edge [line width=3.5pt]    node [below]{$I$} (J11)
        edge  [line width=3.5pt]  node [above]{$V$} (J11);
\end{tikzpicture}
\end{figure}
\vspace{-2ex}
Note that engineers do not explicitly draw ports at the ends of bonds; we follow this notation and simply draw a bond as a thickened edge. Engineers who work with bond graphs often use the terms ``effort" and ``flow" instead of ``voltage" and ``current", respectively.  Thus a bond between two ports in a bond graph is drawn equipped with an effort and flow, rather than a voltage and current, as follows:
\begin{figure}[H]
\centering
\begin{tikzpicture}[circuit ee IEC, set resistor graphic=var resistor IEC
      graphic, scale=0.8, every node/.append style={transform shape}]
[
	node distance=1.5cm,
	mynewelement/.style={
		color=blue!50!black!75,
		thick
	},
	mybondmodpoint/.style={
	rectangle,
	minimum size=3mm,
	very thick,
	draw=red!50!black!50, 
	outer sep=2pt
	}
]		
	\node(J11) {};
	\node (R3) [right=1 of J11] {}
	edge [line width=3.5pt]    node [below]{$F$} (J11)
        edge  [line width=3.5pt]  node [above]{$E$} (J11);
\end{tikzpicture}
\end{figure}
\vspace{-2ex}
\noindent  A bond graph consists of bonds connected together using ``$1$-junctions" and ``$0$-junctions."  These two types of junctions impose equations between the efforts and flows on the attached bonds. The efforts on bonds connected together with a $1$-junction sum to zero, while the flows are all equal. For $0$-junctions the efforts are all equal while the efforts sum to zero. We explain bond graphs with more detail in Chapter \ref{chap:Bond_Graphs}, but for now the following suffices as an example of a bond graph:
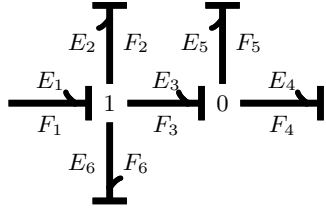
\begin{figure}[H] 
	\centering
\begin{tikzpicture}
[ 
	node distance=1.5cm,
	mynewelement/.style={
		color=blue!50!black!75,
		thick
	},
	mybondmodpoint/.style={
	rectangle,
	minimum size=3mm,
	very thick,
	draw=red!50!black!50, 
	outer sep=2pt
	}
]
	\node (S) {};
	\node (J11) [right of=S]{1}
	edge [inbonde, line width=2.5pt]  node [below]{$F_{1}$} (S)
        edge [inbonde, line width=2.5pt]  node [above]{$E_{1}$} (S);
	\node (R1) [above of=J11]{}
	edge  [inbonde, line width=2.5pt]  node [right]{$F_{2}$} (J11)
        edge  [inbonde, line width=2.5pt] node [left]{$E_{2}$} (J11);
	\node (J01) [right of=J11] {0}
	edge  [inbonde, line width=2.5pt] node  [color=black] [above]{$E_{3}$} (J11)
        edge  [inbonde, line width=2.5pt]  node  [color=black] [below]{$F_{3}$} (J11);
	\node (R2) [right of=J01] {}
	edge  [inbonde, line width=2.5pt] node [below]{$F_{4}$} (J01)
        edge  [inbonde, line width=2.5pt] node [above]{$E_{4}$}  (J01);
	\node (C1) [above of=J01] {}
	edge  [inbonde, line width=2.5pt]  node [right]{$F_{5}$} (J01)
	edge  [inbonde, line width=2.5pt]  node [left]{$E_{5}$} (J01);
	\node (C3) [below of=J11] {}
	edge  [inbonde, line width=2.5pt]  node [right]{$F_{6}$} (J11)
	edge  [inbonde, line width=2.5pt]  node [left]{$E_{6}$} (J11);
\end{tikzpicture}
\end{figure}
\vspace{-2ex}
The arrow at the end of a bond indicates which direction of current flow counts as positive, while the bar is called the ``causal stroke." These are unnecessary for our work, so we adopt a simplified notation without the arrow or bar. Additionally, one may also attach general circuit components, but we will also not consider these.

\section{Plan of the thesis}

In Chapter \ref{chap:Important_Categories} we provide the necessary background for studying four categories $\Fin\Span$, $\Fin\Rel$, $\Fin\Cospan$, and $\Fin\Corel$ as props. In particular, $\Fin\Cospan$ and $\Fin\Corel$ are crucial to our study of networks. Both categories have the natural numbers as objects, while $\Fin\Cospan$ has ``cospans" as morphisms and $\Fin\Corel$ has ``corelations" as morphisms. We then equip $\Fin\Cospan$ and $\Fin\Corel$ with additional structure making them into ``dagger compact categories." In Corollary \ref{cor:presentation} we prove that any prop has a presentation in terms of generators and equations. Then we recall the known presentations for $\Fin\Span$, $\Fin\Cospan$, and $\Fin\Rel$. Proposition \ref{prop:epimorphism} lets us build props as quotient props of other props. All of this work leads to Theorem \ref{thm:fincorel_prop} in Chapter \ref{chap:Chapter2}, where we show that $\Fin\Corel$ is the ``prop for extraspecial commutative Frobenius monoids." This also provides a presentation for $\Fin\Corel$ which we continually use in our work on networks. Our results clarify the relationship between $\Fin\Span$, $\Fin\Rel$, $\Fin\Cospan$, and $\Fin\Corel$ as props.

We begin Chapter \ref{chap:Chapter2} by showing that $\Fin\Corel$ is the prop for extraspecial commutative Frobenius monoids.  Then we define an ``$L$-circuit" as a graph with specified inputs and outputs, together with a labelling set for the egdes of the graph. $L$-circuits are morphisms in the prop $\Circ_L$. In Proposition \ref{prop:lcirc_coproduct} we show that $\Circ_L$ can be viewed as the coproduct of $\Fin\Cospan$ together with a free prop on the set $L$ of labels. We define $\Circ$ to be the prop $\Circ_L$ where $L$ consists of a single element. Then we view $\Circ$ as the category whose morphisms are circuits made of only perfectly conductive wire.  From any morphism in $\Circ$ we extract a cospan of finite sets and then turn the cospan into a corelation. These two processes are functorial so we get a method for sending a circuit made of only perfectly conductive wires to a corelation: \[     \Circ \stackrel{H'}{\longrightarrow} \Fin\Cospan \stackrel{H}{\longrightarrow} 
\Fin\Corel.  \]

There is also a functor $K\maps \Fin\Corel \to \Fin\Rel_k$, where $\Fin\Rel_k$ is the category whose objects are finite dimensional vector spaces  $V$ and whose morphisms $R\maps U\to V$ are linear relations $R\subseteq U \oplus V$. By composing with the above functors $H'$ and $H$ we associate a linear relation $R$ to any circuit made of perfectly conductive wires. On the other hand we get a subspace for any such circuit by first assigning potential and current to each terminal, and then subjecting these variables to the appropriate physical laws. It turns out that these two ways of assigning a subspace to a morphism in $\Circ$ are the same. We call the linear relation associated to a circuit using the composite $KHH'$ the ``behavior" of the circuit and define  the ``black-boxing" functor to be the composite $\blacksquare := KHH'\maps \Circ\to \Fin\Rel_k$. Note that the underlying corelation of a circuit made of perfectly conductive wires completely determines the behavior of the circuit via the functor $K$.

In Chapter \ref{chap:Functors_for_Circuits} we reinterpret $\blacksquare$ as a morphism of props. We do this by introducing the category $\Lag\Rel_k$, whose objects are ``symplectic" vector spaces and whose morphisms are  ``Lagrangian" relations. Then in Proposition \ref{prop:K_1} we prove that the functor $K\maps \Fin\Corel \to \Fin\Rel_k$ actually picks out a Lagrangian subspace for any corelation and thus determines a morphism of props, which we also call $K\maps \Fin\Corel \to \Lag\Rel_k$. Then we redefine black-boxing to be the morphism of props given by the composite \[     \Circ \stackrel{H'}{\longrightarrow} \Fin\Cospan \stackrel{H}{\longrightarrow} 
\Fin\Corel \stackrel{K}{\longrightarrow} \Lag\Rel_k.  \] 

After doing the hard work for circuits made of perfectly conductive wires we show the power of our results on props by easily extending the black-boxing functor to circuits with arbitrary label sets in Theorem \ref{thm:black-boxing_1}. We apply this result in Proposition \ref{prop:black-boxing_1} to the category $\Circ_{RLC}$, whose  morphisms are circuits made of resistors, inductors, and capacitors. Then in Theorem \ref{thm:black-boxing_2} we prove the existence and uniqueness of a morphism of props $\blacksquare \maps \Circ_k \to \Lag\Rel_k$ for any field $k$.

We then recall how control theorists use ``signal-flow diagrams" to draw linear relations. These diagrams make up the category $\SigFlow_k$, which is the free prop generated by the same generators as $\Fin\Rel_k$. Similarly we define the prop $ \Ccirc_k$ as the free prop generated by the same generators as $\Circ_k$. Then there is a strict symmetric monoidal functor $T\maps \Ccirc_k \to \SigFlow_k$ giving a commutative square:
\[
    \xymatrix{
      \Ccirc_k \phantom{ |} \ar[r]^-{P} \ar[d]_{T} & \Circ_k \phantom{ |} \ar[rr]^-{\blacksquare} & & \Lag\Rel_k \phantom{ |} \ar@{^{(}->}[d] \\
      \SigFlow_k \phantom{ |}\ar[rrr]^-{\square} & & & \Fin\Rel_k. \phantom{ |}
      }
\]
Of course, circuits made of perfectly conductive wires are a special case of linear circuits.  We can express this fact using a commutative square:
\[
    \xymatrix{
     \Ccirc \phantom{ |} \ar[r]^-{P} \ar@{^{(}->}[d] & \Circ \phantom{ |} \ar@{^{(}->}[d]  \\
      \Ccirc_k \phantom{ |} \ar[r]^-{P} & \Circ_k. \phantom{ |} 
      }
\]
Combining the diagrams we get a commutative diagram summarizing the relationship between linear circuits, cospans, corelations, and signal-flow diagrams:
\[
    \xymatrix{
     \Ccirc \phantom{ |} \ar[r]^-{P} \ar@{^{(}->}[d] & \Circ \phantom{ |} \ar@{^{(}->}[d] \ar[r]^-{H'} & \Fin\Cospan \phantom{ |} \ar[r]^-{H} &
      \Fin\Corel \phantom{ |} \ar[d]^{K} \\
      \Ccirc_k \phantom{ |} \ar[r]^-{P} \ar[d]_{T} & \Circ_k \phantom{ |} \ar[rr]^-{\blacksquare} & & \Lag\Rel_k \phantom{ |} \ar@{^{(}->}[d] \\
      \SigFlow_k \phantom{ |}\ar[rrr]^-{\square} & & & \Fin\Rel_k. \phantom{ |}
      }
\]

We end the chapter by extending our work to circuits with voltage and current sources. These types of circuits define affine relations instead of linear relations. Our prop framework lets us extend to these types of circuits by showing that affine Lagrangian relations define the morphisms in a prop $\Aff\Lag\Rel_k$. This leads to Theorem \ref{thm:black-boxing_3}, which says that for any field $k$ and label set $L$ there exists a unique morphism of props \[ \blacksquare \maps \Circ_L \to \Aff\Lag\Rel_k \] extending the other black-boxing functor.

In Chapter \ref{chap:Bond_Graphs} we study bond graphs as morphisms in a category. Our goal is to define a category $\BondGraph$, whose morphisms are bond graphs, and then assign a space of efforts and flows as behavior to any bond graph using a functor. We also construct a functor that assigns a space of potentials and currents to any bond graph, which agrees with the way that potential and current relate to effort and flow.

The way we define $\BondGraph$ comes from two different approaches to studying bond graphs. The first way leads us to a subcategory $\Fin\Corel^{\circ}$ of $\Fin\Corel$, while the second way leads us to a subcategory $\Lag\Rel_k^{\circ}$ of $\Lag\Rel_k$. Instead of a commutative square relating these four categories, we obtain a pentagon that commutes up to a natural transformation if we invent a new category $\BondGraph$: 

\vspace{-3ex}

\[
    \xymatrix@C-3pt{
      & \Lag\Rel_k^{\circ} \phantom{ |} \ar[r]^-{i'} & \Lag\Rel_k \\
	\BondGraph \ar[ur]^-{F} \ar[dr]_-{G} \\
 	& \Fin\Corel^{\circ} \ar[r]^-{i}  \ar@<-6ex>@{}[uu]^(.25){}="a"^(.75){}="b" \ar@{<=}^{\alpha} "a";"b"   & \Fin\Corel. \ar[uu]^-{K}
      }
\]

\vspace{1ex}

 In our first approach we view a bond graph as an electrical circuit. We take advantage of our earlier work on circuits and corelations by taking $\Fin\Corel$ to be the category whose morphisms are circuits made of perfectly conductive wires. In this approach a terminal is the object $1$ and a wire is the identity corelation from $1$ to $1$, while a circuit from $m$ terminals to $n$ terminals is a corelation from $m$ to $n$.

Here we think of a port as the object $2$, since a port is a pair of nodes. Then we think of a bond as a pair of wires and hence the identity corelation from $2$ to $2$. Lastly, the two junctions are two different ways of connecting ports together, and thus specific corelations from $2m$ to $2n$. It turns out that we may equip the object $2$ with two different Frobenius monoid structures and these interact with each other similarly to how $1$-junctions and $0$-junctions interact with each other. 

We want the morphisms built from the two Frobenius monoids to correspond to bond graphs, since bond graphs are built out of series and parallel junctions. Unfortunately there are some equations which hold between morphisms made from these Frobenius monoids that do not hold for corresponding bond graphs. We define a category $\Fin\Corel^{\circ}$ using the morphisms that come from these two Frobenius monoids and move on to another attempt at defining $\BondGraph$.

Since bond graphs impose Lagrangian relations between effort and flow we look back at $\Lag\Rel_k$. The relations associated to a $1$-junction make $k\oplus k$ into yet another Frobenius monoid, while the relations associated to a $0$-junction make $k\oplus k$ into a different Frobenius monoid.  The two Frobenius monoid structures interact to form a ``bimonoid." Unfortunately, a bimonoid has equations between morphisms that do not correspond to equations between bond graphs, so this approach also does not result in morphisms that are bond graphs. Regardless, we define the category $\Lag\Rel_k^{\circ}$ using the two Frobenius monoids here. 

Since it turns out that $\Fin\Corel^{\circ}$ and $\Lag\Rel_k^{\circ}$ have corresponding generators we define  $\BondGraph$ as a prop that also has corresponding generators, but with only the equations found in both $\Fin\Corel^{\circ}$ and $\Lag\Rel_k^{\circ}$. By defining $\BondGraph$ in this way we automatically have two functors $F\maps \BondGraph \to \Lag\Rel_k^{\circ}$ and $G\maps \BondGraph \to \Fin\Corel^{\circ}$. The functor $F$ associates effort and flow to a bond graph, while the functor $G$ lets us associate potential and current to a bond graph using the previous work done on $\Fin\Corel$. Then the Lagrangian subspace relating effort, flow, potential, and current: $$\{(V,I,\phi_1,I_1,\phi_2,I_2) | V = \phi_2-\phi_1, I = I_1 =- I_2\}$$ 

\noindent defines a natural transformation in the following diagram:

\vspace{-3ex}

\[
    \xymatrix@C-3pt{
      & \Lag\Rel_k^{\circ} \phantom{ |} \ar[r]^-{i'} & \Lag\Rel_k \\
	\BondGraph \ar[ur]^-{F} \ar[dr]_-{G} \\
 	& \Fin\Corel^{\circ} \ar[r]^-{i}  \ar@<-6ex>@{}[uu]^(.25){}="a"^(.75){}="b" \ar@{<=}^{\alpha} "a";"b"   & \Fin\Corel \ar[uu]^-{K}
      }
\]
Putting this together with the above diagram we get one diagram which encompasses the relationships between circuits, signal-flow diagrams, bond graphs, and their behaviors in category theoretic terms. 

\vspace{-3ex}

\[
    \xymatrix@C-3pt{
     \Ccirc \phantom{ |} \ar[r]^-{P} \ar@{^{(}->}[dd] & \Circ \phantom{ |} \ar@{^{(}->}[dd] \ar[r]^-{H'} & \Fin\Cospan \phantom{ |} \ar[r]^-{H} &
      \Fin\Corel \phantom{ |} \ar[dd]^{K} & \Fin\Corel^{\circ} \phantom{ |} \ar[l] \ar@<-6ex>@{}[dd]^(.25){}="a"^(.75){}="b" \ar@{<=}^{\alpha} "a";"b"  \\
        &&&&& \BondGraph  \ar[ul]_-{G} \ar[dl]^-{F} \\ 
      \Ccirc_k \phantom{ |} \ar[r]^-{P} \ar[d]_{T} & \Circ_k \phantom{ |} \ar[rr]^-{\blacksquare} & & \Lag\Rel_k \phantom{ |} \ar@{^{(}->}[d] & \Lag\Rel_k^{\circ} \phantom{ |} \ar[l]    \\
      \SigFlow_k \phantom{ |}\ar[rrr]^-{\square} & & & \Fin\Rel_k. \phantom{ |}
      }
\]

\chapter{Cospans, corelations, and props}
\label{chap:Important_Categories}

We start by defining the categories $\Fin\Span$, $\Fin\Rel$, $\Fin\Cospan$, and $\Fin\Corel$. Of particular interest are $\Fin\Cospan$ and $\Fin\Corel$ due to their relation to electrical circuits, signal-flow diagrams, and bond graphs. To this end we equip $\Fin\Cospan$ and $\Fin\Corel$ with various structures which, in summary, make them into dagger compact categories, or more informally, make them ``network-like." We also introduce the framework of ``props," i.e.\ strict symmetric monoidal categories whose objects are the natural numbers. It turns out that $\Fin\Span$, $\Fin\Rel$, $\Fin\Cospan$, and $\Fin\Corel$ are all equivalent to props. This fact lets us nicely express their relationships to each other after some preliminary results. 

 In Corollary \ref{cor:presentation} we prove that any prop has a generators and equations presentation. The presentations for  $\Fin\Span, \Fin\Cospan,$ and $\Fin\Rel$ are all well known. We can completely understand how $\Fin\Corel$ is related as a prop to the other three categories by presenting $\Fin\Corel$ as the ``prop for extraspecial commutative Frobenius monoids." In order to arrive at this presentation we need to use Proposition \ref{prop:epimorphism}, which says, among other things, that adding extra equations to a prop creates a quotient prop. It turns out that $\Fin\Corel$ is a quotient of $\Fin\Cospan$, and that this is enough to present $\Fin\Corel$ as a prop in terms of generators and equations.

\section{Cospans and corelations}

We use \emph{cospans} to describe networks with inputs and outputs as morphisms in a category. 
\begin{definition} 
Given a category $\C$ a \define{cospan} in $\C$ is a diagram of the form:

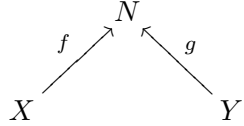
\vspace{-3ex}

\[
  \xymatrix{
    & N \\
    X \ar[ur]^-{f} && Y \ar[ul]_-{g}
  } 
\]

\end{definition}

We often view the apex of a cospan as the set of nodes for some network. Then the idea is that the maps $f$ and $g$ pick out the input and output nodes of the network. This lets us define the composite of two networks so long as we have a method for composing cospans. Indeed there is a well known method for composing cospans by using \emph{pushouts}.

\begin{definition}
Given a diagram of this form in a category $C$: 

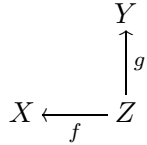
\vspace{-3ex}

\[
  \xymatrix{
    & Y  \\
    X  & Z \ar[l]^-{f} \ar[u]_-{g}
  } 
\]

\noindent the \define{pushout} of the morphisms $f,g$ consists of an object $P$ and two morphisms $i_1\maps X\to P$ and $i_2\maps Y\to P$ such that the diagram:

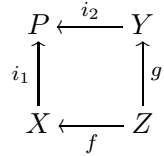
\vspace{-3ex}

\[
  \xymatrix{
    P & Y  \ar[l]_-{i_2} \\
    X \ar[u]^-{i_1} & Z \ar[l]^-{f} \ar[u]_-{g}
  } 
\]

\noindent commutes and such that $P$ is universal for this property. That is, for any object $Q$ with morphisms $j_1\maps X\to Q$ and $j_2\maps Y \to Q$ there exists a unique $u\maps P\to Q$ also making the relevent diagrams commute. 
\end{definition}

Note that a pushout is unique up to unique isomorphism so with a slight abuse of language we say \emph{the} pushout. A consequence of this brand of uniqueness is that pushouts only determine a well-defined composite for \emph{isomorphism classes} of cospans. The notion of isomorphic cospans requies the notion of a map between cospans.

\begin{definition}

A \define{map of cospans} is a morphism $h\maps N \to N'$ between the apices of two cospans such that the diagram:

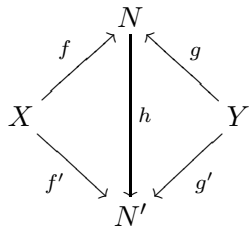
\vspace{-3ex}

\[
  \xymatrix{
    & N \ar[dd]^-{h} \\
    X \ar[ur]^-{f} \ar[dr]_-{f'} && Y \ar[ul]_-{g} \ar[dl]^-{g'} \\
   & N'
  } 
\]
commutes. If $h$ is an isomorphism then we say the cospans are \define{isomorphic}. 

\end{definition}

When a category $\C$ has pushouts we may compose isomorphism classes of cospans in the follwing way. Given two isomorphism classes of cospans with a common foot:

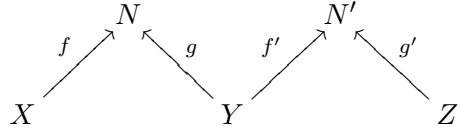
\vspace{-3ex}

\[
  \xymatrix{
    & N && N' \\
    X \ar[ur]^-{f} && Y \ar[ul]_-{g} \ar[ur]^-{f'}  && Z \ar[ul]_-{g'}
  } 
\]

\noindent we take the pushout to form the diagram:

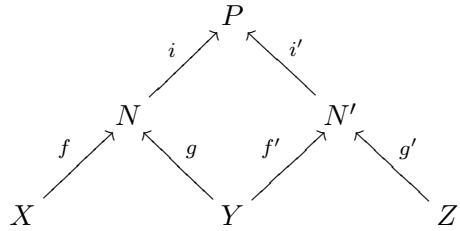
\vspace{-3ex}

\[
  \xymatrix{
	&& P \\
    & N \ar[ur]^-{i} && N'  \ar[ul]_-{i'}  \\
    X \ar[ur]^-{f} && Y \ar[ul]_-{g} \ar[ur]^-{f'}  && Z \ar[ul]_-{g'}
  } 
\]

\noindent and we say the composite is the isomorphism class of the following cospan:

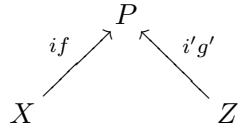
\vspace{-3ex}

\[
  \xymatrix{
    & P \\
    X \ar[ur]^-{i f} && Z \ar[ul]_-{i'g'}
  } 
\]

For any category $\C$ with pushouts, one has a category $\Cospan(\C)$ whose objects are those of $\C$ and whose morphisms are the isomorphism classes of cospans of $\C$ \cite{Be}. The category $\Cospan(\Fin\Set)$, which we rename now, is of particular interest.

\begin{definition}
Let $\Fin\Cospan$ be the category $\Cospan(\Fin\Set)$ which has finite sets as objects and isomorphism classes of cospans  $X\to N \leftarrow Y$ as morphisms, where we compose morphisms using pushouts.
\end{definition}

From this point onward we say cospan when we mean isomorphism class of cospans. To better understand cospans we also study what happens when we reverse the arrows in a cospan.

\begin{definition} 

Given a category $\C$ a \define{span} in $\C$ is a diagram of the form:

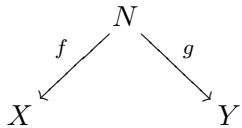
\vspace{-3ex}

\[
  \xymatrix{
    & N \ar[dl]_-{f} \ar[dr]^-{g} \\
    X  && Y 
  } 
\]
\end{definition}

Just as the composite of cospans is defined using pushouts, the composite of spans is defined using \define{pullbacks}. 

\begin{definition}
Given a diagram of this form in a category $C$:

\vspace{-3ex}

\[
  \xymatrix{
    & Y \ar[d]^-{g}  \\
    X   \ar[r]_-{f}  & Z
  } 
\]

\noindent the \define{pullback} of the morphisms $f,g$ consists of an object $P$ and two morphisms $p_1\maps P\to X$ and $p_2\maps P\to Y$ such that the diagram:

\vspace{-3ex}

\[
  \xymatrix{
   P \ar[r]^-{p_2}  \ar[d]_-{p_1} & Y \ar[d]^-{g}  \\
    X   \ar[r]_-{f}  & Z
  } 
\]

\noindent commutes and such that $P$ is universal for this property. That is, for any object $Q$ with morphisms $q_1\maps Q\to X$ and $q_2\maps Q \to Y$ there exists a unique $u\maps Q\to P$ also making the relevant diagrams commute. 
\end{definition}

For any category $\C$ with pullbacks, there is a category $\Span(\C)$ with objects the objects of $\C$ and morphisms the isomorphism classes of spans of $\C$. 

\begin{definition}
Let $\Fin\Span$ be the category $\Span(\Fin\Set)$ which has finite sets as objects and isomorphism classes of spans $X\leftarrow N \to Y$ as morphisms, where we compose morphisms using pullbacks.
\end{definition}

Note also that cospans in a category $\C$ are the same as spans in $\C^{\op}$ while pushouts in a category $\C$ are the same as pullbacks in $\C^{\op}$. In this way $\Fin\Span$ is dual to $\Fin\Cospan$.  There is also an interesting relationship between morphisms $X\to N \leftarrow Y$  in $\Fin\Cospan$  and partitions $P$ of $X+Y$ where ``$+$" denotes disjoint union.


\begin{definition}
A cospan $X \stackrel{f}{\rightarrow} S \stackrel{g}{\leftarrow} Y$
in the category of finite sets is \define{jointly epic} if $f(X) \cup g(Y) = S$.   
\end{definition}

A jointly epic cospan $X \stackrel{f}{\rightarrow} N \stackrel{g}{\leftarrow} Y$ determines a partition of $X+Y$, where two points $p,q \in X+Y$ are in the same part
of the partition if and only if they map to the same point of $N$ via the function
$\langle f, g \rangle \maps X+Y \to N$. To see this consider the following jointly epic cospan of finite sets:
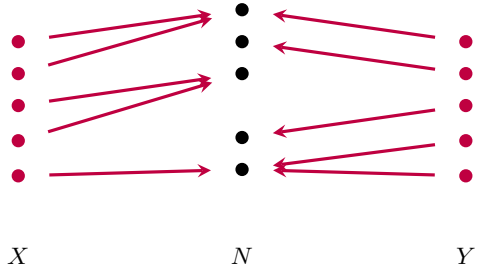
\[
\begin{aligned}
\begin{tikzpicture}[scale=.85, circuit ee IEC]
\node [style=none] (5) at (-.5, -2.25) {$N$};
\node [style=none] (4a) at (-4, -2.25) {$X$};
\node [style=none] (20a) at (3, -2.25) {$Y$};


	   	  \node[contact, outer sep=5pt,fill=gray,color=purple] (p6) at (-4, 1.1) {};
	        \node[contact, outer sep=5pt,fill=gray,color=purple] (p7) at (-4, -.45) {};
	        \node[contact, outer sep=5pt,fill=gray,color=purple] (p8) at (-4, .6) {};
	        \node[contact, outer sep=5pt,fill=gray,color=purple] (p9) at (-4, .1) {};
	        \node[contact, outer sep=5pt,fill=gray,color=purple] (p10) at (-4, -1) {};

	        \node[contact, outer sep=5pt,fill=gray,color=purple] (p-2) at (3, 1.1) {};
	        \node[contact, outer sep=5pt,fill=gray,color=purple] (p0) at (3, .6) {};
	        \node[contact, outer sep=5pt,fill=gray,color=purple] (p1) at (3, 0.1) {};
	        \node[contact, outer sep=5pt,fill=gray,color=purple] (p2) at (3, -.45) {};
	        \node[contact, outer sep=5pt,fill=gray,color=purple] (p3) at (3, -1) {};

		\node [contact, outer sep=5pt] (m1) at (-.5, 1.6) {};
		\node [contact, outer sep=5pt] (m2) at (-.5, 1.1) {};
		\node [contact, outer sep=5pt] (m3) at (-.5, .6) {};
		\node [contact, outer sep=5pt] (m5) at (-.5, -.4) {};
		\node [contact, outer sep=5pt] (m6) at (-.5, -.9) {};

	       \path[color=purple, very thick, shorten >=5pt, shorten <=5pt, ->,
	>=stealth] (p6) edge (m1);
	       \path[color=purple, very thick, shorten >=5pt, shorten <=5pt, ->,
	>=stealth] (p7) edge (m3);
	       \path[color=purple, very thick, shorten >=5pt, shorten <=5pt, ->,
	>=stealth] (p8) edge (m1);
	       \path[color=purple, very thick, shorten >=5pt, shorten <=5pt, ->,
	>=stealth] (p9) edge (m3);
	       \path[color=purple, very thick, shorten >=5pt, shorten <=5pt, ->,
	>=stealth] (p10) edge (m6);

	       \path[color=purple, very thick, shorten >=5pt, shorten <=5pt, ->,
	>=stealth] (p-2) edge (m1);

	       \path[color=purple, very thick, shorten >=5pt, shorten <=5pt, ->,
	>=stealth] (p0) edge (m2);
	       \path[color=purple, very thick, shorten >=5pt, shorten <=5pt, ->,
	>=stealth] (p1) edge (m5);
	       \path[color=purple, very thick, shorten >=5pt, shorten <=5pt, ->,
	>=stealth] (p2) edge (m6);
	       \path[color=purple, very thick, shorten >=5pt, shorten <=5pt, ->,
	>=stealth] (p3) edge (m6);
\end{tikzpicture}
\end{aligned}
\]

\noindent This jointly epic cospan determines the following partition, where elements in the same part of the partition are drawn in the same part with a dashed line:

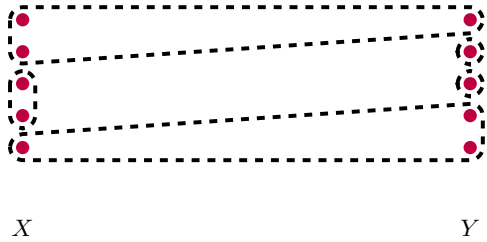
\[
\begin{aligned}
\begin{tikzpicture}[scale=.85, circuit ee IEC]
\node [style=none] (4a) at (-4, -2.25) {$X$};
\node [style=none] (20a) at (3, -2.25) {$Y$};
	        \node[contact, outer sep=5pt,fill=gray,color=purple] (p6) at (-4, 1) {};
	        \node[contact, outer sep=5pt,fill=gray,color=purple] (p7) at (-4, -.5) {};
	        \node[contact, outer sep=5pt,fill=gray,color=purple] (p8) at (-4, .5) {};
	        \node[contact, outer sep=5pt,fill=gray,color=purple] (p9) at (-4, 0) {};
	        \node[contact, outer sep=5pt,fill=gray,color=purple] (p10) at (-4, -1) {};
	        \node[contact, outer sep=5pt,fill=gray,color=purple] (p-2) at (3, 1) {};
	        \node[contact, outer sep=5pt,fill=gray,color=purple] (p0) at (3, .5) {};
	        \node[contact, outer sep=5pt,fill=gray,color=purple] (p1) at (3, 0) {};
	        \node[contact, outer sep=5pt,fill=gray,color=purple] (p2) at (3, -.5) {};
	        \node[contact, outer sep=5pt,fill=gray,color=purple] (p3) at (3, -1) {};
		\draw [rounded corners=5pt, dashed] 
   (node cs:name=p6, anchor=north west) --
   (node cs:name=p8, anchor=south west) --
   (node cs:name=p-2, anchor=south east) --
   (node cs:name=p-2, anchor=north east) --
   cycle;
		\draw [rounded corners=5pt, dashed] 
   (node cs:name=p9, anchor=north west) --
   (node cs:name=p7, anchor=south west) --
   (node cs:name=p7, anchor=south east) --
   (node cs:name=p9, anchor=north east) --
   cycle;
		\draw [rounded corners=5pt, dashed] 
   (node cs:name=p10, anchor=north west) --
   (node cs:name=p10, anchor=south west) --
   (node cs:name=p3, anchor=south east) --
   (node cs:name=p2, anchor=north east) --
   cycle;
		\draw [rounded corners=5pt, dashed] 
   (node cs:name=p0, anchor=north west) --
   (node cs:name=p0, anchor=south west) --
   (node cs:name=p0, anchor=south east) --
   (node cs:name=p0, anchor=north east) --
   cycle;
		\draw [rounded corners=5pt, dashed] 
   (node cs:name=p1, anchor=north west) --
   (node cs:name=p1, anchor=south west) --
   (node cs:name=p1, anchor=south east) --
   (node cs:name=p1, anchor=north east) --
   cycle;
\end{tikzpicture}
\end{aligned}
\]

\noindent Jointly epic cospans determine a category. Given two jointly epic cospans:

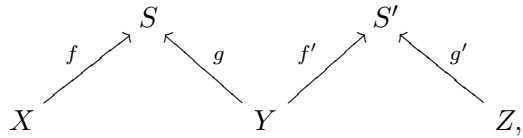
\vspace{-3ex}

\[
  \xymatrix{
    & \; S \;  && \; S'  \\
 \; \;\; X \ar[ur]^{f} \;\; && \; Y \; \ar[ul]_{g} \ar[ur]^{f'} && \;\; Z, \;\; \ar[ul]_{g'}
  }
\]
\noindent we compose them by first forming the pushout:

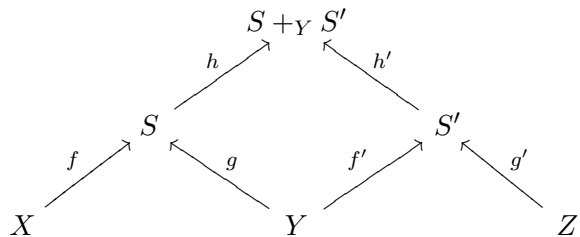
\vspace{-3ex}

\[
  \xymatrix{
    && S+_Y S' \\
    & \; S \; \ar[ur]^h && \; S' \ar[ul]_{h'} \\
 \; \;\; X \ar[ur]^{f} \;\; && \; Y \; \ar[ul]_{g} \ar[ur]^{f'} && \;\; Z \;\; \ar[ul]_{g'}
  }
\]
\noindent and then forcing the composite cospan to be jointly epic:

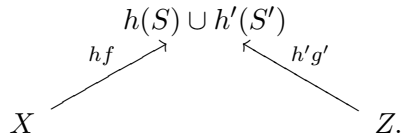
\vspace{-3ex}

\[
  \xymatrix{
    & h(S) \cup h'(S') \\
     \; X \; \ar[ur]^{hf} && \; Z. \ar[ul]_{h'g'} }
\]

There is a notion for composing partitions $P$ of $X+Y$ and $P'$ of $Y+Z$ which, as expected,  is equivalent to composing jointly epic cospans. We refrain from writing the overly wordy definition and instead give a comprehensive example since the idea is simple. Let the following be a partition of $X+Y$:

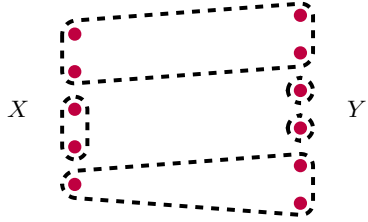
\[
  \begin{tikzpicture}[circuit ee IEC]
		\node [contact, outer sep=5pt,fill=gray,color=purple]  (0) at (-2, 1) {};
		\node [contact, outer sep=5pt,fill=gray,color=purple] (1) at (-2, 0.5) {};
		\node [contact, outer sep=5pt,fill=gray,color=purple] (2) at (-2, -0) {};
		\node [contact, outer sep=5pt,fill=gray,color=purple] (3) at (-2, -0.5) {};
		\node [contact, outer sep=5pt,fill=gray,color=purple] (4) at (-2, -1) {};
		\node [contact, outer sep=5pt,fill=gray,color=purple](5) at (1, 1.25) {};
		\node [contact, outer sep=5pt,fill=gray,color=purple] (6) at (1, 0.75) {};
		\node [contact, outer sep=5pt,fill=gray,color=purple] (7) at (1, 0.25) {};
		\node [contact, outer sep=5pt,fill=gray,color=purple] (8) at (1, -0.25) {};
		\node [contact, outer sep=5pt,fill=gray,color=purple] (9) at (1, -0.75) {};
		\node [contact, outer sep=5pt,fill=gray,color=purple] (10) at (1, -1.25) {};
		\node [style=none] (11) at (-2.75, -0) {$X$};
		\node [style=none] (wd) at (1.75, 0) {$Y$};
		\draw [rounded corners=5pt,  line width=1.5pt,dashed] 
   (node cs:name=0, anchor=north west) --
   (node cs:name=1, anchor=south west) --
   (node cs:name=6, anchor=south east) --
   (node cs:name=5, anchor=north east) --
   cycle;
		\draw [rounded corners=5pt, dashed] 
   (node cs:name=2, anchor=north west) --
   (node cs:name=3, anchor=south west) --
   (node cs:name=3, anchor=south east) --
   (node cs:name=2, anchor=north east) --
   cycle;
		\draw [rounded corners=5pt, dashed] 
   (node cs:name=4, anchor=north west) --
   (node cs:name=4, anchor=south west) --
   (node cs:name=10, anchor=south east) --
   (node cs:name=9, anchor=north east) --
   cycle;
   		\draw [rounded corners=5pt, dashed] 
   (node cs:name=7, anchor=north west) --
   (node cs:name=7, anchor=south west) --
   (node cs:name=7, anchor=south east) --
   (node cs:name=7, anchor=north east) --
   cycle;
   		\draw [rounded corners=5pt, dashed] 
   (node cs:name=8, anchor=north west) --
   (node cs:name=8, anchor=south west) --
   (node cs:name=8, anchor=south east) --
   (node cs:name=8, anchor=north east) --
   cycle;
\end{tikzpicture}
\]
and this be a partition of $Y+Z$:
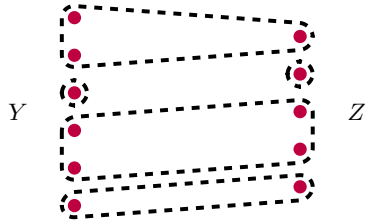
\[
\begin{tikzpicture}[circuit ee IEC]
		\node [style=none] (0) at (-2.75, -0) {$Y$};
		\node [style=none] (1) at (1.75, 0) {$Z$};
		\node [contact, outer sep=5pt,fill=gray,color=purple] (2) at (-2, 1.25) {};
		\node [contact, outer sep=5pt,fill=gray,color=purple] (3) at (-2, 0.75) {};
		\node [contact, outer sep=5pt,fill=gray,color=purple] (4) at (-2, 0.25) {};
		\node [contact, outer sep=5pt,fill=gray,color=purple] (5) at (-2, -0.25) {};
		\node [contact, outer sep=5pt,fill=gray,color=purple] (6) at (-2, -0.75) {};
		\node [contact, outer sep=5pt,fill=gray,color=purple] (7) at (-2, -1.25) {};
		\node [contact, outer sep=5pt,fill=gray,color=purple] (8) at (1, 1) {};
		\node [contact, outer sep=5pt,fill=gray,color=purple] (9) at (1, 0.5) {};
		\node [contact, outer sep=5pt,fill=gray,color=purple] (10) at (1, -0) {};
		\node [contact, outer sep=5pt,fill=gray,color=purple] (11) at (1, -0.5) {};
		\node [contact, outer sep=5pt,fill=gray,color=purple] (12) at (1, -1) {};
		\draw  [rounded corners=5pt,  line width=1.5pt,dashed]  
   (node cs:name=2, anchor=north west) --
   (node cs:name=3, anchor=south west) --
   (node cs:name=8, anchor=south east) --
   (node cs:name=8, anchor=north east) --
   cycle;
		\draw  [rounded corners=5pt,  line width=1.5pt,dashed]  
   (node cs:name=4, anchor=north west) --
   (node cs:name=4, anchor=south west) --
   (node cs:name=4, anchor=south east) --
   (node cs:name=4, anchor=north east) --
   cycle;
		\draw  [rounded corners=5pt,  line width=1.5pt,dashed]  
   (node cs:name=5, anchor=north west) --
   (node cs:name=6, anchor=south west) --
   (node cs:name=11, anchor=south east) --
   (node cs:name=10, anchor=north east) --
   cycle;
		\draw  [rounded corners=5pt,  line width=1.5pt,dashed]  
   (node cs:name=7, anchor=north west) --
   (node cs:name=7, anchor=south west) --
   (node cs:name=12, anchor=south east) --
   (node cs:name=12, anchor=north east) --
   cycle;
		\draw  [rounded corners=5pt,  line width=1.5pt,dashed]  
   (node cs:name=9, anchor=north west) --
   (node cs:name=9, anchor=south west) --
   (node cs:name=9, anchor=south east) --
   (node cs:name=9, anchor=north east) --
   cycle;
\end{tikzpicture}
\]
We compose them by putting elements of $X$ and $Z$ in the same part, if there is a `path' between them:
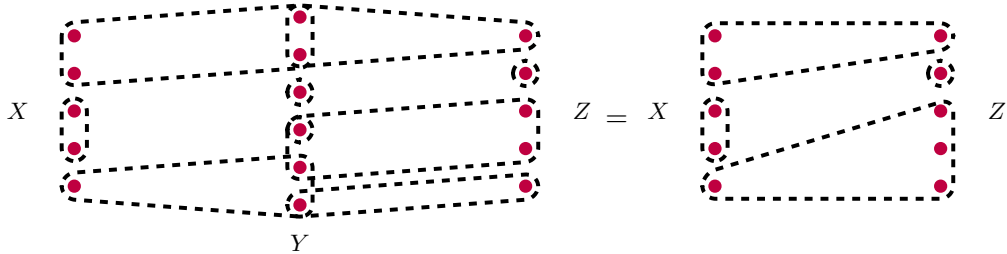
\[
  \begin{aligned}
\begin{tikzpicture}[circuit ee IEC]
		\node [contact, outer sep=5pt,fill=gray,color=purple] (-2) at (1, 1.25) {};
		\node [contact, outer sep=5pt,fill=gray,color=purple] (-1) at (1, 0.75) {};
		\node [contact, outer sep=5pt,fill=gray,color=purple] (0) at (1, 0.25) {};
		\node [contact, outer sep=5pt,fill=gray,color=purple] (1) at (1, -0.25) {};
		\node [contact, outer sep=5pt,fill=gray,color=purple] (2) at (1, -0.75) {};
		\node [contact, outer sep=5pt,fill=gray,color=purple] (3) at (1, -1.25) {};
		\node [style=none] (4) at (-2.75, -0) {$X$};
		\node [style=none] (5) at (4.75, -0) {$Z$};
		\node [contact, outer sep=5pt,fill=gray,color=purple] (6) at (-2, 1) {};
		\node [contact, outer sep=5pt,fill=gray,color=purple] (7) at (-2, -0.5) {};
		\node [contact, outer sep=5pt,fill=gray,color=purple] (8) at (-2, 0.5) {};
		\node [contact, outer sep=5pt,fill=gray,color=purple] (9) at (-2, -0) {};
		\node [contact, outer sep=5pt,fill=gray,color=purple] (10) at (-2, -1) {};
		\node [contact, outer sep=5pt,fill=gray,color=purple] (11) at (4, -0) {};
		\node [contact, outer sep=5pt,fill=gray,color=purple] (12) at (4, -1) {};
		\node [contact, outer sep=5pt,fill=gray,color=purple] (13) at (4, -0.5) {};
		\node [contact, outer sep=5pt,fill=gray,color=purple] (14) at (4, 0.5) {};
		\node [contact, outer sep=5pt,fill=gray,color=purple] (19) at (4, 1) {};
		\node [style=none] (20) at (1, -1.75) {$Y$};
		\node [style=none] (21) at (1, 1.75) {\phantom{$Y$}};
		\draw  [rounded corners=5pt,  line width=1.5pt,dashed]  
   (node cs:name=6, anchor=north west) --
   (node cs:name=8, anchor=south west) --
   (node cs:name=-1, anchor=south east) --
   (node cs:name=-2, anchor=north east) --
   cycle;
		\draw  [rounded corners=5pt,  line width=1.5pt,dashed]  
   (node cs:name=9, anchor=north west) --
   (node cs:name=7, anchor=south west) --
   (node cs:name=7, anchor=south east) --
   (node cs:name=9, anchor=north east) --
   cycle;
		\draw  [rounded corners=5pt,  line width=1.5pt,dashed]  
   (node cs:name=10, anchor=north west) --
   (node cs:name=10, anchor=south west) --
   (node cs:name=3, anchor=south east) --
   (node cs:name=2, anchor=north east) --
   cycle;
		\draw  [rounded corners=5pt,  line width=1.5pt,dashed]  
   (node cs:name=-2, anchor=north west) --
   (node cs:name=-1, anchor=south west) --
   (node cs:name=19, anchor=south east) --
   (node cs:name=19, anchor=north east) --
   cycle;
		\draw  [rounded corners=5pt,  line width=1.5pt,dashed]  
   (node cs:name=0, anchor=north west) --
   (node cs:name=0, anchor=south west) --
   (node cs:name=0, anchor=south east) --
   (node cs:name=0, anchor=north east) --
   cycle;
		\draw  [rounded corners=5pt,  line width=1.5pt,dashed]  
   (node cs:name=1, anchor=north west) --
   (node cs:name=1, anchor=south west) --
   (node cs:name=1, anchor=south east) --
   (node cs:name=1, anchor=north east) --
   cycle;
		\draw  [rounded corners=5pt,  line width=1.5pt,dashed]  
   (node cs:name=1, anchor=north west) --
   (node cs:name=2, anchor=south west) --
   (node cs:name=13, anchor=south east) --
   (node cs:name=11, anchor=north east) --
   cycle;
		\draw  [rounded corners=5pt,  line width=1.5pt,dashed]  
   (node cs:name=3, anchor=north west) --
   (node cs:name=3, anchor=south west) --
   (node cs:name=12, anchor=south east) --
   (node cs:name=12, anchor=north east) --
   cycle;
		\draw  [rounded corners=5pt,  line width=1.5pt,dashed]  
   (node cs:name=14, anchor=north west) --
   (node cs:name=14, anchor=south west) --
   (node cs:name=14, anchor=south east) --
   (node cs:name=14, anchor=north east) --
   cycle;
\end{tikzpicture}
\end{aligned}
\:
  =
\:
\begin{aligned}
\begin{tikzpicture}[circuit ee IEC]
		\node [style=none] (0) at (-2.75, -0) {$X$};
		\node [style=none] (1) at (1.75, -0) {$Z$};
		\node [contact, outer sep=5pt,fill=gray,color=purple] (2) at (-2, 1) {};
		\node [contact, outer sep=5pt,fill=gray,color=purple] (3) at (-2, -0.5) {};
		\node [contact, outer sep=5pt,fill=gray,color=purple] (4) at (-2, 0.5) {};
		\node [contact, outer sep=5pt,fill=gray,color=purple] (5) at (-2, -0) {};
		\node [contact, outer sep=5pt,fill=gray,color=purple] (6) at (-2, -1) {};
		\node [contact, outer sep=5pt,fill=gray,color=purple] (7) at (1, -0) {};
		\node [contact, outer sep=5pt,fill=gray,color=purple] (8) at (1, -1) {};
		\node [contact, outer sep=5pt,fill=gray,color=purple] (9) at (1, -0.5) {};
		\node [contact, outer sep=5pt,fill=gray,color=purple] (10) at (1, 0.5) {};
		\node [contact, outer sep=5pt,fill=gray,color=purple] (13) at (1, 1) {};
		\node [style=none] (20) at (1, -1.75) {\phantom{$Y$}};
		\node [style=none] (21) at (1, 1.75) {\phantom{$Y$}};
		\draw  [rounded corners=5pt,  line width=1.5pt,dashed]  
   (node cs:name=2, anchor=north west) --
   (node cs:name=4, anchor=south west) --
   (node cs:name=13, anchor=south east) --
   (node cs:name=13, anchor=north east) --
   cycle;
		\draw  [rounded corners=5pt,  line width=1.5pt,dashed]  
   (node cs:name=5, anchor=north west) --
   (node cs:name=3, anchor=south west) --
   (node cs:name=3, anchor=south east) --
   (node cs:name=5, anchor=north east) --
   cycle;
		\draw  [rounded corners=5pt,  line width=1.5pt,dashed]  
   (node cs:name=6, anchor=north west) --
   (node cs:name=6, anchor=south west) --
   (node cs:name=8, anchor=south east) --
   (node cs:name=7, anchor=north east) --
   cycle;
		\draw  [rounded corners=5pt,  line width=1.5pt,dashed]  
   (node cs:name=10, anchor=north west) --
   (node cs:name=10, anchor=south west) --
   (node cs:name=10, anchor=south east) --
   (node cs:name=10, anchor=north east) --
   cycle;
\end{tikzpicture}
\end{aligned}
\]

 Similarly,  ``jointly monic" spans  $X \stackrel{f}{\rightarrow} S \stackrel{g}{\leftarrow} Y$ are equivalent to \define{relations} $R\subseteq X\times Y$. Since jointly monic spans are to jointly epic cospans as spans are to cospans, we make the following definition. 

\begin{definition}
Given sets $X$ and $Y$, a \define{corelation} from $X$ to $Y$ is a partition of $X+Y$, or equivalently, an isomorphism class of joinly epic cospans $X \stackrel{f}{\rightarrow} N \stackrel{g}{\leftarrow} Y$.
\end{definition}

\begin{definition}
Let $\Fin\Corel$ be the category with finite sets as objects and corelations as morphisms. 
\end{definition}

Note that although a corelation is a type of cospan, $\Fin\Corel$ is not a subcategory of $\Fin\Cospan$. Taking the composite of two corelations in the category $\Fin\Cospan$ does not necessarily result in a corelation. There is however a functor $H \maps \Fin\Cospan \to \Fin\Corel$ forcing any cospan to be jointly epic. Relations also determine a category $\Fin\Rel$ and there is a functor $M \maps \Fin\Span \to \Fin\Rel$ forcing a span to be jointly monic.

\begin{definition}
Let $\Fin\Rel$ be the category with finite sets as objects and relations as morphisms. 
\end{definition}  
 
In summary we have four important categories: $\Fin\Cospan$, $\Fin\Corel$, $\Fin\Span$, and $\Fin\Rel$. Cospans are to spans as corelations are to relations, and there are functors $H \maps \Fin\Cospan \to \Fin\Corel$, and $M\maps \Fin\Span \to \Fin\Rel$. There is much more to say about the relationship between these four categories, which we do after introducing the framework of props in Section \ref{sec:props}. For now we focus on equipping $\Fin\Cospan$ and $\Fin\Corel$  with more structure with the purpose of making them network-like.

In many examples of networks---including electrical circuits, signal-flow diagrams, and bond graphs---one can often set two networks side by side to form a larger network.  These networks also have a notion of ``swapping" or permuting pieces. For example, in the case of electrical circuits this is just passing wires across each other. In our terms, these structures are encompassed by ``symmetric monoidal categories."

\begin{definition}
A \define{monoidal category}  is a category $\C$ equipped with  
\begin{itemize}
\item a functor $\otimes \maps \C \times \C \to \C$ called the \define{tensor product},
\item an object $I\in \C$ called the unit object
\item a natural isomorphism $a_{x,y,z} \maps (x\otimes y) \otimes z \to x\otimes (y\otimes z)$ called the \define{associator},
\item a natural isomorphism $\lambda_x \maps I\otimes x \to x$ called the \define{left unitor},
\item a natural isomorphism $\rho_x \maps x\otimes I \to x$ called the \define{right unitor},
\end{itemize}
such that the following two kinds of diagrams commute for all objects involved:
\begin{itemize}
\item triangle identity 
\[
  \xymatrix{
    (x\otimes I) \otimes y \ar[rr]^-{a_{x,I,y}}  \ar[dr]_-{\rho_{x}\otimes \id_y} && x\otimes (I\otimes y) \ar[dl]^-{\id_x \otimes \lambda_y}  \\
 & x\otimes y
  } 
\]
\item the pentagon identity
\[
  \xymatrix{
 & (w\otimes x) \otimes (y\otimes z)  \ar[dr]^-{a_{w,x,y\otimes z}} \\
((w\otimes x) \otimes y) \otimes z   \ar[ur]^-{a_{w\otimes x,y, z}} \ar[d]^-{a_{w,x,y}\otimes id_z} && (w\otimes (x \otimes (y\otimes z))) \\
 (w\otimes (x \otimes y)) \otimes z \ar[rr]^-{a_{w, x\otimes y, z}} && w \times ((x\otimes y)\otimes z) \ar[u]^-{id_w \otimes a_{x,y,z}}
  } 
\]

\end{itemize}
\noindent A monoidal category is \define{strict} if the associator, left unitor, and right unitor are identities.
\end{definition}

In a monoidal category the tensor product provides a way that one can set objects ``side by side" with each other and also morphisms ``side by side" with each other. If our category also has a ``braiding" then not only can we tensor, we can also permute. 

\begin{definition}
A \define{braided monoidal category} is a monoidal category $(\C,\otimes, I)$ equipped with a natural isomorphism $B_{x,y} \maps x\otimes y \to y\otimes x$ called the \define{braiding} such that the following two kinds of diagrams commute for all objects involved:
\[
  \xymatrixcolsep{45pt}
  \xymatrix{
 (x\otimes y) \otimes z \ar[r]^-{a_{x,y,z}} \ar[d]_-{B_{x,y} \times \id_z} & x\otimes (y\otimes z) \ar[r]^-{B_{x,y\otimes z}} & (y\otimes z) \otimes x \ar[d]^-{a_{y,z,x}} \\
(y\otimes x) \otimes z \ar[r]^-{a_{y,x,z}} & y\otimes (x\otimes z) \ar[r]^-{\id_y\otimes B_{x,z}} & y\otimes (z\otimes x)
  } 
\]

\[
  \xymatrixcolsep{45pt}
  \xymatrix{
 x\otimes (y \otimes z) \ar[r]^-{a^{-1}_{x,y,z}} \ar[d]_-{\id_x \otimes B_{y,z}} & (x\otimes y) \otimes z \ar[r]^-{B_{x\otimes y, z}} & z\otimes (x \otimes y) \ar[d]^-{a^{-1}_{z,x,y}} \\
x \otimes (z \otimes y)  \ar[r]^-{a^{-1}_{x,z,y}} & (x \otimes z) \otimes y) \ar[r]^-{B_{x,z} \otimes \id_y} & (z \otimes x ) \otimes y
  } 
\]
\end{definition}

 In all of our examples braiding two objects and then braiding them again is the same as doing nothing at all. This  happens in a symmetric monoidal category.

\begin{definition}

A \define{symmetric monoidal category} is a braided monoidal category where the braiding is such that $B_{y,x} B_{x,y} = \id_{x\otimes y}$.

\end{definition}


\begin{example}\label{ex:cospanbraid}
$\Fin\Cospan$ is a symmetric monoidal category where the tensor product is disjoint union, denoted by $+$, and the unit object is the empty set. The braiding morphism $B_{m,n} \maps m+n \to n+m$ is given by the cospan $m+n \stackrel{f}{\rightarrow} m+n \stackrel{g}{\leftarrow} n+m$ where the image of $m$ under $f$ is the same as the last $m$ of the apex, the image of $n$ under $f$ is the first $n$ of the apex, the image of $m$ under $g$ is the first $m$ of the apex, and the image of $n$ under $g$ is the last $n$ of the apex. 
\end{example}

\begin{example}
 $\Fin\Corel$ is also a symmetric monoidal category with $+$ as the tensor product and the empty set as the unit object.  Note that the braiding for $\Fin\Cospan$ is jointly epic so that the same braiding also makes $\Fin\Corel$ into a symmetric monoidal category.
\end{example}

\begin{example}
$\Fin\Span$ is a symmetric monoidal category where the tensor product is also $+$, and the unit object is the empty set. The braiding morphism $B_{m,n} \maps m\times n \to n\times m$ is more easily understood as the relation $$\{(a,b,b,a): a\in m, b\in n\}\subseteq (m\times n) \times (n\times m).$$   This construction also makes $\Fin\Rel$ into a symmetric monoidal category. 
\end{example}

At first glance the braiding morphisms for $\Fin\Cospan$ and $\Fin\Corel$ seem complicated. However,
Joyal and Street \cite{JS1} rigorously justified reasoning in any symmetric monoidal category using ``string diagrams"  so we may draw the morphisms in $\Fin\Cospan$ and $\Fin\Corel$ in this way.

Given a corelation $X \stackrel{f}{\rightarrow} N \stackrel{g}{\leftarrow} Y$ we connect elements of $X$ and $Y$ with strings if they are in the same image of either $f$ or $g$. That is, if elements are in the same part of the partition we connect them with a string, otherwise we do not. Also, we do not bother to draw the elements since the strings themselves tell us about the number of inputs and outputs. 

For example, the following partition in $\Fin\Corel$:
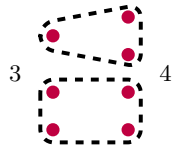
\[
\begin{aligned}
\begin{tikzpicture}
[circuit ee IEC, set resistor graphic=var resistor IEC
      graphic, scale=.5]
\scalebox{1}{
		\node [contact, inner sep=1pt, outer sep=5pt,fill,color=purple]  (0) at (-1, -0) {};
		\node  [contact, inner sep=1pt, outer sep=5pt,fill,color=purple] (2) at (1, 0.5) {};
		\node[contact, inner sep=1pt, outer sep=5pt,fill,color=purple] (3) at (1, -0.5) {};
		\node [contact , inner sep=1pt, outer sep=5pt,fill,color=purple]  (4) at (-1, -1.5) {};
		\node [contact, inner sep=1pt, outer sep=5pt, fill,color=purple] (5) at (-1, -2.5) {};
		\node[contact, inner sep=1pt, outer sep=5pt,fill,color=purple]  (6) at (1, -1.5) {};
		\node [contact, inner sep=1pt, outer sep=5pt, fill,color=purple]  (7) at (1, -2.5) {};
		\node [style=none] (9) at (-2, -1) {3};
		\node [style=none] (10) at (2, -1) {4};
		\draw [rounded corners=5pt,  line width=1.5pt,dashed] 
   (node cs:name=0, anchor=north west) --
   (node cs:name=0, anchor=south west) --
   (node cs:name=3, anchor=south east) --
   (node cs:name=2, anchor=north east) --
   cycle;
		\draw [rounded corners=5pt,  line width=1.5pt, dashed] 
   (node cs:name=4, anchor=north west) --
   (node cs:name=5, anchor=south west) --
   (node cs:name=7, anchor=south east) --
   (node cs:name=6, anchor=north east) --
   cycle;
}
      \end{tikzpicture}
\end{aligned}
\]
is drawn as:
\[
  \xymatrixrowsep{1pt}
  \xymatrixcolsep{5pt}
  
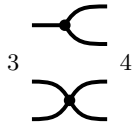
\xymatrix{
    \corelation{.075\textwidth}   
  }
\]
\noindent We draw the unit object as ``nothing," and draw the tensor product of two string diagrams by merely stacking one above the other.  The braiding is represented by a crossing of the strings. For example $B_{1,1}\maps 2 \to 2$ is drawn as:
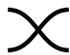
\[
  \xymatrixrowsep{1pt}
  \xymatrixcolsep{5pt}
  \xymatrix{
    \swap{.075\textwidth}
  }
\]


Composition of these string diagrams is essentially the same as composition of the partition pictures. The following example, where we also include the partition notation, shows the idea:
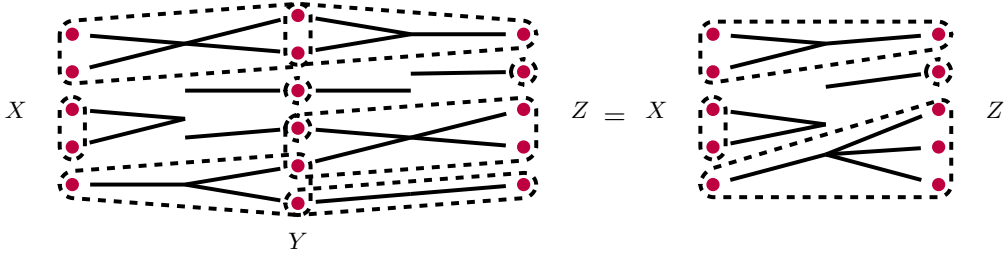
\[
  \begin{aligned}
\begin{tikzpicture}[circuit ee IEC]
		\node [contact, outer sep=5pt, fill=gray,color=purple] (6) at (-2, 1) {};
		\node [contact, outer sep=5pt,fill=gray,color=purple] (7) at (-2, -0.5) {};
		\node [contact, outer sep=5pt,fill=gray,color=purple] (8) at (-2, 0.5) {};
		\node [contact, outer sep=5pt,fill=gray,color=purple] (9) at (-2, -0) {};
		\node [contact, outer sep=5pt,fill=gray,color=purple] (10) at (-2, -1) {};
		\node [style=none] (15) at (-0.5, 0.875) {};
		\node [style=none] (28) at (-0.5, 0.25) {};
		\node [style=none] (16) at (-0.5, -0.125) {};
		\node [style=none] (29) at (-0.5, -0.375) {};
		\node [style=none] (17) at (-0.5, -1) {};
		\node [contact, outer sep=5pt,fill=gray,color=purple] (-2) at (1, 1.25) {};
		\node [contact, outer sep=5pt,fill=gray,color=purple] (-1) at (1, 0.75) {};
		\node [contact, outer sep=5pt,fill=gray,color=purple] (0) at (1, 0.25) {};
		\node [contact, outer sep=5pt,fill=gray,color=purple] (1) at (1, -0.25) {};
		\node [contact, outer sep=5pt,fill=gray,color=purple] (2) at (1, -0.75) {};
		\node [contact, outer sep=5pt,fill=gray,color=purple] (3) at (1, -1.25) {};
		\node [style=none] (18) at (2.5, -1.125) {};
		\node [style=none] (21) at (2.5, 1) {};
		\node [style=none] (22) at (2.5, -0.375) {};
		\node [style=none] (23) at (2.5, 0.475) {};
		\node [style=none] (24) at (2.5, 0.25) {};
		\node [contact, outer sep=5pt,fill=gray,color=purple] (19) at (4, 1) {};
		\node [contact, outer sep=5pt,fill=gray,color=purple] (14) at (4, 0.5) {};
		\node [contact, outer sep=5pt,fill=gray,color=purple] (11) at (4, -0) {};
		\node [contact, outer sep=5pt,fill=gray,color=purple] (13) at (4, -0.5) {};
		\node [contact, outer sep=5pt,fill=gray,color=purple] (12) at (4, -1) {};
		\node [style=none] (4) at (-2.75, -0) {$X$};
		\node [style=none] (5) at (4.75, -0) {$Z$};
		\node [style=none] (20) at (1, -1.75) {$Y$};
		\node [style=none] (30) at (1, 1.75) {\phantom{$Y$}};
		\draw [line width = 1.5pt] (6) to (15.center);
		\draw [line width = 1.5pt] (8) to (15.center);
		\draw [line width = 1.5pt] (-2) to (15.center);
		\draw [line width = 1.5pt] (-1) to (15.center);
		\draw [line width = 1.5pt] (9) to (16.center);
		\draw [line width = 1.5pt] (7) to (16.center);
		\draw [line width = 1.5pt] (10) to (17.center);
		\draw [line width = 1.5pt] (17.center) to (2);
		\draw [line width = 1.5pt] (17.center) to (3);
		\draw [line width = 1.5pt] (3) to (18.center);
		\draw [line width = 1.5pt] (18.center) to (12);
		\draw [line width = 1.5pt] (-2) to (21.center);
		\draw [line width = 1.5pt] (-1) to (21.center);
		\draw [line width = 1.5pt] (21.center) to (19);
		\draw [line width = 1.5pt] (1) to (22.center);
		\draw [line width = 1.5pt] (2) to (22.center);
		\draw [line width = 1.5pt] (22.center) to (11);
		\draw [line width = 1.5pt] (22.center) to (13);
		\draw [line width = 1.5pt] (23.center) to (14);
		\draw [line width = 1.5pt] (28.center) to (0);
		\draw [line width = 1.5pt] (0) to (24.center);
		\draw [line width = 1.5pt] (29.center) to (1);
		\draw [rounded corners=5pt,  line width=1.5pt,dashed]   
   (node cs:name=6, anchor=north west) --
   (node cs:name=8, anchor=south west) --
   (node cs:name=-1, anchor=south east) --
   (node cs:name=-2, anchor=north east) --
   cycle;
		\draw [rounded corners=5pt,  line width=1.5pt,dashed]   
   (node cs:name=9, anchor=north west) --
   (node cs:name=7, anchor=south west) --
   (node cs:name=7, anchor=south east) --
   (node cs:name=9, anchor=north east) --
   cycle;
		\draw [rounded corners=5pt,  line width=1.5pt,dashed]   
   (node cs:name=10, anchor=north west) --
   (node cs:name=10, anchor=south west) --
   (node cs:name=3, anchor=south east) --
   (node cs:name=2, anchor=north east) --
   cycle;
		\draw [rounded corners=5pt,  line width=1.5pt,dashed]   
   (node cs:name=-2, anchor=north west) --
   (node cs:name=-1, anchor=south west) --
   (node cs:name=19, anchor=south east) --
   (node cs:name=19, anchor=north east) --
   cycle;
		\draw [rounded corners=5pt,  line width=1.5pt,dashed]   
   (node cs:name=0, anchor=north west) --
   (node cs:name=0, anchor=south west) --
   (node cs:name=0, anchor=south east) --
   (node cs:name=0, anchor=north east) --
   cycle;
		\draw [rounded corners=5pt,  line width=1.5pt,dashed]   
   (node cs:name=1, anchor=north west) --
   (node cs:name=1, anchor=south west) --
   (node cs:name=1, anchor=south east) --
   (node cs:name=1, anchor=north east) --
   cycle;
		\draw [rounded corners=5pt,  line width=1.5pt,dashed]   
   (node cs:name=1, anchor=north west) --
   (node cs:name=2, anchor=south west) --
   (node cs:name=13, anchor=south east) --
   (node cs:name=11, anchor=north east) --
   cycle;
		\draw [rounded corners=5pt,  line width=1.5pt,dashed]   
   (node cs:name=3, anchor=north west) --
   (node cs:name=3, anchor=south west) --
   (node cs:name=12, anchor=south east) --
   (node cs:name=12, anchor=north east) --
   cycle;
		\draw [rounded corners=5pt,  line width=1.5pt,dashed]   
   (node cs:name=14, anchor=north west) --
   (node cs:name=14, anchor=south west) --
   (node cs:name=14, anchor=south east) --
   (node cs:name=14, anchor=north east) --
   cycle;
\end{tikzpicture}
\end{aligned}
\:
  =
\:
\begin{aligned}
\begin{tikzpicture}[circuit ee IEC]
		\node [style=none] (0) at (-2.75, -0) {$X$};
		\node [style=none] (1) at (1.75, -0) {$Z$};
		\node [contact, outer sep=5pt,fill=gray,color=purple] (2) at (-2, 1) {};
		\node [contact, outer sep=5pt,fill=gray,color=purple] (3) at (-2, -0.5) {};
		\node [contact, outer sep=5pt,fill=gray,color=purple] (4) at (-2, 0.5) {};
		\node [contact, outer sep=5pt,fill=gray,color=purple] (5) at (-2, -0) {};
		\node [contact, outer sep=5pt,fill=gray,color=purple] (6) at (-2, -1) {};
		\node [contact, outer sep=5pt,fill=gray,color=purple] (7) at (1, -0) {};
		\node [contact, outer sep=5pt,fill=gray,color=purple] (8) at (1, -1) {};
		\node [contact, outer sep=5pt,fill=gray,color=purple] (9) at (1, -0.5) {};
		\node [contact, outer sep=5pt,fill=gray,color=purple] (10) at (1, 0.5) {};
		\node [style=none] (11) at (-0.5, 0.875) {};
		\node [style=none] (12) at (-0.5, 0.3) {};
		\node [contact, outer sep=5pt,fill=gray,color=purple] (13) at (1, 1) {};
		\node [style=none] (14) at (-0.5, -0.2) {};
		\node [style=none] (15) at (-0.5, -0.6) {};
		\draw [line width = 1.5pt] (2) to (11.center);
		\draw [line width = 1.5pt] (4) to (11.center);
		\draw [line width = 1.5pt] (11.center) to (13);
		\draw [line width = 1.5pt] (5) to (14.center);
		\draw [line width = 1.5pt] (3) to (14.center);
		\draw [line width = 1.5pt] (15.center) to (7);
		\draw [line width = 1.5pt] (15.center) to (9);
		\draw [line width = 1.5pt] (6) to (15.center);
		\draw [line width = 1.5pt] (15.center) to (8);
		\draw [line width = 1.5pt] (12.center) to (10);
		\draw [rounded corners=5pt,  line width=1.5pt,dashed]   
   (node cs:name=2, anchor=north west) --
   (node cs:name=4, anchor=south west) --
   (node cs:name=13, anchor=south east) --
   (node cs:name=13, anchor=north east) --
   cycle;
		\draw [rounded corners=5pt,  line width=1.5pt,dashed]   
   (node cs:name=5, anchor=north west) --
   (node cs:name=3, anchor=south west) --
   (node cs:name=3, anchor=south east) --
   (node cs:name=5, anchor=north east) --
   cycle;
		\draw [rounded corners=5pt,  line width=1.5pt,dashed]   
   (node cs:name=10, anchor=north west) --
   (node cs:name=10, anchor=south west) --
   (node cs:name=10, anchor=south east) --
   (node cs:name=10, anchor=north east) --
   cycle;
		\draw [rounded corners=5pt,  line width=1.5pt,dashed]   
   (node cs:name=6, anchor=north west) --
   (node cs:name=6, anchor=south west) --
   (node cs:name=8, anchor=south east) --
   (node cs:name=7, anchor=north east) --
   cycle;
\end{tikzpicture}
\end{aligned}
\]

Although we have only explained the way in which corelations can be drawn as string diagrams, we can do the same with cospans. However, in the case of cospans there may be strings which are completely disconnected from the inputs and outputs. It is only because corelations are jointly epic cospans that we can always eliminate completely disconnected components. This fact follows from Theorem \ref{thm:fincorel_prop}, which provides a normal form for any corelation. Another reason for using string diagrams to represent  corelations is because the string diagrams look suspiciously similar to diagrams made of perfectly conductive wires. In fact we shall see that corelations behave like circuits made of perfect wires.

Just as functors preserve the structure of categories, there are functors which preserve these other structures. As expected, a monoidal functor preserves the monoidal structure between monoidal categories and a symmetric monoidal functor preserves the braiding.

\begin{definition}
Let $(\C,\otimes_{\C},I_{\C})$ and $(\D,\otimes_{\D}, I_{\D})$ be two monoidal categories. A \define{monoidal functor} between them is
\begin{itemize}
\item a functor $F\maps \C\to \D$
\item an isomorphism $\epsilon \maps I_{\D} \to F(I_{\C})$
\item a natural isomorphism $\mu_{x,y} \maps F(x) \otimes_{\D} F(y) \to F(x\otimes_{C} y)$ for all $x,y\in \C$
\end{itemize}
satisfying the following conditions:
\begin{itemize}
\item (associativity) For all objects $x,y,z \in \C$ the following diagram commutes:

\vspace{-3ex}

\[
  \xymatrix{
(F(x)\otimes_{\D} F(y))\otimes_{\D} F(z) \ar[rr]^-{a^{\D}_{F(x),F(y),F(z)}}  \ar[d]_-{\mu_{x,y} \otimes \id} && F(x) \otimes_{\D} (F(y)\otimes_{\D} F(z))  \ar[d]^-{\id \otimes \mu_{y,z}} \\
F(x\otimes_{\C} y) \otimes_{\D} F(z)  \ar[d]_-{\mu_{x\otimes_{\C} y, z}} && F(x)\otimes_{\D} (F(y\otimes_{\C} z)) \ar[d]^-{\mu_{x, y\otimes_{\C} z}} \\
F((x\otimes_{\C} y) \otimes_{\C} z) \ar[rr]^-{F(a^{\C}_{x,y,z})} && F(x\otimes_{\C} (y\otimes_{\C} z))
  } 
\]
where $a^{\C}$ and $a^{\D}$ denote the associators of the monoidal categories.

\item (unitality) For all $x\in \C$ the following diagrams commute:

\vspace{-2ex}

\[
  \xymatrix{
I_{\D} \otimes_{\D} F(x) \ar[r]^-{\epsilon \otimes \id} \ar[d]_-{\ell^{\D}_{F(x)}} & F(I_{\C}) \otimes_{\D} F(x) \ar[d]^-{\mu_{I_{\C},x}} \\
F(x) & F(I_{\C} \otimes_{\C} x) \ar[l]^-{F(\ell^{\C}_x)}
  } 
\]

\[
  \xymatrix{
  F(x)  \otimes_{\D} I_{\D} \ar[r]^-{\id \otimes \epsilon} \ar[d]_-{r^{\D}_{F(x)}} & F(x) \otimes_{\D} F(I_{\C}) \ar[d]^-{x, \mu_{I_{\C}}} \\
F(x) & F(x \otimes_{\C} I_{\C} ) \ar[l]^-{F(r^{\C}_x)}
  } 
\]
A monoidal functor is \define{strict} if the maps $\epsilon, \mu_{x,y}$ for all $x,y$ are identities.
\end{itemize}
\end{definition}

\begin{definition}
A monoidal functor $F\maps (\C,\otimes_{\C},I_{\C}) \to (\D,\otimes_{\D}, I_{\D})$ between symmetric monoidal categories is \define{symmetric} if for all $x,y \in \C$ the diagram:

\vspace{-2ex}

\[
  \xymatrixcolsep{45pt}
  \xymatrix{
  F(x)  \otimes_{\D} F(y) \ar[r]^-{B_{F(x),F(y)}} \ar[d]_-{\mu_{x,y}} & F(x) \otimes_{\D} F(y) \ar[d]^-{\mu_{y,x}} \\
F(x \otimes_{\C} y) \ar[r]^-{F(B_{x,y})} & F(x\otimes_{\C} y)
  } 
\]
commutes.
\end{definition}

In Proposition \ref{prop:fincospan_to_fincorel} we show that forcing a cospan to be jointly epic determines  a unique strict symmetric monoidal functor $H\maps \Fin\Cospan \to \Fin\Corel$. This functor plays a key role when we look at the relationships between $\Fin\Cospan$, $\Fin\Corel$, and circuits made of only of perfectly conductive wires. In many examples of networks the distinction between inputs and outputs is arbitrary. For categories this property is encompassed by ``dagger compact categories."


%

\begin{definition}
A \define{dagger category} is a category $\C$ equipped with an involutive contravariant endofunctor $(-)^{\dagger}\maps \C \to \C$ that is identity on objects. This associates to any morphism $f\maps A\to B$ a morphism $f^{\dagger} \maps B\to A$ such that for all $f\maps A\to B$ and $g\maps B\to C$ we have $\mathrm{id}_A = \mathrm{id}_A^{\dagger}$, $(gf)^{\dagger}= f^{\dagger}g^{\dagger}$, and $f^{\dagger\dagger} = f$. A morphism is \define{unitary} if its dagger is also its inverse.
\end{definition}
\begin{definition} A \define{symmetric monoidal dagger category} is a symmetric monoidal category $\C$ with a dagger structure $(-)^{\dagger}\maps \C \to \C$ that is symmetric monoidal and where the associator, unitors, and braiding of $\C$ are unitary. 
\end{definition}

\begin{example}
To see why $\Fin\Corel$ is a dagger category, notice that for any corelation $f\maps m \to n$ we get a corelation $f^{\dagger} \maps n \to m $ defined to give the same partition, but with the input and output formally exchanged. It can be shown that this makes $\Fin\Corel$ into a symmetric monoidal dagger category.
\end{example}

\begin{definition}
\label{def:zigzag}
In a symmetric monoidal category $\C$ an object $A^{*}$ is a \define{dual} of $A$ if it is equipped with two morphisms called the \define{unit} $\eta_A\maps I\to A \otimes A^{*}$ and the \define{counit} $\epsilon_A \maps A^{*} \otimes A \to I$ such that the following diagrams commute:

\vspace{-2ex}

\[
\xymatrix{
I\otimes A \ar[rr]^-{\eta_A \otimes \mathrm{id}_A} \ar[d]_-{r^{-1}_A \ell_A} && (A\otimes A^{*}) \otimes A \ar[d]^-{\alpha_{A,A^{*}, A}} \\
A\otimes I && A\otimes ( A^{*} \otimes A ) \ar[ll]^-{\mathrm{id}_A \otimes \epsilon_A}
}
\]

\[
\xymatrix{
A^{*}\otimes I \ar[rr]^-{ \mathrm{id}_{A^{*}} \otimes \eta_A} \ar[d]_-{\ell^{-1}_A r_A} && A^{*}\otimes (A \otimes A^{*})	\ar[d]^-{\alpha^{-1}_{A^{*}, A,A^{*}}} \\
I\otimes A^{*} && (A^{*}\otimes A) \otimes A^{*} \ar[ll]^-{\epsilon_A \otimes \mathrm{id}_A}
}
\]
A symmetric monoidal category $\C$ is \define{compact closed} if every object $A$ has a dual object $A^{*}$.
\end{definition}

One can show that any two duals of an object are canonically isomorphic, so one may speak of \emph{the} dual.

\begin{example}
 In $\Fin\Corel$ the objects are self-dual. This is most easily seen by looking at the object $1$ where the following morphisms act as the unit and counit respectively:

\[
  \xymatrixrowsep{1pt}
  \xymatrixcolsep{75pt}
  \xymatrix{
    \captwo{.07\textwidth}  &  \cuptwo{.07\textwidth}  \\
     0\to 2 & 2\to 0
  }
\]
The first corelation depicts the partition from $0$ to $2$ where both elements are in the same part, while the second is the same partition, but from $2$ to $0$. The identities in Definition \ref{def:zigzag} are sometimes called the ``zig-zag identities." This choice of name is clear when we express the laws in terms of string diagrams:

\[
  \xymatrixrowsep{1pt}
  \xymatrixcolsep{75pt}
  \xymatrix{
    \zigzaglaw{.07\textwidth}  =  \idone{.07\textwidth} =  \zigzaglawother{.07\textwidth}  \\
  }
\]

\end{example}

\begin{definition}
A \define{dagger compact category} is a symmetric monoidal dagger category that  is compact closed and such that the following diagram commutes for any $A$:

\vspace{-3ex}

\[
  \xymatrix{
    I \ar[r]^-{\epsilon^{\dagger}_A} \ar[dr]_-{\eta_A} & A\otimes A^{*} \ar[d]^-{B_{A,A^{*}}} \\
    & A^{*} \otimes A 
  }
\]
 where $B_{A,A^{*}}$ is the braiding on $A\otimes A^{*}$.
\end{definition}

\begin{example}
\noindent Given a set $n$, the unit corelation $\eta_n \maps 0\to 2n$ is defined as having $n$ parts arranged in the following pattern:
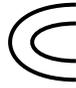
\[
 \xymatrixrowsep{1pt}
 \xymatrixcolsep{75pt}
 \xymatrix{
    \capn{.07\textwidth} 
 }
\]

Note that the dagger of this corelation is precisely the same as the counit $\epsilon_n \maps 2n\to 0$. When we compose $\eta_n$ with the braiding morphism $B_{n,n}$  the corelation does not change.  That is, the resulting corelation still has $n$ parts arranged in the same pattern. Thus we say that $\Fin\Corel$ is a dagger compact category. Fong showed that $\Fin\Cospan$ is also a dagger compact category \cite{BF}.
\end{example}

Having made $\Fin\Corel$ and $\Fin\Cospan$ into dagger compact categories, we are now ready to look at monoids and comonoids inside these two categories. 

\begin{definition}
Given a symmetric monoidal category $\C$, a \define{monoid} $(X,\mu, \iota)$ is an object $X\in \C$ together with a \define{multiplication} $\mu \maps X \otimes X \to X $ and \define{unit} $\iota \maps I \to X$ obeying the associative and unit laws. We can draw these laws using string diagrams as follows:
\[
  \xymatrixrowsep{.1pt}
  \xymatrixcolsep{1pt}
  \xymatrix{
    \assocl{.1\textwidth} &=& \assocr{.1\textwidth} &&&&&& \unitl{.1\textwidth} &=&
    \idone{.1\textwidth} &=&  \unitr{.1\textwidth} \\
     \mu (\mu\otimes \mathrm{id}_X) &=&  \mu(\mathrm{id}_X\otimes \mu) &&&&&& \mu (\iota \otimes \mathrm{id}_X) &=& \mathrm{id}_X  &=&  \mu (\mathrm{id}|_X\otimes \iota)
  }
\]
\end{definition}
\begin{definition}
A \define{comonoid} $(X,\delta,\epsilon)$ in $\C$ is an object $X \in \C$  together with a \define{comultiplication} $\delta \maps X \to X\otimes X $ and \define{counit} $\epsilon \maps X \to I$ obeying the coassociative and counit laws:
\[
  \xymatrixrowsep{.1pt}
  \xymatrixcolsep{1pt}
  \xymatrix{
    \coassocl{.1\textwidth} &=& \coassocr{.1\textwidth} &&&&&& \counitl{.1\textwidth} &=&
    \idone{.1\textwidth} &=&\counitr{.1\textwidth}   \\
     ( \mathrm{id}_X \otimes \delta) \delta &=& (\delta \otimes \mathrm{id}_X) \delta &&&&&& (\epsilon \otimes  \mathrm{id}_X) \delta &=&  \mathrm{id}_X  &=&  ( \mathrm{id}_X\otimes \epsilon)\delta
  }
\]
\end{definition}

\begin{definition}
A \define{Frobenius monoid} in a symmetric monoidal category $\C$ is a monoid $(X,\mu, \iota)$ together with a comonoid $(X,\delta,\epsilon)$ obeying the Frobenius laws:
\[
  \xymatrixrowsep{.1pt}
  \xymatrixcolsep{1pt}
  \xymatrix{
         \frobs{.1\textwidth} &=& \frobx{.1\textwidth} &=& \frobz{.1\textwidth} \\
     (\mathrm{id}_X \otimes \mu) (\delta \otimes \mathrm{id}_X) &=& \delta \mu  &=&  (\mu \otimes \mathrm{id}_X) (\mathrm{id}_X \otimes \delta)
  }
\]
\noindent A Frobenius monoid is
\begin{itemize}
\item \define{commutative} if:
\[
  \xymatrixrowsep{.1pt}
  \xymatrixcolsep{1pt}
  \xymatrix{
    \commute{.1\textwidth} &=& \mult{.07\textwidth} \\
     \mu B &=& \mu
  }
\]
where $B$ is the braiding on $X \otimes X$
\item \define{cocommutative} if:
\[
  \xymatrixrowsep{.1pt}
  \xymatrixcolsep{1pt}
  \xymatrix{
    \cocommute{.1\textwidth} &=& \comult{.07\textwidth} \\
  B \delta &=& \delta
  }
\]
Note that a Frobenius monoid is cocommutative if and only if it is commutative.
\item \define{special} if:
\[
  \xymatrixrowsep{.1pt}
  \xymatrixcolsep{1pt}
  \xymatrix{
    \spec{.1\textwidth} &=&  \idone{.1\textwidth} \\
   \mu \delta &=& \mathrm{id}_X
  }
\]
\item \define{extra} if:
\[
  \xymatrixrowsep{.1pt}
  \xymatrixcolsep{1pt}
  \xymatrix{
\extral{.1\textwidth} &=& \idonezero{.1\textwidth} \\
\epsilon \iota &=& \mathrm{id}_I
  }
\]
\end{itemize}
There is also a notion which is weaker than commutivity, which we need later.
\begin{itemize}
\item A monoid is \define{symmetric} if:
\[
  \xymatrixrowsep{.1pt}
  \xymatrixcolsep{1pt}
  \xymatrix{
    \symmetric{.1\textwidth} &=& \multunit{.07\textwidth} \\
 \epsilon \mu B &=& \epsilon \mu
  }
\]
\item A comonoid is \define{cosymmetric} if:
\[
  \xymatrixrowsep{.1pt}
  \xymatrixcolsep{1pt}
  \xymatrix{
    \cosymmetric{.1\textwidth} &=& \comultcounit{.07\textwidth} \\
 B \delta \iota &=& \delta \iota
  }
\]
Note that a Frobenius monoid is cosymmetric if and only if it is symmetric. 
\end{itemize}
\end{definition}

We use the same style of string diagrams to depict both corelations and morphisms in general symmetric monoidal categories, as in the above definitions, but context will clarify what we mean in any particular situation. One can quickly show that the object $1$ together with the following morphisms in $\Fin\Corel$ make up an extraspecial commutative Frobenius monoid by using string diagrams:  
\[
  \xymatrixrowsep{1pt}
  \xymatrixcolsep{30pt}
  \xymatrix{
    \mult{.1\textwidth}  & \unit{.1\textwidth}  & \comult{.4\textwidth} & \counit{.1\textwidth}  \\
    m\maps 1 + 1 \to 1 & i\maps 0 \to 1 & d\maps 1 \to 1 + 1 & e\maps 1 \to 0
  }
\]
\noindent Note that $(1,m,i,d,e)$ in $\Fin\Cospan$ is a special commutative Frobenius monoid, that is, we lose the \emph{extra} property. After we introduce props we show that the relationship between $\Fin\Cospan$ and $\Fin\Corel$ can also be viewed in terms of imposing the extra law on special commutative Frobenius monoids.

It should be noted that another common way of putting together a monoid and a comonoid is by using a ``bimonoid." It turns out that bimonoids are useful for understanding spans and relations, just as Frobenius monoids are useful for understanding cospans and corelations. Props not only allow us to state these relationships more precisely, but they also provide us with far more tools to study $\Fin\Cospan$, $\Fin\Corel$, $\Fin\Span$, and $\Fin\Rel$. Later, in Chapter \ref{chap:Bond_Graphs}, we use bimonoids and ``weak" bimonoids to study bond graphs. 


\section{Props}
\label{sec:props}

We now introduce the machinery of props. Mac Lane \cite{Ma65} introduced props in 1965 in order to generalize Lawvere's algebraic theories \cite{Law} to cases where the tensor product is not cartesian.   Mac Lane called them ``PROPs," which stands for ``products and permutations," but we prefer to write prop.  After proving that any prop can be presented in terms of generators and equations in Corollary \ref{cor:presentation}, we provide the presentations for $\Fin\Rel$, $\Fin\Span$, and $\Fin\Cospan$. We then prove Proposition \ref{prop:epimorphism} in order to view $\Fin\Corel$ as a quotient prop of $\Fin\Cospan$. 

One major reason to use props is because it is easy to define functors out of props by naming where the generators go and checking that the relevant equations hold. This is how we define the behavior of electrical circuits, signal-flow diagrams, and bond graphs. That is, we view a network as a morphism in a prop; then the behavior of the network corresponds to the image of that morphism under some functor. Even better, the functors are \emph{morphisms of props}.  

\begin{definition}
A \define{prop} is a strict symmetric monoidal category having the natural
numbers as objects, with the tensor product of objects given by addition.  We define a \define{morphism of props} to be a strict symmetric monoidal functor that is the identity on objects.  Let \define{$\PROP$} be the category of props.
\end{definition}

 The defining feature of a prop is that every object is isomorphic to a tensor power of some chosen object $x$. Though not every symmetric monoidal category is a prop, the categories $\Fin\Corel$, $\Fin\Cospan$, $\Fin\Rel$, and $\Fin\Span$ are all equivalent to props as a result of the following proposition.

\begin{proposition} 
\label{prop:strictification_1} 
A symmetric monoidal category $\C$ is equivalent to a prop if and only if there is an object $x \in \C$ such that every object of $\C$ is isomorphic to $x^{\otimes n} = 
x \otimes (x \otimes (x \otimes \cdots ))$ for some $n \in \N$.  
\end{proposition}

\begin{proof} 
See Section \ref{subsec:symmoncats} for a precise statement and proof.  The proof gives a recipe for actually constructing a prop equivalent to $\C$ when this is possible.
\end{proof}

\begin{proposition}
\label{prop:strictification_2}
Suppose $\C$ and $\T$ are props and $F \maps \C \to \T$ is a symmetric monoidal
functor.  Then $F$ is isomorphic to a strict symmetric monoidal functor $G \maps \C \to \T$.   If $F(1) = 1$, then $G$ is a morphism of props.
\end{proposition}

\begin{proof}
See Section \ref{subsec:symmoncats} for a precise statement and proof.  
\end{proof}

To better understand  props we study their \emph{algebras}. The idea is that an algebra provides a view of the morphisms in a prop as morphisms in some other category. As an example to keep in mind, consider the yet to be defined prop ``$\Circ$," whose morphisms are circuits made only of perfectly conductive wires.  Then a homset $\Circ(m,n)$  is  the set of possible circuits made only of perfectly conductive wires  with $m$ inputs and $n$ outputs.  In such a case black-boxing is some algebra  $F \maps \Circ \to \C$ that outputs the behavior of each circuit as a morphism in $\C$.

\begin{definition}
If $\T$ is a prop and $\C$ is a symmetric monoidal category, an 
\define{algebra of} $\T$ \define{in} $\C$ is a strict symmetric monoidal functor 
$F \maps \T \to \C$.   We define a morphism of algebras of $\T$ in  
$\C$ to be a monoidal natural transformation between such functors. 
\end{definition}

 In some contexts $\T(m,n)$  acts as the set of ``operations" with $m$ inputs and $n$ outputs and an algebra for $T$ equips an object in another category with these operations. In any case, by determining the algebras of a prop we gain a better understanding of it.

Our goal is to determine the algebras of $\Fin\Corel$ viewed as a prop. This also provides us with a presentation of $\Fin\Corel$ as a prop. We then study the relationships between $\Fin\Cospan$, $\Fin\Corel$, $\Fin\Span$, and $\Fin\Rel$ in terms of their algebras.  Before determining the algebras of $\Fin\Corel$ we look at some other props whose algebras are already known. 

\begin{example}
\label{ex:finset}
 By Proposition \ref{prop:strictification_1}  $\Fin\Set$ is equivalent to a prop. We henceforth write $\Fin\Set$ for this equivalent prop. We can identify this prop with a skeleton of the category of finite sets and functions, having finite ordinals $0, 1, 2, \dots$ as objects. It is known that the algebras of $\Fin\Set$ are commutative monoids \cite{Pi}. Consider the unique functions  $m \maps 2 \to 1$ and $i\maps 0 \to 1$. 

It is easy to check that $(1,m,i)$ is a commutative monoid in $\Fin\Set$. Then if $F\maps \Fin\Set \to \C$ is a strict symmetric monoidal functor, the object $F(1)$ becomes a commutative monoid in $\C$ with multiplication $F(m) \maps F(1) \otimes F(1) \to F(1)$ and unit $F(i)$. It is more work to check that any given commutative monoid in $\C$ comes from a unique choice of $F$. 

Additionally, morphisms between algebras of $\Fin\Set$ in $\C$ correspond to morphisms 
of commutative monoids in $\C$. As a result we say that $\Fin\Set$ is \define{the prop for 
commutative monoids}.
\end{example}

\begin{example}
\label{ex:fincospan}
Once again by Proposition \ref{prop:strictification_1} $\Fin\Cospan$ is equivalent to a prop, which we refer to as $\Fin\Cospan$. We again identify the objects of $\Fin\Cospan$ with finite ordinals.  Lack \cite{La} showed that $\Fin\Cospan$ is the prop for special commutative Frobenius monoids.  To see this first notice that the cospans $2\to 1 \leftarrow 1$ and $0 \to 1 \leftarrow 1$ make $1 \in \Fin\Cospan$ into a commutative monoid. This induces a morphism of props $\Fin\Set \to \Fin\Cospan$ assigning to any function $f \maps X \to Y$ between finite sets a cospan $X \stackrel{f}{\rightarrow} Y \stackrel{\mathrm{id}}{\leftarrow} Y$. Similarly the cospans $ 1 \to 1 \leftarrow 2$ and $1 \to 1 \leftarrow 0$ make $1\in\Fin\Cospan$ into a cocommutative comonoid. This induces a morphism of props $\Fin\Set^{\mathrm{op}} \to \Fin\Cospan$ assigning to any function $f \maps X \to Y$ between finite sets a cospan $Y \stackrel{\mathrm{id}}{\rightarrow} Y \stackrel{f}{\leftarrow} X$. 

It is easy to check that the commutative monoid and cocommutative comonoid structures on $1 \in \Fin\Corel$ fit together to form a special commutative Frobenius monoid.  The difficult part is to find an elegant proof that the category of algebras of $\Fin\Cospan$ in any strict symmetric monoidal category $\C$ is equivalent to the category of special commutative Frobenius monoids in $\C$. Lack showed this and thus we say that $\Fin\Cospan$ is the \define{prop for special commutative Frobenius monoids}.

\end{example}
 
\begin{example}
\label{ex:finspan}
By Proposition \ref{prop:strictification_1} $\Fin\Span$ is equivalent to a prop, which we call $\Fin\Span$. Lack \cite{La} has shown that $\Fin\Span$ is the prop for bicommutative bimonoids. A \define{bicommutative bimonoid} in $\C$ is a
commutative monoid $(x,\mu,\iota)$ and a cocommutative comonoid 
$(x,\delta,\epsilon)$ in $\C$ which together obey the \define{bimonoid laws}:
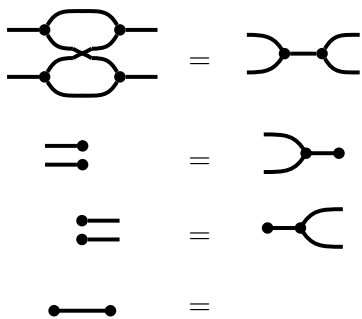
\[
  \xymatrixrowsep{1pt}
  \xymatrixcolsep{5pt}
  \xymatrix{
    \bi{.13\textwidth}  &= & \frobx{.11\textwidth} \\
     \parcounit{.06\textwidth}  &= & \multunit{.07\textwidth}  \\
     \parunit{.06\textwidth}   &=  & \comultcounit{.07\textwidth} \\
	\extral{.09\textwidth} & = &
  }
\]
 To see that $\Fin\Span$ is the prop for bicommutative bimonoids first notice that the spans $2 \stackrel{\mathrm{id}}{\leftarrow} 2 \rightarrow 1$ and $0\leftarrow 0 \to 1$ make the object $1\in \Fin\Span$ into a commutative monoid. This induces a morphisms of props $\Fin\Set \to \Fin\Span$ assigning to any function $f \maps X \to Y$ between finite sets a span $X \stackrel{f}{\leftarrow} Y \stackrel{\mathrm{id}}{\rightarrow} Y$. The spans $1 \leftarrow 2 \stackrel{\mathrm{id}}{\rightarrow}  2$ and $1\leftarrow 0 \to 0$ make $1\in \Fin\Span$ into a cocommutative comonoid.  This induces a morphism of props $\Fin\Set^{\mathrm{op}} \to \Fin\Span$ assigning to any function $f \maps X \to Y$ between finite sets a span $Y \stackrel{\mathrm{id}}{\leftarrow} Y \stackrel{f}{\rightarrow} X$.

It follows that the two functions $m\maps 2 \to 1$ and $i\maps 0 \to 1$ induce four spans  $m' \maps 2 \to 1$, $i' \maps 0 \to 1$, $d' \maps 1 \to 2$, and $e' \maps 1 \to 0$, which make $1$ into a bicommutative bimonoid.
\end{example}

\begin{example}
\label{ex:finrel}
By Proposition \ref{prop:strictification_1} $\Fin\Rel$ is equivalent to a prop, which we call $\Fin\Rel$. It is the prop for special bicommutative bimonoids.  The object $1 \in \Fin\Rel$ together with the four spans $m' \maps 2 \to 1$, $i' \maps 0 \to 1$, $d' \maps 1 \to 2$, and $e' \maps 1 \to 0$ now make up a special bicommutative bimonoid. This means there is a morphism of props $\Fin\Span \to \Fin\Rel$ sending any span to a relation, which forces a span to be jointly monic. Notice that this morphism of props sends each generator of $\Fin\Span$ to its coresponding generator in $\Fin\Rel$. However in $\Fin\Rel$ an additional equation holds: the special law $d'm'=id_{1}$.

The relationship between forcing a span to be jointly monic and forcing a bicommutative bimonoid to be special is as follows. Taking the jointly monic part of a span is equivalent to iteratively replacing the span $1\leftarrow 2 \to 1$ with the identity span $1\leftarrow 1 \rightarrow 1$. This iterative replacement is nothing more than forcing the equation $d'm'=id_{1}$ to hold in $\Fin\Span$.

Going from $\Fin\Span$ to $\Fin\Rel$ by imposing the special law hints at the way we shall go from $\Fin\Cospan$ to $\Fin\Corel$ by imposing the extra law. That is, forcing a cospan to be jointly epic is equivalent to iteratively replacing the cospan $0\to 1 \leftarrow 0$ with the identity cospan $0\to 0\leftarrow 0$. The formal statement and proof requires some additional results.
\end{example}

Recall that each of the props $\Fin\Span$, $\Fin\Rel$, $\Fin\Cospan$, and $\Fin\Corel$ were also dagger compact categories. Thus we make the following definitions. 

\begin{definition}
A \define{dagger prop} is a prop that is also a symmetric monoidal dagger category. A \define{dagger compact prop} is a prop that is also a dagger compact category.
\end{definition}

 Most of the morphisms of props we look at also preserve the dagger structure of our dagger props. 

\begin{definition}
A \define{symmetric monoidal dagger functor} is a symmetric monoidal functor $F\maps \C \to \D$ between symmetric monoidal dagger categories such that $F((-)^{\dagger}) = (F(-))^{\dagger}$, $\mu_{x,y}\maps F(x) \otimes F(y) \to F(x\otimes y)$ is unitary for all $x,y\in \C$, and $\eta \maps I \to F(I)$ is unitary. A \define{morphism of dagger props} is a morphism of props $F\maps C\to D$ between dagger props such that  $F((-)^{\dagger}) = (F(-))^{\dagger}$.
\end{definition}

Next we prove that any prop, much like any group, has a presentation in terms of generators and equations. 

\section{Presenting props}

Obtaining a group from a free group using quotients is actually a special case of using a coequalizer.  Thus, coequalizers underlie the technique of presenting groups in terms of generators and relations.   We use coequalizers to obtain props from free props, and to present props with generators and equations. 

\begin{definition}
In a category $\C$ a diagram:
\[
    \xymatrix{
      X \ar@<-.25ex>[r]_-{f} \ar@<1ex>[r]^-{g} & Y \ar[r]^-{p} & Z
    }
\]
is called a \define{coequalizer} if 
\begin{enumerate}
\item $p f = p g$;
\item $p$ is universal for this property: i.e.\ if $q\maps Y\to W$ is a morphism in $\C$ such that $q f = q g$ then there is a unique morphism $q'\maps Z\to W$ such that $q' p = q$. 
\end{enumerate}
We also call $Z$ \emph{the} \define{coequalizer} because it is unique up to canonical isomorphism. A morphism $p\maps Y \to Z$ is called a \define{regular epimorphism} if there is a coequalizer diagram:
\[
    \xymatrix{
      X \ar@<-.25ex>[r]_-{f} \ar@<1ex>[r]^-{g} & Y \ar[r]^-{p} & Z.
    }
\]
\end{definition}
\noindent The idea here is that the coequalizer is the quotient of $Y$ by equations imposed with the maps $f$ and $g$. This is how we present props with generators and equations.

\begin{definition} We define the category of \define{signatures} to be the functor category $\Set^{\N \times \N}$.  This has as objects the functors $\N \times \N \to \Set$ and as morphisms the natural transformations between such functors, where $\N\times \N$ is a discrete category.
\end{definition}
 
\begin{proposition}
\label{prop:monadic}
There is a forgetful functor
\[           U \maps \PROP \to \Set^{\N \times \N}  \]
sending any prop to its underlying signature and any morphism of props to its
underlying morphism of signatures.  This functor is \define{monadic}: that is, it has a left
adjoint 
\[          F \maps \Set^{\N \times \N} \to \PROP  \]
and the category of algebras of the monad $UF$ is equivalent, via the canonical
comparison functor, to the category $\PROP$.
\end{proposition}

\begin{proof} This follows from the fact that props are algebras of a multi-sorted or
`typed' Lawvere theory, together with a generalization of Lawvere's fundamental
result \cite{Law} to this context.  For details, see Appendix \ref{sec:appendix}.  \end{proof}

 We say that an \define{element} of $\Sigma \maps \N\times \N \to \Set$ is an element of some set $\Sigma(m,n)$. We say $ f\maps  m \to  n$ when $f$ is an element of some signature $\Sigma$ with $f \in \Sigma(m,n)$. For any signature $\Sigma$, we call $F(\Sigma)$ the \define{free prop} on $\Sigma$ where the elements of the signature $\Sigma$ are the generators of the prop $F(\Sigma)$. An important corollary to this proposition is that any prop can be obtained from a free prop using a coequalizer.

\begin{corollary}
\label{cor:presentation}
The category $\PROP$ is cocomplete, and any prop $\T$ is the coequalizer of some diagram
\[
    \xymatrix{
      F(E) \ar@<-.5ex>[r]_{\rho} \ar@<.5ex>[r]^{\lambda} & F(\Sigma).
    }
\]
\end{corollary}

\begin{proof}
This follows from the general theory of monads \cite{BW}.  We can take $\Sigma = U(\T)$
and $E = UFU(\T)$, with $\lambda = FU\epsilon_{\T}$ and $\rho = \epsilon_{FU(\T)}$,
where $\epsilon \maps FU \To 1$ is the counit of the adjunction in Proposition
\ref{prop:monadic}.  \end{proof}

To see how this proposition lets us define props in terms of generators and equations we give an example in complete detail. The underlying idea is that elements of the signature $\Sigma$ provide generators for the morphisms in the prop $\T$ while elements of the signature $E$ impose equations between morphisms in $\T$. Due to this we often refer to the elements of $\Sigma$ as generators and elements of $E$ as equations.

\begin{example}

We now present the \define{prop for monoids}. We first need generators that correspond to the multiplication $\mu \maps x\otimes x\to x$ and the unit $\iota \maps I \to x$ associated to some monoid $x$. Thus $\Sigma$ is a functor $\Sigma\maps \N \times \N \to \Set$ where the element $\mu \in \Sigma(2,1)$ corresponds to the multiplication and the element $\iota \in \Sigma(0,1)$ corresponds to the unit. Since a monoid only comes equipped with two morphisms we do not need any other elements so $\Sigma(m,n)$ is the empty set for another other $(m,n)$. Thus the free prop $F(\Sigma)$ has two elements $\mu \maps 2 \to 1$ and $\iota \maps 0 \to 1$ as generators.

More generally, picking a generating set works the same way, but we can be more succinct. In this case we instead say that $\Sigma$ is a signature with elements $\mu\maps 2 \to 1$ and $\iota\maps 0 \to 1$,  but this really means $\Sigma$ is a functor chosen in the above way.

Next, note that there are a total of $3$ equations describing the laws of a monoid: one for associativity and two for the unital laws. Thus we need the signature $E\maps \N\times \N \to \Set$  to pick out sets for each pair $(m,n)$ corresponding to the arity of the expressions needed to write the equations of a  monoid. For example, associativity is expressed via morphisms $ 3 \to 1$, while the unit laws are expressed with morphisms $1 \to 1$. Thus define $E(3,1)$ to be a singleton set and $E(1,1)$ to be a $2$ element set since there is one equation of arity $(3,1)$ and two equations of arity $(1,1)$. In short we say that $E$ is the signature with $3$ elements called $f\maps 3 \to 1$ and $g,h\maps 1 \to 1$. This defines another free prop $F(E)$.

To impose the equations of a monoid we use prop morphisms $\lambda,\rho \maps F(E)\to F(\Sigma)$ between the two free props. For each $e\in E(m,n)$ the morphisms $\lambda(e)$ and $\rho(e)$ in $F(\Sigma)$ are the left and right sides of an equation, respectively.


We define $\lambda$ in the following way:

\vspace{-3ex}

\[      \lambda(f) =  \mu (\mu \otimes \id_1) \]
\[      \lambda(g) =  \mu (\iota \otimes \id_1) \]
\[      \lambda(h) =  \id_1 \]
and $\rho$ in the following way:

\vspace{-3ex}

\[      \rho(f) =   \mu (\id_1 \otimes \mu) \]
\[      \rho(g) =  \id_1 \]
\[      \rho(h) =  \mu( \id_1 \otimes \iota) \]

These prop morphisms are well-defined just by naming where the generators go because $F(E)$ is a free prop.  Notice that $\lambda$ sends $f$ to one side of the associativity equation, $g$ to one side of the left unit equation, and $h$ to one side of the right unit equation. Meanwhile $\rho$ sends each element of $E$ to the other side of each equation.

Now we take the coequalizer, which we can do since $\PROP$ is cocomplete:
\[
    \xymatrix{
      F(E) \ar@<-.5ex>[r]_-{\rho} \ar@<.5ex>[r]^-{\lambda} & F(\Sigma) \ar[r] &  \T.
    }
\]
This results in the prop $\T$ which has the same morphisms as $F(\Sigma)$ but with the required equations holding between the above morphisms.  Since $\T$ is the prop generated by two morphisms making $1$ into a monoid, with no additional properties, we say that $\T$ is the \define{prop for monoids}. 

\end{example}

Presenting a prop in the above way can be streamlined. In the above example for monoids,  we instead say that $\Sigma$ consists of generators $m\maps 2 \to 1$ and $i\maps 0 \to 1$ while $E$ consists of the equations for a monoid. Then the coequalizer $\T$ is the prop for monoids. If we know that a category is equivalent to a prop for some algebraic object then we can easily present the prop as a coequalizer. 

%

\begin{example}
\label{ex:fincospan_presentation}
To present $\Fin\Cospan$ we can take $\Sigma$ to be the signature with elements $m \maps 2\to 1, i \maps 0 \to 1, d \maps 1 \to 2, e \maps 1\to 0,$ and let the equations be those governing a special commutative Frobenius monoid with multiplication $m$, unit $i$, comultiplication $d$ and counit $e$.  We obtain a coequalizer diagram:
\[
    \xymatrix{
      F(E) \ar@<-.5ex>[r]_-{\rho} \ar@<.5ex>[r]^-{\lambda} & F(\Sigma) \ar[r] & \Fin\Cospan.
    }
\]
\end{example}


There is yet another useful consequence that comes from presenting props with coequalizers. Given a prop $\T$ presented by $\Sigma$ and $E$ we have the following commuting diagram:
\[
    \xymatrix{
      F(E) \ar@<-.5ex>[r]_-{\rho} \ar@<.5ex>[r]^-{\lambda} & F(\Sigma) \ar[r]^-{t} & \T.
    }
\]
By the universal property of the coequalizer if we have another commuting diagram:
\[
    \xymatrix{
      F(E) \ar@<-.5ex>[r]_-{\rho} \ar@<.5ex>[r]^-{\lambda} & F(\Sigma)  \ar[r]^-{f} & P.
    }
\]
then there is a unique morphism of props $h\maps \T \to P$ making the resulting diagrams commute. 
Thus if we have a prop $P$ which has the same generators as $\T$ and includes the equations for $\T$ then there automatically exists a morphism of props from $P$ to $\T$. For example, if $\T$ is the prop for monoids and $P$ is the prop for commutative monoids then there must be a unique morphism $f\maps \T \to P$ since every commutative monoid is a monoid.

The more general idea is that, roughly,  adding generators to a presentation of a prop $P$ gives a new prop $P'$ having $P$ as a sub-prop, while adding equations gives a new prop $P'$ that is a quotient of $P$.  This is encompassed by the following proposition about props, however we prove even more general statements in the appendix.

\begin{proposition}
\label{prop:epimorphism}
 Suppose we have two coequalizer diagrams in $\PROP$:
\[
    \xymatrix{
      F(E) \ar@<-.5ex>[r]_-{\rho} \ar@<.5ex>[r]^-{\lambda} & F(\Sigma) \ar[r]^-{\pi} & P   
      }
\] 
\[
	\xymatrix{     
       F(E') \ar@<-.5ex>[r]_-{\rho'} \ar@<.5ex>[r]^-{\lambda'} & F(\Sigma') \ar[r]^-{\pi'} & P'
    }
\]
together with morphisms $f \maps E \to E'$, $g \maps \Sigma \to \Sigma'$ such that these squares commute:
\[ 
\xymatrix{
	F(E) \ar[r]^-{\lambda} \ar[d]_-{F(f)} & F(\Sigma) \ar[d]^-{F(g)} \\
	F(E') \ar[r]^-{\lambda'}                        & F(\Sigma') 
	}
\qquad \qquad
	\xymatrix{
	F(E) \ar[r]^-{\rho} \ar[d]_-{F(f)} & F(\Sigma) \ar[d]^-{F(g)} \\
	F(E') \ar[r]^-{\rho'}                        & F(\Sigma') 
	}
\] 
By the universal property of $P$  there exists a unique morphism $h \maps P \to P'$ making the square at right commute:
\[ 
\xymatrix{
F(E) \ar@<-.5ex>[r]_-{\rho} \ar@<.5ex>[r]^-{\lambda} \ar[d]_-{F(f)} 
& F(\Sigma) \ar[r]^-{\pi} \ar[d]_-{F(g)}             & P \ar[d]^-{h} \\
F(E') \ar@<-.5ex>[r]_-{\rho'} \ar@<.5ex>[r]^-{\lambda'} 
&	F(\Sigma') \ar[r]^-{\pi'}                        & P' 
	}
\]
In this situation, adding extra equations makes $P'$ into a quotient object of $P$.  More precisely, and also more generally: If $g$ is an epimorphism, then $h$ is a regular epimorphism. 

\end{proposition}

\begin{proof} 
Given that $g$ is an epimorphism in $\Set^T$, it is a regular epimorphism.
So is $F(g)$, since left adjoints preserve regular epimorphisms, and so is $\pi'$, by definition.
It follows that $\pi'  F(g) = h  \pi$ is a regular epimorphism, and thus so is $h$.
This means that the following is a coequalizer diagram:
\[ 
\xymatrix{
F(E) \ar@<-.5ex>[r]_-{\pi \rho} \ar@<.5ex>[r]^-{\pi  \lambda} & P \ar[r]^-{h} & P'  
	} \qedhere
\]
\end{proof}

We are now ready to prove that $\Fin\Corel$ is the prop for extraspecial commutative Frobenius monoids. Roughly speaking, this completes the relationship between spans, cospans, relations, corelations. The following table summarizes the idea: going from spans to relations is akin to imposing the special condition on extra bicommutative bimonoids, while going from cospans to corelations is akin to imposing the extra condition on special bicommutative Frobenius monoids.

\vspace{1ex}

\begin{figure}[H]
\begin{tabular}{c|c}
  \textbf{spans} & \textbf{cospans} \\
  extra bicommutative bimonoids & special bicommutative Frobenius monoids \\ \hline
  \textbf{relations} & \textbf{corelations} \\
  extra\textit{special} bicommutative bimonoids & \textit{extra}special bicommutative Frobenius monoids \\
\end{tabular}
\end{figure}

\chapter{Circuits and linear relations} 
\label{chap:Chapter2}

In the last chapter we described a symmetric monoidal category $\Fin\Corel$ whose morphisms are corelations.  We saw that $\Fin\Corel$ is equivalent to a prop where a morphism $f \maps m \to n$ is a corelation from an $m$-element set to an $n$-element set.  In this chapter we shall give a presentation for the prop $\Fin\Corel$ and make mathematically precise the relationship that the props $\Fin\Corel$ and $\Fin\Cospan$ have with electrical circuits and signal-flow diagrams.   

We begin by presenting $\Fin\Corel$ as a prop in Theorem \ref{thm:fincorel_prop}. For now this result clarifies the relationship between $\Fin\Span$, $\Fin\Rel$, $\Fin\Cospan$,  and $\Fin\Corel$ as props, but later it helps us to understand electrical circuits.  To begin studying electrical circuits as morphisms we define an ``$L$-circuit" to be a graph labelled by some set $L$, together with a cospan picking out inputs and outputs.  Depending on the nature of the set $L$ an $L$-circuit may represent many different kinds of electrical circuits. For example, if $L$ consists of non-negative real numbers then an $L$-circuit may be viewed as a circuit with inputs and ouputs made up of ``resistors," where the real number associated to an edge is the ``resistance" of the resistor. Given a set $L$ there is a prop $\Circ_L$ whose morphisms are $L$-circuits.   Unsurprisingly $\Circ_L$ is closely related to $\Fin\Cospan$ and $\Fin\Corel$; we show in Proposition \ref{prop:lcirc_coproduct} that $\Circ_L$ is the coproduct of $\Fin\Cospan$ and a prop coming from the set $L$.  

In the special case where $L=\{\ell\}$ is a one element set all of the edges in an $L$-circuit have the same label; this is the same as an $L$-circuit with unlabelled edges.  For this reason we think of such $L$-circuits as being made of only perfectly conductive wires between terminals and we define $\Circ := \Circ_{\{\ell\}}$.  In order to study more complicated circuits using our framework we first understsand $\Circ$ as a prop. Due to Proposition \ref{prop:lcirc_coproduct} any morphism in $\Circ$ has an underlying cospan, and due to Proposition \ref{prop:fincospan_to_fincorel}, there is a unique morphism of props from $\Fin\Cospan$ to $\Fin\Corel$. Thus any circuit made perfectly conductive wire has an underlying corelation. In \ref{def:blackbox_functor} we use the underlying corelation for a circuit to define a black-boxing functor for circuits made only of perfectly conductive wire. This definition agrees with the black-boxing functor defined by Baez and Fong \cite{BF}, but has the nice property of being completely determined by $\Fin\Corel$. That is, the behavior of any circuit made of only perfectly conductive wire is completely determined by the underlying corelation. As a result, when we study bond graphs in Chapter \ref{chap:Bond_Graphs} we take the view that corelations are the same as circuits made only of perfectly conductive wire.


\section{$\Fin\Corel$ as a Prop}

Though the following result is purely categorical we shall also use it quite often when dealing with electrical circuits and bond graphs.

\begin{theorem}\label{thm:fincorel_prop}
$\Fin\Corel$ is the prop for extraspecial commutative Frobenius monoids.
\end{theorem}

\begin{proof}

Let $\Sigma, E, \lambda, \rho$ present the prop $\Fin\Cospan$; the prop for special commutative Frobenius monoids where $F(\Sigma)$ is generated by $m \maps 2\to 1, i \maps 0 \to 1, d \maps 1 \to 2,$ and $e \maps 1\to 0$. So the following is a coequalizer diagram:

\[
    \xymatrix{
      F(E) \ar@<-.5ex>[r]_-{\rho} \ar@<.5ex>[r]^-{\lambda} & F(\Sigma) \ar@<.5ex>[r]^-{\pi} & \Fin\Cospan.
    }
\]

Let $\T$ be the prop for extraspecial commutative Frobenius monoids presented with $\Sigma, E', \lambda', \rho'$ so the following is a coequalizer diagram:

\[
    \xymatrix{
      F(E') \ar@<-.5ex>[r]_-{\rho'} \ar@<.5ex>[r]^-{\lambda'} & F(\Sigma) \ar@<.5ex>[r]^-{\pi'} & \T.
    }
\]

It follows from Proposition \ref{prop:epimorphism} that since $\Fin\Cospan$ and $\T$ are presented with the same generators, there exists a unique $h\maps \Fin\Cospan \to \T$ which is a regular epimorphism. In other words the following is another coequalizer:

\[
    \xymatrix{
      F(E) \ar@<-.5ex>[r]_-{\pi\rho} \ar@<.5ex>[r]^-{\pi\lambda} & \Fin\Cospan \ar[r]^-{h} & T.
    }
\]

Next we let  $E_{\textrm{Ex}}$ be a signature with just one element, namely $ex\maps 0\to 0$. Let the maps $\rho_{\textrm{Ex}}, \lambda_{\textrm{Ex}}\maps  F(E_{\textrm{Ex}}) \to F(\Sigma)$ present the ``extra" equation. That is, $\lambda_{\textrm{Ex}}(ex\maps 0 \to 0) = e i \maps 0 \to 0$ and $\rho_{\textrm{Ex}}(ex\maps 0 \to 0)=\id_0\maps 0 \to 0$.

We claim that the following is a coequalizer diagram:
\[
    \xymatrix{
      F(E_{\textrm{Ex}}) \ar@<-.5ex>[r]_-{\rho_{\textrm{Ex}}} \ar@<.5ex>[r]^-{\lambda_{\textrm{Ex}}} & F(\Sigma)  \ar[r]^-{\pi} & \Fin\Cospan \ar[r]^-{h} & \T.
    }
\]
\noindent First $\pi( \lambda_{\textrm{Ex}}(ex)) = 0\rightarrow 1 \leftarrow 0 = \pi(\lambda(ex))$ and $\pi( \rho_{\textrm{Ex}}(ex)) = 0\rightarrow 0 \leftarrow 0 = \pi(\rho(ex))$. Then since $h$ is a coequalizer for $\pi\lambda$ and $\pi\rho$ we get $h\pi\lambda_{\textrm{Ex}}(ex) = h\pi\lambda(ex)=h\pi\rho(ex)=h\pi\rho_{\textrm{Ex}}(ex)$ so that the diagram commutes. Next we check the universal property. 

Let $\C$ and $g\maps \Fin\Cospan \to \C$ be such that the diagram commutes:
\[
    \xymatrix{
      F(E_{\textrm{Ex}}) \ar@<-.5ex>[r]_-{\rho_{\textrm{Ex}}} \ar@<.5ex>[r]^-{\lambda_{\textrm{Ex}}} & F(\Sigma)  \ar[r]^-{\pi} & \Fin\Cospan \ar[r]^-{g} & \C.
    }
\]
\noindent Note then that the following diagram commutes:
\[
    \xymatrix{
      F(E) \ar@<-.5ex>[r]_-{\pi\rho} \ar@<.5ex>[r]^-{\pi\lambda} & \Fin\Cospan \ar[r]^-{g} & \C
    }
\]
\noindent and since the following is a coequalizer diagram:
\[
    \xymatrix{
      F(E) \ar@<-.5ex>[r]_-{\pi\rho} \ar@<.5ex>[r]^-{\pi\lambda} & \Fin\Cospan \ar[r]^-{h} & \T
    }
\]
\noindent there exists a unique $q\maps \T \to \C$ making the following commute:
\[
    \xymatrix{
      \Fin\Cospan \ar@<-.5ex>[r]^-{h} \ar@<-.5ex>[dr]_-{g} & \T \ar@<.5ex>[d]^-{q} \\
	& \C.
    }
\]
\noindent Thus our claim is proven and so the following is a coequalizer diagram:
 \[
    \xymatrix{
      F(E_{\textrm{Ex}}) \ar@<-.5ex>[r]_-{\rho_{\textrm{Ex}}} \ar@<.5ex>[r]^-{\lambda_{\textrm{Ex}}} & F(\Sigma)  \ar[r]^-{\pi} & \Fin\Cospan \ar[r]^-{h} & \T.
    }
\]

Now we show that the following is also a coequalizer diagram:
 \[
    \xymatrix{
      F(E_{\textrm{Ex}}) \ar@<-.5ex>[r]_-{\rho_{\textrm{Ex}}} \ar@<.5ex>[r]^-{\lambda_{\textrm{Ex}}} & F(\Sigma)  \ar[r]^-{\pi} & \Fin\Cospan \ar[r]^-{H} & \Fin\Corel.
    }
\]
\noindent Here $H\maps \Fin\Cospan \to \Fin\Corel$ is the canonical one $H(n \stackrel{f}{\rightarrow} a
  \stackrel{g}{\leftarrow} m) = (n \stackrel{f}{\rightarrow} \im[f,g]
  \stackrel{g}{\leftarrow} m)$ restricting the apex to the joint images of $f$ and $g$. See Fong \cite[Cor 3.18]{Fo2} for a proof that this restriction is functorial. The diagram commutes because $\pi\lambda_{\textrm{Ex}}(ex) = 0\rightarrow 1 \leftarrow 0$ and  $\pi\rho_{\textrm{Ex}}(ex) = 0\rightarrow 0 \leftarrow 0$, which are both mapped to $0\rightarrow 0 \leftarrow 0$ under $H$. Now we show the universal property. 

Let $\C$ and $A\maps \Fin\Cospan \to \C$ be such that:
 \[
    \xymatrix{
      F(E_{\textrm{Ex}}) \ar@<-.5ex>[r]_{\rho_{\textrm{Ex}}} \ar@<.5ex>[r]^{\lambda_{\textrm{Ex}}} & F(\Sigma)  \ar[r]^-{\pi} & \Fin\Cospan \ar[r]^-{A} & \C
    }
\]
\noindent also commutes. Note that $H\maps \Fin\Cospan \to \Fin\Corel$ is surjective on morphisms so if we show that any cospan $n \stackrel{f}{\rightarrow} a
  \stackrel{g}{\leftarrow} m$ has the same image as its jointly epic part $n \stackrel{f'}{\rightarrow} \im[f,g] \stackrel{g'}{\leftarrow} m$ then we can define a map $A'\maps \Fin\Corel \to \C$ by sending each jointly epic cospan to its image as a cospan under $A$. We check this now. First:
  \begin{align*}
    A\big(n \stackrel{f}{\rightarrow} a \stackrel{g}{\leftarrow} m\big)
    &= A\big(n \stackrel{f'}{\rightarrow} \im[f,g] \stackrel{g'}{\leftarrow} m\big) +
    A\big(0 \rightarrow (a-\im[f,g]) \leftarrow 0\big).\\
\end{align*}

\vspace{-1ex}
\noindent Note that $0 \rightarrow (a-\im[f,g]) \leftarrow 0$ is the same as the tensoring of the cospan $0\rightarrow 1\leftarrow 0$ with itself $(a-\im[f,g])$-many times so that:
\begin{align*}
   A\big(n \stackrel{f}{\rightarrow} a \stackrel{g}{\leftarrow} m\big) &= A\big(n \stackrel{f'}{\rightarrow} \im[f,g] \stackrel{g'}{\leftarrow} m\big) +
    A(0\rightarrow 1\leftarrow 0) + \cdots + A(0\rightarrow 1\leftarrow 0) \\
    &= A\big(n \stackrel{f'}{\rightarrow} \im[f,g] \stackrel{g'}{\leftarrow} m\big) +
   A\pi\lambda_{\textrm{Ex}}(ex) +\cdots + A\pi\lambda_{\textrm{Ex}}(ex) \\
    &= A\big(n \stackrel{f'}{\rightarrow} \im[f,g] \stackrel{g'}{\leftarrow} m\big)+
	A\pi\rho_{\textrm{Ex}}(ex) +\cdots + A\pi\rho_{\textrm{Ex}}(ex) \\
    &= A\big(n \stackrel{f'}{\rightarrow} \im[f,g] \stackrel{g'}{\leftarrow} m\big)+
	A(0\rightarrow 0\leftarrow 0) + \cdots + A(0\rightarrow 0\leftarrow 0) \\
    &= A\big(n \stackrel{f'}{\rightarrow} \im[f,g] \stackrel{g'}{\leftarrow} m\big)+
	\id_0 + \cdots + \id_0 \\
    &= A\big(n \stackrel{f'}{\rightarrow} \im[f,g] \stackrel{g'}{\leftarrow} m\big)
  \end{align*}

Thus:
 \[
    \xymatrix{
      F(E_{\textrm{Ex}} \ar@<-.5ex>[r]_-{\rho_{\textrm{Ex}}} \ar@<.5ex>[r]^-{\lambda_{\textrm{Ex}}} & F(\Sigma)  \ar[r]^-{\pi} & \Fin\Cospan \ar[r]^-{H} & \Fin\Corel
    }
\]
\noindent is a coequalizer diagram. Now $\T$ and $\Fin\Corel$ are both coequalizers in the following diagram:
\[
  \xymatrixrowsep{15pt}
  \xymatrixcolsep{35pt}
  \xymatrix{
    &  \Fin\Cospan
    \ar[dd]^{\id_{\Fin\Cospan}} \ar[r] & \T \ar[dd] \\
    F(E_{\mathrm{Ex}}) \ar@<0.6ex>[ur]^{\lambda_{\textrm{Ex}}} \ar@<-0.6ex>[ur]_{\rho_{\textrm{Ex}}}
    \ar@<0.6ex>[dr]^{\lambda_{\textrm{Ex}}}
    \ar@<-0.6ex>[dr]_{\rho_{\textrm{Ex}}}\\
    & \Fin\Cospan \ar[r] & \Fin\Corel.
  }
\]
\noindent The two triangles commute so that $\T$ and $\Fin\Corel$ are coequalizers of isomorphic diagrams and thus must be isomorphic. Note that this also expresses $\Fin\Corel$ as a `quotient' of $\Fin\Cospan$.
\end{proof}

Not only is the method for turning a cospan into a corelation functorial, but it determines a unique morphism of props. 

\begin{proposition}
\label{prop:fincospan_to_fincorel}
The map $H \maps \Fin\Cospan \to \Fin\Corel$ is a unique morphism of props.
\end{proposition}

\begin{proof}
We need only show uniqueness of $H$. It suffices to show that $H$ is uniquely determined on the generators $m , i, d, e$.   By Example \ref{ex:fincospan_presentation}, 
$H(m) \maps 2 \to 1, H(i) \maps 0 \to 1, H(d) \maps 1 \to 2$ and $H(e) \maps 1 \to 0$ must make $1 \in \Fin\Corel$ into a special commutative
Frobenius monoid.   

There is a unique corelation from $0$ to $1$, so $H(i)$ is uniquely determined.
Similarly, $H(e)$ is uniquely determined.  Each partition of a 3-element
set gives a possible choice for $H(m)$.   If we write $H(m)$ as 
a partition of the set $2 + 1$ with $2 = \{a,b\}$ and $1 = \{c\}$ these choices are:
\[      \{\{a,b,c\}\},  \;\;
\{\{a\}, \{b\}, \{c\}\} , \;\;
\{\{a,b\}, \{c\} \},  \;\;
\{\{a\}, \{b,c\} \},  \;\;
\{\{a,c\}, \{b\}\} .
\]
The commutative law rules out those that are not invariant under switching
$a$ and $b$, leaving us with these:
\[      \{\{a,b,c\}\},  \;\;
\{\{a\}, \{b\}, \{c\}\}, \;\;
\{\{a,b\}, \{c\} \} .
\]
The associative law rules out the last of these.  Either of the remaining two
makes $1 \in \Fin\Corel$ into a commutative monoid.   Dually, there are two
choices for $H(d)$ making $1$ into a commutative comonoid.  However, 
the `special' law forces $H(m) H(d)$ to be the identity corelation.
This ensures that $H(m)$ is the partition with just one block, namely $\{\{a,b,c\}\}$,
and similarly for $H(d)$.  
\end{proof}

Since $\Fin\Corel$ is the prop for extraspecial commutative Frobenius monoids we may as well use this to give a presentation of $\Fin\Corel$. 

\begin{example}
\label{ex:fincospan_to_fincorel}
To present $\Fin\Corel$ we can use the same signature $\Sigma$ as for $\Fin\Cospan$, but the equations $E'$ include one additional equation, the so-called \define{extra} law:
\[    e i = \id_0  \]
in the definition of extraspecial commutative Frobenius monoids. 
Thus we obtain a coequalizer diagram:
\[
    \xymatrix{
      F(E')\ar@<-.5ex>[r]_{\rho'} \ar@<.5ex>[r]^{\lambda'} & F(\Sigma) \ar[r] & \Fin\Corel.
    }
\]
More concisely we say that the prop $\Fin\Corel$ is generated by the following morphisms: 
\[
  \xymatrixrowsep{1pt}
  \xymatrixcolsep{30pt}
  \xymatrix{
    \mult{.1\textwidth}  & \unit{.1\textwidth}  & \comult{.4\textwidth} & \counit{.1\textwidth}  \\
    m\maps 1 + 1 \to 1 & i\maps 0 \to 1 & d\maps 1 \to 1 + 1 & e\maps 1 \to 0
  }
\]
subject to the laws of an extraspecial commutative Frobenius monoid.
\end{example}

Note by Theorem \ref{thm:fincorel_prop} that symmetric monoidal functors $F\maps \Fin\Corel \to \C$, correspond precisely to extraspecial commutative Frobenius monoids in $\C$. Thus whenever we have a category $\C$ with such an object we automatically have a functor $\Fin\Corel \to \C$. Later we see that for a field $k$, the one dimensional vector space $k$ can be made into two different extraspecial commutative Frobenius monoids. This allows us to define a functor assigning to any corelation a vector space of ``potentials" and ``currents."


\section{$L$-Circuits}

Rosebrugh, Sabadini and Walters \cite{RSW2} explained how to construct a category where the morphisms are circuits made of wires with ``circuit elements" on them, such as resistors, inductors and capacitors.    We use their work to construct a symmetric monoidal category $\Circ_L$ for any set $L$ of circuit elements. 

To begin with, a circuit consists of a finite graph with ``wires" as edges:

\begin{definition}
A \define{graph} is a finite set $E$ of \define{edges} and a finite set $N$ of \define{nodes} equipped with a pair of functions $s,t \maps E \to N$ assigning to 
each edge its \define{source} and \define{target}.  We say that $e \in E$ is an edge \define{from} $s(e)$ \define{to} $t(e)$.  
\end{definition}

We then label the edges with elements of some set $L$:

\begin{definition}
Given a set $L$ of \define{labels}, an \define{$L$-graph} is a graph $s,t\maps E\to N$
 equipped with a function $\ell \maps E \to L$ assigning a label to each edge.
\end{definition}

For example, we can describe a circuit made from resistors by labeling each
edge with a `resistance' chosen from the set $L = (0,\infty)$.   Here is a typical
$L$-graph in this case:

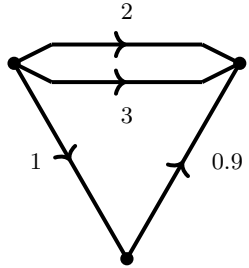
\begin{center}
    \begin{tikzpicture}[circuit ee IEC, set resistor graphic=var resistor IEC graphic]
\scalebox{1}{	
      \node[contact]         (A) at (0,0) {};
      \node[contact]         (B) at (3,0) {};
      \node[contact]         (C) at (1.5,-2.6) {};
      \coordinate         (ua) at (.5,.25) {};
      \coordinate         (ub) at (2.5,.25) {};
      \coordinate         (la) at (.5,-.25) {};
      \coordinate         (lb) at (2.5,-.25) {};
      \path (A) edge (ua);
      \path (A) edge (la);
      \path (B) edge (ub);
      \path (B) edge (lb);
      \path (ua) edge  [->-=.5] node[label={[label distance=1pt]90:{$2$}}] {} (ub);
      \path (la) edge  [->-=.5] node[label={[label distance=1pt]270:{$3$}}] {} (lb);
      \path (A) edge  [->-=.5] node[label={[label distance=2pt]180:{$1$}}] {} (C);
      \path (C) edge  [->-=.5] node[label={[label distance=2pt]0:{$0.9$}}] {} (B);
}
    \end{tikzpicture}
\end{center}

\noindent 
However, to make circuits into the morphisms of a category, we need to specify
some inputs and outputs.  We could do this by picking out two subsets of the set of nodes, but it turns out to be better to use two maps into the set of nodes.

\begin{definition}
Given a set $L$ and finite sets $X$ and $Y$, an \define{$L$-circuit from $X$ to $Y$} is a cospan of finite sets:
\[
  \xymatrix{
    & N \\
    X \ar[ur]^{i} && Y \ar[ul]_o
  }
\]
together with an $L$-graph:
\[ \xymatrix{L & E \ar@<2.5pt>[r]^{s} \ar@<-2.5pt>[r]_{t} \ar[l]_{\ell} & N.} \] 
We call the sets $X$, $Y$, and $\partial N = X \cup Y$ the \define{inputs},  \define{outputs}, and \define{terminals} of the $L$-circuit, respectively.
\end{definition}

\noindent
Here is an example of an $L$-circuit:

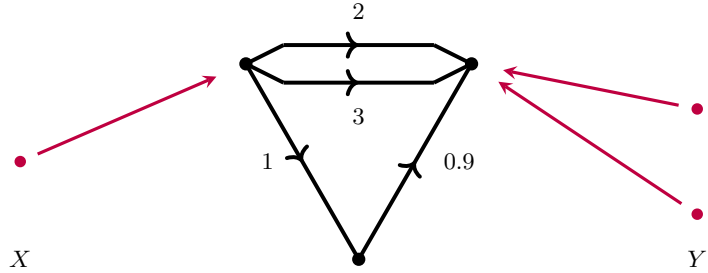
\begin{center}
    \begin{tikzpicture}[circuit ee IEC, set resistor graphic=var resistor IEC graphic]
\scalebox{1}{
      {\node[circle,draw,inner sep=1pt,fill=gray,color=purple]         (x) at
	(-3,-1.3) {};
	\node at (-3,-2.6) {$X$};}
      \node[contact]         (A) at (0,0) {};
      \node[contact]         (B) at (3,0) {};
      \node[contact]         (C) at (1.5,-2.6) {};
      {\node[circle,draw,inner sep=1pt,fill=gray,color=purple]         (y1) at
	(6,-.6) {};
	  \node[circle,draw,inner sep=1pt,fill=gray,color=purple]         (y2) at
	  (6,-2) {};
	  \node at (6,-2.6) {$Y$};}
      \coordinate         (ua) at (.5,.25) {};
      \coordinate         (ub) at (2.5,.25) {};
      \coordinate         (la) at (.5,-.25) {};
      \coordinate         (lb) at (2.5,-.25) {};
      \path (A) edge (ua);
      \path (A) edge (la);
      \path (B) edge (ub);
      \path (B) edge (lb);
      \path (ua) edge  [->-=.5] node[label={[label distance=1pt]90:{$2$}}] {} (ub);
      \path (la) edge  [->-=.5] node[label={[label distance=1pt]270:{$3$}}] {} (lb);
      \path (A) edge  [->-=.5] node[label={[label distance=2pt]180:{$1$}}] {} (C);
      \path (C) edge  [->-=.5] node[label={[label distance=2pt]0:{$0.9$}}] {} (B);
      {
	\path[color=purple, very thick, shorten >=10pt, shorten <=5pt, ->, >=stealth] (x) edge (A);
	\path[color=purple, very thick, shorten >=10pt, shorten <=5pt, ->, >=stealth] (y1) edge (B);
	\path[color=purple, very thick, shorten >=10pt, shorten <=5pt, ->, >=stealth] (y2)
      edge (B);}
}
    \end{tikzpicture}
  \end{center}
\noindent
Note that in this example, two points of $Y$ map to the same node. In general the input and output sets do not need to map injectively; we make no additional restrictions so that our work applies as generally as possible. 

For example, we can compose these two $L$-circuits:
 
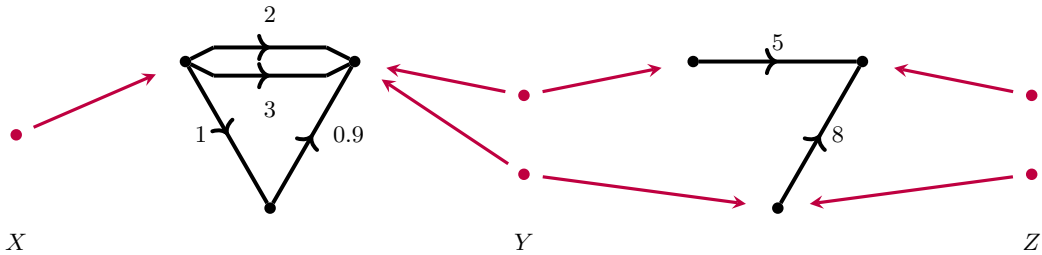
\begin{center}
    \begin{tikzpicture}[circuit ee IEC, set resistor graphic=var resistor IEC
      graphic,scale=.75]
      \node[circle,draw,inner sep=1pt,fill=gray,color=purple]         (x) at
      (-3,-1.3) {};
      \node at (-3,-3.2) {\footnotesize $X$};
      \node[circle,draw,inner sep=1pt,fill]         (A) at (0,0) {};
      \node[circle,draw,inner sep=1pt,fill]         (B) at (3,0) {};
      \node[circle,draw,inner sep=1pt,fill]         (C) at (1.5,-2.6) {};
      \node[circle,draw,inner sep=1pt,fill=gray,color=purple]         (y1) at
      (6,-.6) {};
      \node[circle,draw,inner sep=1pt,fill=gray,color=purple]         (y2) at
      (6,-2) {};
      \node at (6,-3.2) {\footnotesize $Y$};
      \coordinate         (ua) at (.5,.25) {};
      \coordinate         (ub) at (2.5,.25) {};
      \coordinate         (la) at (.5,-.25) {};
      \coordinate         (lb) at (2.5,-.25) {};
      \path (A) edge (ua);
      \path (A) edge (la);
      \path (B) edge (ub);
      \path (B) edge (lb);
      \path (ua) edge  [->-=.5] node[label={[label distance=1pt]90:{\footnotesize $2$}}] {} (ub);
      \path (la) edge  [->-=.5] node[label={[label distance=1pt]270:{\footnotesize $3$}}] {} (lb);
      \path (A) edge  [->-=.5] node[label={[label distance=-1pt]180:{\footnotesize $1$}}] {} (C);
      \path (C) edge  [->-=.5] node[label={[label distance=-1pt]0:{\footnotesize $0.9$}}] {} (B);
      \path[color=purple, very thick, shorten >=10pt, shorten <=5pt, ->, >=stealth] (x) edge (A);
      \path[color=purple, very thick, shorten >=10pt, shorten <=5pt, ->, >=stealth] (y1) edge (B);
      \path[color=purple, very thick, shorten >=10pt, shorten <=5pt, ->, >=stealth] (y2)
      edge (B);

      \node[circle,draw,inner sep=1pt,fill]         (A') at (9,0) {};
      \node[circle,draw,inner sep=1pt,fill]         (B') at (12,0) {};
      \node[circle,draw,inner sep=1pt,fill]         (C') at (10.5,-2.6) {};
      \node[circle,draw,inner sep=1pt,fill=gray,color=purple]         (z1) at
      (15,-.6) {};
      \node[circle,draw,inner sep=1pt,fill=gray,color=purple]         (z2) at (15,-2) {};
      \node at (15,-3.2) {\footnotesize $Z$};
      \path (A') edge  [->-=.5] node[above] {\footnotesize $5$} (B');
      \path (C') edge  [->-=.5] node[right] {\footnotesize $8$} (B');
      \path[color=purple, very thick, shorten >=10pt, shorten <=5pt, ->, >=stealth] (y1) edge (A');
      \path[color=purple, very thick, shorten >=10pt, shorten <=5pt, ->, >=stealth] (y2)
      edge (C');
      \path[color=purple, very thick, shorten >=10pt, shorten <=5pt, ->, >=stealth] (z1) edge (B');
      \path[color=purple, very thick, shorten >=10pt, shorten <=5pt, ->, >=stealth]
      (z2) edge (C');
    \end{tikzpicture}
\end{center}

\noindent
and obtain this one: 

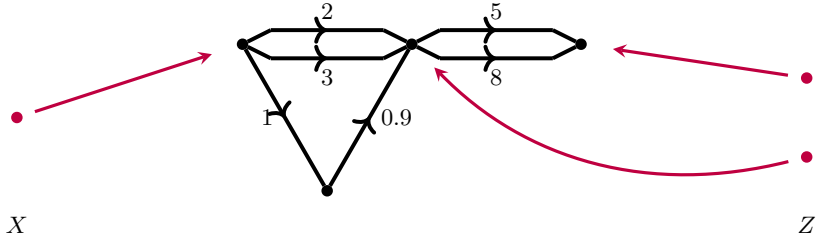
\begin{center}
    \begin{tikzpicture}[circuit ee IEC, set resistor graphic=var resistor IEC
      graphic,scale=0.75]
      \node[circle,draw,inner sep=1pt,fill=gray,color=purple]         (x) at (-4,-1.3) {};
      \node at (-4,-3.2) {\footnotesize $X$};
      \node[circle,draw,inner sep=1pt,fill]         (A) at (0,0) {};
      \node[circle,draw,inner sep=1pt,fill]         (B) at (3,0) {};
      \node[circle,draw,inner sep=1pt,fill]         (C) at (1.5,-2.6) {};
      \node[circle,draw,inner sep=1pt,fill]         (D) at (6,0) {};
      \coordinate         (ua) at (.5,.25) {};
      \coordinate         (ub) at (2.5,.25) {};
      \coordinate         (la) at (.5,-.25) {};
      \coordinate         (lb) at (2.5,-.25) {};
      \coordinate         (ub2) at (3.5,.25) {};
      \coordinate         (ud) at (5.5,.25) {};
      \coordinate         (lb2) at (3.5,-.25) {};
      \coordinate         (ld) at (5.5,-.25) {};
      \path (A) edge (ua);
      \path (A) edge (la);
      \path (B) edge (ub);
      \path (B) edge (lb);
      \path (B) edge (ub2);
      \path (B) edge (lb2);
      \path (D) edge (ud);
      \path (D) edge (ld);
      \node[circle,draw,inner sep=1pt,fill=gray,color=purple]         (z1) at
      (10,-.6) {};
      \node[circle,draw,inner sep=1pt,fill=gray,color=purple]         (z2) at (10,-2) {};
      \node at (10,-3.2) {\footnotesize $Z$};
      \path (ua) edge  [->-=.5] node[above] {\footnotesize $2$} (ub);
      \path (la) edge  [->-=.5] node[below] {\footnotesize $3$} (lb);
      \path (A) edge  [->-=.5] node[left] {\footnotesize $1$} (C);
      \path (C) edge  [->-=.5] node[right] {\footnotesize $0.9$} (B);
      \path (ub2) edge  [->-=.5] node[above] {\footnotesize $5$} (ud);
      \path (lb2) edge  [->-=.5] node[below] {\footnotesize $8$} (ld);
      \path[color=purple, very thick, shorten >=10pt, shorten <=5pt, ->, >=stealth] (x) edge (A);
      \path[color=purple, very thick, shorten >=10pt, shorten <=5pt, ->, >=stealth] (z1)
      edge (D);
      \path[bend left, color=purple, very thick, shorten >=10pt, shorten <=5pt, ->, >=stealth] (z2)
      edge (B);
    \end{tikzpicture}
  \end{center}

We formalize this using composition of cospans.  Given cospans 
$X \stackrel{i}{\rightarrow} N \stackrel{o}{\leftarrow} Y$ and 
$Y\stackrel{i'}{\rightarrow} N' \stackrel{o'}{\leftarrow} Z$, their composite is $X \stackrel{fi}{\longrightarrow} N+_Y N'
\stackrel{f'o'}{\longleftarrow} Z$, where the finite set $N+_Y N'$ and the functions $f$ and $f'$ are defined by this pushout:
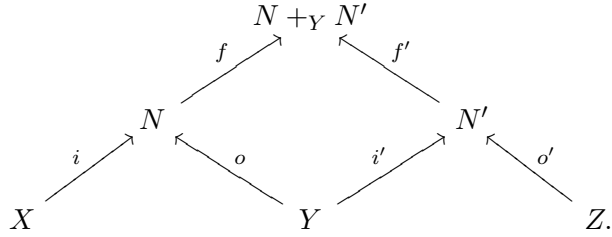
\[
  \xymatrix{
    && N+_YN' \\
    & \; N \; \ar[ur]^f && \; N' \ar[ul]_{f'} \\
 \; \;\; X \ar[ur]^i \;\; && \; Y \; \ar[ul]_o \ar[ur]^{i'} && \;\; Z. \;\; \ar[ul]_{o'}
  }
\]
To make the composite cospan into an $L$-circuit, we must choose an $L$-graph whose set of nodes is $N +_Y N'$.  We use this $L$-graph:
\[   \xymatrix{L & E + E' \ar@<2.5pt>[rr]^{j(s+s')} \ar@<-2.5pt>[rr]_{j(t+t')} \ar[l]_{\langle \ell, \ell'\rangle} && N +_Y N' } \]
where 
\[  \langle \ell, \ell' \rangle \maps E + E' \to L \]
is the `copairing' of $\ell$ and $\ell$, i.e.\, the function that
equals $\ell$ on $E$ and $\ell'$ on $E'$, while the functions
\[  s + s', t + t' \maps E + E' \to N + N' \]
are coproducts, and
\[   j_{N,N'} \maps N + N' \to N +_Y N' \]
is the natural map from the coproduct to the pushout.  

However, the pushout is only unique up to isomorphism, so to make composition
associative we must use isomorphism classes of $L$-circuits. Since an 
$L$-circuit can be seen as a graph with an extra structure, an isomorphism 
between $L$-circuits is an isomorphism of graphs that preserves this extra
structure:

\begin{definition}
Two $L$-circuits
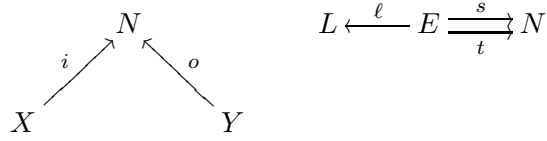
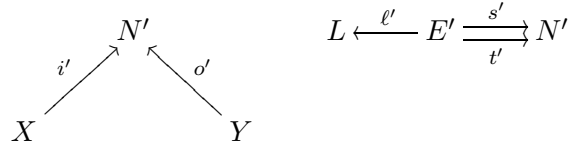
\[
  \xymatrix{
    & N \\
    X \ar[ur]^{i} && Y \ar[ul]_o
  } \qquad
\xymatrix{L & E \ar@<2.5pt>[r]^{s} \ar@<-2.5pt>[r]_{t} \ar[l]_{\ell} & N} 
\]
and
\[ 
  \xymatrix{
    & N' \\
    X \ar[ur]^{i'} && Y \ar[ul]_{o'}
  } \qquad
 \xymatrix{L & E' \ar@<2.5pt>[r]^{s'} \ar@<-2.5pt>[r]_{t'} \ar[l]_{\ell'} & N'} \] 
are \define{isomorphic} if there are bijections 
\[   f_E \maps E \to E', \quad f_N \maps N \to N'  \]
such that these diagrams commute:
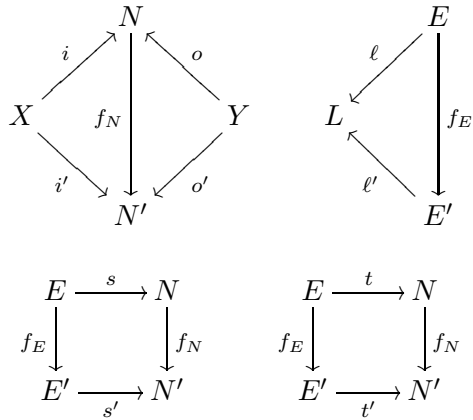
\[ 
  \xymatrix{
    & N \ar[dd]_{f_N} \\
    X \ar[ur]^{i} \ar[dr]_{i'} && Y \ar[ul]_o \ar[dl]^{o'} \\
    & N' 
  }   \qquad  
\xymatrix{&  E \ar[dd]^{f_E} \ar[dl]_{\ell} \\  L \\  & E' \ar[ul]^{\ell'} }
\]
\[ 
 \xymatrix{E \ar[d]_{f_E} \ar[r]^s & N \ar[d]^{f_N} \\ E' \ar[r]_{s'} & N'} 
\qquad 
\xymatrix{E \ar[d]_{f_E} \ar[r]^t & N \ar[d]^{f_N} \\ E' \ar[r]_{t'} & N'} 
\] 
\end{definition}
\noindent
It is easy to check that composition 
of $L$-circuits is well-defined and associative at the level of isomorphism classes.  

  We can also tensor two $L$-circuits by setting them side by side.  This uses coproducts, or disjoint unions, both of sets and of $L$-graphs.  For example, tensoring this
$L$-circuit:
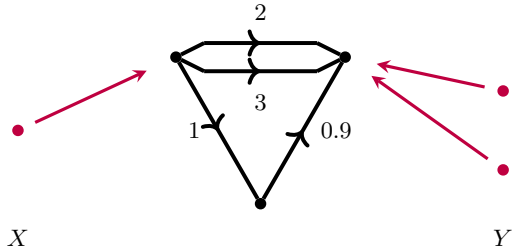
\begin{center}
      \begin{tikzpicture}[circuit ee IEC, set resistor graphic=var resistor IEC
	graphic,scale=.75]
	\node[circle,draw,inner sep=1pt,fill=gray,color=purple]         (x) at
	(-2.8,-1.3) {};
	\node at (-2.8,-3.2) {\footnotesize $X$};
	\node[circle,draw,inner sep=1pt,fill]         (A) at (0,0) {};
	\node[circle,draw,inner sep=1pt,fill]         (B) at (3,0) {};
	\node[circle,draw,inner sep=1pt,fill]         (C) at (1.5,-2.6) {};
	\node[circle,draw,inner sep=1pt,fill=gray,color=purple]         (y1) at
	(5.8,-.6) {};
	\node[circle,draw,inner sep=1pt,fill=gray,color=purple]         (y2) at
	(5.8,-2) {};
	\node at (5.8,-3.2) {\footnotesize $Y$};
	\coordinate         (ua) at (.5,.25) {};
	\coordinate         (ub) at (2.5,.25) {};
	\coordinate         (la) at (.5,-.25) {};
	\coordinate         (lb) at (2.5,-.25) {};
	\path (A) edge (ua);
	\path (A) edge (la);
	\path (B) edge (ub);
	\path (B) edge (lb);
	\path (ua) edge  [->-=.5] node[label={[label distance=0pt]90:{\footnotesize $2$}}] {} (ub);
	\path (la) edge  [->-=.5] node[label={[label distance=0pt]270:{\footnotesize $3$}}] {} (lb);
	\path (A) edge  [->-=.5] node[label={[label distance=-2pt]180:{\footnotesize $1$}}] {} (C);
	\path (C) edge  [->-=.5] node[label={[label distance=-2pt]0:{\footnotesize $0.9$}}] {} (B);
	\path[color=purple, very thick, shorten >=10pt, shorten <=5pt, ->, >=stealth] (x) edge (A);
	\path[color=purple, very thick, shorten >=10pt, shorten <=5pt, ->, >=stealth] (y1) edge (B);
	\path[color=purple, very thick, shorten >=10pt, shorten <=5pt, ->, >=stealth] (y2)
	edge (B);
      \end{tikzpicture}
\end{center}
with this one:
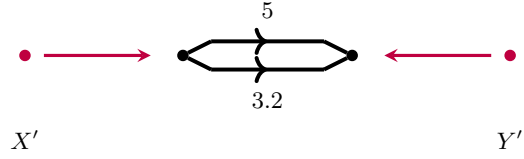
\begin{center}
      \begin{tikzpicture}[circuit ee IEC, set resistor graphic=var resistor IEC
	graphic,scale=.75]
	\node[circle,draw,inner sep=1pt,fill=gray,color=purple]         (x) at
	(-2.8,0) {};
	\node at (-2.8,-1.5) {\footnotesize $X'$};
	\node[circle,draw,inner sep=1pt,fill]         (A) at (0,0) {};
	\node[circle,draw,inner sep=1pt,fill]         (B) at (3,0) {};
	\node[circle,draw,inner sep=1pt,fill=gray,color=purple]         (y1) at
	(5.8,0) {};
	\node at (5.8,-1.5) {\footnotesize $Y'$};
	\coordinate         (ua) at (.5,.25) {};
	\coordinate         (ub) at (2.5,.25) {};
	\coordinate         (la) at (.5,-.25) {};
	\coordinate         (lb) at (2.5,-.25) {};
	\path (A) edge (ua);
	\path (A) edge (la);
	\path (B) edge (ub);
	\path (B) edge (lb);
	\path (ua) edge  [->-=.5] node[label={[label distance=0pt]90:{\footnotesize $5$}}] {} (ub);
	\path (la) edge  [->-=.5] node[label={[label distance=0pt]270:{\footnotesize $3.2$}}] {} (lb);
	\path[color=purple, very thick, shorten >=10pt, shorten <=5pt, ->, >=stealth] (x) edge (A);
	\path[color=purple, very thick, shorten >=10pt, shorten <=5pt, ->, >=stealth] (y1) edge (B);
      \end{tikzpicture}
\end{center}
give this one:
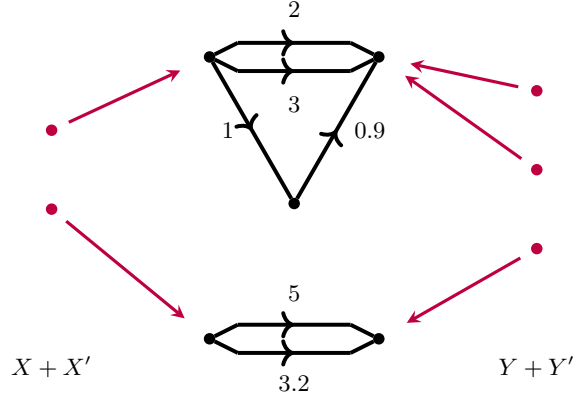
\begin{center}
      \begin{tikzpicture}[circuit ee IEC, set resistor graphic=var resistor IEC
	graphic,scale=.75]
	\node[circle,draw,inner sep=1pt,fill=gray,color=purple]         (x) at
	(-2.8,-1.3) {};
	\node at (-2.8,-5.5) {\footnotesize $X+X'$};
	\node[circle,draw,inner sep=1pt,fill]         (A) at (0,0) {};
	\node[circle,draw,inner sep=1pt,fill]         (B) at (3,0) {};
	\node[circle,draw,inner sep=1pt,fill]         (C) at (1.5,-2.6) {};
	\node[circle,draw,inner sep=1pt,fill=gray,color=purple]         (y1) at
	(5.8,-.6) {};
	\node[circle,draw,inner sep=1pt,fill=gray,color=purple]         (y2) at
	(5.8,-2) {};
	\node at (5.8,-5.5) {\footnotesize $Y+Y'$};
	\coordinate         (ua) at (.5,.25) {};
	\coordinate         (ub) at (2.5,.25) {};
	\coordinate         (la) at (.5,-.25) {};
	\coordinate         (lb) at (2.5,-.25) {};
	\path (A) edge (ua);
	\path (A) edge (la);
	\path (B) edge (ub);
	\path (B) edge (lb);
	\path (ua) edge  [->-=.5] node[label={[label distance=1pt]90:{\footnotesize $2$}}] {} (ub);
	\path (la) edge  [->-=.5] node[label={[label distance=1pt]270:{\footnotesize $3$}}] {} (lb);
	\path (A) edge  [->-=.5] node[label={[label distance=-2pt]180:{\footnotesize $1$}}] {} (C);
	\path (C) edge  [->-=.5] node[label={[label distance=-2pt]0:{\footnotesize $0.9$}}] {} (B);
	\path[color=purple, very thick, shorten >=10pt, shorten <=5pt, ->, >=stealth] (x) edge (A);
	\path[color=purple, very thick, shorten >=10pt, shorten <=5pt, ->, >=stealth] (y1) edge (B);
	\path[color=purple, very thick, shorten >=10pt, shorten <=5pt, ->, >=stealth] (y2)
	edge (B);
	\node[circle,draw,inner sep=1pt,fill=gray,color=purple]         (x2) at
	(-2.8,-2.7) {};
	\node[circle,draw,inner sep=1pt,fill]         (A1) at (0,-5) {};
	\node[circle,draw,inner sep=1pt,fill]         (B1) at (3,-5) {};
	\node[circle,draw,inner sep=1pt,fill=gray,color=purple]         (y3) at
	(5.8,-3.4) {};
	\coordinate         (ua1) at (.5,-4.75) {};
	\coordinate         (ub1) at (2.5,-4.75) {};
	\coordinate         (la1) at (.5,-5.25) {};
	\coordinate         (lb1) at (2.5,-5.25) {};
	\path (A1) edge (ua1);
	\path (A1) edge (la1);
	\path (B1) edge (ub1);
	\path (B1) edge (lb1);
	\path (ua1) edge  [->-=.5] node[label={[label distance=0pt]90:{\footnotesize $5$}}] {} (ub1);
	\path (la1) edge  [->-=.5] node[label={[label distance=0pt]270:{\footnotesize $3.2$}}] {} (lb1);
	\path[color=purple, very thick, shorten >=10pt, shorten <=5pt, ->,
	>=stealth] (x2) edge (A1);
	\path[color=purple, very thick, shorten >=10pt, shorten <=5pt, ->,
	>=stealth] (y3) edge (B1);
      \end{tikzpicture}
\end{center}
In general, given $L$-circuits
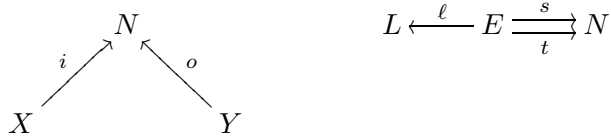
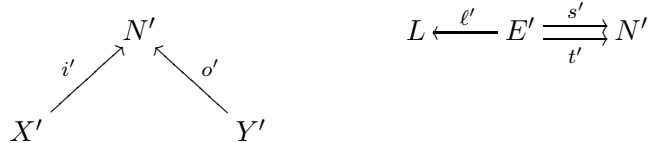
\[  \xymatrix{
    & N \\
    X \ar[ur]^{i} && Y \ar[ul]_o
  } \qquad \qquad \; 
\xymatrix{L & E \ar@<2.5pt>[r]^{s} \ar@<-2.5pt>[r]_{t} \ar[l]_{\ell} & N} 
\]
and
\[ 
  \xymatrix{
    & N' \\
    X' \ar[ur]^{i'} && Y' \ar[ul]_{o'}
  } \qquad \qquad \;
 \xymatrix{L & E' \ar@<2.5pt>[r]^{s'} \ar@<-2.5pt>[r]_{t'} \ar[l]_{\ell'} & N'} \] 
their tensor product is
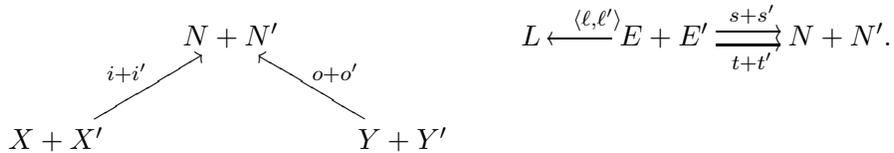
\[  \xymatrix{
    & N + N'\\
    X + X' \ar[ur]^{i+i'} && Y + Y'\ar[ul]_{o+o'}
  } \qquad
\xymatrix{L & E + E' \ar@<2.5pt>[r]^{s+s'} \ar@<-2.5pt>[r]_{t+t'} \ar[l]_{\langle \ell, \ell' \rangle} & N + N'.} 
\]
This operation is well-defined at the level of isomorphism classes.  Indeed, we obtain a symmetric monoidal category:

\begin{proposition}
\label{prop:lcirc_as_symmoncat}
For any set $L$, there is a symmetric monoidal category \define{$\Circ_L$} where the objects are finite sets, the morphisms are isomorphism classes of $L$-circuits, and composition and the tensor product are defined as above.
\end{proposition}

\begin{proof} 
This was proved by Rosebrugh, Sabadini and Walters \cite{RSW2}, who call this 
category \break $\mathrm{Csp}(\mathrm{Graph}_\mathrm{fin}/\Sigma)$,
 where $\Sigma$ is their name for the label set $L$.  Another style of proof uses Fong's theory of decorated cospans \cite{Fo1}, which implies that $\Circ_L$ is a compact closed symmetric monoidal
category.  Baez and Fong used this method in the special case where $L$ is the set of ``positive elements" in a field \cite[Cor.\ 5.3]{BF}, but the argument does not depend on the nature of the set $L$.  
\end{proof}

In fact, the work of Courser \cite{Co} and Stay \cite{St} implies that $\Circ_L$ comes from a compact closed symmetric monoidal bicategory.  The objects in this bicategory are still finite sets, but the morphisms in this bicategory are actual $L$-circuits, not isomorphism classes.  We expect this bicategorical approach to become important, but for now we are content to work with mere categories.


Fix a set $L$. Then by Proposition \ref{prop:strictification_1}, $\Circ_L$ is equivalent to a prop.   Henceforth, by a slight abuse of language, we use \define{$\Circ_L$} to denote this prop. To understand $\Circ_L$ as a prop first notice one very important fact: the object $1$ equipped with the generating cospans $m \maps 2\to 1, i \maps 0 \to 1, d \maps 1 \to 2$ and $e \maps 1\to 0$, thought of as $L$-circuits with no edges, is a special commutative Frobenius monoid in $\Circ_L$. It is no coincidence that $m,i,d$, and $e$ subject to the laws of a special commutative Frobenius monoid are also the generators of $\Fin\Cospan$. 

The idea is that the labelled edges in an $L$-circuit correspond to elements of $L$, while cospans $m,i,d,$ and $e$ let us put the labelled edges together in interesting ways. In the prop framework a labelled edge in an $L$-circuit is thought of as a morphism $\ell \maps 1 \to 1$ which is labelled by some element of $L$. If the prop $\Circ_L$ were generated only by morphisms $\ell$ like this then the morphisms of $\Circ_L$ would only look like parallel edges. 

 Thus the overall idea is that an $L$-circuit can be built from generating cospans and labelled edges; this intuition matches what we find out about the algebras of $\Circ_L$.  To describe the algebras of $\Circ_L$, we make a somewhat nonstandard definition. We say a set $L$ acts on an object $x$ if for each element of $L$ we have a morphism from $x$ to itself:


\begin{definition} 
An \define{action} of a set $L$ on an object $x$ in a category $\C$ is a function
$\alpha \maps L \to \hom(x,x)$.   We also call this an \define{$L$-action}.
Given $L$-actions $\alpha \maps L \to \hom(x,x)$ 
and $\beta \maps L \to \hom(y,y)$, a \define{morphism of $L$-actions} is
a morphism $f \maps x \to y$ in $\C$ such that 
$f \alpha(\ell) = \beta(\ell) f$ for all $\ell \in L$.   
\end{definition}

\begin{proposition}
\label{prop:lcirc_algebras}
An algebra of $\Circ_L$ in a strict symmetric monoidal category $\C$ is a 
special commutative Frobenius monoid in $\C$ whose underlying object 
is equipped with an action of $L$.  A morphism of algebras of $\Circ_L$ in $\C$ is 
a morphism of special commutative Frobenius monoids that is also a morphism
of $L$-actions.

\end{proposition}

\begin{proof}
This was proved by Rosebrugh, Sabadini and Walters \cite{RSW2}, though stated
in quite different language.  
\end{proof}

\noindent
We may thus say that $\Circ_L$ is the prop for special commutative Frobenius monoids whose underlying object is equipped with an action of $L$.  

Unsurprisingly, $\Circ_L$ is coproduct of two props: the prop for special commutative Frobenius monoids and the prop for $L$-actions.  To describe the latter, consider a signature with one unary operation for each element of $L$, and no other operations.  For simplicity we call this signature simply $L$.   The free prop $F(L)$ has a morphism $\ell\maps 1\to 1$ for each $\ell \in L$.  For any strict symmetric monoidal category $\C$, the category of algebras of $F(L)$ in $\C$ is the category of $L$-actions and morphisms of $L$-actions.  We thus call $F(L)$ \define{the prop for $L$-actions}.

\begin{proposition} 
\label{prop:lcirc_coproduct}
$\Circ_L$ is the coproduct of $\Fin\Cospan$ and the prop for $L$-actions.
\end{proposition}

\begin{proof} 
Let
\[  \xymatrix{
F(E)\ar@<-.5ex>[r]_-{\rho} \ar@<.5ex>[r]^-{\lambda} & F(\Sigma)}\]
be the presentation of $\Fin\Cospan$ given in Example \ref{ex:fincospan_presentation}.  Here $\Sigma$ is the signature with elements 
$m \maps 2\to 1, i \maps 0 \to 1, d \maps 1 \to 2$ and
$e \maps 1\to 0$, and the equations are the laws for a special commutative Frobenius monoid.

Since left adjoints preserve colimits, we have a natural isomorphism 
$F(\Sigma)+ F(L) \cong F(\Sigma+L)$.   Let $\iota \maps F(\Sigma) \to 
F(\Sigma+L)$ be the resulting monomorphism.
By Corollary \ref{cor:presentation} we can form the coequalizer $X$ here:
\[ \xymatrix{
F(E)\ar@<-.5ex>[r]_-{\iota \rho} \ar@<.5ex>[r]^-{\iota \lambda} & F(\Sigma+L) \ar[r] & X. } \]
We claim that $X \cong \Circ_L$.  

On the one hand, there is a morphism $f \maps F(\Sigma) \to \Circ_L$ sending $m, i, d$ and $e$ in $F\Sigma$ to the corresponding morphisms in $\Circ_L$, and $f \lambda = f \rho$ because these morphisms make $1 \in \Circ_L$ into
a special commutative Frobenius monoid.    We thus have a commutative diagram
\[
    \xymatrix{
      F(E)\ar@<-.5ex>[r]_-{\iota \rho} \ar@<.5ex>[r]^-{\iota \lambda} & F(\Sigma+L) \ar[r]^-{j} \ar[dr]_{f} & X
\\
 & & \Circ_L . }
\]
By the universal property of the coequalizer, there is a unique morphism $g \maps 
X \to \Circ_L$ with $g j = f$.  On the other hand, the object $1 \in X$ is, by construction, a special commutative Frobenius monoid with $L$-action.  By Proposition \ref{prop:lcirc_algebras} we thus obtain an algebra of $\Circ_L$ in $X$, that is, a morphism $h$ as follows:
\[
    \xymatrix{
      F(E)\ar@<-.5ex>[r]_-{\iota\rho} \ar@<.5ex>[r]^-{\iota\lambda} & F(\Sigma+L) \ar[r]^-{j} \ar[dr]_{f} & X \ar@<.5ex>[d]^{g}
\\
 & & \Circ_L . \ar@<.5ex>[u]^{h} }
\]
It is easy to check that $h f = j$ by seeing how both sides act on the elements
$m, i, d, e,\ell$ of $F(\Sigma + L)$.  By the universal property 
of $X$ we have $h g = 1$, and because the Frobenius monoid in $\Circ_L$ has no nontrivial automorphisms we also have $g h = 1$.  Thus $X \cong \Circ_L$ and 
\begin{equation}
\label{eq:lcirc_1}
\xymatrix{F(E)\ar@<-.5ex>[r]_-{\rho} \ar@<.5ex>[r]^-{\lambda} & F(\Sigma + L)\ar[r]^{f} & \Circ_L } 
\end{equation}
is a coequalizer.

More generally, given presentations of props $X_1$ and $X_2$, we have coequalizers:
\[
    \xymatrix{
      F(E_i) \ar@<-.5ex>[r]_{\rho_i} \ar@<.5ex>[r]^{\lambda_i} & F(\Sigma_i)
	\ar@<0ex>[r]^-{g_i} & X_i ,
    }  \qquad i = 1, 2
  \]
whose coproduct is another coequalizer:
\[
    \xymatrix{
      F(E_1)+F(E_2) \ar@<-.5ex>[rr]_{\rho_1 + \rho_2} \ar@<.5ex>[rr]^{\lambda_1 + \lambda_2} && F(\Sigma_1)+F(\Sigma_2) 		         	\ar@<0ex>[rr]^-{g_1 + g_2} &&  X_1 + X_2.
    }
  \]
Since left adjoints preserve coproducts, we obtain a coequalizer 
\[
    \xymatrix{
      F(E_1+E_2) \ar@<-.5ex>[r]_{} \ar@<.5ex>[r]^{} & F(\Sigma_1+\Sigma_2) 		         	\ar@<0ex>[r]^-{} &  X_1 + X_2
    }
  \]
which gives a presentation for $X_1 + X_2$. 

In our situation $E \cong E + 0$ where $0$ is the initial or `empty' signature, 
so we have a coequalizer 
\[
    \xymatrix{
      F(E) + F(0)\ar@<-.5ex>[r]_{} \ar@<.5ex>[r]^{} & F(\Sigma) + F(L) 		         	\ar@<0ex>[r]^-{} & \Fin\Corel + F(L)
    }
  \]
Combining this with Equation (\ref{eq:lcirc_1}), we see $\Circ_L \cong \Fin\Cospan + F(L)$.
\end{proof}


With this result we think of $\Circ_L$ has being generated by morphisms:
 \[
  \xymatrixrowsep{1pt}
  \xymatrixcolsep{30pt}
  
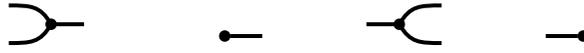
\xymatrix{
    \mult{.1\textwidth}  & \unit{.1\textwidth}  & \comult{.4\textwidth} & \counit{.1\textwidth} 
  }
\]
subject to the laws of a special commutative Frobenius monoid, together with a generator
 \[
  \xymatrixrowsep{1pt}
  \xymatrixcolsep{30pt}
  \xymatrix{
 \singlegen{.1\textwidth} 
  }
\]
for each $\ell$ in $L$. This is exactly what we want since $\Circ_L$ is the coproduct of two props; this means that it is generated by the generators of of both props. To see how these fits with our initial definition of an $L$-circuit, consider the earlier example:

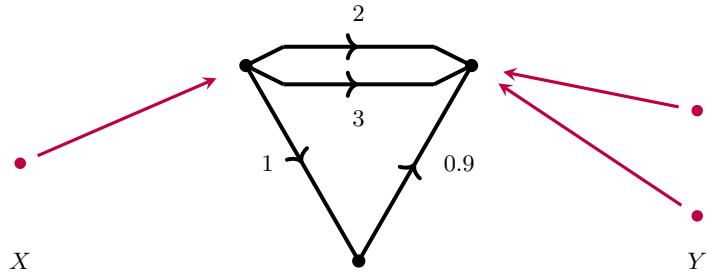
\begin{center}
    \begin{tikzpicture}[circuit ee IEC, set resistor graphic=var resistor IEC graphic]\label{ex:lcircuit}
\scalebox{1}{
      {\node[circle,draw,inner sep=1pt,fill=gray,color=purple]         (x) at
	(-3,-1.3) {};
	\node at (-3,-2.6) {$X$};}
      \node[contact]         (A) at (0,0) {};
      \node[contact]         (B) at (3,0) {};
      \node[contact]         (C) at (1.5,-2.6) {};
      {\node[circle,draw,inner sep=1pt,fill=gray,color=purple]         (y1) at
	(6,-.6) {};
	  \node[circle,draw,inner sep=1pt,fill=gray,color=purple]         (y2) at
	  (6,-2) {};
	  \node at (6,-2.6) {$Y$};}
      \coordinate         (ua) at (.5,.25) {};
      \coordinate         (ub) at (2.5,.25) {};
      \coordinate         (la) at (.5,-.25) {};
      \coordinate         (lb) at (2.5,-.25) {};
      \path (A) edge (ua);
      \path (A) edge (la);
      \path (B) edge (ub);
      \path (B) edge (lb);
      \path (ua) edge  [->-=.5] node[label={[label distance=1pt]90:{$2$}}] {} (ub);
      \path (la) edge  [->-=.5] node[label={[label distance=1pt]270:{$3$}}] {} (lb);
      \path (A) edge  [->-=.5] node[label={[label distance=2pt]180:{$1$}}] {} (C);
      \path (C) edge  [->-=.5] node[label={[label distance=2pt]0:{$0.9$}}] {} (B);
      {
	\path[color=purple, very thick, shorten >=10pt, shorten <=5pt, ->, >=stealth] (x) edge (A);
	\path[color=purple, very thick, shorten >=10pt, shorten <=5pt, ->, >=stealth] (y1) edge (B);
	\path[color=purple, very thick, shorten >=10pt, shorten <=5pt, ->, >=stealth] (y2)
      edge (B);}
}
    \end{tikzpicture}
  \end{center}

We build this morphism from the generators of the prop $\Circ_L$ by doing the following composite:
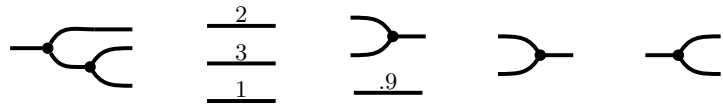
\[
  \xymatrixrowsep{.1pt}
  \xymatrixcolsep{2pt}
  \xymatrix{
    \coassocl{.9\textwidth} &&& \lcircpropexample{.9\textwidth}  &&& \lcircpropexampleother{.9\textwidth}  &&& \mult{.9\textwidth}  &&& \comult{.9\textwidth}
  }
\]
giving us a morphism $f\maps 1 \to 2$ in the prop $\Circ_L$.


\section{Circuits made only of perfectly conductive wires}
\label{sec:circuits_of_just_wires}

The simplest circuits are those made solely of perfectly conductive wires.  Electrical engineers would consider these circuits trivial.   Nonetheless, they provide the foundation on which all our
results rest.   In this case the underlying graph of the circuit has unlabeled edges---or equivalently, the set of labels has a single element.    So, we make the following definition:

\begin{definition}
\label{defn:Circ}
Let $\Circ$ be symmetric monoidal category $\Circ_L$ where $L = \{\ell\}$ is a 1-element set.
\end{definition}

A morphism in $\Circ$ can be viewed an isomorphism class of cospans of finite sets
\[
  \xymatrix{
    & N \\
    X \ar[ur]^{i} && Y \ar[ul]_o
  }
\]
together with a graph $\Gamma$ having $N$ as its set of vertices since the labelling becomes trivial:
\[\Gamma = \{ \xymatrix{ E \ar@<2.5pt>[r]^{s} \ar@<-2.5pt>[r]_{t} & N} \} \] 
For example, a morphism in $\Circ$ might look like this:
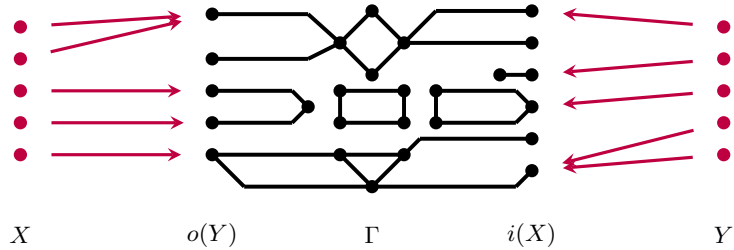
\[
\begin{aligned}
\begin{tikzpicture}[scale=.85, circuit ee IEC]
\node [contact, outer sep=5pt] (6) at (-3, 1.2) {};
\node [contact, outer sep=5pt] (7) at (-3, -0.5) {};
\node [contact, outer sep=5pt] (8) at (-3, 0.5) {};
\node [contact, outer sep=5pt] (9) at (-3, -0) {};
\node [contact, outer sep=5pt] (10) at (-3, -1) {};
\node [contact, outer sep=5pt] (15) at (-1, 0.75) {};
	\coordinate         (15a) at (-1.5,1.2) {};
	\coordinate         (15b) at (-1.5,.5) {};
\node [contact, outer sep=5pt] (18) at (0, 0.75) {};
	\coordinate         (18a) at (.5,1.25) {};
\node [contact, outer sep=5pt](14) at (-0.5, 1.25) {};
\node [contact, outer sep=5pt](19) at (-0.5, 0.25) {};
\node [contact, outer sep=5pt](28) at (1.5, 0.25) {};
\node [contact, outer sep=5pt] (16) at (-1.5, -0.25) {};
	\coordinate         (16a) at (-1.75,0) {};
	\coordinate         (16b) at (-1.75,-.5) {};
\node [contact, outer sep=5pt] (29) at (0.5, -0.5) {};
\node [contact, outer sep=5pt] (11) at (0.5, 0.0) {};
\node [contact, outer sep=5pt] (17) at (-0.5, -1.5) {};
	\coordinate         (17a) at (-2.5,-1.5) {};
\node [contact, outer sep=5pt] (32) at (-1, -1) {};
\node [contact, outer sep=5pt] (33) at (0, -1) {};
	\coordinate         (33a) at (.25,-.75) {};
\node [contact, outer sep=5pt] (-2) at (2, 1.25) {};
\node [contact, outer sep=5pt] (-1) at (2, 0.75) {};
\node [contact, outer sep=5pt] (0) at (2, 0.25) {};
\node [contact, outer sep=5pt] (1) at (2, -0.25) {};
	\coordinate         (1a) at (1.75,-.5) {};
	\coordinate         (1b) at (1.75,0) {};
\node [contact, outer sep=5pt] (2) at (2, -0.75) {};
\node [contact, outer sep=5pt] (3) at (2, -1.25) {};
	\coordinate         (3a) at (1.75,-1.5) {};
\node [contact, outer sep=5pt] (40) at (0, -0.5) {};
\node [contact, outer sep=5pt] (41) at (0, 0.0) {};
\node [contact, outer sep=5pt] (42) at (-1, -0.5) {};
\node [contact, outer sep=5pt] (43) at (-1, 0.0) {};
\node [style=none] (4) at (-3, -2.25) {$o(Y)$};
\node [style=none] (20) at (2, -2.25) {$i(X)$};
\node [style=none] (4a) at (-6, -2.25) {$X$};
\node [style=none] (20a) at (5, -2.25) {$Y$};
\node [style=none] (5) at (-.5, -2.25) {$\Gamma$};
	        \node[contact, outer sep=5pt,fill=gray,color=purple] (p6) at (-6, 1) {};
	        \node[contact, outer sep=5pt,fill=gray,color=purple] (p7) at (-6, -.5) {};
	        \node[contact, outer sep=5pt,fill=gray,color=purple] (p8) at (-6, .5) {};
	        \node[contact, outer sep=5pt,fill=gray,color=purple] (p9) at (-6, 0) {};
	        \node[contact, outer sep=5pt,fill=gray,color=purple] (p10) at (-6, -1) {};
	        \node[contact, outer sep=5pt,fill=gray,color=purple] (p-2) at (5, 1) {};
	        \node[contact, outer sep=5pt,fill=gray,color=purple] (p0) at (5, .5) {};
	        \node[contact, outer sep=5pt,fill=gray,color=purple] (p1) at (5, 0) {};
	        \node[contact, outer sep=5pt,fill=gray,color=purple] (p2) at (5, -.5) {};
	        \node[contact, outer sep=5pt,fill=gray,color=purple] (p3) at (5, -1) {};
\draw  (6.center) to (15a.center);
\draw  (8.center) to (15b.center);
\draw  (15.center) to (15a.center);
\draw  (15.center) to (15b.center);
\draw  (-2.center) to (18a.center);
\draw  (18.center) to (18a.center);
\draw  (-1.center) to (18.center);
\draw  (9.center) to (16a.center);
\draw  (7.center) to (16b.center);
\draw  (10.center) to (17a.center);
\draw  (17a.center) to (17.center);
\draw  (33a.center) to (33.center);
\draw  (33a.center) to (2.center);
\draw  (17.center) to (3a.center);
\draw  (3a.center) to (3.center);
\draw  (28.center) to (0.center);
\draw (1a.center) to (1.center);
\draw  (1a.center) to (29.center);
\draw  (1b.center) to (11.center);
\draw  (1.center) to (1b.center);
\draw  (15.center) to (14.center);
\draw  (15.center) to (19.center);
\draw  (18.center) to (14.center);
\draw  (18.center) to (19.center);
\draw  (16a.center) to (16.center);
\draw  (16b.center) to (16.center);
\draw  (29.center) to (11.center);
\draw  (32.center) to (17.center);
\draw  (32.center) to (10.center);
\draw  (33.center) to (17.center);
\draw  (33.center) to (32.center);
\draw  (40.center) to (41.center);
\draw  (40.center) to (42.center);
\draw  (41.center) to (43.center);
\draw  (43.center) to (42.center);
	       \path[color=purple, very thick, shorten >=5pt, shorten <=5pt, ->,
	>=stealth] (p6) edge (6);
	       \path[color=purple, very thick, shorten >=4pt, shorten <=5pt, ->,
	>=stealth] (p7) edge (7);
	       \path[color=purple, very thick, shorten >=5pt, shorten <=5pt, ->,
	>=stealth] (p8) edge (6);
	       \path[color=purple, very thick, shorten >=5pt, shorten <=5pt, ->,
	>=stealth] (p9) edge (9);
	       \path[color=purple, very thick, shorten >=5pt, shorten <=5pt, ->,
	>=stealth] (p10) edge (10);
	       \path[color=purple, very thick, shorten >=5pt, shorten <=5pt, ->,
	>=stealth] (p-2) edge (-2);
	       \path[color=purple, very thick, shorten >=5pt, shorten <=5pt, ->,
	>=stealth] (p0) edge (0);
	       \path[color=purple, very thick, shorten >=5pt, shorten <=5pt, ->,
	>=stealth] (p1) edge (1);
	       \path[color=purple, very thick, shorten >=5pt, shorten <=5pt, ->,
	>=stealth] (p2) edge (3);
	       \path[color=purple, very thick, shorten >=5pt, shorten <=5pt, ->,
	>=stealth] (p3) edge (3);
\end{tikzpicture}
\end{aligned}
\]

By Proposition \ref{prop:lcirc_coproduct} we can identify $\Circ$ with the coproduct $\Fin\Cospan + F(\{\ell\})$. This gives us the generators of $\Circ$, which we draw as follows:
\begin{corollary}
The prop $\Circ$ is generated by the following morphisms:
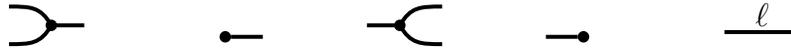
\[
  \xymatrixrowsep{1pt}
  \xymatrixcolsep{30pt}
  \xymatrix{
    \mult{.1\textwidth}  & \unit{.1\textwidth}  & \comult{.4\textwidth} & \counit{.1\textwidth}  & \singlegen{.1\textwidth} \\
  }
\]
where the first four generators are subject to the laws of a special commutative Frobenius monoid.


\end{corollary}

Just as in Example \ref{ex:lcircuit}, we could rewrite the above morphism in $\Circ$ as being built from these generators. Each of the unlabelled edges would be labelled by $\ell$ and the other generators would be used to glue appropriate edges together and to get the right number of inputs and outputs. 

As we shall see in detail later, for the behavior of a circuit made of perfectly conductive wires, all that matters about the underlying graph is whether any given pair of nodes lie in the same connected component or not: that is, whether or not there exists a path of edges and their reverses from one node to another.   If they lie in the same component, current can flow between them; otherwise not, and this is all there is to say.  We may thus replace the graph $\Gamma$ by its set of connected components, $\pi_0(\Gamma)$. There is map $p_\Gamma \maps N \to \pi_0(\Gamma)$ sending each node to the connected component it lies in.  We thus obtain obtain a cospan of finite sets:
\[
  \xymatrix{
    & \pi_0(\Gamma) \\
    X \ar[ur]^{p_\Gamma i} && Y \ar[ul]_{p_\Gamma o}
  }
\]
In particular, the above example gives this cospan:
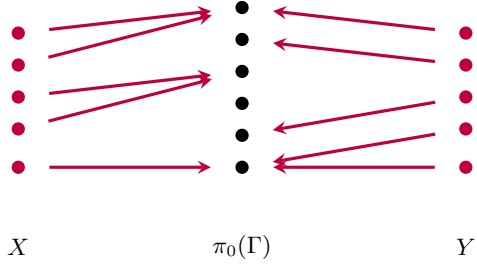
\[
\begin{aligned}
\begin{tikzpicture}[scale=.85, circuit ee IEC]
\node [style=none] (5) at (-.5, -2.25) {$\pi_0(\Gamma)$};
\node [style=none] (4a) at (-4, -2.25) {$X$};
\node [style=none] (20a) at (3, -2.25) {$Y$};


	        \node[contact, outer sep=5pt,fill=gray,color=purple] (p6) at (-4, 1.1) {};
	        \node[contact, outer sep=5pt,fill=gray,color=purple] (p7) at (-4, -.4) {};
	        \node[contact, outer sep=5pt,fill=gray,color=purple] (p8) at (-4, .6) {};
	        \node[contact, outer sep=5pt,fill=gray,color=purple] (p9) at (-4, .1) {};
	        \node[contact, outer sep=5pt,fill=gray,color=purple] (p10) at (-4, -1) {};

	        \node[contact, outer sep=5pt,fill=gray,color=purple] (p-2) at (3, 1.1) {};
	        \node[contact, outer sep=5pt,fill=gray,color=purple] (p0) at (3, .6) {};
	        \node[contact, outer sep=5pt,fill=gray,color=purple] (p1) at (3, 0.1) {};
	        \node[contact, outer sep=5pt,fill=gray,color=purple] (p2) at (3, -.4) {};
	        \node[contact, outer sep=5pt,fill=gray,color=purple] (p3) at (3, -1) {};

\node [contact, outer sep=5pt] (m1) at (-.5, 1.5) {};
\node [contact, outer sep=5pt] (m2) at (-.5, 1) {};
\node [contact, outer sep=5pt] (m3) at (-.5, .5) {};
\node [contact, outer sep=5pt] (m4) at (-.5, 0) {};
\node [contact, outer sep=5pt] (m5) at (-.5, -.5) {};
\node [contact, outer sep=5pt] (m6) at (-.5, -1) {};

	       \path[color=purple, very thick, shorten >= 5pt, shorten <=5pt, ->,
	>=stealth] (p6) edge (m1);
	       \path[color=purple, very thick, shorten >= 5pt, shorten <=5pt, ->,
	>=stealth] (p7) edge (m3);
	       \path[color=purple, very thick, shorten >=5pt, shorten <=5pt, ->,
	>=stealth] (p8) edge (m1);
	       \path[color=purple, very thick, shorten >=5pt, shorten <=5pt, ->,
	>=stealth] (p9) edge (m3);
	       \path[color=purple, very thick, shorten >=5pt, shorten <=5pt, ->,
	>=stealth] (p10) edge (m6);

	       \path[color=purple, very thick, shorten >=5pt, shorten <=5pt, ->,
	>=stealth] (p-2) edge (m1);

	       \path[color=purple, very thick, shorten >=5pt, shorten <=5pt, ->,
	>=stealth] (p0) edge (m2);
	       \path[color=purple, very thick, shorten >=5pt, shorten <=5pt, ->,
	>=stealth] (p1) edge (m5);
	       \path[color=purple, very thick, shorten >=5pt, shorten <=5pt, ->,
	>=stealth] (p2) edge (m6);
	       \path[color=purple, very thick, shorten >=5pt, shorten <=5pt, ->,
	>=stealth] (p3) edge (m6);
\end{tikzpicture}
\end{aligned}
\]
This way of simplifying a circuit defines a functor 
\[   H' \maps \Circ \to \Fin\Cospan.  \]
If instead we write $\Circ$ as the coproduct $\Fin\Cospan + F(\{\ell\})$   then there exists  a unique morphism of props
\[              H' \maps \Circ \to \Fin\Cospan \]
such that 
\[           H'(f) = f  \]
for any morphism $f$ in $\Fin\Cospan$ and
\[           H'(\ell) = 1_1. \]
In other words, $H'$ does nothing to morphisms in the sub-prop $\Fin\Cospan$, while it sends 
the morphism $\ell$ to the identity.  It is thus a retraction for the inclusion of $\Fin\Cospan$ in $\Circ$. Now, each edge of $\Gamma$ corresponds to a copy of the morphism $\ell$.   The functor $G$ has 
the effect of collapsing each edge of $\Gamma$ to a point, since it sends $\ell$ to the identity.  The result is still a cospan of finite sets where the apex is the set of connected components $\pi_0(\Gamma)$. 

Recall by Proposition \ref{prop:fincospan_to_fincorel} there is a unique morphism of props $H\maps \Fin\Cospan \to \Fin\Corel$ which forces a cospan to become jointly epic. In our example we obtain the following
jointly epic cospan:
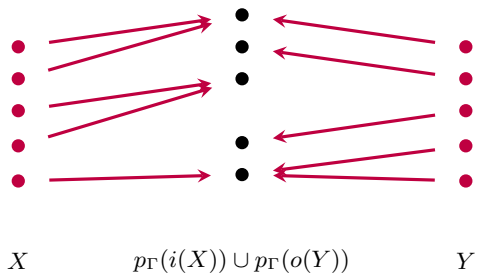
\[
\begin{aligned}
\begin{tikzpicture}[scale=.85, circuit ee IEC]
\node [style=none] (5) at (-.5, -2.25) {$p_\Gamma(i(X)) \cup p_\Gamma(o(Y))$};
\node [style=none] (4a) at (-4, -2.25) {$X$};
\node [style=none] (20a) at (3, -2.25) {$Y$};


	   	  \node[contact, outer sep=5pt,fill=gray,color=purple] (p6) at (-4, 1.1) {};
	        \node[contact, outer sep=5pt,fill=gray,color=purple] (p7) at (-4, -.45) {};
	        \node[contact, outer sep=5pt,fill=gray,color=purple] (p8) at (-4, .6) {};
	        \node[contact, outer sep=5pt,fill=gray,color=purple] (p9) at (-4, .1) {};
	        \node[contact, outer sep=5pt,fill=gray,color=purple] (p10) at (-4, -1) {};

	        \node[contact, outer sep=5pt,fill=gray,color=purple] (p-2) at (3, 1.1) {};
	        \node[contact, outer sep=5pt,fill=gray,color=purple] (p0) at (3, .6) {};
	        \node[contact, outer sep=5pt,fill=gray,color=purple] (p1) at (3, 0.1) {};
	        \node[contact, outer sep=5pt,fill=gray,color=purple] (p2) at (3, -.45) {};
	        \node[contact, outer sep=5pt,fill=gray,color=purple] (p3) at (3, -1) {};

		\node [contact, outer sep=5pt] (m1) at (-.5, 1.6) {};
		\node [contact, outer sep=5pt] (m2) at (-.5, 1.1) {};
		\node [contact, outer sep=5pt] (m3) at (-.5, .6) {};
		\node [contact, outer sep=5pt] (m5) at (-.5, -.4) {};
		\node [contact, outer sep=5pt] (m6) at (-.5, -.9) {};
	       \path[color=purple, very thick, shorten >=5pt, shorten <=5pt, ->,
	>=stealth] (p6) edge (m1);
	       \path[color=purple, very thick, shorten >=5pt, shorten <=5pt, ->,
	>=stealth] (p7) edge (m3);
	       \path[color=purple, very thick, shorten >=5pt, shorten <=5pt, ->,
	>=stealth] (p8) edge (m1);
	       \path[color=purple, very thick, shorten >=5pt, shorten <=5pt, ->,
	>=stealth] (p9) edge (m3);
	       \path[color=purple, very thick, shorten >=5pt, shorten <=5pt, ->,
	>=stealth] (p10) edge (m6);

	       \path[color=purple, very thick, shorten >=5pt, shorten <=5pt, ->,
	>=stealth] (p-2) edge (m1);

	       \path[color=purple, very thick, shorten >=5pt, shorten <=5pt, ->,
	>=stealth] (p0) edge (m2);
	       \path[color=purple, very thick, shorten >=5pt, shorten <=5pt, ->,
	>=stealth] (p1) edge (m5);
	       \path[color=purple, very thick, shorten >=5pt, shorten <=5pt, ->,
	>=stealth] (p2) edge (m6);
	       \path[color=purple, very thick, shorten >=5pt, shorten <=5pt, ->,
	>=stealth] (p3) edge (m6);
\end{tikzpicture}
\end{aligned}
\]
In short, we can simplify a circuit made of ideal conductive wires
in two successive stages:
\[     \Circ \stackrel{H'}{\longrightarrow} \Fin\Cospan \stackrel{H}{\longrightarrow} 
\Fin\Corel.  \]
In Section \ref{sec:linear_relations} we define a black-boxing functor 
\[   \blacksquare \maps \Circ \to \Fin\Rel_k, \]
which maps any such circuit to its `behavior': that is, the linear relation $R\subseteq k^{2n}$ 
 that it imposes between potentials and currents at its inputs and outputs. 
 We construct this
by composing $HH' \maps \Circ \to \Fin\Corel$ with a functor
\[     K \maps \Fin\Corel \to \Fin\Rel_k  .\]
This makes precise the sense in which the behavior of such a circuit depends only on
its underlying corelation.

 This functor was constructed by Baez and Fong \cite{BF} using `decorated cospans,' but we do so using props. Since both approaches produce the same functor we state ours as a definition. The functor $K$ is defined by finding an extraspecial commutative Frobenius monoid in $\Fin\Rel_k$.  In Chapter \ref{chap:Functors_for_Circuits}, more specifically in Proposition \ref{prop:K_1}, we show that even though $K$ is not a morphism of props, it does determine one. Namely, $K\maps \Fin\Corel \to \Lag\Rel_k$ is a morphism of props, where $\Lag\Rel_k$ has ``Lagrangian relations" as morphisms.


\section{Linear relations}
\label{sec:linear_relations}

Since the simplest circuits impose \emph{linear} relations between potentials and currents at their terminals, the study of circuits forces us to think carefully about linear relations.  As we shall see, for any field $k$ there is a prop $\Fin\Rel_k$ where a morphism $f \maps m \to n$ is a linear relation from $k^m$ to $k^n$.  A presentation for this prop has been worked out by Erbele and Baez \cite{BE,E} and independently by Bonchi, Soboci\'nski and Zanasi \cite{BSZ,BSZ2,Za}.  Since we need some facts about this presentation to describe the black-boxing of circuits, we recall it here.

For any field $k$, there is a category where an object is a finite-dimensional vector space over $k$, while a morphism from $U$ to $V$ is a \define{linear relation}, meaning a linear subspace $L \subseteq U \oplus V$.    We write a linear relation
from $U$ to $V$ as $L \maps U \asrelto V$. Note that we also use the same notation for linear maps since they are a special case of linear relations.  Since the direct sum $U \oplus V$ is also the cartesian product of $U$ and $V$, a linear relation is a relation in the usual sense, and we can compose linear relations $L \maps U \asrelto V$ and $L' \maps V \asrelto W$ in the usual way:
\[         L'L = \{(u,w) \colon \; \exists\; v \in V \;\; (u,v) \in L \textrm{ and } 
(v,w) \in L'\} \]
the result being a linear relation $L'L \maps U \asrelto W$.  Given linear relations 
$L \maps U \asrelto V$ and $L' \maps U' \asrelto V'$, the direct sum of subspaces gives a linear relation $L \oplus L' \maps U \oplus U' \asrelto V \oplus V'$, and this gives our category a symmetric monoidal structure.  By Proposition \ref{prop:strictification_1}, this symmetric monoidal category is equivalent to a prop.  Concretely, we may describe this prop as follows:

\begin{definition}
Let \define{$\Fin\Rel_k$} be the prop where a morphism $f \maps m \to n$ is a linear
relation from $k^m$ to $k^n$, composition is the usual composition of relations, and the symmetric monoidal structure is given by direct sum.
\end{definition}

To give a presentation of $\Fin\Rel_k$, we use a simple but nice fact: the object 
$1 \in \Fin\Rel_k$, or in less pedantic terms the 1-dimensional vector space $k$, is an extraspecial Frobenius monoid in two fundamentally different ways.   To understand these, first note that for 
any linear relation $L \maps U \asrelto V$ there is a linear relation 
$L^\dagger \maps V \asrelto U$ given by
\[     L^\dagger = \{(v,u) : \; (u,v) \in L  \} .\]
This makes $\Fin\Rel_k$ into a dagger compact category \cite{AC,BE,Se}.  Also recall that a 
linear map is a special case of a linear relation.

The first way of making $k$ into an extraspecial commutative Frobenius monoid in $\Fin\Rel_k$ uses these morphisms:
\begin{itemize}
\item 
as comultiplication, the linear map called \define{duplication}:
\[  \Delta \maps k \asrelto k^2 \]
\[  \Delta = \{(x,x,x): \; x \in k\} \subseteq k \oplus k^2 \]
\item 
as counit, the linear map called \define{deletion}:
\[  ! \maps k \asrelto \{0\} \]
\[  ! = \{ (x,0) : \; x \in k\} \subseteq k \oplus \{0\}  \]
\item
as multiplication, the linear relation called \define{coduplication}:
\[    \Delta^\dagger \maps  k^2 \asrelto k \]
\[    \Delta^\dagger =\{(x,x,x) : \;  x \in k \} \subseteq k^2 \oplus k \]
\item 
as unit, the linear relation called \define{codeletion}:
\[   !^\dagger \maps \{0\} \asrelto k \]
\[   !^\dagger = \{(0,x): \; x \in k \} \subseteq \{0\} \oplus k .\]
\end{itemize}
We call this the \define{duplicative Frobenius structure} on $k$.   In circuit theory this is
important for working with the electric potential.  The reason is that in a circuit of ideal conductive wires the potential is constant on each connected component, so wires like this:
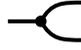
\[
  \xymatrixrowsep{10pt}
  \xymatrixcolsep{10pt}
  \xymatrix@1{
    \comult{.075\textwidth} 
    }
\]
have the effect of duplicating the potential. 

The second way of making $k$ into an extraspecial commutative Frobenius monoid in
$\Fin\Rel_k$ uses these morphisms:
\begin{itemize}
\item as multiplication, the linear map called \define{addition}:
\[   +\maps  k^2 \asrelto k \]
\[   + = \{(x,y,x+y) : \; x,y \in k \} \subseteq k^2 \oplus k \] 
\item as unit, the linear map called \define{zero}:
\[    0 \maps  \{0\} \asrelto k \]
\[    0 = \{(0,x) : \; x \in k \} \subseteq \{0\} \oplus k \]
\item as comultiplication, the linear relation called \define{coaddition}:
\[   +^\dagger \maps  k \asrelto k \oplus k \]
\[   +^\dagger = \{(x+y,x,y) : \; x,y \in k \} \subseteq k \oplus k^2 \]
\item as counit, the linear relation called \define{cozero}:
\[     0^\dagger \maps k \asrelto \{0\} \]
\[    0^\dagger = \{(x,0)\} \subseteq k \oplus \{0\} .\] 
\end{itemize}
We call this the \define{additive Frobenius structure} on $k$.   In circuit theory this structure is
important for working with electric current.  The reason is that Kirchhoff's current law says that the sum of input currents must equal the sum of the output currents, so wires like this:
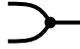
\[
  \xymatrixrowsep{10pt}
  \xymatrixcolsep{10pt}
  \xymatrix@1{
    \mult{.075\textwidth} 
    }
\]
have the effect of adding currents.

The prop $\Fin\Rel_k$ is generated by the eight morphisms listed above together with a
morphism for each element $c \in k$, namely the map from $k$ to itself given by 
multiplication by $c$.   We denote this simply as $c$:
\[    \begin{array}{cccl}
c \maps & k &\asrelto & k \\
        & x & \mapsto &  c x.
\end{array}
\] 

From the generators we can build two other important morphisms:
\begin{itemize}
\item the \define{cup} $\cup \maps k^2 \asrelto \{0\}$: this is the composite of coduplication 
$\Delta^\dagger \maps k^2 \to k$ and deletion $! \maps k \asrelto \{0\}$.  
\item the \define{cap} $\cap \maps \{0\} \to k^2$: this is the composite of codeletion $!^\dagger \maps \{0\} \to k$ and duplication $\Delta \maps k \to k^2$;
\end{itemize}
These are the unit and counit for an adjunction making $k$ into its own dual.  Since every object in $\Fin\Rel_k$ is a tensor product of copies of $k$, every object becomes self-dual.  Thus, $\Fin\Rel_k$ becomes a dagger compact category.  This explains the use of the dagger notation for half of the
eight morphisms listed above.     

Thanks to the work of Baez and Erbele \cite{BE,E} and also Bonchi, Soboci\'nski and Zanasi \cite{BSZ,BSZ2,Za}, the prop $\Fin\Rel_k$ has a presentation of this form:
\[ \xymatrix{
F(E_k) \ar@<-.5ex>[r]_-{\rho_k} \ar@<.5ex>[r]^-{\lambda_k} & F(\Sigma) + F(\Sigma) + F(k) \ar[r]^-{\square} & \Fin\Rel_k
 }
 \]
Here the signature $\Sigma$ has elements $m \maps 2\to 1, \i \maps 0 \to 1, d \maps 1 \to 2, e \maps 1\to 0$. The morphisms in the free prop $F(\Sigma) + F(\Sigma) + F(k)$ can be drawn as string diagrams, and these roughly match what control theorists call ``signal-flow diagrams."   The meaning of these generators is best understood in terms of the linear relations they are mapped to under $\square$, so we also use them to draw linear relations. We discuss signal-flow diagrams in Section \ref{sec:sigflowdiagrams}, but for now only use the diagrams in reference to linear relations.
 In this presentation for $\Fin\Rel_k$, 
the first copy of $F(\Sigma)$ is responsible for the duplicative Frobenius structure on 
$k$, so we call its generators
\begin{itemize}
\item \define{coduplication}, $\Delta^\dagger \maps 2 \to 1$
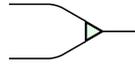
\[
  \xymatrix@1{
    \sigflowpiccodup{.15\textwidth} 
  }
\]
 \item \define{codeletion}, $!^\dagger \maps 0 \to 1$
\[
  \xymatrix@1{
    \;\;\quad\quad\sigflowpiccodel{.1\textwidth} 
  }
\]
\item \define{duplication}, $\Delta \maps 1 \to 2$
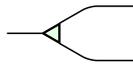
\[
  \xymatrix@1{
    \sigflowpicdup{.15\textwidth}\;\;\quad\quad 
  }
\]

\item \define{deletion}, $! \maps 1 \to 1$
\[
  \xymatrix@1{
    \sigflowpicdel{.1\textwidth} 
  }
\]
\end{itemize}
The second copy of $F(\Sigma)$ is responsible for the additive Frobenius structure,
so we call its generators
\begin{itemize}
\item 
\define{addition}, $+ \maps 2 \to 1$
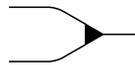
\[
  \xymatrix@1{
    \sigflowpicadd{.15\textwidth} 
  }
\]
\item \define{coaddition}, $+^\dagger \maps 1 \to 2$
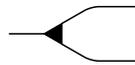
\[
  \xymatrix@1{
    \sigflowpiccoadd{.15\textwidth} 
  }
\]
\item \define{zero}, $0 \maps 0 \to 1$ 
\[
  \xymatrix@1{
    \;\;\quad\quad \sigflowpiczero{.1\textwidth} 
  }
\]
\item \define{cozero}, $0^\dagger \maps 1 \to 0$
\[
  \xymatrix@1{
    \sigflowpiccozero{.1\textwidth} \quad\quad\;\;
  }
\]
\end{itemize}
Finally, we have a copy of $F(k)$, consisting of elements we call
\begin{itemize}
\item  \define{scalar multiplication}, $c \maps 1 \to 1$
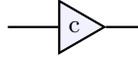
\[
  \xymatrixrowsep{5pt}
  \xymatrix@1{
    \SigLabelpic{.15\textwidth}\\
  }
\]
\end{itemize}
 one for each $c \in k$.  All these generators are mapped by $\square$ to the previously described morphisms with the same names in $\Fin\Rel_k$.   

We do not need a complete list of the equations in this presentation of $\Fin\Rel_k$, but among them are equations saying that in $\Fin\Rel_k$
\begin{enumerate}
\item $(k, \Delta^\dagger, !, \Delta, !^\dagger)$ is an extraspecial commutative Frobenius monoid;
\item $(k, +, 0, +^\dagger, 0^\dagger)$ is an extraspecial commutative Frobenius monoid;
\item $(k, +, 0, \Delta, !)$ is a bicommutative bimonoid;
\item $(k, \Delta^\dagger, !^\dagger, +^\dagger, 0^\dagger)$ is a bicommutative bimonoid.
\end{enumerate}

Note that one could draw the above equations in terms of the above string diagrams. In Section \ref{sec:circuits_of_just_wires} we looked at circuits made of perfectly conductive wires
and described morphisms of props
\[     \Circ \stackrel{H'}{\longrightarrow} \Fin\Cospan \stackrel{H}{\longrightarrow} 
\Fin\Corel.  \]
  However, so far we have only briefly touched on the ``behavior" of such circuits: that is, what they actually \emph{do}.  A circuit provides a relation between potentials and currents at its inputs and outputs.  For a circuit with $m$ inputs and $n$ outputs, this is a linear relation on $2m+2n$ variables.  We now describe a functor called black-boxing, which takes any circuit of ideal conductive wires and extracts this linear relation. Baez and Fong \cite{BF}   used decorated cospans to arrive at the same functor.

 In Theorem \ref{thm:fincorel_prop} we saw that $\Fin\Corel$ is the prop for extraspecial commutative Frobenius monoids.  Thus, the presentation of $\Fin\Rel_k$ gives two prop morphisms from $\Fin\Corel$ to $\Fin\Rel_k$.  The first is relevant to potentials, while the second is relevant to currents \cite{BF}.    In any symmetric monoidal category, the tensor product of two  monoids is a monoid in standard way, and dually for comonoids.  In the same way, the tensor product of extraspecial commutative Frobenius monoids becomes another extraspecial commutative Frobenius monoid.   Thus, we can make $k \oplus k$ into an extraspecial commutative Frobenius monoid in $\Fin\Rel_k$ where the first copy of $k$ has the duplicative Frobenius structure and the second copy has the additive structure.   Thanks to Theorem \ref{thm:fincorel_prop}, this determines a strict symmetric monoidal functor
\[    K \maps \Fin\Corel \to \Fin\Rel_k. \]
Composing this with 
\[     HH' \maps \Circ \to  \Fin\Corel  \]
we get the \define{black-boxing} functor 
\[   \blacksquare = KHH' \maps \Circ \to \Fin\Rel_k  .\]

\begin{definition}
\label{def:blackbox_functor}
 We call the functor 
\[    KHH' \maps \Circ \to \Fin\Rel_k  .\]
the \define{black-boxing functor for ideal circuits} and write $\blacksquare := KHH'$.
\end{definition}

In Proposition \ref{prop:K_1} we show that the black-boxing functor is determined, as a morphism of props, by where it sends the generators of $\Circ$, so we may as well provide that information right now:

\[
  \xymatrixrowsep{10pt}
  \xymatrixcolsep{10pt}
  \xymatrix@1{
    \mult{.075\textwidth} \ar@{|->}@<.25ex>[r]^-{\blacksquare} \;\; & \;\; \{( \phi_1, I_1,\phi_2, I_2,\phi_3, I_3) :\; \phi_1= \phi_2 = \phi_3,  I_1+ I_2=I_3 \}  \\
    \comult{.075\textwidth} \ar@{|->}@<.25ex>[r]^-{\blacksquare} \;\; & \;\; \{( \phi_1, I_1,\phi_2, I_2,\phi_3, I_3)  :\;  \phi_1= \phi_2 = \phi_3, I_1= I_2 + I_3\}  \\
 \unit{.075\textwidth} \ar@{|->}@<.25ex>[r]^-{\blacksquare} \;\; & \;\; \{(\phi_2, I_2) : \; I_2 = 0\}  \hspace{34ex} \\
\counit{.075\textwidth} \ar@{|->}@<.25ex>[r]^-{\blacksquare} \;\; & \;\; \{( \phi_1, I_1) :\; I_1= 0\} 
\hspace{34ex} \\
\singlegen{.075\textwidth} \ar@{|->}@<.25ex>[r]^-{\blacksquare} \;\; & \;\; \{( \phi_1,I_1,\phi_2,I_2) :\; \phi_1=\phi_2,  I_1= I_2 \} 
 \hspace{18ex} \\
 \hspace{20ex}
  }
\]
\vskip -1em \noindent
Here $\ell$ is the generator corresponding to a perfectly conductive wire; black-boxing maps it to the identity morphism on $k^2$.  Since black-boxing is a symmetric monoidal functor, we can decompose a large circuit made of ideal conductive wires into simple building blocks in order to determine the relation it imposes between the potentials and currents at its inputs and outputs.

In terms of string diagrams black-boxing does the following:

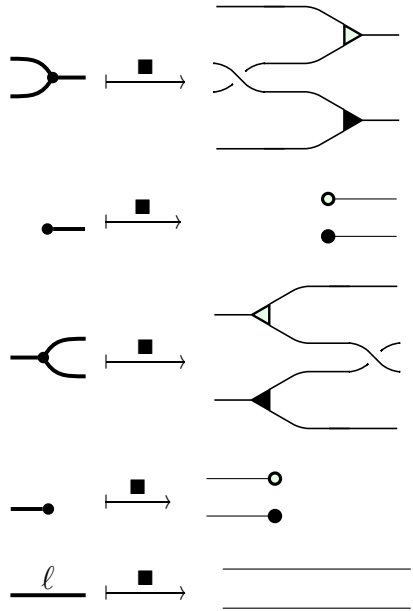
\[
  \xymatrixrowsep{5pt}
  \xymatrix@1{
  \mult{.075\textwidth} \;\; \ar@{|->}@<.25ex>[r]^-{\blacksquare} & \SigMultpic{.20\textwidth}  \\
  \unit{.075\textwidth} \;\; \ar@{|->}@<.25ex>[r]^-{\blacksquare}  & \quad\quad\quad\quad \;\SigUnitpic{.10\textwidth} \\
 \comult{.075\textwidth} \;\; \ar@{|->}@<.25ex>[r]^-{\blacksquare} & \SigCoMultpic{.20\textwidth}  \\
 \counit{.075\textwidth}  \;\; \ar@{|->}@<.25ex>[r]^-{\blacksquare} & \SigCoUnitpic{.10\textwidth} \quad\quad \quad\quad\quad \\
\singlegen{.075\textwidth} \;\; \ar@{|->}@<.25ex>[r]^-{\blacksquare}  & \SigDoubleIdpic{.20\textwidth}
  }
\]
\vskip -1em \noindent
In the next Chapter we characterize the linear relations that are picked out by $\blacksquare$. Indeed, this lets us reinterpret $\blacksquare$ as a morphism of props. We then extend black-boxing to $\Circ_L$ for when $L$ is the field $k$. Then we extend black-boxing even further to allow for affine relations instead of just linear relations.

\chapter{Black-boxing}
\label{chap:Functors_for_Circuits}

We begin this chapter by showing that $\blacksquare \maps \Circ \to \Fin\Rel_k$ determines a morphism of props. This follows because linear relations which arise from our circuits are actually ``Lagrangian" relations. These kinds of relations determine another prop $\Lag\Rel_k$ whose objects are symplectic vector spaces $(k\oplus k)^n$. In Proposition \ref{prop:K_1} we redefine $\blacksquare$ as a morphism of props $\blacksquare \maps \Circ \to \Lag\Rel_k$.

In Theorem \ref{thm:black-boxing_1} we extend black-boxing to the category $\Circ_L$ where $L$ is an arbitrary set. To do so we take advantage of the presentation for $\Circ_L$  as $\Fin\Cospan + F(L)$. We then apply our work to particular types of circuits by choosing $L$ to obtain ``RLC circuits": that is, circuits made of resistors, capacitors, and inductors. This happens in  Proposition \ref{prop:black-boxing_1} where we view black-boxing of RLC circuits as a morphism of props. In Theorem \ref{thm:black-boxing_2} we look at the case where $L$ is the field $k$ and this gives us the black-boxing functor for passive linear circuits.  

We then use our framework to understand the relationship between linear circuits and signal-flow diagrams. We define $\SigFlow_k$ to be the free prop with the same generators as $\Fin\Rel_k$. Since this is a free prop while $\Circ_k$ is not, there is no useful morphism of props from $\Circ_k$ to $\SigFlow_k$. Thus we consider the free prop $\Ccirc_k$ which has $\Circ_k$ as a quotient, and a morphism of props from this to  $\SigFlow_k$, which we think of as a method for turning a circuit into a signal-flow diagram. 

This is all summarized by the commutative diagram in Theorem \ref{thm:black-and-white-boxing}:

\vspace{-2ex}

\[
    \xymatrix{
     \Ccirc \phantom{ |} \ar[r]^-{P} \ar@{^{(}->}[d] & \Circ \phantom{ |} \ar@{^{(}->}[d] \ar[r]^-{H'} & \Fin\Cospan \phantom{ |} \ar[r]^-{H} &
      \Fin\Corel \phantom{ |} \ar[d]^{K} \\
      \Ccirc_k \phantom{ |} \ar[r]^-{P} \ar[d]_{T} & \Circ_k \phantom{ |} \ar[rr]^-{\blacksquare} & & \Lag\Rel_k \phantom{ |} \ar@{^{(}->}[d] \\
      \SigFlow_k \phantom{ |}\ar[rrr]^-{\square} & & & \Fin\Rel_k. \phantom{ |}
      }
\]

%


\section{Black-boxing as a morphism of props}
\label{sec:black-boxing_conductive}

The black-boxing functor as described so far is not a morphism of props, since it sends the object
$1 \in \Circ$ to the object $2 \in \Fin\Rel_k$, that is, the vector space $k^2$.   However, 
it can reinterpreted as a morphism of props with the help of some symplectic geometry.  Instead of linear relations between finite-dimensional vector spaces, we use Lagrangian relations between symplectic vector spaces.  For a detailed explanation of this idea, see Baez and Fong \cite{BF}; here we simply state the key facts.

\begin{definition}
A \define{symplectic vector space} $V$ over a field $k$ is a finite-dimensional vector
space equipped with a \define{symplectic structure} $\omega$, meaning a map $\omega \maps V \times  V \to k$ that is:
\begin{itemize}
\item bilinear,
\item alternating: $\omega(v,v) = 0$ for all $v \in V$,
\item nondegenerate: if $\omega(u,v) = 0 $ for all $u \in V$ then $v = 0$.
\end{itemize}
\end{definition}
\noindent
There is a standard way to make $k\oplus k$ into a symplectic vector space, namely
\[   \omega((\phi, I), (\phi',I')) =  \phi I' - \phi' I . \]
Given two symplectic vector spaces $(V_1,\omega_1)$ and $(V_2, \omega_2)$,
we give their direct sum the symplectic structure 
\[   (\omega_1 \oplus \omega_2)((u_1,u_2), (v_1, v_2)) = \omega_1(u_1, v_1) + 
\omega_2(u_2, v_2)  .\]
In what follows, whenever we treat as $(k \oplus k)^n$ as a symplectic vector space, we
give it the symplectic structure obtained by taking a direct sum of copies of $k \oplus k$ with 
the symplectic structure described above.   Every symplectic vector space is isomorphic to 
$(k \oplus k)^n$ for some $n$, so every symplectic vector space is even dimensional \cite[Thm.\ 21.2]{GS}. 

The concept of a ``Lagrangian relation" looks subtle at first, but it has become clear 
in mathematical physics that for many purposes this is the right notion of
morphism between symplectic vector spaces \cite{We1,We2}.    Lagrangian relations
are also known as `canonical relations'.  The definition has a few prerequisites: 

\begin{definition}
Given a symplectic structure $\omega$ on a vector space $V$, we define its \define{conjugate} to be the symplectic structure $\overline\omega = -
\omega$, and write the conjugate symplectic vector space $(V,\overline\omega)$ as
$\overline V$.
\end{definition}

\begin{definition}
A subspace $L$ of symplectic vector space $(V,\omega)$ is \define{isotropic} if
$\omega (v,w) = 0$ for all $v,w\in V$.   It is \define{Lagrangian} if is isotropic
and not properly contained in any other isotropic subspace.
\end{definition}

\noindent
One can show that a subspace $L \subseteq V$ is Lagrangian if and only it is
isotropic and $\dim(L) = \frac{1}{2}\dim(V)$.  This condition is often easier
to check.  

\begin{definition}
Given symplectic vector spaces $(V,\omega)$ and $(V',\omega')$, a
\define{linear Lagrangian relation} $L \maps V \asrelto V'$ is a Lagrangian
subspace $L \subseteq \overline{V} \oplus V'$.
\end{definition}

We need the conjugate symplectic structure on $V$ to show that the identity
relation is a linear Lagrangian relation.   With this twist, linear Lagrangian relations are 
also closed under composition \cite[Prop.\ 6.8]{BF}.  Thus, there is a category with symplectic vector spaces as objects and linear Lagrangian relations as morphisms.  This becomes symmetric monoidal using direct sums: in particular, if the linear relations $L \maps U \asrelto V$ and $L' \maps U' \asrelto V'$ are Lagrangian, so is $L \oplus L' \maps U \oplus U' \asrelto V \oplus V'$.  One can show using Proposition \ref{prop:strictification_1} that this symmetric monoidal category is equivalent to the following prop.

\begin{definition}
Let \define{$\Lag\Rel_k$} be the prop where a morphism $f \maps m \to n$ is a  linear Lagrangian relation from $(k \oplus k)^m$ to $(k \oplus k)^n$, composition is the usual composition of relations, and the symmetric monoidal structure is given by direct sum.
\end{definition}

We can now redefine
 the functor $K \maps \Fin\Corel \to \Fin\Rel_k$, and the black-boxing
functor $\blacksquare \maps \Circ \to \Fin\Rel_k$, to be morphisms of props taking values in 
$\Lag\Rel_k$.  This is the view we take henceforth.

\begin{proposition}
\label{prop:K_1}
The strict symmetric monoidal functor $K \maps \Fin\Corel \to \Fin\Rel_k$
maps any morphism $f \maps m \to n$ is any morphism in $\Fin\Cospan$ to 
a Lagrangian linear relation $K(f) \maps (k \oplus k)^m \asrelto (k \oplus k)^n$.  
It thus defines a morphism of props, which we call
\[   K \maps \Fin\Corel \to \Lag\Rel_k . \]
\end{proposition}

\begin{proof}
First note that $\Lag\Rel_k$ is a symmetric monoidal 
subcategory of $\Fin\Rel_k$: composition and direct sum for linear Lagrangian relations is a special case of composition and direct sum for linear relations.  Second, note that while $K$ applied to the object $n \in \Fin\Cospan$ gives the vector space $(k \oplus k)^n$, which is the object $2n$ in $\Fin\Rel_k$, this is the object $n$ in $\Lag\Rel_k$.  Thus, to check that strict symmetric monoidal functor $K \maps \Fin\Cospan \to \Fin\Rel_k$ defines a morphism of props from $\Fin\Cospan$ to $\Lag\Rel_k$, we just need to check that $K(f)$ is Lagrangian for each
generator $f$ of $\Fin\Cospan$.  We have
\[
\begin{array}{lcl}
  K(m) &=& \{(\phi_1,I_1,\phi_2,I_2,\phi_3,I_3) :\;  \phi_1= \phi_2 = \phi_3, \; I_1+ I_2=I_3\}  \\
  K(d) &=& \{(\phi_1,I_1,\phi_2,I_2,\phi_3,I_3) :\; \phi_1= \phi_2 = \phi_3, \;  I_1= I_2 + I_3\}  \\
  K(i) &=& \{(\phi_2, I_2) : \; I_2 = 0\}   \\
  K(e) &=& \{(\phi_1, I_1) :\; I_1= 0\}
\end{array}
\]
In each case the relation is an isotropic subspace of half
the total dimension, so it is Lagrangian.
\end{proof}

We can characterize this new improved $K$ as follows:

\begin{proposition}
\label{prop:K_2}
There exists a unique morphism of props
\[    K \maps \Fin\Corel \to \Lag\Rel_k \]
sending the extraspecial commutative Frobenus monoid $1 \in \Fin\Corel$ to the extraspecial commutative Frobenius monoid $k \oplus k \in \Lag\Rel_k$, where the first copy of $k$ is equipped with its duplicative Frobenius structure and the second is equipped with its additive Frobenius structure.
\end{proposition}

\begin{proof} 
Existence follows from \ref{prop:K_1}; uniqueness
follows from the fact that $\Fin\Corel$ is the prop for extraspecial commutative Frobenius
monoids.
\end{proof}

We can now reinterpret black-boxing of circuits of perfectly conductive wires as a morphism of props:

\begin{definition}
We define \define{black-boxing of circuits made of perfect wire} to be the morphism of props $\blacksquare \maps \Circ \to \Lag\Rel_k$ given by the composite
\[     \Circ \stackrel{H'}{\longrightarrow} \Fin\Cospan \stackrel{H}{\longrightarrow} 
\Fin\Corel \stackrel{K}{\longrightarrow} \Lag\Rel_k. \]
\end{definition}


\section{Black-boxing linear circuits}
\label{sec:black-boxing}

Black-boxing circuits of ideal conductive wires is just the first step. 
We now extend black-boxing to circuits made of wires labeled by elements of any set $L$.  The elements of $L$ play the role of `circuit elements' such as resistors, inductors and capacitors.  The extended black-boxing functor can be chosen so that these circuit elements are mapped to arbitrary Lagrangian linear relations from $k\oplus k$ to itself.   

The key to doing this is Proposition \ref{prop:lcirc_coproduct}, which says that
\[   \Circ_L \cong \Fin\Cospan + F(L) . \]
We also need Propositions \ref{prop:fincospan_to_fincorel} and 
\ref{prop:K_2}, which give prop morphisms $H \maps \Fin\Cospan \to \Fin\Corel$
and $K \maps \Fin\Corel \to \Lag\Rel_k$, respectively.

\begin{theorem} 
\label{thm:black-boxing_1}
For any field $k$ and label set $L$, there exists a unique morphism of props 
\[ \blacksquare \maps \Circ_L \cong \Fin\Cospan + F(L) \to \Lag\Rel_k \] 
such that $\blacksquare \vert_{\Fin\Cospan}$ is the composite
\[    
       \Fin\Cospan \stackrel{H}{\longrightarrow} 
       \Fin\Corel \stackrel{K}{\longrightarrow} \Lag\Rel_k  \]
and $\blacksquare \vert_{F(L)}$ maps each $\ell \in L$ to an arbitrarily chosen
Lagrangian linear relation from $k \oplus k$ to itself.    
\end{theorem}

\begin{proof}
By the universal property of the coproduct and the fact that $F(L)$ is the prop for 
$L$-actions, there exists a unique morphism of props
\[     \blacksquare \maps \Fin\Cospan + F(L) \to \Lag\Rel_k \]
such that $\blacksquare \vert_{\Fin\Cospan} = K  H $ and $\phi \vert_{F(L)}$ maps 
each $\ell \in L$ to an arbitrarily chosen Lagrangian linear relation from $k \oplus k$ to itself. 
\end{proof}

We can apply this theorem to circuits made of resistors.   Any resistor has a \define{resistance} $R$, which is a positive real number.  Thus, if we take the label set $L$ to be $\R^+$, we obtain a prop $\Circ_L$ that models circuits made of resistors.  Electrical engineers typically draw a resistor as a wiggly line:
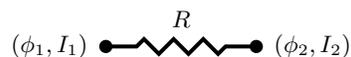
\[
\begin{tikzpicture}[circuit ee IEC, set resistor graphic=var resistor IEC graphic]
\node[contact] (C) at (0,2) {};
\node[contact] (D) at (2,2) {};
\node (A) at (-.75,2) {$(\phi_1,I_1)$};
\node (B) at (2.75,2) {$(\phi_2,I_2)$};
  \draw (0,2) to [resistor={info={$R$}}] ++(2,0);
\end{tikzpicture}
\]
Here $(\phi_1,I_1) \in \R \oplus \R$ are the \define{potential} and \define{current} at the resistor's input and $(\phi_2,I_2) \in \R \oplus \R$ are the potential and current at its output.  To define a black-boxing functor 
\[  \blacksquare \maps \Circ_L \to \Fin\Rel_k   \]
we need to choose a linear relation between these four quantities for each choice of the
resistance $R$.   We do this using Kirchhoff's current law which requires that the current flowing in equals the current flowing out: $I_1 = I_2$.  To relate current and potential we use Ohm's Law, which says that $V=RI$ where $I = I_1 = I_2$ is called the \define{current through} the resistor and $V = \phi_2 - \phi_1$ is called the \define{voltage across} the resistor.   Thus, for each $R \in \R^+$ we choose
\[ \blacksquare (R) = \{(\phi_1,I_1,\phi_2,I_2) : \; \phi_2-\phi_1 = R I_1, I_1 = I_2 \} .\]

We could stop here, but suppose we also want to include inductors and capacitors. 
An inductor comes with an inductance $L \in \R^+$ (not to be confused with our notation for a label set), while a capacitor comes with a capacitance $C \in \R^+$. These circuit elements are drawn as follows:
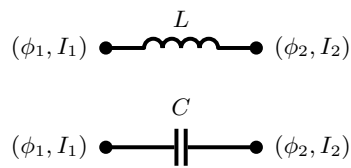
\[
\begin{tikzpicture}[circuit ee IEC, set resistor graphic=var resistor IEC graphic]
\node[contact] (C) at (0,2) {};
\node[contact] (D) at (2,2) {};
\node (A) at (-.75,2) {$(\phi_1,I_1)$};
\node (B) at (2.75,2) {$(\phi_2,I_2)$};
  \draw (0,2) to [inductor={info={$L$}}] ++(2,0);
\end{tikzpicture}
\]
\[
\begin{tikzpicture}[circuit ee IEC, set resistor graphic=var resistor IEC graphic]
\node[contact] (C) at (0,2) {};
\node[contact] (D) at (2,2) {};
\node (A) at (-.75,2) {$(\phi_1,I_1)$};
\node (B) at (2.75,2) {$(\phi_2,I_2)$};
  \draw (0,2) to [capacitor={info={$C$}}] ++(2,0);
\end{tikzpicture}
\]
These circuit elements apply to time-dependent currents and voltages, and they 
impose the relations $V = L \dot{I}$ and $I = C \dot{V}$, where the dot stands
for the time derivative.  Engineers deal with this using the Laplace transform.  
As explained in detail elsewhere \cite{BE,BF}, this comes down to adjoining a variable $s$ to the field $\mathbb{R}$ and letting $k = \R(s)$ be the field of rational functions in one real variable.   The variable $s$ has the meaning of a time derivative. If we now use $V$ and $I$ to denote the
Laplace transforms of current and voltage, we thus obtain the relations $V = sLI$ for the inductor and $I = sCV$ for the capacitor.  Thus if we extend our label set to the disjoint union of three copies
of $\R^+$, defining
\[  RLC = \R^+ + \R^+ + \R^+ ,\]
we obtain a prop $\RLCCirc$ that describes circuits of resistors, inductors and capacitors.  The name ``$\RLCCirc$" is a bit of a pun, since electrical engineers 
call a circuit made of one resistor, one inductor and one capacitor an ``$RLC$ circuit."

To construct the black-boxing functor
\[   \blacksquare \maps \RLCCirc \to \Fin\Rel_k \]
we specify it separately on each kind of circuit element.  Thus, on the first copy 
of $\R^+$, corresponding to resistors, we set
\[ \blacksquare (R) = \{(\phi_1,I_1,\phi_2,I_2) : \;\phi_2-\phi_1 = R I_1, \;  I_1 = I_2\}\]
as before.  On the second copy we set
\[ \blacksquare (L) = \{(\phi_1,I_1,\phi_2,I_2) : \; \phi_2-\phi_1 = s L I_1, \;  I_1 = I_2\}\]
and on the third we set
\[ \blacksquare (C) = \{(\phi_1,I_1,\phi_2,I_2) : \; sC(\phi_2-\phi_1) = I_1, \; \; I_1 = I_2\}  .\]
We have:

\begin{proposition}
\label{prop:black-boxing_1}
If $f \maps m \to n$ is any morphism in $\RLCCirc$, the linear relation $\blacksquare(f): (k \oplus k)^m\asrelto (k \oplus k)^n$ is Lagrangian.  We thus
obtain a morphism of props
\[   \blacksquare \maps \RLCCirc \to \Lag\Rel_k  \]
where $k = \R(s)$.
\end{proposition}

\begin{proof}
By Theorem \ref{thm:black-boxing_1} it suffices to check that 
the linear relations $\blacksquare(R), \blacksquare(L)$ and $\blacksquare(C)$ are
Lagrangian for any $R,L,C \in \R^+$.  To do this, one can check that these relations are 2-dimensional isotropic subspaces of $\overline{(k \oplus k)} \oplus (k \oplus k)$.
\end{proof}

A similar result was proved by Baez and Fong \cite{BF,Fo2} using different methods: decorated cospan categories rather than props.  In their work, resistors, inductors and capacitors were subsumed in a mathematically more natural class of circuit elements.   We can do something similar here.   At the same time, we might as well generalize to an arbitrary field $k$ and work with the prop $\Circ_k$, meaning $\Circ_L$ where the label set $L$ is taken to be $k$.   

\begin{definition} We call a morphism in $\Circ_k$ a \define{linear circuit}.
\end{definition}

Engineers might instead call such a morphism a `passive' linear circuit \cite{BF}, but
we will never need any other kind.

\begin{theorem}
\label{thm:black-boxing_2}
For any field $k$ there exists a unique morphism of props
\[   \blacksquare \maps \Circ_k \cong \Fin\Cospan + F(k) \to \Lag\Rel_k \]
such that $\blacksquare \vert_{\Fin\Cospan}$ is the composite
\[    \Fin\Cospan \stackrel{H}{\longrightarrow} 
       \Fin\Corel \stackrel{K}{\longrightarrow} \Lag\Rel_k  \]
and for each \(Z \in k\), the linear Lagrangian relation
\[   \blacksquare(Z) \maps k \oplus k \asrelto k \oplus k \]
is given by
\[  \blacksquare (Z) =\{(\phi_1,I_1,\phi_2,I_2) : \; \phi_2-\phi_1 = Z I_1 ,\;  I_1 = I_2\}. \]
\end{theorem}

\begin{proof} 
By Theorem \ref{thm:black-boxing_1} it suffices to check that $\blacksquare(Z)$ is a linear Lagrangian relation for any \(Z \in k\).    To do this, check that $\blacksquare(Z)$ is a 2-dimensional isotropic subspace of $\overline{(k \oplus k)} \oplus (k \oplus k)$.
\end{proof}

In electrical engineering, $Z$ is called the ``impedance:" a circuit
element with one input and one output has \define{impedance} $Z$ if the voltage across it is $Z$ times the current through it.   Resistance is a special case of impedance.  In particular, a perfectly conductive wire has impedance zero.  Mathematically, this fact is reflected in an inclusion of props
\[        \Circ \hookrightarrow \Circ_k  \]
that sends the generator $\ell$ in $\Circ \cong \Fin\Cospan + F(\{\ell\})$ to 
the generator $0 \in k$ in  $\Circ \cong \Fin\Cospan + F(k)$, while it is the identity
on $\Fin\Cospan$.    Black-boxing for linear circuits then extends black-boxing
as previously defined for circuits of ideal conductive wires.   That is, we have a commutative
triangle:
\begin{equation}
\label{eq:extension_of_black-boxing}
    \xymatrix{
        \Circ \phantom{ |} \ar[dr]^-{\blacksquare} \ar@{^{(}->}[d] &  \\
        \Circ_k  \phantom{ |}  \ar[r]_-{\blacksquare} & \Lag\Rel_k. }
\end{equation}


\section{Signal-flow diagrams}
\label{sec:sigflowdiagrams}

Control theory is the branch of engineering that studies the behavior of open dynamical systems: that is, systems with inputs and outputs.  Control theorists have intensively studied \emph{linear} open dynamical systems, and they specify these using signal-flow diagrams.  We now know that signal-flow diagrams are a syntax for linear relations.  In other words, we can see signal-flow diagrams as morphisms in a free prop that maps onto the prop of linear relations, $\Fin\Rel_k$.   This is a nice example of functorial semantics drawn from engineering.

The machinery of props lets us map circuit diagrams to signal-flow diagrams in a manner compatible with composition, addressing a problem raised by Willems \cite{Willems}.  To do this, because circuits as we have defined them obey some nontrivial equations, while the prop of signal-flow diagrams is free, we need to introduce free props that map onto $\Circ$ and $\Circ_k$.  

In Section \ref{sec:linear_relations} we discussed this presentation of $\Fin\Rel_k$: 
\[ \xymatrix{
F(E_k) \ar@<-.5ex>[r]_-{\rho_k} \ar@<.5ex>[r]^-{\lambda_k} & F(\Sigma) + F(\Sigma) + F(k) \ar[r]^-{\square} & \Fin\Rel_k. }
 \]
The morphisms in the free prop $F(\Sigma) + F(\Sigma) + F(k)$ can be drawn as string diagrams, and these roughly match what control theorists call signal-flow diagrams.  
So, we make the following definition:

\begin{definition} Define \define{$\SigFlow_k$} to be $F(\Sigma) + F(\Sigma) + F(k)$. We call a morphism in $\SigFlow_k$ a \define{signal-flow diagram}.
\end{definition}

The prop $\SigFlow_k$ is free on eight generators together with one generator for each element of  $k$.  The meaning of these generators is best understood in terms of the linear relations they are mapped to under $\square$.  We discussed those linear relations in Section \ref{sec:black-boxing_conductive}.   So, we give the generators of $\SigFlow_k$ the same names.  Baez and Erbele also drew pictures of them, loosely modeled after the notation in signal-flow diagrams \cite{BE}. Recall that the generators of the first copy of $F(\Sigma)$ are:
\begin{itemize}
\item \define{coduplication}, $\Delta^\dagger \maps 2 \to 1$
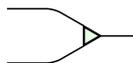
\[
  \xymatrix@1{
    \sigflowpiccodup{.15\textwidth} 
  }
\]
 \item \define{codeletion}, $!^\dagger \maps 0 \to 1$
\[
  \xymatrix@1{
    \;\;\quad\quad\sigflowpiccodel{.1\textwidth} 
  }
\]
\item \define{duplication}, $\Delta \maps 1 \to 2$
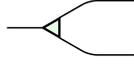
\[
  \xymatrix@1{
    \sigflowpicdup{.15\textwidth}\;\;\quad\quad 
  }
\]

\item \define{deletion}, $! \maps 1 \to 1$
\[
  \xymatrix@1{
    \sigflowpicdel{.1\textwidth} 
  }
\]
\end{itemize}
The generators of the second copy of $F(\Sigma)$ are:
\begin{itemize}
\item 
\define{addition}, $+ \maps 2 \to 1$
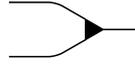
\[
  \xymatrix@1{
    \sigflowpicadd{.15\textwidth} 
  }
\]
\item \define{coaddition}, $+^\dagger \maps 1 \to 2$
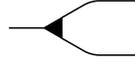
\[
  \xymatrix@1{
    \sigflowpiccoadd{.15\textwidth} 
  }
\]
\item \define{zero}, $0 \maps 0 \to 1$ 
\[
  \xymatrix@1{
    \;\;\quad\quad \sigflowpiczero{.1\textwidth} 
  }
\]
\item \define{cozero}, $0^\dagger \maps 1 \to 0$
\[
  \xymatrix@1{
    \sigflowpiccozero{.1\textwidth} \quad\quad\;\;
  }
\]
\end{itemize}
The generators of $F(k)$ are:
\begin{itemize}
\item for each $c \in k$, \define{scalar multiplication}, $c \maps 1 \to 1$
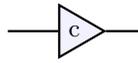
\[
  \xymatrixrowsep{5pt}
  \xymatrix@1{
    \SigLabelpic{.15\textwidth}\\
  }
\]
\end{itemize}
\vskip -2.5em
Since $\SigFlow_k$ is a free prop, while $\Circ_k$ is not, there is no useful morphism of props from $\Circ_k$ to $\SigFlow_k$.   However, there is a free prop $\Ccirc_k$ having $\Circ_k$ as a quotient, and a morphism from this free prop to $\SigFlow_k$.  This morphism lifts the black-boxing functor described in Theorem \ref{thm:black-boxing_2} to a morphism between free props.   

\begin{definition}
For any set $L$ define the prop \define{$\Ccirc_L$} by
\[     \Ccirc_L = F(\Sigma) + F(L).  \]
Here $L$ stands for the signature with one unary operation for each element of $L$ while $\Sigma$ is the signature with elements $m \maps 2\to 1, i \maps 0 \to 1, d \maps 1 \to 2$ and $e \maps 1\to 0$.
\end{definition}

In Example \ref{ex:fincospan_presentation} we saw this presentation for $\Fin\Cospan$:
\[ \xymatrix{
F(E)\ar@<-.5ex>[r]_-{\rho} \ar@<.5ex>[r]^-{\lambda} & F(\Sigma) \ar[r] & \Fin\Cospan } \]
where the equations in $E$ are the laws for a special commutative Frobenius monoid.   In the proof of Proposition \ref{prop:lcirc_coproduct} we derived this presentation for $\Circ_L$:
\[ \xymatrix{
F(E)\ar@<-.5ex>[r]_-{\iota \rho} \ar@<.5ex>[r]^-{\iota \lambda} & F(\Sigma) + F(L) \ar[r] & \Fin\Cospan + F(L) \cong \Circ_L} \]
where $\iota$ is the inclusion of $F(\Sigma)$ in $F(\Sigma) + F(L)$.  Since 
$F(\Sigma) + F(k) = \Ccirc_k$, we can rewrite this as 
\[ \xymatrix{
F(E)\ar@<-.5ex>[r]_-{\iota \rho} \ar@<.5ex>[r]^-{\iota \lambda} & \Ccirc_L \ar[r] & \Circ_L } \]
The last arrow here, which we call $P \maps \Ccirc_L \to \Circ_L$, imposes the laws of a special commutative Frobenius monoid on the object $1$.  

The most important case of this construction is when $L$ is some field $k$:

\begin{theorem}
\label{thm:black-and-white-boxing}
For any field $k$, there is a strict symmetric monoidal functor $T \maps \Ccirc_k \to \SigFlow_k$
giving a commutative square of strict symmetric monoidal functors
\[
    \xymatrix{
      \Ccirc_k \phantom{ |} \ar[r]^-{P} \ar[d]_{T} & \Circ_k \phantom{ |} \ar[rr]^-{\blacksquare} & & \Lag\Rel_k \phantom{ |} \ar@{^{(}->}[d] \\
      \SigFlow_k \phantom{ |}\ar[rrr]^-{\square} & & & \Fin\Rel_k. \phantom{ |}
      }
\]
\end{theorem}

\noindent
The horizontal arrows in this diagram are morphisms of props.  The vertical ones are not, because they send the object $1$ to the object $2$.

\begin{proof}
We define the strict symmetric monoidal functor $T \maps \Ccirc_k \to \SigFlow_k$
as follows.  It sends the object $1$ to $2$ and has the following action on the generating morphisms of $\Ccirc_k = F(\Sigma) + F(k)$:
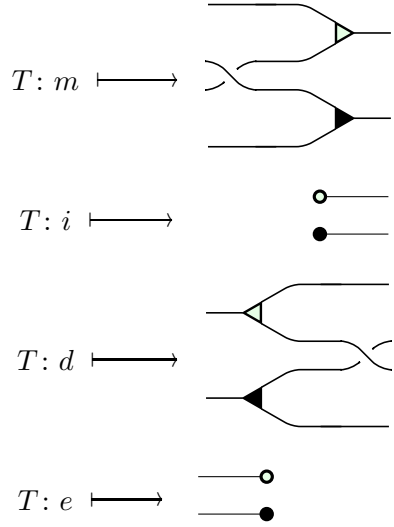
\[
  \xymatrixrowsep{5pt}
  \xymatrix@1{
  T \maps m \;\; \ar@{|->}@<.25ex>[r] & \SigMultpic{.20\textwidth}  \\
 T \maps i \;\; \ar@{|->}@<.25ex>[r] & \quad\quad\quad\quad \;\SigUnitpic{.10\textwidth} \\
 T \maps d \;\; \ar@{|->}@<.25ex>[r] & \SigCoMultpic{.20\textwidth}  \\
 T \maps e \;\; \ar@{|->}@<.25ex>[r] & \SigCoUnitpic{.10\textwidth} \quad\quad \quad\quad\quad \\
  }
\]
and for each element $Z \in k$, 
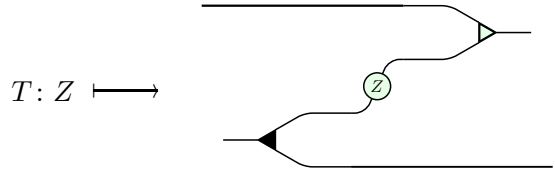
\[
  \xymatrixrowsep{5pt}
  \xymatrix@1{
   T \maps Z \;\; \ar@{|->}@<.25ex>[r] & \quad \SigLabelEdgepic{.32\textwidth}\\
  }
\]
where we use string diagram notation for morphisms in $\SigFlow_k$.    To check that $T$ with these properties exists and is unique, let $\SigFlow_k^{\mathrm{ev}}$ be the full subcategory of $\SigFlow_k$ whose objects are even natural numbers.  This becomes a prop if we rename each object $2n$, calling it $n$.  Then, since $\Ccirc_k$ is free, there exists a unique morphism of props $T \maps \Ccirc _k\to \SigFlow_k^{\mathrm{ev}}$ defined on generators as above.   Since $\SigFlow_k^{\mathrm{ev}}$ is a symmetric monoidal subcategory of $\SigFlow_k$, we can reinterpret $T$ as a strict symmetric monoidal functor $T \maps \Ccirc_k \to \SigFlow_k$, and this too is uniquely determined by its action on the generators.

To prove that the square in the statement of the theorem commutes, it suffices to check it on the generators of $\Ccirc_k$.   For this we use the properties of black-boxing stated in Theorem \ref{thm:black-boxing_2}.   First, note that these morphisms in $\SigFlow_k$:
\[
  \xymatrixrowsep{5pt}
  \xymatrix@1{
  T(m) \;\; & =& \SigMultpic{.20\textwidth}  \\
 T(i) \;\; &=& \quad\quad\quad\quad \;\SigUnitpic{.10\textwidth} \\
 T(d) \;\; &=& \SigCoMultpic{.20\textwidth}  \\
 T(e) \;\; &=& \SigCoUnitpic{.10\textwidth} \quad\quad \quad\quad\quad \\
  }
\]
are mapped by $\square$ to the same multiplication, unit, comultiplication and counit
on $k \oplus k$ as given by $\blacksquare(m), \blacksquare(i), \blacksquare(d)$
and $\blacksquare(e)$.   Namely, these four linear relations make $k \oplus k$ into Frobenius monoid where the first copy of $k$ has the duplicative Frobenius structure and the second copy has the additive Frobenius structure.  Second, note that the morphism
\[   T(Z) = \SigLabelEdgepic{.32\textwidth}  \]
in $\SigFlow_k$ is mapped by $\square$ to the linear relation
\[   \{(\phi_1,I_1,\phi_2,I_2) : \; \phi_2-\phi_1 = Z I_1 ,\;  I_1 = I_2\}, \]
while Theorem \ref{thm:black-boxing_2} states that $\blacksquare(Z)$ is the same
relation, viewed as a Lagrangian linear relation.
 \end{proof}

Recall that when $L = \{\ell\}$, we call $\Circ_L$ simply $\Circ$.    We have seen that the map $\{\ell\} \to k$ sending $\ell$ to $0$ induces a morphism of props $\Circ \hookrightarrow \Circ_k$, which expresses how circuits of perfectly conductive wires are a special case of linear circuits.  We can define a morphism of props $\Ccirc \hookrightarrow \Ccirc_k$ in an analogous
way.  Due to the naturality of the above construction, we obtain a commutative square
\[
    \xymatrix{
     \Ccirc \phantom{ |} \ar[r]^-{P} \ar@{^{(}->}[d] & \Circ \phantom{ |} \ar@{^{(}->}[d]  \\
      \Ccirc_k \phantom{ |} \ar[r]^-{P} & \Circ_k \phantom{ |} 
      }
\]
We can combine this with the commutative square in Theorem \ref{thm:black-and-white-boxing} and the commutative triangle in Equation \ref{eq:extension_of_black-boxing}, which expands to a square when we use the definition of black-boxing for circuits of perfectly conductive wires.  The resulting diagram summarizes the relationship between linear circuits, cospans, corelations, and signal-flow diagrams:
\[
    \xymatrix{
     \Ccirc \phantom{ |} \ar[r]^-{P} \ar@{^{(}->}[d] & \Circ \phantom{ |} \ar@{^{(}->}[d] \ar[r]^-{H'} & \Fin\Cospan \phantom{ |} \ar[r]^-{H} &
      \Fin\Corel \phantom{ |} \ar[d]^{K} \\
      \Ccirc_k \phantom{ |} \ar[r]^-{P} \ar[d]_{T} & \Circ_k \phantom{ |} \ar[rr]^-{\blacksquare} & & \Lag\Rel_k \phantom{ |} \ar@{^{(}->}[d] \\
      \SigFlow_k \phantom{ |}\ar[rrr]^-{\square} & & & \Fin\Rel_k. \phantom{ |}
      }
\]

In conclusion, we warn the reader that Erbele \cite{E} uses a different definition of $\SigFlow_k$.  His prop with this name is free on the following generators:
\begin{itemize}
\item addition, $+ \maps 2 \to 1$
\item zero, $0 \maps 0 \to 1$
\item duplication, $\Delta \maps 1 \to 2$
\item deletion, $! \maps 1 \to 0$
\item for each $c \in k$, scalar multiplication $c \maps 1 \to 1$
\item the cup, $\cup \maps 2 \to 0$
\item the cap, $\cap \maps 0 \to 2$
\end{itemize}
The main advantage is that string diagrams for morphisms in his prop more closely resemble the signal-flow diagrams actually drawn by control theorists; however, see his discussion of some subtleties.  All the results above can easily be adapted to Erbele's definition. 

\section{Voltage and current sources}
\label{sec:affine}

In the previous work on electrical circuits by Baez and Fong \cite{BF}, batteries were not included.  Resistors, capacitors, and inductors define linear relations between potential and current.   Batteries, also known as ``voltage sources," define \emph{affine} relations between these quantities.  The same is true of current sources.    Thus, to handle these additional circuit elements, we need a black-boxing functor that takes values in a different prop.  The ease with which we can do this illustrates the flexibility of working with props. In what follows, we continue to work over an arbitrary field $k$, which in electrical engineering is either $\R$ or $\R(s)$.

A voltage source is typically drawn as follows:
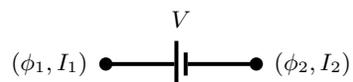
\[
\begin{tikzpicture}[circuit ee IEC, set resistor graphic=var resistor IEC graphic]
\node[contact] (C) at (0,2) {};
\node[contact] (D) at (2,2) {};
\node (A) at (-.75,2) {$(\phi_1,I_1)$};
\node (B) at (2.75,2) {$(\phi_2,I_2)$};
  \draw (0,2) to [battery={info={$V$}}] ++(2,0);
\end{tikzpicture}
\]
It sets the difference between the output and input potentials to a
constant value $V \in k$.    Thus, to define a black-boxing functor 
for voltage sources, we want to set
\[  \blacksquare(V) =  \{(\phi_1,I_1,\phi_2,I_2) : \; \phi_2-\phi_1 = V, \;  I_1 = I_2 \} .\]
Similarly, a current source is drawn as
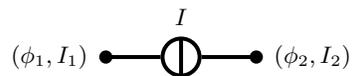
\[
\begin{tikzpicture}[circuit ee IEC, set resistor graphic=var resistor IEC graphic]
\node[contact] (C) at (0,2) {};
\node[contact] (D) at (2,2) {};
\node (A) at (-.75,2) {$(\phi_1,I_1)$};
\node (B) at (2.75,2) {$(\phi_2,I_2)$};
  \draw (0,2) to [current source={info={$I$}}] ++(2,0);
\end{tikzpicture}
\]
and it fixes the current at both input and output to a constant value $I$, giving
this relation:
\[  \blacksquare(I) = \{(\phi_1,I_1,\phi_2,I_2) : \;  I_1 = I_2 = I \} .\]

We could define a black-boxing functor suitable for voltage and current sources
by using a prop where the morphisms $f \maps m \to n$ are arbitrary relations
from $(k \oplus k)^m$ to $(k \oplus k)^n$.  However, the relations shown above 
are better than average.  First, they are `affine relations': that is, translates of linear
subspaces of $(k \oplus k)^m \oplus (k \oplus k)^n$.  For voltage sources we have
\[    \blacksquare(V)  = (0,0,V,0) +  \{(\phi_1,I_1,\phi_1,I_1) \}  \]
and for current sources we have
\[    \blacksquare(I) =  (I,0,I,0) +\{(\phi_1,0,\phi_2,0) \} .\]
Second, these affine relations are `Lagrangian': that is, they are translates of
Lagrangian linear relations.    

Thus, we proceed as follows:

\begin{definition}  
Given symplectic vector spaces $(V,\omega)$ and $(V',\omega')$,
an \define{Lagrangian affine relation} $R \maps V \asrelto V'$ is an affine subspace
$R \subseteq \overline{V} \oplus V'$ that is also a Lagrangian subvariety of 
$\overline{V} \oplus V'$.
\end{definition}

Here recall that a subset $A$ of a vector space is said to be an 
\define{affine subspace} if it is closed under affine linear combinations: 
if $a,a' \in A$ then so is $t a + (1-t)a'$ for all $t \in \R$.   A subvariety
$R \subseteq \overline{V} \oplus V'$ is said to be \define{Lagrangian} if each 
of its tangent spaces, when identified with a linear subspace of
$\overline{V} \oplus V'$, is Lagrangian.   If $R$ is an affine subspace of 
$\overline{V} \oplus V'$ it is automatically a subvariety, and it is either empty or 
a translate $L + (v,v')$ of some linear subspace $L \subseteq \overline{V} \oplus V'$.
In the latter case all its tangent spaces become the same when identified with linear subspaces of $\overline{V} \oplus V'$: they are all simply $L$.   We thus have:

\begin{proposition}
Given symplectic vector spaces $(V,\omega)$ and $(V',\omega')$, any Lagrangian affine relation $R \maps V \asrelto V'$ is either empty or a translate of a Lagrangian linear relation.
\end{proposition}

This allows us to construct the following category:

\begin{proposition}
There is a category where the objects are symplectic vector spaces, the morphisms
are Lagrangian affine relations, and composition is the usual composition of relations.  This is a symmetric monoidal subcategory of the category of sets and relations with the symmetric monoidal structure coming from the cartesian product of sets.  
\end{proposition}

\begin{proof}
It suffices to check that morphisms are closed under composition and tensor product and that the braiding is a Lagrangian affine relation. Let $R \maps U \asrelto V$ and $S \maps V \asrelto W$ be two Lagrangian affine relations. If the composite $S R$ is empty then we are done, so suppose it is not.  Thus there exist $u\in U, v\in V, w\in W$ such that $(u,v)\in R, (v,w) \in S$, and $(u,w) \in S R $.   Consider the subspaces $-(u,v) + R= L$ and $-(v,w) + S = M$. These are both affine subspaces containing the origin and thus are linear subspaces. Since $R$ and $S$ is are Lagrangian affine subspaces so are their translates $L$ and $M$, so these are Lagrangian linear relations from $U$ to $W$.   It follows that $M L\maps V \asrelto W $ is a Lagrangian linear relation.  We claim that $S R = (u,w) + M L$ so that morphisms in our proposed category are closed under composition.  First write $R = L +(u,v)$ and $S = M + (v,w)$ so that
\begin{align*}
S R &=  \{(x,z)  | \exists y \in V \text{ s.t. } (x,y) \in R \text{ and } (y,z) \in S\} \\
& = \{(x,z) | \exists y \in V \text{ s.t. } (x,y) \in L+(u,v) \text{ and } (y,z) \in M+ (v,w)\} \\[-3em]
& = \{(x,z) | \exists y \in V \text{ s.t. } 
	\begin{aligned}
                            \\ 
                             \\
          x= l_1 + u \\ 
          y = l_2 + v \\
          (l_1,l_2) \in L
        \end{aligned}
	\begin{aligned} \\
                              \\
            y = m_1 + v \\
           z = m_2 + w \\\hspace{2ex}
         (m_1,m_2) \in M\} \\
	\end{aligned}
\end{align*}
which gives us that $l_2 = m_1$ so finally we have
\begin{align*}
S  R & = \{(l_1+u, m_2+w) | \exists l_2 \in V \text{ s.t. } (l_1,l_2) \in L \text{ and } (l_2,m_2) \in M\}\\
&  = (u,w) + M L
\end{align*}
as desired.

The tensor product of Lagrangian affine relations $R \maps U \asrelto V$ and $R'  \maps U' \asrelto V'$ is given by
\[  R \oplus R' = \{(u,u',v,v'):  \; (u,v) \in R, \; (u',v') \in R'\} \maps U \oplus U' \asrelto V \oplus V' ,\]
and this is a Lagrangian affine relation because it is a translate of a Lagrangian linear relation.  Finally, note that the braiding morphism $B_{U,V} \maps U\oplus V \asrelto V \oplus U$ defined by $ B_{U,V} = \{(u,v,v,u) |u\in U, v\in V\}$ is a Lagrangian relation so it is an affine Lagrangian relation.
\end{proof}

By Proposition \ref{prop:strictification_1}, the above symmetric monoidal category is equivalent to the following prop:

\begin{definition}
Let \define{$\Aff\Lag\Rel_k$} be the prop where a morphism $f \maps m \to n$ is an
affine Lagrangian relation from $(k \oplus k)^m$ to $(k \oplus k)^n$, composition
is the usual composition of relations, and the symmetric monoidal structure is given 
as above.
\end{definition}

We can extend the black-boxing functor from linear circuits to a circuits that
include voltage and/or current sources.   The target of this extended black-boxing functor will be, not $\Lag\Rel_k$, but $\Aff\Lag\Rel_k$.  

\begin{theorem} 
\label{thm:black-boxing_3}
For any field $k$ and label set $L$, there exists a unique morphism of props 
\[ \blacksquare \maps \Circ_L \cong \Fin\Cospan + F(L) \to \Aff\Lag\Rel_k \] 
such that $\blacksquare \vert_{\Fin\Cospan}$ is the composite
\[    
       \Fin\Cospan \stackrel{H}{\longrightarrow} 
       \Fin\Corel \stackrel{K}{\longrightarrow} \Lag\Rel_k \hookrightarrow \Aff\Lag\Rel_k \]
and $\blacksquare \vert_{F(L)}$ maps each $\ell \in L$ to an arbitrarily chosen
Lagrangian affine relation from $k \oplus k$ to itself.    
\end{theorem}

\begin{proof}
The proof mimics that of Theorem \ref{thm:black-boxing_1}.
\end{proof}

Using the formulas given above for the relations between potentials and currents for voltage sources and current sources, we can use this theorem to define black-boxing functors for circuits that include these additional circuit elements.   A similar strategy can be used to define black-boxing for circuits containing other nonlinear circuit elements, such as transistors.  We merely need to expand the target of this black-boxing functor to include all the relations between potentials and currents that arise.

\chapter{Bond graphs}
\label{chap:Bond_Graphs}

In this chapter we study bond graphs, which are another diagrammatic tool used by engineers.  Recall that a bond graph consists of edges called bonds which are connected to one another via $1$-junctions and $0$-junctions. Although bonds go between ports we do not draw the ports since neither do engineers. Also, recall that each bond is associated with a pair of real numbers called effort, $E$, and flow, $F$, where effort is a generalization of voltage and flow is a generalization of current. We adopt a notion of inputs and ouputs for bond graphs where think of effort and flow as going from left to right. That is, bonds which go into a junction from the left are inputs while bonds going out to the right are outputs.

 The $1$-junction puts the associated bonds into a ``series" connection.  A $1$-junction with $n$ bonds, where $k$ are inputs and $n-k$ are outputs,  imposes the equations  $$\sum_{i=1}^{k} E_i = \sum_{i=k+1}^{n} E_i $$  and $$ F_1 = F_2 = \cdots = F_n $$ In other words the sum of efforts going into a $1$-junction is equal to the sum of efforts going out while all of the flows are equal. These equations come from using a different convention than what engineers use, but this does not cause any issues. As an example consider the following $1$-junction with $3$ bonds:  

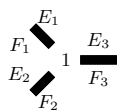
\begin{figure}[H] 
	\centering
\begin{tikzpicture}[circuit ee IEC, set resistor graphic=var resistor IEC
      graphic, scale=0.8, every node/.append style={transform shape}]
[
	node distance=1.5cm,
	mynewelement/.style={
		color=blue!50!black!75,
		thick
	},
	mybondmodpoint/.style={
	rectangle,
	minimum size=3mm,
	very thick,
	draw=red!50!black!50, 
	outer sep=2pt
	}
]		
	\node(J11) {$\mathrm{1}$};
	\node (R2) [ below left of=J11] {}
	edge  [line width=3.5pt]   node [right = .1, below]{$F_2$} (J11)
        edge  [line width=3.5pt]   node [above=.15, left]{$E_2$} (J11);
	\node (R1) [ above left of=J11] {}
	edge [line width=3.5pt]    node [below = .15, left ]{$F_1$} (J11)
        edge  [line width=3.5pt]   node [right=.1, above]{$E_1$} (J11);
	\node (C1) [right of=J11] {}
	edge [line width=3.5pt]    node [below]{$F_3$} (J11)
        edge [line width=3.5pt]    node [above]{$E_3$} (J11);
      \end{tikzpicture} 
\end{figure}

\vspace{-1ex}
\noindent Since there are $3$ bonds and thus $3$ ports this is sometimes called a ``$3$-port" while a junction with $n$ bonds is sometimes called an ``$n$-port."  Using our convention the efforts obey the equation $E_1+E_2=E_3$. Meanwhile the flows  obey the equation $F_1=F_2=F_3$. We say that this $1$-junction adds efforts and coduplicates flows. We may also turn this picture around to get another $1$-junction with one input and two outputs:

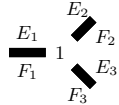
\begin{figure}[H] 
	\centering
\begin{tikzpicture}[circuit ee IEC, set resistor graphic=var resistor IEC
      graphic, scale=0.8, every node/.append style={transform shape}]
[
	node distance=1.5cm,
	mynewelement/.style={
		color=blue!50!black!75,
		thick
	},
	mybondmodpoint/.style={
	rectangle,
	minimum size=3mm,
	very thick,
	draw=red!50!black!50, 
	outer sep=2pt
	}
]		
	\node(J11) {$\mathrm{1}$};
	\node (R2) [ below right of=J11] {}
	edge  [line width=3.5pt]   node [ left = .1, below]{$F_3$} (J11)
        edge  [line width=3.5pt]   node [above=.15, right]{$E_3$} (J11);
	\node (R1) [ above right of=J11] {}
	edge [line width=3.5pt]    node [below=.15, right]{$F_2$} (J11)
        edge  [line width=3.5pt]   node [left=.1, above ]{$E_2$} (J11);
	\node (C1) [left of=J11] {}
	edge [line width=3.5pt]    node [below]{$F_1$} (J11)
        edge [line width=3.5pt]    node [above]{$E_1$} (J11);
      \end{tikzpicture} 
\end{figure}

\vspace{-1ex}

\noindent We say that this junction coadds effort and duplicates flow since now $E_1=E_2+E_3$ and $F_1=F_2=F_3$. 

The other type of junction is the $0$-junction, i.e.\ the ``parallel" junction. A $0$-junction with $n$ bonds, where $k$ are inputs and $n-k$ are outputs,  imposes the equations $$\sum_{i=1}^{k} F_i = \sum_{i=k+1}^{n} F_i = 0$$ and $$E_1=E_2=\cdots= E_n.$$  When there are $3$ bonds this junction is drawn as follows:
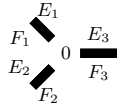
\begin{figure}[H] 
	\centering
\begin{tikzpicture}[circuit ee IEC, set resistor graphic=var resistor IEC
      graphic, scale=0.8, every node/.append style={transform shape}]
[
	node distance=1.5cm,
	mynewelement/.style={
		color=blue!50!black!75,
		thick
	},
	mybondmodpoint/.style={
	rectangle,
	minimum size=3mm,
	very thick,
	draw=red!50!black!50, 
	outer sep=2pt
	}
]		
	\node(J11) {$\mathrm{0}$};
	\node (R2) [ below left of=J11] {}
	edge  [line width=3.5pt]   node [ right = .1, below]{$F_2$} (J11)
        edge  [line width=3.5pt]   node [above=.15, left]{$E_2$} (J11);
	\node (R1) [ above left of=J11] {}
	edge [line width=3.5pt]    node [below = .15, left ]{$F_1$} (J11)
        edge  [line width=3.5pt]   node [right = .1, above]{$E_1$} (J11);
	\node (C1) [right of=J11] {}
	edge [line width=3.5pt]    node [below]{$F_3$} (J11)
        edge [line width=3.5pt]    node [above]{$E_3$} (J11);
      \end{tikzpicture} 
\end{figure}

\vspace{-1ex}

\noindent Now $F_1+F_2=F_3$ and $E_1=E_2=E_3$ so we say that this junction adds flow and coduplicates effort. Similarly, the reflected $0$-junction coadds flow and duplicates effort. This is drawn as:
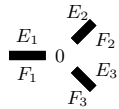
\begin{figure}[H] 
	\centering
\begin{tikzpicture}[circuit ee IEC, set resistor graphic=var resistor IEC
      graphic, scale=0.8, every node/.append style={transform shape}]
[
	node distance=1.5cm,
	mynewelement/.style={
		color=blue!50!black!75,
		thick
	},
	mybondmodpoint/.style={
	rectangle,
	minimum size=3mm,
	very thick,
	draw=red!50!black!50, 
	outer sep=2pt
	}
]		
	\node(J11) {$\mathrm{0}$};
	\node (R2) [ below right of=J11] {}
	edge  [line width=3.5pt]   node [left = .1, below]{$F_3$} (J11)
        edge  [line width=3.5pt]   node [above=.15, right]{$E_3$} (J11);
	\node (R1) [ above right of=J11] {}
	edge [line width=3.5pt]    node [below=.15, right]{$F_2$} (J11)
        edge  [line width=3.5pt]   node [ left=.1, above]{$E_2$} (J11);
	\node (C1) [left of=J11] {}
	edge [line width=3.5pt]    node [below]{$F_1$} (J11)
        edge [line width=3.5pt]    node [above]{$E_1$} (J11);
      \end{tikzpicture} 
\end{figure}

\vspace{-1ex}
\noindent and imposes equations $F_1=F_2+F_3$ and $E_1=E_2=E_3$.

These four junctions can be stuck together in various ways to make larger bond graphs. There are also rules for simplifying bond graphs that are derived from the equations above. However, some bond graphs never appear in the engineering literature though one could imagine drawing them. This is typically because the result is trivial or does not have a useful physical interpretation in electrical engineering. One simple example is the following:
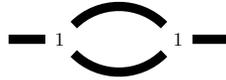
\begin{figure}[H]
\centering
\begin{tikzpicture}[circuit ee IEC, set resistor graphic=var resistor IEC
      graphic, scale=0.8, every node/.append style={transform shape}]
[
	node distance=1.5cm,
	mynewelement/.style={
		color=blue!50!black!75,
		thick
	},
	mybondmodpoint/.style={
	rectangle,
	minimum size=3mm,
	very thick,
	draw=red!50!black!50, 
	outer sep=2pt
	}
]		
	\node (J11) {1};
	\node (C1) [right of=J11] {}
        edge [line width=3.5pt]   node [left]{} (J11);
	\node (J12) [ left =1.5 of J11] {1}
    edge [line width=3.5pt, in =225, out=-45, looseness=1]   node [left]{} (J11)
    edge [line width=3.5pt, in =135, out=45, looseness=1]   node [left]{} (J11);
	\node (J13) [left of=J12] {}
        edge [line width=3.5pt]   node [left]{} (J12);
      \end{tikzpicture} 
\end{figure}
\vspace{-2ex}

\noindent Viewing this as the composite of some  $2$-output morphism and some $2$-input morphism presents a mild roadblock.  Since in a category we cannot arbitrarily stop compatible morphisms from composing, such a bond graph must be given meaning. Also, there are no junctions with fewer than $3$ total inputs and outputs, but with some experience using bond graphs one can imagine constructing such junctions. To overcome these obstacles we construct a prop $\BondGraph$ such that: 

\begin{enumerate}
\item $\BondGraph$ has generating morphisms corresponding to $1$-junctions and $0$-junctions for $n=3$ and $n=1$.
\item The morphisms in $\BondGraph$ obey relations that can be derived from the equations governing the junctions.
\item There is a functor assigning to any morphism in $\BondGraph$ a Lagrangian subspace of efforts and flows consistent with equations governing junctions.
\item There is also a functor assigning to any morphism in $\BondGraph$ a Lagrangian subspace of potentials and currents consistent with the laws governing potential and current along wires. 
\end{enumerate}

\noindent More explicitly, this means that the prop $\BondGraph$ shall have two generators of type $(2,1)$ and two of type $(1,2)$, as well as two of type $(0,1)$ and two of type $(1,0)$. Thus one may wonder about morphisms corresponding to junctions that have either two bonds or no bonds, i.e when $n=0$ or $n=2$. Since $\BondGraph$ is a prop we can build such morphisms as long as we have the proper relations. For example, the extra law dictates that the only junction between no ports is the junction that does nothing. The $n=2$ case corresponds to either the identity morphism due to the special law, or different kinds of caps and cups. The functors in $3)$ and $4)$ suggest two ways of defining the category $\BondGraph$. Neither approach succeeds, but putting them together we can succeed in defining this category.

In our first attempt at defining $\BondGraph$ we look more closely at the object $2$ in $\Fin\Corel$, where now we think of $\Fin\Corel$ as the category whose morphisms are circuits made only of perfectly conductive wire. We can do this because the functor $K\maps \Fin\Corel \to \Lag\Rel_k$ in Proposition \ref{prop:K_2} completely determines the behavior of a circuit made only of perfectly conductive wires. The reason we look at the object $2$ is that in this view the object $1$ is seen as a single terminal. Since bonds go between pairs of terminals where current on one terminal is equal and opposite to current on the other, it makes sense to think of $2$ as the end of a bond, i.e.\ a port.

 We equip the object $2$ in $\Fin\Corel$ with two different Frobenius monoid structures. The first Frobenius monoid has a multiplication and comultiplication that behave similarly to $1$-junctions. Due to this we think of the unit and counit of this Frobenius monoid as being unary versions of $1$-junctions. The other Frobenius monoid has a multiplication and comultiplication that behave similarly to $0$-junctions, so we also think of the unit and counit here as unary versions of $0$-junctions. The overall idea is that by thinking of the object $2$ as a pair of terminals, the two Frobenius structures on $2$ correspond with $1$- and  $0$-junctions together with unary versions. 

However, we quickly see that this method is not the way to define $\BondGraph$, although we still get an idea as to what properties $\BondGraph$ should have. There are some equations that hold between morphisms coming from these Frobenius monoids which do not hold between  the corresponding bond graphs. This approach culminates with the definition of $\Fin\Corel^{\circ}$, which is the subcategory of $\Fin\Corel$ generated by the two aforementioned Frobenius structures. 

 In our second attempt to define $\BondGraph$ we look at the category $\Lag\Rel_k$ where there some morphisms coming from Frobenius monoids that also seem to correspond to $1$- and $0$-junctions. This is expected since a bond graph is automatically associated with a Lagrangian subspace. However once again some issues arise; there are equations that hold between some of the Lagrangian relations which do not hold between corresponding bond graphs. Thus instead of defining $\BondGraph$ using the Frobenius structures in $\Lag\Rel_k$  we define another prop $\Lag\Rel_k^{\circ}$ using them.

Finally, we notice that $\Fin\Corel^{\circ}$ and $\Lag\Rel_k^{\circ}$ actually have corresponding generators  together with many corresponding equations. We define $\BondGraph$ to be the prop generated by morphisms corresponding to the generators of $\Fin\Corel^{\circ}$ and $\Lag\Rel_k^{\circ}$, obeying the equations found in both $\Fin\Corel^{\circ}$ and $\Lag\Rel_k^{\circ}$. This definition lets us define two functors out of $\BondGraph$:
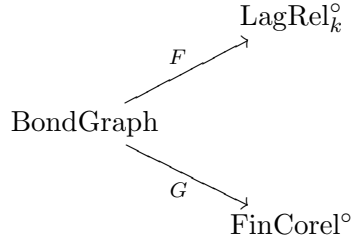
\[
    \xymatrix@C-3pt{
      & \Lag\Rel_k^{\circ} \\
	\BondGraph \ar[ur]^-{F} \ar[dr]_-{G} \\
 	& \Fin\Corel^{\circ} 
      }
\]
 \noindent The functor $G$ equips a bond graph with the behavior of a circuit by composing with $K\maps \Fin\Corel \to \Lag\Rel_k$, while $F$ equips a bond graph with the behavior of a bond graph in terms of effort and flow. Lastly, we show in Theorem \ref{thm:natural} that there is a natural transformation between these two functors which comes from the relationship between effort, flow, potential, and current given by the equations $V=\phi_2-\phi_1$ and $I=I_1=-I_2$.  All of this is summarized in the following diagram:
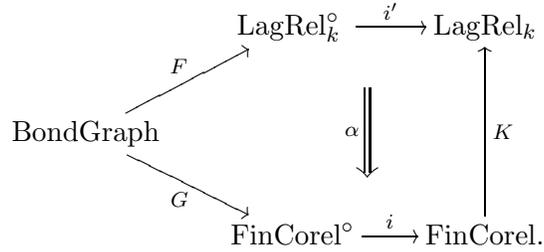
\[
    \xymatrix@C-3pt{
      & \Lag\Rel_k^{\circ} \phantom{ |} \ar[r]^-{i'} & \Lag\Rel_k \\
	\BondGraph \ar[ur]^-{F} \ar[dr]_-{G} \\
 	& \Fin\Corel^{\circ} \ar[r]^-{i}  \ar@<-6ex>@{}[uu]^(.25){}="a"^(.75){}="b" \ar@{<=}^{\alpha} "a";"b"   & \Fin\Corel. \ar[uu]_-{K}
      }
\]

\section{$1$- and $0$-junctions}

The first Frobenius monoid structure on the object $2$ arises naturally as a monad constructed from an adjunction. We shall see that the multiplication and comultiplication morphisms associated to this Frobenius monoid resemble $1$-junctions because of the relations obeyed. Later, we apply the functor $K\maps \Fin\Corel \to \Lag\Rel_k$ to these morphisms, and see that they are mapped to the behaviors corresponding to $1$-junctions. This is a more definitive reason for the association. 

First recall the generators of $\Fin\Corel$.
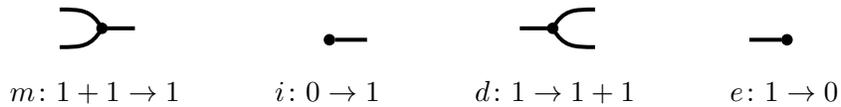
\[
  \xymatrixrowsep{1pt}
  \xymatrixcolsep{30pt}
  \xymatrix{
    \mult{.1\textwidth}  & \unit{.1\textwidth}  & \comult{.4\textwidth} & \counit{.1\textwidth}  \\
    m\maps 1 + 1 \to 1 & i\maps 0 \to 1 & d\maps 1 \to 1 + 1 & e\maps 1 \to 0
  }
\]
Also, recall that the object $1$ in $\Fin\Corel$ is self-dual, where the two morphisms  $ d \circ i \maps 0 \to 2$ and $e\circ m \maps 2\to 0$ are the unit and counit. We draw these two composites as follows:
\[
  \xymatrixrowsep{1pt}
  \xymatrixcolsep{8pt}
  \xymatrix{
    \captwo{.07\textwidth}  :=  & \comultcounit{.07\textwidth} & = & d\circ i \maps 0\to 2 \\
    \cuptwo{.07\textwidth}  :=  & \multunit{.07\textwidth} & = & e\circ m\maps 0\to 2
  }
\]
\noindent Finally, recall that these morphisms obey the zig-zag identities:
\[
  \xymatrixrowsep{1pt}
  \xymatrixcolsep{75pt}
  \xymatrix{
    \zigzaglaw{.07\textwidth}  =  \idone{.07\textwidth} =  \zigzaglawother{.07\textwidth}  \\
  }
\]

As a result of obeying the zig-zag identities, the two morphisms form an adjunction in $\Fin\Corel$, viewed as a one-object 2-category. Now recall that from any adjunction one can construct a monad  \cite{RS}. For us this makes $2$ in $\Fin\Corel$ into a monoid with the morphisms $\mathrm{id}_1 + (e\circ m) + \mathrm{id}_1\maps 4 \to 2$ and $d\circ i\maps 0 \to 2$, acting as multiplication and the unit respectively. These are drawn as follows:

\[
  \xymatrixrowsep{1pt}
  \xymatrixcolsep{5pt}
  \xymatrix{
    m_2 &:= & \monadmult{.1\textwidth}  &= &   \mathrm{id}_1 + (e\circ m) + \mathrm{id}_1\maps 4 \to 2  \\
    i_2 &:= & \captwo{.1\textwidth}  &= & d\circ i \maps 0\to 2  \phantom{hadalssdedddi} \\
  }
\]

Although the way in which $(2,m_2,i_2)$ forms a monoid is well known, we give the diagrammatic proof using our string diagrams in $\Fin\Corel$ for completeness. Associativity follows from only the monoidal structure of $\Fin\Corel$, while the left and right unit laws follow from the zig-zag identities governing adjunctions:

\[
  \xymatrixrowsep{1pt}
  \xymatrixcolsep{5pt}
  \xymatrix{
    \monadassocl{.1\textwidth}  &= &\monadassocm{.1\textwidth}  &= &\monadassocr{.1\textwidth} \\
  }
\]
\[
  \xymatrixrowsep{1pt}
  \xymatrixcolsep{5pt}
  \xymatrix{
    \monadunitl{.1\textwidth}  &= & \identitytwo{.1\textwidth} &= & \monadunitr{.1\textwidth} \\
  }
\]

Since  $\Fin\Corel$ is dagger compact, $2$ can also be made into a comonoid.  First, we turn $m_2$ and $i_2$ around to get two more corelations which we call $d_2$ and $e_2.$ These are drawn in the following way:
\[
  \xymatrixrowsep{1pt}
  \xymatrixcolsep{5pt}
  \xymatrix{
    d_2 & = & \monadcomult{.07\textwidth}   &= &   \mathrm{id}_1 + (d\circ i) + \mathrm{id}_1\maps 2 \to 4  \\
   e_2 & = &  \cuptwo{.07\textwidth}  &= & e\circ m \maps 0\to 2  \phantom{hadalsedddi} \\
  }
\]
\noindent and they give us a comonoid $(2,d_2,e_2)$:
\[
  \xymatrixrowsep{1pt}
  \xymatrixcolsep{5pt}
  \xymatrix{
    \monadcoassocl{.1\textwidth}  &= &\monadcoassocm{.1\textwidth}  &= &\monadcoassocr{.1\textwidth} \\
  }
\]
\[
  \xymatrixrowsep{1pt}
  \xymatrixcolsep{5pt}
  \xymatrix{
    \monadcounitl{.1\textwidth}  &= & \identitytwo{.1\textwidth} &= & \monadcounitr{.1\textwidth} \\
  }
\]

On the other hand, $1$-junctions  impose equations which allow one to simplify and redraw bond graphs. These simplifications correspond to how the multiplication and comultiplication morphisms here interact. These laws are just a small part of the larger story: how the monoid $(2,m_2,i_2)$ interacts with the comonoid $(2,d_2,e_2)$.

\begin{theorem}\label{thm:series}
The object $(2,m_2,i_2,d_2,e_2)$ in $\Fin\Corel$ is an extraspecial symmetric Frobenius monoid.
\end{theorem}

\begin{proof}

We already have the monoid and comonoid structure so we begin with the Frobenius laws. These follow due only to the monoidal structure of $\Fin\Corel$: 

\[
  \xymatrixrowsep{1pt}
  \xymatrixcolsep{5pt}
  \xymatrix{
    \monadfrobl{.1\textwidth}  &= & \monadfrobm{.1\textwidth} &= & \monadfrob{.1\textwidth} \\
  }
\]

\noindent The other Frobenius law is proven in a similar manner. For the extraspecial structure, it is necessary to first show the extra property, $e_2\circ i_2 = \mathrm{id}_0.$ This comes from first using the special property $m\circ d=\mathrm{id}_1$ and then the extra property $i\circ e=\mathrm{id}_0$:

\[
  \xymatrixrowsep{1pt}
  \xymatrixcolsep{5pt}
  \xymatrix{
    \monadextra{.1\textwidth}  &= & \monadextram{.1\textwidth} &= & \extral{.1\textwidth} \phantom{hf}=  \\
  }
\]

\noindent Using this fact we can show the special property, $d_2\circ m_2 = \mathrm{id}_2$:
\[
  \xymatrixrowsep{1pt}
  \xymatrixcolsep{5pt}
  \xymatrix{
    \monadspec{.1\textwidth}  &= & \identitytwo{.1\textwidth}   \\
  }
\]

\noindent Since the multiplication for $1$ is commutative, it is also symmetric. We use this symmetric  property to show the multiplication for $2$ is also symmetric. We begin by using the symmetric property to remove one braiding. Then due to the naturality of the braiding we are able to remove another braiding. The last equation merely makes the picture look nicer.
\[
  \xymatrixrowsep{1pt}
  \xymatrixcolsep{5pt}
  \xymatrix{
    \monadsymml{.1\textwidth}  &= & \monadsymmm{.1\textwidth}  &= & \monadsymmr{.1\textwidth}  &= & \monadsymmrr{.1\textwidth}   \\
  }
\]

\noindent We can further simplify this by using the naturality of the braiding again, which allows us to move the cup past the braiding. Since we are in a symmetric monoidal category we can then cancel the two braidings:
\[
  \xymatrixrowsep{1pt}
  \xymatrixcolsep{5pt}
  \xymatrix{
    \monadsymmrr{.1\textwidth}   &=  & \monadsymmend{.13\textwidth} 
  }
\qedhere 
\]
\end{proof}


Next we equip $2$ with another Frobenius monoid structure through another very common construction. Recall that from a pair of monoids $(X,m_X,i_X)$ and $(Y,m_Y,i_Y)$ in a braided monoidal category, there is a standard way to make $X\otimes Y$ into a monoid with $$ ( m_ X\otimes m_Y)\circ(\mathrm{id}_X \otimes B_{X,Y} \otimes \mathrm{id}_Y)\maps (X\otimes Y) \otimes (X\otimes Y)\to (X\otimes Y)$$ as multiplication and $$i_X\otimes i_Y\maps I \to (X\otimes Y)$$ as the unit. 
\noindent Consider the object $1+1 =2 \in \Fin\Corel$ with the following morphisms:

\[
  \xymatrixrowsep{1pt}
  \xymatrixcolsep{8pt}
  \xymatrix{
    \mu_2 & :=& \parmult{.07\textwidth}  & =  & ( m + m)\circ(\mathrm{id}_1 + B_{1,1} + \mathrm{id}_1)\maps 4\to 2  \\
    \iota_2 & := & \parunit{.07\textwidth}  &  =  & i + i\maps 0 \to 2  \phantom{hadasffaasfaldfai} 
  }
\]
\noindent Then $(2,\mu_2,\iota_2)$ is a monoid in $\Fin\Corel$. Since both monoids in this construction are $(1,m,i)$, our new monoid inherits the same properties that $(1,m,i)$ has. Using the dagger structure in $\Fin\Corel$, we also obtain corelations $\delta_2$ and  $\epsilon_2$, which when drawn as string diagrams are reflections of $\mu_2$ and $\iota_2$ respectively: 

\[
  \xymatrixrowsep{1pt}
  \xymatrixcolsep{8pt}
  \xymatrix{
    \delta_2 & :=& \parcomult{.07\textwidth}  & =  & (\mathrm{id}_1 + B_{1,1} + \mathrm{id}_1) \circ ( d + d)\maps 2\to 4  \\
    \epsilon_2 &:= & \parcounit{.07\textwidth}  &  =  & e + e\maps 2 \to 0  \phantom{hadasffsfaddldfai} 
  }
\]

\begin{theorem}\label{thm:parallel}
The object $(2, \mu_2,\iota_2, \delta_2, \epsilon_2) $  in $\Fin\Corel$ is an extraspecial commutative Frobenius monoid.
\end{theorem}

\begin{proof}

Each law is inherited from the extraspecial commutative Frobenius monoid $1$ due to the naturality of the braiding. We start with the monoid laws:
\[
  \xymatrixrowsep{1pt}
  \xymatrixcolsep{5pt}
  \xymatrix{
    \parassocl{.12\textwidth}  &= & \parassocr{.12\textwidth}   
  }
\]

\vspace{-3ex}

\[
  \xymatrixrowsep{1pt}
  \xymatrixcolsep{5pt}
  \xymatrix{
    \parunitl{.1\textwidth}  &= & \identitytwo{.12\textwidth} &= & \parunitr{.11\textwidth} 
  }
\]

\noindent Similarly we get a comonoid structure:
\[
  \xymatrixrowsep{1pt}
  \xymatrixcolsep{5pt}
  \xymatrix{
    \parcoassocl{.12\textwidth}  &= & \parcoassocr{.12\textwidth}   
  }
\]

\vspace{-3ex}

\[
  \xymatrixrowsep{1pt}
  \xymatrixcolsep{5pt}
  \xymatrix{
    \parcounitl{.1\textwidth}  &= & \identitytwo{.12\textwidth} &= & \parcounitr{.1\textwidth} 
  }
\]

 \noindent We next have the Frobenius laws:
\[
  \xymatrixrowsep{1pt}
  \xymatrixcolsep{5pt}
  \xymatrix{
    \parfrobl{.12\textwidth}  &= & \parfrobm{.12\textwidth} &= & \parfrobr{.12\textwidth} 
  }
\]

\noindent Then we have the extra and special laws:
\[
  \xymatrixrowsep{1pt}
  \xymatrixcolsep{5pt}
  \xymatrix{
	\parextra{.11\textwidth} & = & \\
    \parspec{.12\textwidth}  &= & \identitytwo{.12\textwidth}  
  }
\]

\noindent Finally, we have commutativity:
\[
  \xymatrixrowsep{1pt}
  \xymatrixcolsep{5pt}
  \xymatrix{
    \parcoml{.13\textwidth}  &= &  \parcomm{.13\textwidth} &= &  \parcomr{.13\textwidth}   
  }
\qedhere 
\]
\end{proof}

The equations imposed by $0$-junctions let us simplify and redraw bond graphs. This Frobenius monoid perfectly captures the ways in which we can do so.  Next we show that the way in which these two Frobenius monoids interact corresponds nicely to the way in which $1$- and $0$-junctions interect. However, we shall see some issues arise, which force us to consider a second approach.


\section{Bond graphs as corelations}
We have found two Frobenius monoid structures on the object $2 \in \Fin\Corel$, one related to $1$-junctions and one related to $0$-junctions.  Surpisingly,  the two structures $(2,m_2,i_2,d_2,e_2)$ and $(2, \mu_2, \iota_2, \delta_2, \epsilon_2)$ combine to form a pair of ``weak bimonoids."  These were introduced by Pastro and Street in their work on quantum categories \cite{PS}.  

\begin{definition} A \define{weak bimonoid} in a braided monoidal bicategory $\C$ is a pair consisting of a  monoid $(X,\mu, \iota)$ in $\C$ together with a comonoid $(X,\delta,\epsilon)$ in $\C$ obeying:

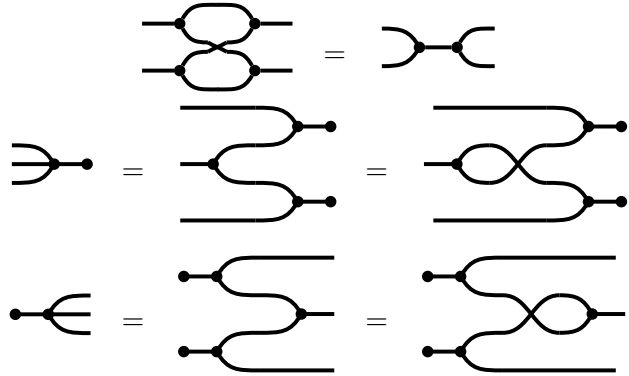
\[
  \xymatrixrowsep{1pt}
  \xymatrixcolsep{5pt}
  \xymatrix{
    \bi{.13\textwidth}  &= & \frobx{.11\textwidth} 
  }
\]

\vspace{-3ex}

\[
  \xymatrixrowsep{1pt}
  \xymatrixcolsep{5pt}
  \xymatrix{
    \weakbil{.07\textwidth}  &= & \weakbim{.13\textwidth} &= & \weakbir{.16\textwidth}  \\
    \weakcobil{.07\textwidth}  &= & \weakcobim{.14\textwidth} &= & \weakcobir{.16\textwidth}
  }
\]

\end{definition}

This differs from a \define{bimonoid}, where the second and third sets of equations are replaced by:
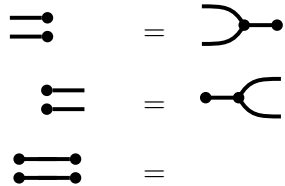
\[
  \xymatrixrowsep{.1pt}
  \xymatrixcolsep{5pt}
  \xymatrix{
     \parcounit{.06\textwidth}  &= & \multunit{.07\textwidth}  \\
     \parunit{.06\textwidth}   &=  & \comultcounit{.07\textwidth} \\
	\parextra{.09\textwidth} & = &
}
\]
It is a fun exercise to show that every bimonoid is a weak bimonoid using string diagrams. Frobenius monoids and bimonoids are the more commonly studied structures that combine monoids and comonoids, while weak bimonoids are much newer. Pastro and Street also proved the following fact, which we use to show how our monoids interact.

\begin{theorem}[Pastro, Street, \cite{PS}]\label{thm:RS}
If $X$ is a special commutative Frobenius monoid in a braided monoidal category then $X\otimes X$, equipped with the following morphisms, is a weak bimonoid:
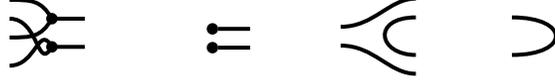
\[
  \xymatrixrowsep{1pt}
  \xymatrixcolsep{25pt}
  \xymatrix{
 	\weirdmult{.06\textwidth}  &  \parunit{.06\textwidth}    &  \monadcomult{.07\textwidth}   &\cuptwo{.07\textwidth}  \\
  }
\]
\end{theorem}

Since the Frobenius monoid $(1,m,i,d,e)$ is also commutative, we can replace the multiplication morphism with the following:
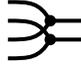
\[
  \xymatrixrowsep{1pt}
  \xymatrixcolsep{25pt}
  \xymatrix{
 	\parmult{.07\textwidth} \\
  }
\]
\vspace{-3ex}

\noindent and the theorem still holds. This gives us a weak bimonoid, but since  $(1,m,i,d,e)$ is an extraspecial commutative Frobenius monoid our weak bimonoid will obey some additional equations.  In fact,  because $\Fin\Corel$ is dagger compact there is a second weak bimonoid.

\begin{theorem}\label{thm:weakbimonoid}
$(2,\mu_2,\iota_2, d_2,e_2)$ and $(2,m_2,i_2,\delta_2,\epsilon_2)$ are weak bimonoids. Additionally:
\begin{itemize}
\item  the extra law for both weak bimonoids holds: $\epsilon_2\circ i_2 =\mathrm{id}_0= e_2 \circ \iota_2$
\item  $m_2\circ \delta_2 = \mu_2 \circ d_2$
\item $(m_2\circ \delta_2)^2 = m_2 \circ \delta_2$.
\end{itemize}
\end{theorem}
\begin{proof}
 We first prove that $(2,\mu_2,\iota_2, d_2,e_2)$ is a weak bimonoid. Note that proving $(2,m_2,i_2,\delta_2,\epsilon_2)$ is a weak bimonoid amounts to the same thing. Our diagrammatic proofs are similar to those used by Pastro and Street \cite{PS}.

The first weak bimonoid law to show is: $$(\mu_2 + \mu_2)\circ (\mathrm{id}_2 + B_{2,2} + \mathrm{id}_2) \circ (d_2 + d_2) = d_2\circ \mu_2. $$ We begin by using the naturality of the braiding; then we use the cocommutativity of the comonoid $(1,d,e)$.  Following this we use the Frobenius and special laws for $(1,m,i,d,e)$:

\vspace{-1ex}

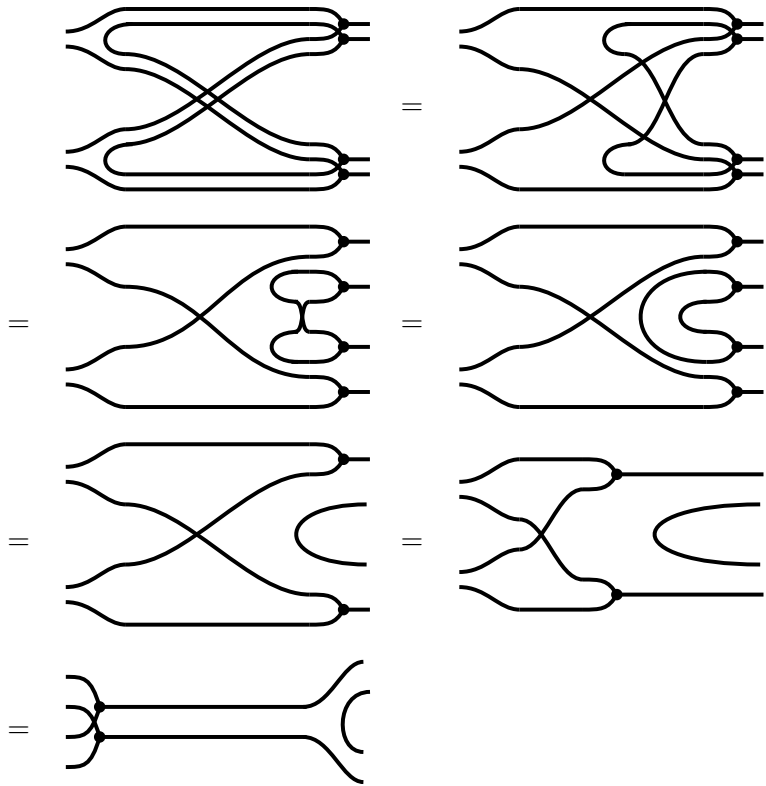
\[
  \xymatrixrowsep{1pt}
  \xymatrixcolsep{5pt}
  \xymatrix{
  &   & \bigbimonoidone{.35\textwidth}  &= & \bigbimonoidtwo{.35\textwidth}  \\
     &= & \bigbimonoidthree{.35\textwidth}  &= &   \bigbimonoidfour{.35\textwidth}\\
	&= &     \bigbimonoidfive{.35\textwidth} &= &  \bigbimonoidsix{.35\textwidth} \\
 	&=& \bigbimonoidseven{.35\textwidth}
  }
\]


%
%
%

\vspace{-1ex}

\noindent Next we have the other weak bimonoid laws:
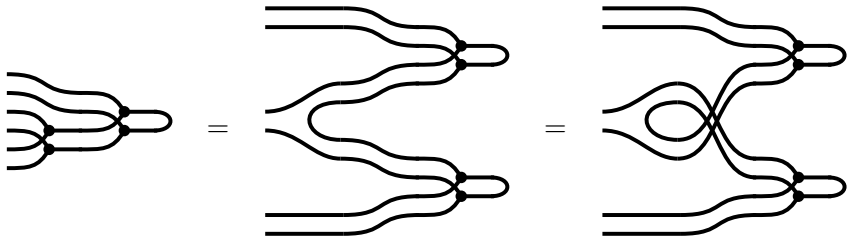
\[
  \xymatrixrowsep{1pt}
  \xymatrixcolsep{5pt}
  \xymatrix{
    \weakbimonoidtwol{.12\textwidth}  &= & \weakbimonoidtwom{.12\textwidth} &= & \weakbimonoidtwor{.12\textwidth}   \\
  }
\]
\noindent The easiest way to see that these are true is to note that these are both corelations from $6$ inputs to $0$ outputs where all of the inputs are connected to each other. Thus they are both the partition with all inputs in the same part, and no outputs, so they are the same corelation. 

The last weak bimonoid law  is the following:
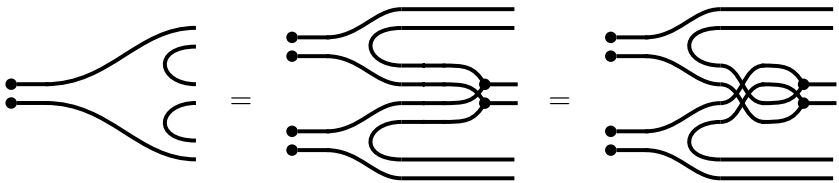
\[
  \xymatrixrowsep{1pt}
  \xymatrixcolsep{5pt}
  \xymatrix{
    \weakbimonoidcotwol{.1\textwidth}  &= &\weakbimonoidcotwom{.1\textwidth}  &= &\weakbimonoidcotwor{.12\textwidth} \\
  }
\]
\noindent This comes from passing the two inner unit morphisms through the braiding and then using the unit laws to delete them. Then we merely have two caps in between two units as we want. This shows that $(2,\mu_2,\iota_2, d_2,e_2)$ is a weak bimonoid.

 We prove only one of the extra laws, since the other is similar:
\[
  \xymatrixrowsep{1pt}
  \xymatrixcolsep{5pt}
  \xymatrix{
   \bimonoidextral{.07\textwidth} &= &  \bimonoidextram{.13\textwidth} \enspace  &= & \extral{.12\textwidth} 	\phantom{hf}=  \\
}
\]

\noindent We also have $m_2\circ \delta_2 =d\circ m= \mu_2 \circ d_2$:
\[
  \xymatrixrowsep{1pt}
  \xymatrixcolsep{5pt}
  \xymatrix{
   \strangelawl{.07\textwidth} &= &  \frobx{.1\textwidth} &= &  \strangelawr{.07\textwidth} \\
}
\]
\noindent One way to see this is that these all describe the corelation from $2$ to $2$ where all of the inputs and outputs are in the same part. We may also understand this by noticing that commutativity allows the inner ``loop" to be drawn as a cup or cap.
\noindent Finally, $(m_2\circ \delta_2)^2 = m_2 \circ \delta_2$, comes from $d\circ m$ being idempotent, since $d\circ m \circ d\circ m = d\circ id_2\circ m = d\circ m$. It follows that $ (\mu_2 \circ d_2)^2 =  \mu_2 \circ d_2$.
\end{proof}

The weak bimonoid laws encapsulate properties that are known to hold for interacting $1$- and $0$-junctions, including their unary versions. Unfortunately the law: $$m_2\circ \delta_2 = \mu_2 \circ d_2$$ is not a property that holds for bond graphs. It would indicate that a $0$-junction followed by a $1$-junction is equivalent to the composite in the other order, but this is not true for bond graphs.

However, it is still interesting to consider the following category.

\begin{definition}
Let $\define{\Fin\Corel^{\circ}}$ be the symmetric monoidal subcategory of $\Fin\Corel$ whose objects are $0,2,4,...$ and whose morphisms are generated by composing and tensoring the morphisms:

\[
  \xymatrixrowsep{1pt}
  \xymatrixcolsep{25pt}
  \xymatrix{
 	\monadmult{.07\textwidth}  &  \captwo{.07\textwidth}    &  \parcomult{.06\textwidth}   & \parcounit{.06\textwidth}  \\
 	\parmult{.06\textwidth}  &  \parunit{.06\textwidth}    &  \monadcomult{.07\textwidth}   & \cuptwo{.07\textwidth}  
  }
\] 
i.e.\ the morphisms $\{m_2,i_2,\delta_2,\epsilon_2,\mu_2,\iota_2, d_2,e_2\}$ in $\Fin\Corel.$ 
\end{definition}

 Note that $\Fin\Corel^{\circ}$ is equivalent to a prop since the objects are generated by $2$.

\begin{conjecture}\label{con:presentation}
The prop $\Fin\Corel^{\circ}$ can be presented with $8$ generators corresponding to $\{m_2,i_2,\delta_2,\epsilon_2,\mu_2,\iota_2, d_2,e_2\}$ and the equations shown in Theorem \ref{thm:series}, Theorem \ref{thm:parallel}, and Theorem \ref{thm:weakbimonoid}.
\end{conjecture}

By definition, given a symmetric monoidal functor $F\maps \Fin\Corel^{\circ} \to \C$, the object $F(2)$ together with the associated morphisms obey the required equations. However, the converse is once again much more difficult to show. 

 In Proposition \ref{prop:blackbox} we shall see how the functor $K\maps \Fin\Corel \to \Lag\Rel_k$ acts on $\Fin\Corel^{\circ}$. For now we look back at some Frobenius monoids in $\Lag\Rel_k$ that arise from the equations imposed by $1$- and $0$-junctions. These Frobenius monoids also come close to modelling bond graphs, but once again we get too many relations.


\section{Bond graphs as Lagrangian relations}

 As we saw in Chapter \ref{chap:Functors_for_Circuits}, any circuit can be associated with a Lagrangian subspace of possible potential and current pairs, which is then called the ``behavior" of the circuit. Similarly, a bond graph can be associated with a space of possible effort and flow pairs, which is called the behavior of the bond graph. These too are Lagrangian subspaces, and they come from the equations imposed by $1$- and $0$-junctions. Best of all, the Lagrangian subspaces associated to $1$- and $0$-junctions can all be written in terms of the \emph{additive} and \emph{duplicative} structures on the vector space $k$ defined in Chapter \ref{chap:Chapter2}.

The $0$-junction with two inputs and one output determines the following Lagrangian subspace: $$ \Delta^{\dagger} \oplus + = \{(E_1,\ldots, F_3): \;  E_1 = E_2 = E_3, F_1 + F_2 = F_3,\} \maps k^4\asrelto k^2. $$ Here the efforts are coduplicated while the flows are added together. Recall that in Section \ref{sec:linear_relations} we used string diagrams to depict the additive and duplicative Frobenius structures in $\Fin\Rel_k$. Using these we draw the Lagrangian subspace as follows: 
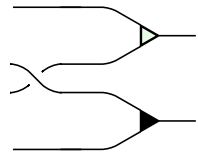
\[
  \xymatrixrowsep{1pt}
  \xymatrixcolsep{5pt}
  \xymatrix{
\SigMultpic{.20\textwidth}
  }
\]
By switching the inputs and outputs we also have a Lagrangian subspace associated to a $0$-junction with one input and two outputs: $$ \Delta \oplus +^{\dagger}= \{(E_1,\ldots, F_3): \;  E_1 = E_2 = E_3,  F_1 = F_2 + F_3\} \maps k^2\asrelto k^4. $$ Again we may draw such a Lagrangian subspace with our string diagrams:

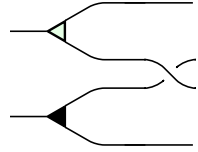
\[
  \xymatrixrowsep{1pt}
  \xymatrixcolsep{5pt}
  \xymatrix{
\SigCoMultpic{.20\textwidth}
  }
\]

In a similar way, a $1$-junction with two inputs and one outputs  has corresponding Lagrangian subspace: $$ + \oplus \Delta^{\dagger}=\{(E_1,\ldots, F_3): \; E_1 + E_2 = E_3, F_1 = F_2 = F_3\} \maps k^4\asrelto k^2 $$
so that now efforts are added and flows are coduplicated. This is drawn as:
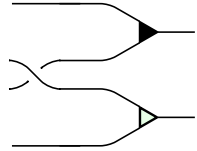
\[
  \xymatrixrowsep{1pt}
  \xymatrixcolsep{5pt}
  \xymatrix{
\SigOtherMultpic{.20\textwidth}
  }
\]
\noindent Similarly, the $1$-junction with one input and two outputs has corresponding Lagrangian subspace:
 $$+^{\dagger} \oplus \Delta=\{(E_1,\ldots, F_3): \; E_1 = E_2 + E_3, F_1 = F_2 = F_3\} \maps k^2\asrelto k^4 $$
which may be drawn as:
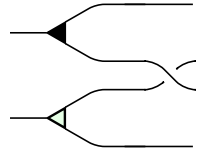
\[
  \xymatrixrowsep{1pt}
  \xymatrixcolsep{5pt}
  \xymatrix{
\SigOtherCoMultpic{.20\textwidth}
  }
\]

Note that  $(k\oplus k,\Delta^{\dagger} \oplus +, !^{\dagger} \oplus 0,  \Delta \oplus +^{\dagger}, ! \oplus 0^{\dagger})$ was the extraspecial commutative Frobenius monoid we used to define the black-boxing functor $\blacksquare \maps \Circ_L \to \Lag\Rel_k$ in Theorem \ref{thm:black-boxing_1}. Since the multiplication and comultiplication of this Frobenius monoid are the two Lagrangian subspaces associated to the two types of $0$-junctions, a natural attempt at defining unary $0$-junctions is using the unit and counit of this Frobenius monoid.

We draw the unit: $$!^{\dagger} \oplus 0 = \{(E, F) : \; F= 0\} \maps \{0\} \asrelto k^2 $$
\noindent as:

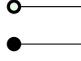
\[
  \xymatrixrowsep{1pt}
  \xymatrixcolsep{5pt}
  \xymatrix{
\SigUnitpic{.10\textwidth}
  }
\]

%

\noindent We draw the counit: $$! \oplus 0^{\dagger}= \{(E, F) : \; F= 0\} \maps k^2\asrelto \{0\} $$
\noindent as:

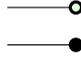
\[
  \xymatrixrowsep{1pt}
  \xymatrixcolsep{5pt}
  \xymatrix{
\SigCoUnitpic{.10\textwidth}
  }
\]

By tensoring the additive structure on the left with the duplicative structure on the right we construct another extraspecial commutative Frobenius monoid $(k\oplus k,+ \oplus \Delta^{\dagger},  0 \oplus !^{\dagger}, +^{\dagger} \oplus \Delta,  0^{\dagger} \oplus !)$. Here the multiplication and comultiplication are the Lagrangian subspaces associated to the two types of $1$-junctions. Thus it is natural to try defining unary $1$-junctions using the unit and counit for this Frobenius monoid. 

%

The unit: $$0 \; \oplus \; !^{\dagger} = \{(E, F) : \; E= 0\} \maps \{0\} \asrelto k^2 $$
\noindent is drawn as:

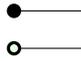
\[
  \xymatrixrowsep{1pt}
  \xymatrixcolsep{5pt}
  \xymatrix{
\SigOtherUnitpic{.10\textwidth}
  }
\]

%

\noindent The counit: $$0^{\dagger} \; \oplus \; ! = \{(E, F) : \; E= 0\} \maps k^2\asrelto \{0\} $$
\noindent is drawn as:

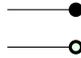
\[
  \xymatrixrowsep{1pt}
  \xymatrixcolsep{5pt}
  \xymatrix{
\SigOtherCoUnitpic{.10\textwidth}
  }
\]

%

Since the multiplications and comultiplications for the two Frobenius monoids come from the equations that define $0$- and $1$-junctions it is no surprise that they correspond to these junctions. However, when the morphisms from the two Frobenius monoids interact something unexepcted occurs. We get equations that are too strong. 

\begin{theorem} 
\label{thm:lagrel}
We have two extraspecial commutative Frobenius monoids
\begin{itemize}
 \item $(k\oplus k,+ \oplus \Delta^{\dagger}, 0 \oplus !^{\dagger}, +^{\dagger} \oplus \Delta,  0^{\dagger} \oplus !)$ 
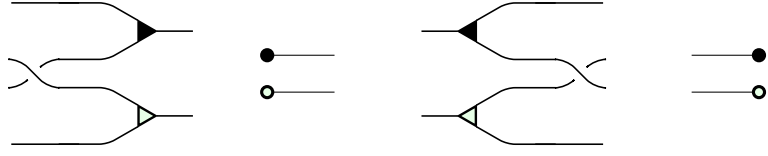
\[
  \xymatrixrowsep{1pt}
  \xymatrixcolsep{5pt}
  \xymatrix{
\SigOtherMultpic{.20\textwidth} & \SigOtherUnitpic{.10\textwidth} & \SigOtherCoMultpic{.20\textwidth} & \SigOtherCoUnitpic{.10\textwidth}
  }
\]
 \item $(k\oplus k,\Delta^{\dagger} \oplus +, !^{\dagger} \oplus 0, \Delta \oplus +^{\dagger}, ! \oplus 0^{\dagger})$ 
\[
  \xymatrixrowsep{1pt}
  \xymatrixcolsep{5pt}
  
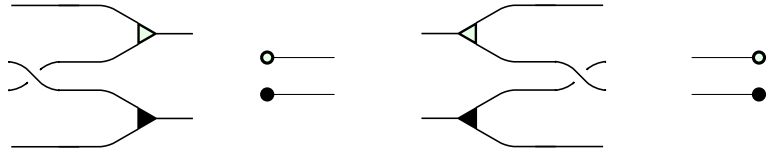
\xymatrix{
\SigMultpic{.20\textwidth} & \SigUnitpic{.10\textwidth} & \SigCoMultpic{.20\textwidth} & \SigCoUnitpic{.10\textwidth} \\
  }
\]
\end{itemize}
\noindent which come together as two bimonoids
\begin{itemize}
\item $(k\oplus k,\Delta^{\dagger} \oplus +, !^{\dagger} \oplus 0, +^{\dagger} \oplus \Delta, 0^{\dagger} \oplus !)$

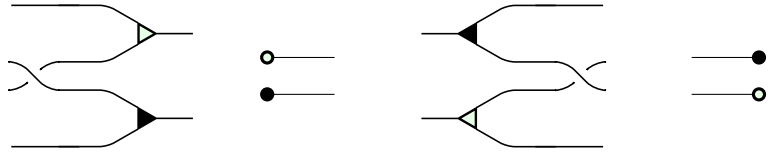
\[
  \xymatrixrowsep{1pt}
  \xymatrixcolsep{5pt}
  \xymatrix{
\SigMultpic{.20\textwidth} & \SigUnitpic{.10\textwidth} & \SigOtherCoMultpic{.20\textwidth} & \SigOtherCoUnitpic{.10\textwidth} \\
  }
\]
\item $(k\oplus k,+ \oplus \Delta^{\dagger}, 0 \oplus !^{\dagger}, \Delta \oplus +^{\dagger}, ! \oplus 0^{\dagger}).$
\[
  \xymatrixrowsep{1pt}
  \xymatrixcolsep{5pt}
  
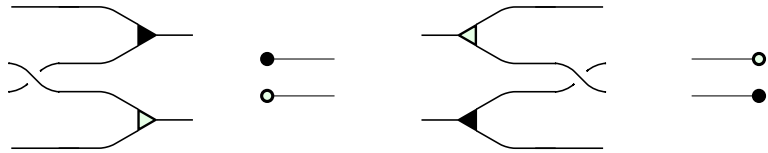
\xymatrix{
\SigOtherMultpic{.20\textwidth} & \SigOtherUnitpic{.10\textwidth} & \SigCoMultpic{.20\textwidth} & \SigCoUnitpic{.10\textwidth} \\
  }
\]
\end{itemize}
\noindent Additionally: $$(\Delta^{\dagger} \oplus +)\circ (+^{\dagger} \oplus \Delta) = ((+ \oplus \Delta^{\dagger})\circ (\Delta \oplus +^{\dagger}))^{-1}$$
\end{theorem}
\begin{proof}
These facts follow from relations shown by Bonchi, Soboci\'nski, and Zanasi \cite{Be}.
\end{proof}

We define another category using the above ideas.

\begin{definition}
Let $\Lag\Rel_k^{\circ}$ be the subcategory of $\Lag\Rel_k$ generated by the $8$ morphisms:
$$(+ \oplus \Delta^{\dagger},  0 \oplus !^{\dagger}, +^{\dagger} \oplus \Delta,  0^{\dagger} \oplus !,\Delta^{\dagger} \oplus +, !^{\dagger} \oplus 0, \Delta \oplus +^{\dagger},  ! \oplus 0^{\dagger})$$
which are drawn as:

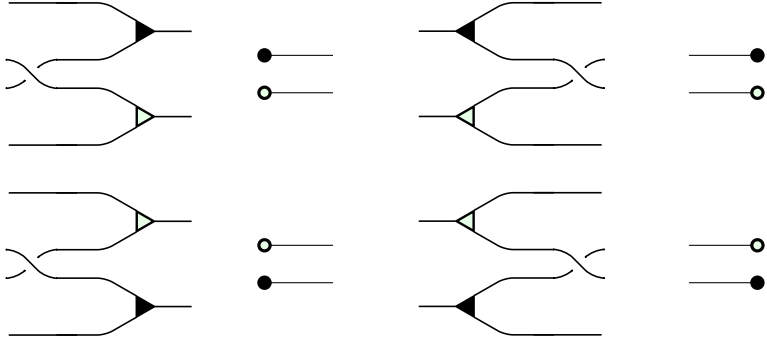
\[
  \xymatrixrowsep{1pt}
  \xymatrixcolsep{5pt}
  \xymatrix{
\SigOtherMultpic{.20\textwidth} & \SigOtherUnitpic{.10\textwidth} & \SigOtherCoMultpic{.20\textwidth} & \SigOtherCoUnitpic{.10\textwidth}\\
\SigMultpic{.20\textwidth} & \SigUnitpic{.10\textwidth} & \SigCoMultpic{.20\textwidth} & \SigCoUnitpic{.10\textwidth} \\
  }
\]

 Note that $\Lag\Rel_k^{\circ}$ is equivalent to a prop.  Henceforth we refer to this prop as $\Lag\Rel^{\circ}_k$.
\end{definition}

\begin{conjecture}
We conjecture that $\Lag\Rel_k^{\circ}$ is presented by generators and equations corresponding to Theorem \ref{thm:lagrel}.
\end{conjecture}

To summarize the two approaches we say that $\Fin\Corel^{\circ}$ correctly captures the interaction between $1$- and $0$-junctions, while $\Lag\Rel_k^{\circ}$ correctly captures the behavior of each junction separately.


\section{The prop of bond graphs}
\label{sec:functors}

Using the framework of props we can describe a category having only the equations that appear in both $\Fin\Corel^{\circ}$ and $\Lag\Rel_k^{\circ}$. Since these two props have a correspondence between generators as well as many overlapping equations, this is easy enough to do.
\begin{definition}
\label{def:BondGraph}
Let $\BondGraph$ be the prop generated by the $8$ morphisms:
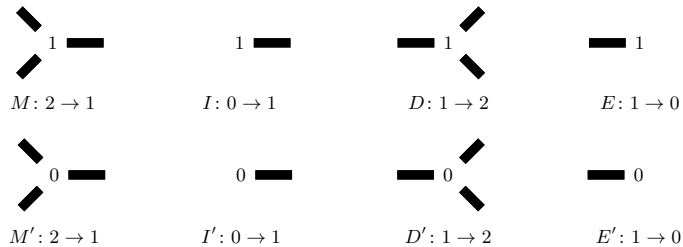
\[
  \xymatrixrowsep{1pt}
  \xymatrixcolsep{25pt}
  \xymatrix@1{

\\
}
\]

such that:
\begin{itemize}
\item $(1,M,I,D,E)$ is an extraspecial symmetric Frobenius monoid.
\end{itemize}

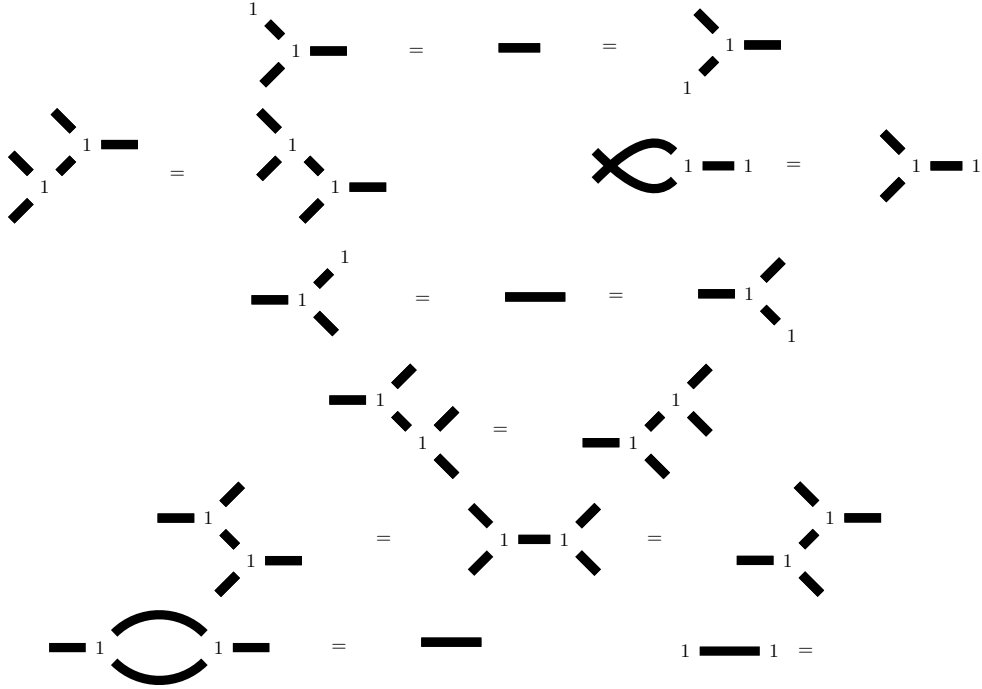
\vspace{-3ex}


\[
  \xymatrixrowsep{1pt}
  \xymatrixcolsep{15pt}
  \xymatrix@1{
\begin{aligned}
%
 
\end{aligned}
  }
\]

\begin{itemize}
\item $(1,M',I',D', E')$ is an extraspecial symmetric Frobenius monoid.
\end{itemize}

\vspace{-3ex}

\[
  \xymatrixrowsep{5pt}
  \xymatrixcolsep{15pt}
  \xymatrix@1{
\begin{aligned}
%
 
  }
\]

\noindent is idempotent.
\end{definition}
Further, we equip $\BondGraph$ with a dagger structure where if $f\maps m\to n$ is a morphism then $f^\dagger\maps n\to m$ is the reflected version of $f$. This is the unique way to make this prop into a symmetric monoidal dagger category such that  \[   M^\dagger = D, \quad I^\dagger = E, \quad M'^\dagger = D', \quad I'^\dagger = E'.\] 

\noindent Composing the inclusion  $i \maps \Fin\Corel^{\circ} \to \Fin\Corel$ with $K\maps \Fin\Corel \to \Lag\Rel_k$
we obtain a symmetric monoidal dagger functor, whose action on generators we can describe explicitly as follows:

\begin{proposition}\label{prop:blackbox}
$K i \maps \Fin\Corel^{\circ} \to \Lag\Rel_k $ is a symmetric monoidal dagger functor that is determined by the following:
\[\arraycolsep=1pt\def\arraystretch{.75}
\begin{array}{lcl}
  Ki(m_2) &=& \{(\phi_1,\ldots\thinspace , I_6): \; \phi_1= \phi_5, I_1 = I_5,  \phi_4 = \phi_6, I_4 = I_6, \phi_2 = \phi_3, I_2 + I_3 = 0\}  \\
 & & \maps k^8\asrelto k^4 \\ \vspace{.5ex}
  Ki(i_2) &=& \{(\phi_1, I_1, \phi_2,I_2) : \; I_1 +I_2 = 0, \phi_1=\phi_2\}   \\
 & & \maps \{0\}\asrelto k^4 \\ \vspace{.5ex}
  Ki(\mu_2) &=& \{(\phi_1,\ldots\thinspace , I_6): \; \phi_1= \phi_3 =\phi_5,\phi_2= \phi_4 =\phi_6, I_1+I_3=I_5, I_2+I_4=I_6\}  \\
 & & \maps k^8\asrelto k^4 \\ \vspace{.5ex}
  Ki(\iota_2) &=& \{(\phi_1, I_1,\phi_2,I_2) : \; I_1 = 0,I_2 =0\}   \\
 & & \maps \{0\}\asrelto k^4 
\end{array}
\]
In terms of string diagrams we draw:
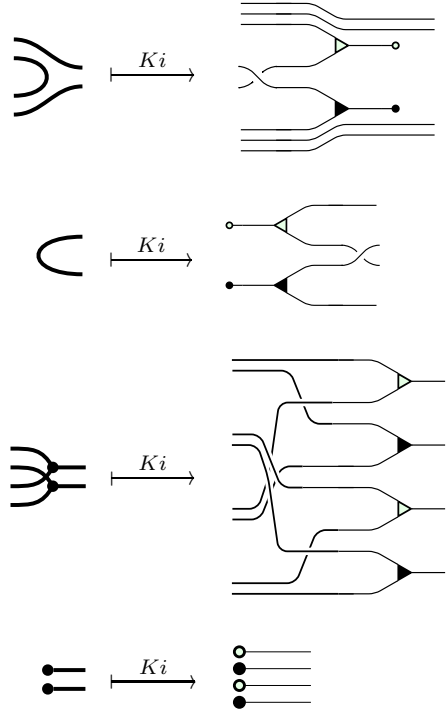
\[
  \xymatrixrowsep{5pt}
  \xymatrixcolsep{30pt}
  \xymatrix{
  \monadmult{.1\textwidth}  \;\;  \ar@{|->}@<.75ex>[r]^-{Ki} &  \;\; \BondtoSigMultpic{.2\textwidth}   \\
    \captwo{.1\textwidth}  \;\;  \ar@{|->}@<.25ex>[r]^-{Ki} &   \;\;\BondtoSigUnitpic{.15\textwidth} \hspace{5ex}  \\
  \parmult{.1\textwidth}  \;\;   \ar@{|->}@<.5ex>[r]^-{Ki} &  \;\; \BondtoSigOtherMultpic{.2\textwidth} \\
   \parunit{.1\textwidth}  \;\;   \ar@{|->}@<.25ex>[r]^-{Ki}&   \;\; \BondtoSigOtherUnitpic{.1\textwidth} \hspace{9ex} \\
}
\]

\end{proposition}

\begin{proof}
$K i$ is a symmetric monoidal dagger functor since $K$  and $i$ are. It is tedious, but not difficult, to check that $K i$ does act on the morphisms $m_2,i_2,\mu_2,\iota_2$ in the above way. We just write them in terms of $m,d,i,e$ and then compose the results in $\Lag\Rel_k$. Since $Ki$ is a dagger functor the following is a consequence:
\[\arraycolsep=1pt\def\arraystretch{.75}
\begin{array}{lcl}
  Ki(d_2) &=& \{(\phi_1,\ldots\thinspace , I_6): \; \phi_1= \phi_3, I_1 = I_3,  \phi_2 = \phi_6, I_2 = I_6, \phi_4 = \phi_4, I_4 + I_5 = 0\} \\
 & & \maps k^4\asrelto k^8 \\ \vspace{.5ex}
  Ki(e_2) &=& \{(\phi_1, I_1, \phi_2,I_2) :\; I_1 +I_2= 0, \phi_1 = \phi_2\} \\
 & & \maps k^4\asrelto \{0\} \\ \vspace{.5ex}
  Ki(\delta_2) &=& \{(\phi_1,\ldots\thinspace , I_6): \; \phi_1= \phi_3 =\phi_5,\phi_2= \phi_4 =\phi_5, I_1=I_3+ I_5, I_2=I_4+ I_6\} \\
 & & \maps k^4\asrelto k^8 \\ \vspace{.5ex}
  Ki(\epsilon_2) &=& \{(\phi_1, I_1, \phi_2,I_2) :\; I_1= 0, I_2 = 0\} \\
 & & \maps k^4\asrelto \{0\}. 
\end{array}
\]
In terms of string diagrams we draw:
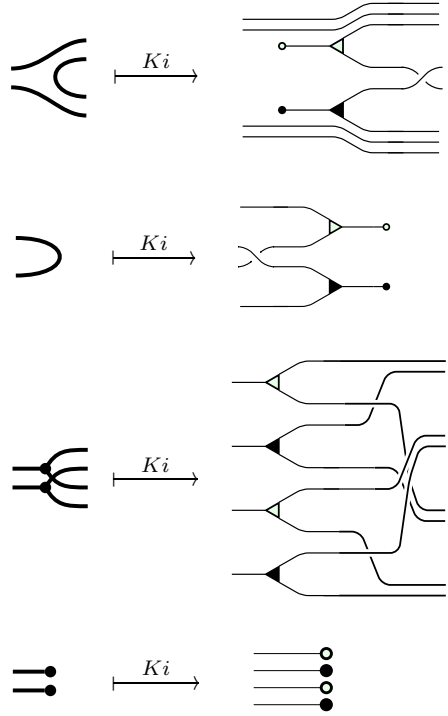
\[
  \xymatrixrowsep{5pt}
  \xymatrixcolsep{30pt}
  \xymatrix{
   \monadcomult{.1\textwidth}  \;\; \ar@{|->}@<.75ex>[r]^-{Ki} &  \;\; \BondtoSigCoMultpic{.2\textwidth}   \\
   \cuptwo{.1\textwidth}  \;\; \ar@{|->}@<.5ex>[r]^-{Ki} &  \;\; \BondtoSigCoUnitpic{.15\textwidth} \hspace{5ex}  \\
    \parcomult{.1\textwidth}  \;\; \ar@{|->}@<.5ex>[r]^-{Ki} &  \;\; \BondtoSigOtherCoMultpic{.2\textwidth}\\
    \parcounit{.1\textwidth}  \;\; \ar@{|->}@<.25ex>[r]^-{Ki} &  \;\; \BondtoSigOtherCoUnitpic{.1\textwidth} \hspace{9ex}
}
\]

\noindent The only thing to show is that the symmetric monoidal functor $Ki$ is uniquely determined by its action on these $8$ morphisms. This follows from the fact that these $8$ morphisms generate $\Fin\Corel^{\circ}$ as a prop.
\end{proof}

So far we have the following diagram of symmetric monoidal dagger functors between props:
\[
    \xymatrix@C-3pt{
      & \Lag\Rel_k^{\circ} \phantom{ |} \ar[r]^-{i'} & \Lag\Rel_k \\
 	& \Fin\Corel^{\circ} \ar[r]^-{i}  & \Fin\Corel \ar[u]_-{K}
      }
\]
By using $\BondGraph$ we can extend this to a diagram that commutes up to a natural transformation.

\begin{proposition}\label{prop:functors}
There exist unique morphisms of dagger props $G\maps \BondGraph \to \Fin\Corel^{\circ}$ and $F\maps \BondGraph \to \Lag\Rel_k^{\circ}$ defined by:
\begin{itemize}
\item $G(M) = m_2$
\end{itemize}
\vspace{-3ex}
\[
  \xymatrixrowsep{5pt}
  \xymatrix@1{
 *+[uuuuu]{
\begin{tikzpicture}[circuit ee IEC, set resistor graphic=var resistor IEC
      graphic, scale=0.8, every node/.append style={transform shape}]
[
	node distance=1.5cm,
	mynewelement/.style={
		color=blue!50!black!75,
		thick
	},
	mybondmodpoint/.style={
	rectangle,
	minimum size=3mm,
	very thick,
	draw=red!50!black!50, 
	outer sep=2pt
	}
]		
\node(G) at +(-1,0) {\large $G\maps$};
	\node(J11) {$\mathrm{1}$};
	\node (R2) [ below left of=J11] {}
	edge  [line width=3.5pt]   node [below]{} (J11)
        edge  [line width=3.5pt]   node [above]{} (J11);
	\node (R1) [ above left of=J11] {}
	edge [line width=3.5pt]    node [below]{} (J11)
        edge  [line width=3.5pt]   node [above]{} (J11);
	\node (C1) [right of=J11] {}
	edge [line width=3.5pt]    node [right]{} (J11)
        edge [line width=3.5pt]    node [left]{} (J11);
		\node [style=none] (www) at (0, -1) {};
      \end{tikzpicture} 
}
 \ar@{|->}@<.25ex>[r] & \quad \monadmult{.1\textwidth}
  }
\]
\vspace{-4.5ex}
\begin{itemize}
\item $G(I) =  i_2$
\end{itemize}
\vspace{-3ex}
\[
  \xymatrixrowsep{5pt}
  \xymatrix@1{
 *+[u]{
\begin{tikzpicture}[circuit ee IEC, set resistor graphic=var resistor IEC
      graphic, scale=0.8, every node/.append style={transform shape}]
[
	node distance=1.5cm,
	mynewelement/.style={
		color=blue!50!black!75,
		thick
	},
	mybondmodpoint/.style={
	rectangle,
	minimum size=3mm,
	very thick,
	draw=red!50!black!50, 
	outer sep=2pt
	}
]		
\node(G) at +(-1,0) {\large $G\maps$};
	\node (J11) {$\mathrm{1}$};
	\node (R1) [right of=J11] {}
	edge  [line width=3.5pt]   node [below]{} (J11)
        edge  [line width=3.5pt]   node [above]{} (J11);
	\node (D1) [ above left of=J11] {};
	\node (D2) [ below left of=J11] {};
		\node [style=none] (www) at (0, -1) {};
      \end{tikzpicture} 
	}
  \ar@{|->}@<.25ex>[r] & \quad \captwo{.1\textwidth}
  }
\]
\vspace{-4.5ex}
\begin{itemize}
\item $G(M') = \mu_2$
\end{itemize}
\vspace{-3ex}
\[
  \xymatrixrowsep{5pt}
  \xymatrix@1{
 *+[u]{
\begin{tikzpicture}[circuit ee IEC, set resistor graphic=var resistor IEC
      graphic, scale=0.8, every node/.append style={transform shape}]
[
	node distance=1.5cm,
	mynewelement/.style={
		color=blue!50!black!75,
		thick
	},
	mybondmodpoint/.style={
	rectangle,
	minimum size=3mm,
	very thick,
	draw=red!50!black!50, 
	outer sep=2pt
	}
]		
\node(G) at +(-1,0) {\large $G\maps$};
	\node(J11) {$\mathrm{0}$};
	\node (R2) [ below left of=J11] {}
	edge [line width=3.5pt]    node [below]{} (J11)
        edge   [line width=3.5pt]  node [above]{} (J11);
	\node (R1) [ above left of=J11] {}
	edge [line width=3.5pt]   node [below]{} (J11)
        edge  [line width=3.5pt]   node [above]{} (J11);
	\node (C1) [right of=J11] {}
	edge [line width=3.5pt]    node [right]{} (J11)
        edge  [line width=3.5pt]   node [left]{} (J11);
		\node [style=none] (www) at (0, -1) {};
      \end{tikzpicture} 
}
  \ar@{|->}@<.25ex>[r] &\quad \parmult{.1\textwidth}
  }
\]
\vspace{-4.5ex}
\begin{itemize}
\item $G(I') =   \iota_2$
\end{itemize}
\vspace{-3ex}
\[
  \xymatrixrowsep{5pt}
  \xymatrix@1{
 *+[u]{
\begin{tikzpicture}[circuit ee IEC, set resistor graphic=var resistor IEC
      graphic, scale=0.8, every node/.append style={transform shape}]
[
	node distance=1.5cm,
	mynewelement/.style={
		color=blue!50!black!75,
		thick
	},
	mybondmodpoint/.style={
	rectangle,
	minimum size=3mm,
	very thick,
	draw=red!50!black!50, 
	outer sep=2pt
	}
]		
\node(G) at +(-1,0) {\large $G\maps$};
	\node (J11) {$\mathrm{0}$};
	\node (R1) [right of=J11] {}
	edge [line width=3.5pt]    node [below]{} (J11)
        edge [line width=3.5pt]    node [above]{} (J11);
	\node (D1) [ above left of=J11] {};
	\node (D2) [ below left of=J11] {};
		\node [style=none] (www) at (0, -1) {};
      \end{tikzpicture} 
	}
  \ar@{|->}@<.25ex>[r] & \quad \parunit{.1\textwidth}
  }
\]
\vspace{-4.5ex}
\begin{itemize}
\item $F(M) = + \oplus \Delta^{\dagger}$
\end{itemize}
\vspace{-3ex}
\[
  \xymatrixrowsep{5pt}
  \xymatrix@1{
 *+[u]{
\hspace{12ex}
\begin{tikzpicture}[circuit ee IEC, set resistor graphic=var resistor IEC
      graphic, scale=0.8, every node/.append style={transform shape}]
[
	node distance=1.5cm,
	mynewelement/.style={
		color=blue!50!black!75,
		thick
	},
	mybondmodpoint/.style={
	rectangle,
	minimum size=3mm,
	very thick,
	draw=red!50!black!50, 
	outer sep=2pt
	}
]		
\node(F) at +(-1,0) {\large $F\maps$};
	\node(J11) {$\mathrm{1}$};
	\node (R2) [ below left of=J11] {}
	edge [line width=3.5pt]    node [below]{} (J11)
        edge  [line width=3.5pt]   node [above]{} (J11);
	\node (R1) [ above left of=J11] {}
	edge [line width=3.5pt]    node [below]{} (J11)
        edge [line width=3.5pt]    node [above]{} (J11);
	\node (C1) [right of=J11] {}
	edge  [line width=3.5pt]   node [right]{} (J11)
        edge [line width=3.5pt]    node [left]{} (J11);
		\node [style=none] (www) at (0, -1) {};
      \end{tikzpicture} 
}
  \ar@{|->}@<.25ex>[r] &  \;  \;  \SigOtherMultpic{.20\textwidth}
  }
\]
\vspace{-4.5ex}
\begin{itemize}
\item $F(I) =  0 \; \oplus \; !^{\dagger} $
\end{itemize}
\vspace{-3ex}
\[
  \xymatrixrowsep{5pt}
  \xymatrix@1{
 *+[u]{ \;  \;  \;  \;  
\begin{tikzpicture}[circuit ee IEC, set resistor graphic=var resistor IEC
      graphic, scale=0.8, every node/.append style={transform shape}]
[
	node distance=1.5cm,
	mynewelement/.style={
		color=blue!50!black!75,
		thick
	},
	mybondmodpoint/.style={
	rectangle,
	minimum size=3mm,
	very thick,
	draw=red!50!black!50, 
	outer sep=2pt
	}
]		
\node(F) at +(-1,0) {\large $F\maps$};
	\node (J11) {$\mathrm{1}$};
	\node (R1) [right of=J11] {}
	edge  [line width=3.5pt]   node [below]{} (J11)
        edge [line width=3.5pt]    node [above]{} (J11);
	\node (D1) [ above left of=J11] {};
	\node (D2) [ below left of=J11] {};
		\node [style=none] (www) at (0, -1) {};
      \end{tikzpicture} 
	}
  \ar@{|->}@<.25ex>[r] & \;  \; \SigOtherUnitpic{.10\textwidth}
  }
\]
\vspace{-4.5ex}
\begin{itemize}
\item $F(M') = \Delta^{\dagger} \oplus + $
\end{itemize}
\vspace{-3ex}
\[
  \xymatrixrowsep{5pt}
  \xymatrix@1{
 *+[u]{
\hspace{12ex}
\begin{tikzpicture}[circuit ee IEC, set resistor graphic=var resistor IEC
      graphic, scale=0.8, every node/.append style={transform shape}]
[
	node distance=1.5cm,
	mynewelement/.style={
		color=blue!50!black!75,
		thick
	},
	mybondmodpoint/.style={
	rectangle,
	minimum size=3mm,
	very thick,
	draw=red!50!black!50, 
	outer sep=2pt
	}
]		
\node(F) at +(-1,0) {\large $F\maps$};
	\node(J11) {$\mathrm{0}$};
	\node (R2) [ below left of=J11] {}
	edge  [line width=3.5pt]   node [below]{} (J11)
        edge [line width=3.5pt]   node [above]{} (J11);
	\node (R1) [ above left of=J11] {}
	edge [line width=3.5pt]    node [below]{} (J11)
        edge [line width=3.5pt]    node [above]{} (J11);
	\node (C1) [right of=J11] {}
	edge [line width=3.5pt]    node [right]{} (J11)
        edge [line width=3.5pt]   node [left]{} (J11);
		\node [style=none] (www) at (0, -1) {};
      \end{tikzpicture} 
}
  \ar@{|->}@<.25ex>[r] & \;  \;  \SigMultpic{.20\textwidth}
  }
\]
\vspace{-4.5ex}
\begin{itemize}
\item $F(I') =   !^{\dagger} \oplus 0$
\end{itemize}
\vspace{-3ex}
\[
  \xymatrixrowsep{5pt}
  \xymatrix@1{
 *+[u]{ \;  \;  \;  \;  
\begin{tikzpicture}[circuit ee IEC, set resistor graphic=var resistor IEC
      graphic, scale=0.8, every node/.append style={transform shape}]
[
	node distance=1.5cm,
	mynewelement/.style={
		color=blue!50!black!75,
		thick
	},
	mybondmodpoint/.style={
	rectangle,
	minimum size=3mm,
	very thick,
	draw=red!50!black!50, 
	outer sep=2pt
	}
]		
\node(F) at +(-1,0) {\large $F\maps$};
	\node (J11) {$\mathrm{0}$};
	\node (R1) [right of=J11] {}
	edge  [line width=3.5pt]   node [below]{} (J11)
        edge [line width=3.5pt]    node [above]{} (J11);
	\node (D1) [ above left of=J11] {};
	\node (D2) [ below left of=J11] {};
		\node [style=none] (www) at (0, -1) {};
      \end{tikzpicture} 
	}
 \ar@{|->}@<.25ex>[r] & \;  \; \SigUnitpic{.10\textwidth}
  }
\]
\end{proposition}
\label{prop:functor}
\vspace{-4.5ex}

\begin{proof}
We first prove uniqueness of $G$. Suppose that $G$ is a morphism of dagger props such that $G(M) = m_2$,  $G(I) =  i_2$,  $G(M') = \mu_2$, and $G(I') =   \iota_2$ as pictured above. Note that since $G$ is a dagger functor,  $G$ must do the following to the other $4$ generators:

\begin{itemize}
\item $G(D) = d_2$
\end{itemize}
\vspace{-3ex}
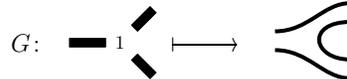
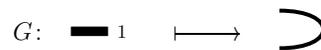
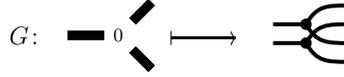
\[
  \xymatrixrowsep{5pt}
  \xymatrix@1{
 *+[u]{
\begin{tikzpicture}[circuit ee IEC, set resistor graphic=var resistor IEC
      graphic, scale=0.8, every node/.append style={transform shape}]
[
	node distance=1.5cm,
	mynewelement/.style={
		color=blue!50!black!75,
		thick
	},
	mybondmodpoint/.style={
	rectangle,
	minimum size=3mm,
	very thick,
	draw=red!50!black!50, 
	outer sep=2pt
	}
]		
\node(G) at +(-1.5,0) {\large $G\maps$};
	\node(J11) {$\mathrm{1}$};
	\node (R2) [ below right of=J11] {}
	edge  [line width=3.5pt]   node [below]{} (J11)
        edge  [line width=3.5pt]   node [above]{} (J11);
	\node (R1) [ above right of=J11] {}
	edge [line width=3.5pt]    node [below]{} (J11)
        edge  [line width=3.5pt]   node [above]{} (J11);
	\node (C1) [left of=J11] {}
	edge [line width=3.5pt]    node [right]{} (J11)
        edge [line width=3.5pt]    node [left]{} (J11);
		\node [style=none] (www) at (0, -1) {};
      \end{tikzpicture} 
}
 \ar@{|->}@<.25ex>[r] & \quad \monadcomult{.1\textwidth}
  }
\]
\vspace{-4.5ex}
\begin{itemize}
\item $G(E) =  e_2$
\end{itemize}
\vspace{-3ex}
\[
  \xymatrixrowsep{5pt}
  \xymatrix@1{
 *+[u]{
\begin{tikzpicture}[circuit ee IEC, set resistor graphic=var resistor IEC
      graphic, scale=0.8, every node/.append style={transform shape}]
[
	node distance=1.5cm,
	mynewelement/.style={
		color=blue!50!black!75,
		thick
	},
	mybondmodpoint/.style={
	rectangle,
	minimum size=3mm,
	very thick,
	draw=red!50!black!50, 
	outer sep=2pt
	}
]		
\node(G) at +(-1.5,0) {\large $G\maps$};
	\node (J11) {$\mathrm{1}$};
	\node (R1) [left of=J11] {}
	edge  [line width=3.5pt]   node [below]{} (J11)
        edge  [line width=3.5pt]   node [above]{} (J11);
	\node (D1) [ above right of=J11] {};
	\node (D2) [ below right of=J11] {};
		\node [style=none] (www) at (0, -1) {};
      \end{tikzpicture} 
	}
  \ar@{|->}@<.25ex>[r] & \quad \cuptwo{.1\textwidth}
  }
\]
\vspace{-4.5ex}
\begin{itemize}
\item $G(D') = \delta_2$
\end{itemize}
\vspace{-3ex}
\[
  \xymatrixrowsep{5pt}
  \xymatrix@1{
 *+[u]{
\begin{tikzpicture}[circuit ee IEC, set resistor graphic=var resistor IEC
      graphic, scale=0.8, every node/.append style={transform shape}]
[
	node distance=1.5cm,
	mynewelement/.style={
		color=blue!50!black!75,
		thick
	},
	mybondmodpoint/.style={
	rectangle,
	minimum size=3mm,
	very thick,
	draw=red!50!black!50, 
	outer sep=2pt
	}
]		
\node(G) at +(-1.5,0) {\large $G\maps$};
	\node(J11) {$\mathrm{0}$};
	\node (R2) [ below right of=J11] {}
	edge [line width=3.5pt]    node [below]{} (J11)
        edge   [line width=3.5pt]  node [above]{} (J11);
	\node (R1) [ above right of=J11] {}
	edge [line width=3.5pt]   node [below]{} (J11)
        edge  [line width=3.5pt]   node [above]{} (J11);
	\node (C1) [left of=J11] {}
	edge [line width=3.5pt]    node [right]{} (J11)
        edge  [line width=3.5pt]   node [left]{} (J11);
		\node [style=none] (www) at (0, -1) {};
      \end{tikzpicture} 
}
  \ar@{|->}@<.25ex>[r] &\quad \parcomult{.1\textwidth}
  }
\]
\vspace{-4.5ex}
\begin{itemize}
\item $G(E') = \epsilon_2$
\end{itemize}
\vspace{-3ex}
\[
  \xymatrixrowsep{5pt}
  \xymatrix@1{
 *+[u]{
\begin{tikzpicture}[circuit ee IEC, set resistor graphic=var resistor IEC
      graphic, scale=0.8, every node/.append style={transform shape}]
[
	node distance=1.5cm,
	mynewelement/.style={
		color=blue!50!black!75,
		thick
	},
	mybondmodpoint/.style={
	rectangle,
	minimum size=3mm,
	very thick,
	draw=red!50!black!50, 
	outer sep=2pt
	}
]		
\node(G) at +(-1.5,0) {\large $G\maps$};
	\node (J11) {$\mathrm{0}$};
	\node (R1) [left of=J11] {}
	edge [line width=3.5pt]    node [below]{} (J11)
        edge [line width=3.5pt]    node [above]{} (J11);
	\node (D1) [ above right of=J11] {};
	\node (D2) [ below right of=J11] {};
		\node [style=none] (www) at (0, -1) {};
      \end{tikzpicture} 
	}
 \ar@{|->}@<.25ex>[r] & \quad \parcounit{.1\textwidth}
  }
\]
\vspace{-3ex}

We know what $G$ does to all $8$ generators, so $G$ is the unique morphism of props acting in this way on those generators.  Similarly $F$ is the unique morphism of props such that $F(M) = + \oplus \Delta^{\dagger}$, $F(I) =   0 \oplus !^{\dagger}$, $F(M') = \Delta^{\dagger} \oplus +$, and $F(I') =  !^{\dagger} \oplus 0$. On the other generators, $F$ acts as follows:

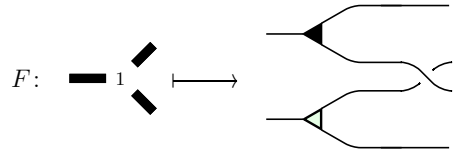
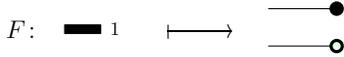
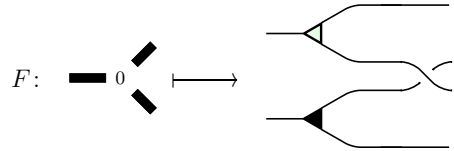
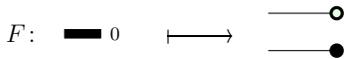
\begin{itemize}
\item  $F(D) = +^{\dagger} \oplus \Delta $
\end{itemize}
\vspace{-3ex}
\[
  \xymatrixrowsep{5pt}
  \xymatrix@1{
 *+[u]{ \hspace{10ex}
\begin{tikzpicture}[circuit ee IEC, set resistor graphic=var resistor IEC
      graphic, scale=0.8, every node/.append style={transform shape}]
[
	node distance=1.5cm,
	mynewelement/.style={
		color=blue!50!black!75,
		thick
	},
	mybondmodpoint/.style={
	rectangle,
	minimum size=3mm,
	very thick,
	draw=red!50!black!50, 
	outer sep=2pt
	}
]		
\node(F) at +(-1.5,0) {\large $F\maps$};
	\node(J11) {$\mathrm{1}$};
	\node (R2) [ below right of=J11] {}
	edge [line width=3.5pt]    node [below]{} (J11)
        edge  [line width=3.5pt]   node [above]{} (J11);
	\node (R1) [ above right of=J11] {}
	edge [line width=3.5pt]    node [below]{} (J11)
        edge [line width=3.5pt]    node [above]{} (J11);
	\node (C1) [left of=J11] {}
	edge  [line width=3.5pt]   node [right]{} (J11)
        edge [line width=3.5pt]    node [left]{} (J11);
		\node [style=none] (www) at (0, -1) {};
      \end{tikzpicture} 
}
  \ar@{|->}@<.25ex>[r] &   \SigOtherCoMultpic{.20\textwidth}
  }
\]
\vspace{-4.5ex}
\begin{itemize}
\item $F(E) =   ! \oplus 0^{\dagger}$
\end{itemize}
\vspace{-3ex}
\[
  \xymatrixrowsep{5pt}
  \xymatrix@1{
 *+[u]{
\begin{tikzpicture}[circuit ee IEC, set resistor graphic=var resistor IEC
      graphic, scale=0.8, every node/.append style={transform shape}]
[
	node distance=1.5cm,
	mynewelement/.style={
		color=blue!50!black!75,
		thick
	},
	mybondmodpoint/.style={
	rectangle,
	minimum size=3mm,
	very thick,
	draw=red!50!black!50, 
	outer sep=2pt
	}
]		
\node(F) at +(-1.5,0) {\large $F\maps$};
	\node (J11) {$\mathrm{1}$};
	\node (R1) [left of=J11] {}
	edge  [line width=3.5pt]   node [below]{} (J11)
        edge [line width=3.5pt]    node [above]{} (J11);
	\node (D1) [ above right of=J11] {};
	\node (D2) [ below right of=J11] {};
		\node [style=none] (www) at (0, -1) {};
      \end{tikzpicture} 
	}
  \ar@{|->}@<.25ex>[r] &  \SigOtherCoUnitpic{.10\textwidth}
  }
\]
\vspace{-4.5ex}
\begin{itemize}
\item $F(D') = \Delta \oplus +^{\dagger}$
\end{itemize}
\vspace{-3ex}
\[
  \xymatrixrowsep{5pt}
  \xymatrix@1{
 *+[u]{
\hspace{10ex}
\begin{tikzpicture}[circuit ee IEC, set resistor graphic=var resistor IEC
      graphic, scale=0.8, every node/.append style={transform shape}]
[
	node distance=1.5cm,
	mynewelement/.style={
		color=blue!50!black!75,
		thick
	},
	mybondmodpoint/.style={
	rectangle,
	minimum size=3mm,
	very thick,
	draw=red!50!black!50, 
	outer sep=2pt
	}
]		
\node(F) at +(-1.5,0) {\large $F\maps$};
	\node(J11) {$\mathrm{0}$};
	\node (R2) [ below right of=J11] {}
	edge  [line width=3.5pt]   node [below]{} (J11)
        edge [line width=3.5pt]   node [above]{} (J11);
	\node (R1) [ above right of=J11] {}
	edge [line width=3.5pt]    node [below]{} (J11)
        edge [line width=3.5pt]    node [above]{} (J11);
	\node (C1) [left of=J11] {}
	edge [line width=3.5pt]    node [right]{} (J11)
        edge [line width=3.5pt]   node [left]{} (J11);
		\node [style=none] (www) at (0, -1) {};
      \end{tikzpicture} 
}
 \ar@{|->}@<.25ex>[r] & \SigCoMultpic{.20\textwidth}
  }
\]
\vspace{-4.5ex}
\begin{itemize}
\item $F(E') =  ! \oplus 0^{\dagger}$
\end{itemize}
\vspace{-3ex}
\[
  \xymatrixrowsep{5pt}
  \xymatrix@1{
 *+[u]{
\begin{tikzpicture}[circuit ee IEC, set resistor graphic=var resistor IEC
      graphic, scale=0.8, every node/.append style={transform shape}]
[
	node distance=1.5cm,
	mynewelement/.style={
		color=blue!50!black!75,
		thick
	},
	mybondmodpoint/.style={
	rectangle,
	minimum size=3mm,
	very thick,
	draw=red!50!black!50, 
	outer sep=2pt
	}
]		
\node(F) at +(-1.5,0) {\large $F\maps$};
	\node (J11) {$\mathrm{0}$};
	\node (R1) [left of=J11] {}
	edge  [line width=3.5pt]   node [below]{} (J11)
        edge [line width=3.5pt]    node [above]{} (J11);
	\node (D1) [ above right of=J11] {};
	\node (D2) [ below right of=J11] {};
		\node [style=none] (www) at (0, -1) {};
      \end{tikzpicture} 
	}
  \ar@{|->}@<.25ex>[r] & \SigCoUnitpic{.10\textwidth}
  }
\]
\vspace{-3ex}

Next we show that these morphisms of dagger props exist. Let $G$ be defined on the generators of $\BondGraph$ as described above. We must show that the necessary relations hold in $\Fin\Corel^{\circ}$. Most of the required relations have been shown or follow immediately from what has been shown. The only ones to check are the idempotency laws:
\begin{itemize}
\item $(G(M') \circ G(D)\circ G(M)\circ G(D'))^2 = G(M') \circ G(D)\circ G(M)\circ G(D')$
\item $(G(M) \circ G(D')\circ G(M')\circ G(D))^2 = G(M) \circ G(D')\circ G(M')\circ G(D)$ 
\end{itemize}

\noindent We have that :
\begin{align*}
   (G(M') \circ G(D)\circ G(M)\circ G(D'))^2 &= (\mu_2 \circ d_2 \circ m_2\circ \delta_2)\circ (\mu_2 \circ d_2 \circ m_2\circ \delta_2) \\
    &= (\mu_2 \circ d_2 \circ \mu_2 \circ d_2)\circ ( m_2\circ \delta_2 \circ m_2\circ \delta_2)\\
    &= (\mu_2 \circ d_2 \circ m_2\circ \delta_2) \\
    &= G(M') \circ G(D)\circ G(M)\circ G(D')
\end{align*}

\noindent One can similarly show $(G(M) \circ G(D')\circ G(M')\circ G(D))^2 = G(M) \circ G(D')\circ G(M')\circ G(D)$. Thus such a morphism of props exists.

 Next we show that $F$ exists. Once again we define it on the generators as above and check the necessary relations. The only ones left to check are again the idempotency laws:

\begin{itemize}
\item $(F(M) \circ F(D')\circ F(M')\circ F(D))^2 = F(M) \circ F(D')\circ F(M')\circ F(D)$ 
\item $(F(M') \circ F(D)\circ F(M)\circ F(D'))^2 = F(M') \circ F(D)\circ F(M)\circ F(D')$
\end{itemize}

\noindent We have that:
\vspace{-1ex}
\begin{align*}
  F(M) \circ F(D')\circ F(M')\circ F(D) &=( + \oplus \Delta^{\dagger})\circ ( \Delta \oplus +^{\dagger}) \circ (\Delta^{\dagger} \oplus +) \circ ( +^{\dagger} \oplus \Delta) \\
    &= \mathrm{id}_2
\end{align*}

\vspace{-2ex}

\noindent Since $\mathrm{id}_2\circ \mathrm{id}_2 = \mathrm{id}_2$, we get:
\begin{align*}
  (F(M) \circ F(D')\circ F(M')\circ F(D))^2 &=F(M) \circ F(D')\circ F(M')\circ F(D). 
\end{align*}

\noindent Similarly one can show that:
\begin{align*}
  (F(M') \circ F(D)\circ F(M)\circ F(D'))^2 &= F(M') \circ F(D)\circ F(M)\circ F(D'). 
\end{align*} 
Thus $F$ is a morphism of props. Finally, one can see $F$ is a morphism of dagger props by noting that $F(f^\dagger) = F(f)^\dagger$ whenver $f$ is one of our chosen generators of $\BondGraph$, and similarly for $G$. The dagger structures associated to $\Fin\Corel^{\circ}$,  $\Lag\Rel_k^{\circ}$, and $\BondGraph$ are all reflections. 
\end{proof}

We have now assembled the following diagram where the arrows are symmetric monoidal dagger functors:
\[
    \xymatrix@C-3pt{
      & \Lag\Rel_k^{\circ} \phantom{ |} \ar[r]^-{i'} & \Lag\Rel_k \\
	\BondGraph \ar[ur]^-{F} \ar[dr]_-{G} \\
 	& \Fin\Corel^{\circ} \ar[r]^-{i}  & \Fin\Corel \ar[uu]_-{K}
      }
\]
This does not commute. However, we shall prove that it commutes up to a natural transformation. The idea is that $\Fin\Corel^{\circ}$ is the category that allows us to understand bond graphs in terms of potential and current, while $\Lag\Rel_k^{\circ}$ is the category that allows us to understand bond graphs in terms of effort and flow. Then the Lagrangian relation relating potential and current to effort and flow:
$$\{(V,I,\phi_1,I_1,\phi_2,I_2) | V = \phi_2-\phi_1, I = I_1 =- I_2\}$$ 

\noindent defines a natural transformation between the two functors shown from $\BondGraph$ to $\Lag\Rel_k$.  We can also draw this Lagrangian relation:

\[
  \xymatrixrowsep{1pt}
  \xymatrixcolsep{5pt}
  \xymatrix{
\BondtoCircuitConverterpic{.25\textwidth}  \\
  }
\]

\begin{theorem}\label{thm:natural}
There is a natural transformation $\alpha \maps  i'F \Rightarrow KiG$ such that  $$\alpha_1 = \{(V,I,\phi_1,I_1,\phi_2,I_2) | V = \phi_2-\phi_1, I = I_1 =- I_2\}$$ and   $\alpha_n = \bigoplus_1^n \alpha_1$ for any $n\in \BondGraph$.
\end{theorem}

\begin{proof}
For $1 \in \BondGraph$ let $\alpha_1 \maps k^{2}\asrelto k^4$ be defined by: $$\alpha_1 = \{(V,I,\phi_1,I_1,\phi_2,I_2) | V = \phi_2-\phi_1, I = I_1 =- I_2\}.$$ 
Then we define $\alpha_n = \bigoplus_1^n \alpha_1$. Note since $\Lag\Rel_k$ is a dagger category we have that $$\alpha_1^\dagger = \{(\phi_1,I_1,\phi_2,I_2, V,I,) | V = \phi_2-\phi_1, I = I_1 = -I_2\}$$ 
and that  $\alpha_n^\dagger = \bigoplus_1^n \alpha_1^\dagger$. It can be shown that $\alpha_n^\dagger \circ \alpha_n = \mathrm{id}_{k^{2n}}$ so $\alpha_n^\dagger$ is a left inverse of $\alpha_n$. We want to show the following square commutes for any $f\maps m\to n$ in $\BondGraph$:

\[ 
 \xymatrix{
k^{2m} \ar[d]_{\alpha_m} \ar[r]^{i'F (f)} & k^{2n} \ar[d]^{\alpha_n} \\
 k^{4n} \ar[r]_{KiG(f) } & k^{4m}
} 
\] 
\noindent
Thus we need to show that $$i'F (f) \circ \alpha_m = \alpha_n \circ KiG (f)$$ which is equivalent to showing  $$\alpha_n^\dagger \circ i'F (f) \circ \alpha_m = KiG (f).$$ We proceed by showing that both sides determine strict monoidal functors that are equal on identities, generators, and the braiding. Since all of the morphisms in $\BondGraph$ are built up by starting with generators and repeatedly composing, tensoring, and braiding, this proves that the two functors are equal. 

We first show that the left side, $\alpha_n^\dagger \circ KiG (f) \circ \alpha_m$, defines a functor $$\alpha^\dagger \circ KiG \circ \alpha \maps \BondGraph \to \Lag\Rel_k$$ that sends each object $m$ in $\BondGraph$ to the object $2m$ in $\Lag\Rel_k$ and has $$(\alpha^\dagger \circ KiG \circ \alpha)f =\alpha_n^\dagger \circ KiG (f) \circ \alpha_m$$ for a morphism $f\maps m \to n$. It is clear that identities are preserved so to prove it is a functor we need only show that $$\alpha_n^\dagger \circ KiG (fg) \circ \alpha_m= \alpha_k^\dagger \circ KiG (f) \circ \alpha_n\circ \alpha_n^\dagger \circ KiG (g) \circ \alpha_m$$ for any $f\maps n\to k$ and  $g\maps m\to n$.
We prove this inductively. First, we need that $$KiG(f) \alpha_m = \alpha_n\alpha_n^\dagger KiG(f)\alpha_m$$ for any generator $f\maps m\to n$ and for the braiding $B_{1,1}\maps 2\to 2$. These   calculations are left until the end of the proof.  Next we show that if 
\begin{align*}
   KiG(f) \alpha_a & = \alpha_b\alpha_b^\dagger KiG(f)\alpha_a\\	
   KiG(g) \alpha_c & = \alpha_d\alpha_d^\dagger KiG(g)\alpha_c\\	
\end{align*} 

\vspace{-4ex}

\noindent for two morphisms $f$ and $g$, then $$KiG (f\oplus g)\alpha_m = \alpha_n \alpha_n^\dagger KiG(f\oplus  g)\alpha_m.$$ Let $f\maps a \to b$ and $g \maps c \to d$, where $a+c=m$ and $b+d=n$. Then we get
\begin{align*}
   KiG(f\oplus g)\alpha_m &= KiG(f\oplus g)\circ (\alpha_a \oplus \alpha_c) \\
    & = (KiG(f)\oplus\ KiG(g))\circ(\alpha_a\oplus \alpha_c )\\
    & = KiG(f)\alpha_a \oplus KiG(g)\alpha_c \\
    &= (\alpha_b\alpha_b^\dagger KiG(f)\alpha_a)\oplus 		
	(\alpha_d\alpha_d^\dagger KiG(g)\alpha_c)  \\
    &=(\alpha_b\alpha_d)(\alpha_b^\dagger\alpha_d^\dagger)
	(KiG(f)\oplus KiG(g))\alpha_a\alpha_c\\
    &=\alpha_n\alpha_n^\dagger(KiG(f)\oplus KiG(g))\alpha_m\\
  &=\alpha_n\alpha_n^\dagger KiG (f\oplus g)\alpha_m.\\
\end{align*}

\vspace{-3ex}

\noindent Finally, we show that if 
\begin{align*}
   KiG(f) \alpha_k & = \alpha_n\alpha_n^\dagger KiG(f)\alpha_k\\	
   KiG(g) \alpha_m & = \alpha_k\alpha_k^\dagger KiG(g)\alpha_m\\	
\end{align*} 

\vspace{-3ex}

\noindent for two morphisms $f$ and $g$, then $$KiG(fg)\alpha_m=\alpha_n\alpha_n^\dagger KiG(fg)\alpha_m.$$ 
\noindent Let $f\maps k\to n$ and $ g\maps m\to k$.
\begin{align*}
   KiG(fg)\alpha_m & = KiG(f)KiG(g)\alpha_m \\		             
    &= KiG(f)\alpha_k\alpha_k^\dagger KiG(g)\alpha_m \\
    &= \alpha_n\alpha_n^\dagger KiG(f)\alpha_k\alpha_k^\dagger KiG (g)\alpha_m \\
    &= \alpha_n\alpha_n^\dagger KiG(f) KiG(g)\alpha_m \\
    &= \alpha_n\alpha_n^\dagger KiG(fg)\alpha_m. \\
\end{align*} 

\vspace{-3ex}

\noindent Since any morphism in $\BondGraph$ can be built from generators and braiding morphisms by repeatedly tensoring and composing, this shows that $$KiG(f)\alpha_m=\alpha_n\alpha_n^\dagger KiG (f)\alpha_m$$ for any morphism $f\maps m\to n$. Then finally we can say that 
\begin{align*}
    \alpha_n^\dagger \circ KiG (fg) \circ \alpha_m &= \alpha_n^\dagger \circ KiG (f) KiG(g) \circ \alpha_m \\
    &=\alpha_n^\dagger \circ KiG (f) \alpha_k\alpha_k^\dagger KiG (g) \circ \alpha_m 
\end{align*}

\vspace{-1.5ex}

\noindent so that we have a functor.

 Next we show that both functors,  $\alpha^\dagger \circ KiG \circ \alpha \maps \BondGraph \to \Lag\Rel_k$ and $i'F\maps \BondGraph \to \Lag\Rel_k$, are strict monoidal. It is clear that $i'F$ is a strict monoidal functor since it is the composite of two strict monoidal functors. Thus, it suffices to show that $\alpha^\dagger \circ KiG \circ \alpha$ is strict monoidal.

For an object $m$ in $\BondGraph$ we have $(\alpha^\dagger \circ KiG \circ \alpha)(m) = 2m $ so that $$(\alpha^\dagger \circ KiG \circ \alpha)(m) + (\alpha^\dagger \circ KiG \circ \alpha)(n)  = 2m + 2n = (\alpha^\dagger \circ KiG \circ \alpha)(m+n). $$  The unit object in $\Lag\Rel_k$ is $0$ and is equal to $(\alpha^\dagger \circ KiG \circ \alpha)(0)$. Thus $\alpha^\dagger \circ KiG \circ \alpha$ is strict monoidal.

Since both functors are strict monoidal  the two functors are equal if they are equal on the generators, the braiding, and the identities. We give the proof for the morphism $M$.
\begin{align*}
\alpha_1^{\dagger} \circ KiG (M) &= \{(\phi_5,I_5,\phi_6,I_6,V,I) | V = \phi_6-\phi_5, I = I_5 =- I_6\}\\[-4em]
& \qquad \circ \{(\phi_1,\ldots, I_6): 
	\begin{aligned}
                            \\ 
                             \\
         \phi_1= \phi_5 \phantom{hh}\\ 
          I_1 = I_5 \phantom{hh} \\
          \phi_4 = \phi_6 \phantom{hh}
        \end{aligned}
	\begin{aligned} \\
                              \\
             I_4 = I_6 \phantom{hhhh}\\
          \phi_2 = \phi_3 \phantom{hlhh} \\\hspace{2ex}
         \phantom{hh} I_2 + I_3 = 0\}. \\
	\end{aligned}
\end{align*}

 \noindent From relation composition this results in:
\begin{align*}
\alpha_1^{\dagger} \circ KiG(M) &= \{(\phi_1,\ldots, I_4,V,I) | V = \phi_4-\phi_1, I = I_1 =- I_4 \\[-4em]
&\qquad	 \qquad \qquad  \qquad \enspace \thinspace\thinspace\thinspace \begin{aligned} \\
                            \\ 
          I_1 = I_5,   \\
        \end{aligned}
	\begin{aligned} \\
                              \\
  \phi_2 = \phi_3,  \\
	\end{aligned}
	\begin{aligned} \\
\\
     I_2 + I_3 = 0\}. \\
	\end{aligned}
\end{align*}

\noindent Then we compose with
\begin{align*}
\alpha_2 = \{(V,'I',V'',I'',\phi_1,\ldots ,I_4) | V' = \phi_2-\phi_1, I' = I_1 =- I_2 \phantom{hhhhh} \\[-4em]
\quad \begin{aligned} \\
                             \\
         \phantom{hhh} V'' = \phi_4-\phi_3,\\ 
        \end{aligned}
	\begin{aligned} \\
                              \\
         I''= I_3 = -I_4 \} \phantom{hhh} \\
	\end{aligned}
\end{align*}

\noindent which gives us  $$  \alpha_1^{\dagger} \circ KiG(M) \circ \alpha_2 = \{(V',I',V'', I'',V,I): \; V'+V''=V,  I = I' = I''\}.$$

\noindent This is the same as the Lagragian subspace $$ i'F(M) = \{(E_1,F_1,E_2,F_2,E_3 , F_3): \; E_1+E_2=E_3,  F_1 = F_2 = F_3\}.$$

The proof of the other generators is similar. Thus the two functors are equal on all morphisms in $\BondGraph$, which shows that  $$\alpha_n^\dagger \circ i'F (f) \circ \alpha_m = KiG (f)$$  for any morphism $f\in \BondGraph$. Thus $\alpha$ is a natural transformation, as desired.

Finally, we check that $KiG(f) \alpha_m = \alpha_n \alpha^\dagger_n KiG(f) \alpha_m$ when $f$ is a generator or the braiding $B_{1,1}$. Consider again the generator $M\maps 2\to 1$. We need that $$KiG(M) \alpha_2 = \alpha_1\alpha_1^\dagger KiG(M)\alpha_2.$$ \vspace{1ex} We have $KiG(M) \alpha_2  = \{(V,I,V',I',\phi_5,I_5, \phi_6,I_6)\}$ such that 

\vspace{-5ex}

\begin{align*}
   V+V'=\phi_6-\phi_5 &&& I=I'=I_5=-I_6. \\		
\end{align*} 

\vspace{-5ex}

\noindent By composing with  $ \alpha_1^\dagger $ we get $ \{(V,I,V',I', V'',I'')\}$ such that 

\vspace{-5ex}

\begin{align*}
   V+V'=\phi_6-\phi_5 &&& I=I'=I_5=-I_6 \\	
     V''=\phi_6-\phi_5  &&&      I''=I_5=-I_6.  &&&  \\
\end{align*} 

\vspace{-5ex}

\noindent These equations simplify to become

\vspace{-5ex}

\begin{align*}
   V+V'=V'' &&&     I=I'=I''. &&& \\	
\end{align*}

\vspace{-5ex}

\noindent Now if we compose with $\alpha_1$ we get $ \{(V,I,V',I',\phi_7,I_7, \phi_8,I_8)\}$ such that

\vspace{-5ex}

\begin{align*}
   V+V'=V'' &&&     I=I'=I'' &&& \\	
    V''=\phi_8-\phi_7  &&& I''=I_7=-I_8    \\
\end{align*} 

\vspace{-5ex}

\noindent which can be reduced to

\vspace{-5ex}

\begin{align*}
   V+V'=\phi_8-\phi_7 &&&     I=I_7=-I_8.\\	
\end{align*} 

\vspace{-5ex}

This  is $KiG(M) \alpha_2 $ with a change of names. We can see the general pattern with just this generator. Composing with $\alpha_n\alpha_n^\dagger$ simply renames the potentials and currents without changing any of the relations. For potentials, the relations  defining the subspaces deal only with the differences in potentials. Thus by writing the difference in potential as voltage, we get relations that define the subspace in terms of voltage. However, this does nothing because voltage is the difference in potential. So when we convert back to potential, the same relations hold. A similar process ensures that the current can be renamed through the same process without changing the subspace. With similar reasoning we conclude that the equation holds for the other generators as well as the braiding.
\end{proof}

Recall that at the end of Section \ref{sec:sigflowdiagrams} we presented a commututing diagram summarizing the relationship between linear circuits, cospans, corelations, and signal-flow diagrams. Combining that diagram with the result from Theorem \ref{thm:natural} we show how bond graphs fit into the same story:
\[
    \xymatrix@C-3pt{
     \Ccirc \phantom{ |} \ar[r]^-{P} \ar@{^{(}->}[dd] & \Circ \phantom{ |} \ar@{^{(}->}[dd] \ar[r]^-{H'} & \Fin\Cospan \phantom{ |} \ar[r]^-{H} &
      \Fin\Corel \phantom{ |} \ar[dd]^{K} & \Fin\Corel^{\circ} \phantom{ |} \ar[l] \ar@<-6ex>@{}[dd]^(.25){}="a"^(.75){}="b" \ar@{<=}^{\alpha} "a";"b"  \\
        &&&&& \BondGraph  \ar[ul]_-{G} \ar[dl]^-{F} \\ 
      \Ccirc_k \phantom{ |} \ar[r]^-{P} \ar[d]_{T} & \Circ_k \phantom{ |} \ar[rr]^-{\blacksquare} & & \Lag\Rel_k \phantom{ |} \ar@{^{(}->}[d] & \Lag\Rel_k^{\circ} \phantom{ |} \ar[l]    \\
      \SigFlow_k \phantom{ |}\ar[rrr]^-{\square} & & & \Fin\Rel_k. \phantom{ |}
      }
\]


\chapter{Appendix}
\label{chap:appendix}
\section{The mathematics of props}
\label{sec:appendix}

\subsection{Props from symmetric monoidal categories}
\label{subsec:symmoncats}

There is a 2-category $\Symm\Mon\Cat$ where:

\begin{itemize}
\item objects are symmetric monoidal categories,
\item morphisms are symmetric monoidal functors, and 
\item 2-morphisms are monoidal natural transformations.   
\end{itemize}
Here our default notions are the `weak' ones (which Mac Lane \cite{Ma13} calls `strong'),  where all laws hold up to coherent natural isomorphism.     Props, on the other hand, are \emph{strict} symmetric monoidal categories where every object is \emph{equal} to a natural number, and morphisms between them are \emph{strict} symmetric monoidal functors sending each object $n$ to itself.   Categorical structures found in nature are often weak.  Thus, to study them using props, one needs to `strictify' them.  Thanks to conversations with Steve Lack we can state the following results, which accomplish this strictification.   

The first question is when a symmetric monoidal category $\C$ is equivalent, as an object of $\Symm\Mon\Cat$, to a prop.    In other words: when is does there exist a prop $\T$ and symmetric monoidal functors $j \maps \T \to \C$, $k \maps \C \to \T$ together with monoidal natural isomorphisms $jk \cong 1_\C$ and $kj \cong 1_\T$?
This is answered by Proposition \ref{prop:strictification_1}:

\vskip 1em 
\noindent \textbf{Proposition 2.2.2}  \textit{A symmetric monoidal category $\C$ is equivalent, as an object of $\Symm\Mon\Cat$, to a prop if and only if there is an object $x \in \C$ such that every object of $\C$ is isomorphic to the $n$th tensor power of $x$ for some $n \in \N$.}

\begin{proof}
The `only if' condition is obvious, so suppose that $\C$ is any symmetric monoidal category with $x \in \C$ such that every object of $\C$ is isomorphic to a tensor power of $x$.   We use a method due to A.\ J.\ Power, based on this lemma:

\begin{lemma}[Lemma 3.3, \cite{Po}]  
\label{lem:power} 1. Any functor $f \maps A \to B$ can be factored as $je$ where $e$ is bijective on objects and $j$ is fully faithful.  

2.  Given a square that commutes up to a natural isomorphism $\alpha$:
\[
  \xymatrix{
     A \ddrrtwocell<\omit>{<0>\alpha} \ar[rr]^h \ar[dd]_u && B \ar[dd]^v \\ 
    & &\\
    C \ar[rr]_{g} && D 
  }
\]
where $h$ is bijective on objects and $g$ is fully faithful, there exists a unique functor $w \maps B \to C$ and natural transformation $\beta \maps v \Rightarrow g w$ such that $wh = u$ and $\beta h = \alpha$.  Moreover $\beta$ is a natural isomorphism.
\end{lemma}

Let $\NN$ be the strict symmetric monoidal category with one object for each natural number and only identity morphisms, with the tensor product of objects being given by addition.  There exists a symmetric monoidal functor $f \maps \NN \to \C$ that sends the $n$th object of $\NN$ to $x^{\otimes n} = x \otimes (x \otimes (x \otimes \cdots))$.   By Part 1 of Lemma \ref{lem:power} we can factor $f$ as a composite
\[   \NN \stackrel{e}{\longrightarrow} \T \stackrel{j}{\longrightarrow} \C \]
where $e$ is bijective on objects and $j$ is fully faithful.     By our condition on $\C$, $f$ is essentially surjective.  It follows that $j$ is also essentially surjective, and thus an equivalence of categories.   We claim that $\T$ can be given the structure of a strict symmetric monoidal category making $j$ symmetric monoidal.    It will follow that $\T$ is a prop and $j \maps \T \to \C$ can be promoted to an equivalence in $\Symm\Mon\Cat$.

To prove the claim, we use the 2-monad $P$ on $\Cat$ whose strict algebras are strict symmetric monoidal categories.    For any category $\A$, $P(\A)$ is the `free strict symmetric monoidal category' on $\A$.  Explicitly, an object of $P(\A)$ consists of a finite list $(a_1\dots,a_n)$ of objects of $\A$.  A morphism from $(a_1,\dots,a_n)$ to $(b_1,\dots b_m)$ exists only if $n=m$, in which case it consists of a permutation $\sigma \in S_n$ and a morphism from $a_i$ to $b_{\sigma(i)}$ in $\A$ for each $i$.    For the rest of the 2-monad structure see for example \cite[Sec.\ 4.1]{FGHW}.  Any symmetric monoidal category can be made into a pseudoalgebra of $P$, and then the pseudomorphisms between such pseudoalgebras are the symmetric monoidal functors.   

Power's method \cite{Po} applies to 2-monads that preserve the class of functors that are bijective on objects.  The 2-monad $P$ has this property.

We can take $\NN$ above to be $P(1)$, since they are isomorphic.    Any object $x \in \C$ determines a functor $1 \to \C$ and so a pseudomorphism $F \maps P(1) \to \C$.    Replacing $F$ by an equivalent pseudomorphism if necessary, we can assume $F(n) = x^{\otimes n}$, so the situation in the first paragraph holds with this choice of $F$, and we can factor $F$ as $P(1) \stackrel{e}{\to} \T \stackrel{j}{\to} \C$ as before.  Since $F$ is a pseudomorphism, this square commutes up to a natural isomorphism $\alpha$:
\[
  \xymatrix{
     P(P(1)) \ddrrtwocell<\omit>{<0>\alpha} \ar[rr]^{P(e)} \ar[d]_{m_1} && P(\T) \ar[d]^{P(j)} \\ 
 P(1) \ar[d]_e   & & P(\C) \ar[d]^{a} \\
    \T \ar[rr]_{j} && \C 
  }
\]
where $a$ comes from the pseudoalgebra structure  and $m_1$ comes the multiplication in the 2-monad. The functor $P(e)$ is bijective on objects because $P$ is, and $j\maps \T \to \C$ is fully faithful.  Thus, by Part 2 of Lemma \ref{lem:power}, there exists a unique functor $w\maps P(\T) \to \T$ and natural isomorphism $\beta \maps a P(j) \to jw$ such that $w P(e) =e m_1$ and $\beta P(e) = \alpha$.   Thus, $w$ makes $\T$ into a strict algebra of $P$ and $\beta$ makes $j$ into a pseudomorphism from $\T$ to $\C$.  This proves the claim: $\T$ has been given the structure of a symmetric monoidal category for which $j \maps \T \to \C$ is a symmmetric monoidal functor.
\end{proof}

The second question is when a symmetric monoidal functor $f \maps \C \to \T$
between props is isomorphic, in $\Symm\Mon\Cat$, to a morphism of props.   In other words: when is there a morphism of props $g \maps \C \to T$ and a monoidal natural isomorphism $f \cong g$?  This is answered by Proposition \ref{prop:strictification_2}:

\vskip 1em
\noindent \textbf{Proposition 2.2.3} \textit{Suppose $\C$ and $\T$ are props and $f \maps \C \to \T$ is a symmetric monoidal functor. Then $f$ is isomorphic to a strict symmetric monoidal functor $g \maps \C \to  \T$.   If $f(1) = 1$, then $g$ is a morphism of props.}

\begin{proof}
As in the previous proof, let $P$ be the 2-monad on $\Cat$ whose strict algebras
are strict monoidal categories.    The objects of $P(1)$ correspond to natural
numbers, with tensor product being given addition, so we can write the $n$th object
as $n$.  There is a unique strict monoidal functor $e \maps P(1) \to C$ with $e(n) = n$ for all $n$.   By a result of Blackwell, Kelly and Power \cite[Cor.\ 5.6]{BKP}, any free algebra of a 2-monad is `flexible', meaning that pseudomorphisms out of this algebra are isomorphic to strict morphisms.   Thus, the symmetric monoidal functor $f e \maps P(1) \to \T$ is isomorphic, in $\Symm\Mon\Cat$, to a strict symmetric monoidal functor $h \maps P(1) \to \T$.      Let $\alpha \maps fe \Rightarrow h$ be the isomorphism.

We can define a strict symmetric monoidal functor $g \maps \C \to \T$ as follows.
On objects, define $g(n) = h(n)$.  For any morphism $\phi \maps m \to n$, 
there is a unique morphism $g(\phi)$ making this square commute:
\[    
 \xymatrix{
     f(m)  \ar[r]^{f(\phi)} \ar[d]_{\alpha_m} & f(n) \ar[d]^{\alpha_n} \\ 
      g(m) \ar[r]_{g(\phi)} & g(n)  
  }
\]
One can check that $g$ is a strict symmetric monoidal functor.  The above square
gives a natural isomorphism between $f$ and $g$, which by abuse of language
we could call $\alpha \maps f \Rightarrow g$.   It is easy to check that this is
a monoidal natural isomorphism.
\end{proof}

It is worth noting that Propositions \ref{prop:strictification_1} and \ref{prop:strictification_2}, and the proofs just given, generalize straightforwardly from props to `$C$-colored' props, with 
$\N$ replaced everywhere by the free commutative monoid on the set of colors, $C$. 

\subsection{The adjunction between props and signatures}
\label{subsec:monadic}

Our goal in this section is to prove Proposition \ref{prop:monadic}:

\vskip 1em 
\noindent \textbf{Proposition 2.3.3} \textit{There is a forgetful functor
\[           U \maps \PROP \to \Set^{\N \times \N}  \]
sending any prop to its underlying signature and any morphism of props to its
underlying morphism of signatures.  This functor is monadic.}

\begin{proof}
The plan of the proof is as follows.  We show that props are models of a typed Lawvere theory 
$\Theta_\PROP$ whose set of types is $\N \times \N$.  We write this as
follows:
\[              \PROP \simeq \Mod(\Theta_\PROP)  .\]
This lets us apply the following theorem, which says that for any typed
Lawvere theory $\Theta$ with $T$ as its set of types, the forgetful functor
\[       U \maps \Mod(\Theta) \to \Set^T \]
is monadic.    The desired result follows.

\begin{theorem} 
\label{thm:monadic}
 If $\Theta$ is a typed Lawvere theory with $T$ as its set of types,
then the forgetful functor $U\maps \Mod(\Theta) \to \Set^T$ is monadic.
\end{theorem}

\begin{proof}
The origin of this result may be lost in the mists of time, though the case $T = 1$
is famous, and was proved in Lawvere's thesis \cite{Law}. This result appears as 
Theorem A.41 in Ad\'amek, Rosick\'y and Vitale's book on algebraic theories \cite{ARV}.
Trimble has proved a generalization where $\Set$ is replaced by any category $\C$ that 
is cocomplete and has the property that finite products distribute over colimits \cite{Tr}.
An even more general result appears in the work of Nishisawa and Power \cite{NP}.
\end{proof}

Now let us define the terms here and see how this result applies to our situation.
First, for any set $T$, let $\N[T]$ be the set of finite linear combinations of elements
of $T$ with natural number coefficients.  This becomes a commutative monoid under
addition, in fact the free commutative monoid on $T$.  Define a \define{$T$-typed Lawvere theory} to be a category $\Theta$ with finite products whose set of objects is $\mathbb{N}[T]$, with the product of objects given by addition in $\N[T]$.   We call the elements of $T$ \define{types}.  

Suppose $\Theta$ is an $T$-typed Lawvere theory.  
Let \define{$\Mod(\Theta)$} be the category whose objects are functors
$F \maps \Theta \to \Set$ preserving finite products, and whose morphisms are
natural transformations between such functors.  We call an object of $\Mod(\Theta)$ a \define{model} of $\Theta$, and call a morphism in $\Mod(\Theta)$ a 
\define{morphism of models}.

There is an inclusion $T \hookrightarrow \N[T]$, since $\N[T]$ is the free commutative
monoid on $T$.  Thus, any model $M$ of $\Theta$ gives, for each type
$t \in T$, a set $M(t)$.   Similarly, any morphism of models 
$\alpha \maps M \to M'$ gives, for each type $t \in T$, a function
$\alpha_t \maps M(t) \to M'(t)$.  Indeed, there is a functor
\[         U \maps \Mod(\Theta) \to \Set^T   \]
with $U(M)(t) = M(t)$ for each model $M$ and $U(\alpha)(t) = \alpha_t$ for each
morphism of models $\alpha \maps M \to M'$.   If we call $\Set^T$ the category of 
\define{signatures} for $T$-typed Lawvere theories, then $U$ sends models to their
underlying signatures and morphisms of models to morphisms of their underlying
signatures.  Theorem \ref{thm:monadic} says that $U$ is monadic.

To complete the proof of Proposition \ref{prop:monadic} we need to give 
a typed Lawvere theory $\Theta_{\PROP}$ whose models are props.  We can do this by giving a 
``sketch."  Since this idea has been described very carefully by Barr and Wells \cite{BW} and others, we content ourselves with a quick intuitive explanation of the special case we really need which could be called a ``products sketch."   This is a way of presenting a $T$-typed Lawvere
theory by specifying a set $T$ of generating objects (or types), a set of generating morphisms between formal products of these generating objects, and a set of relations given as commutative diagrams.  These commutative diagrams can involve the generating morphisms and also morphisms built from these using the machinery available in a category with finite products.  

To present the typed Lawvere theory $\Theta_{\PROP}$ we start by taking $T = \N \times \N$.  For the purposes of easy comprehension, we call the corresponding generating objects $\hom(m,n)$ for $m,n \in \N$, since these will be mapped by any model of $\Theta_{\PROP}$ to 
the homsets that a prop must have.   We then include the following generating morphisms:

\begin{itemize}
\item
For any $m$ we include a morphism $\iota_m \maps 1 \rightarrow \hom(m,m)$. These give rise to the identity morphisms in any model of $\Theta_{\PROP}$.
\item
For any $\ell,m,n$, we include a morphism $\circ_{\ell,m,n} \maps \hom(m,n) \times \hom(\ell,m) \rightarrow \hom(\ell,n)$. These give us the ability to compose morphisms in any model of $\Theta_{\PROP}$.
\item 
For any $m,n,m',n'$, we include a morphism $\otimes_{m,n,m',n'} \maps \hom(m',n') \times \hom(m,n) \rightarrow \hom(m+m',n+n')$. These allow us to take the tensor product of morphisms in any model of $\Theta_{\PROP}$.
\item 
For any $m,m'$, we include a morphism $b_{m,m'} \maps 1 \rightarrow \hom(m+m',m'+m)$. These give the braidings in any model of $\Theta_{\PROP}$.
\end{itemize}

Finally, we impose relations via the following commutative diagrams.  In these diagrams, unlabeled arrows are morphisms provided by the structure of a category with 
finite products.  We omit the subscripts on morphisms since they can be inferred 
from context.  We begin with a set of diagrams, one for each $m,n \in \N$, that ensure associativity of composition in any model of $\Theta_{\PROP}$:

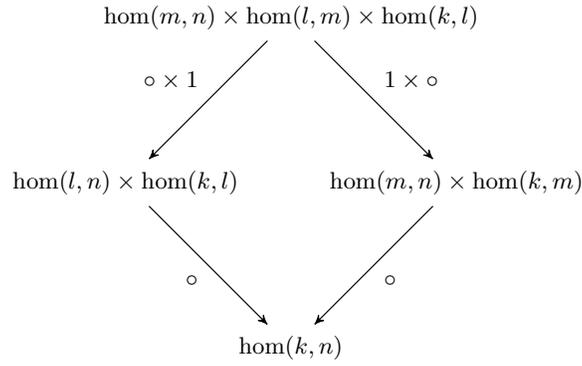
\begin{center}
\begin{tikzpicture}[->,>=stealth',node distance=2.2cm, auto]
 \node (A) {$\hom(m,n)\times \hom(l,m)\times \hom(k,l)$};
 \node (B) [below of=A,left of=A] {$\hom(l,n)\times \hom(k,l)$};
 \node (C) [below of=A,right of=A] {$\hom(m,n)\times \hom(k,m)$};
 \node (D) [right of=B,below of=B] {$\hom(k,n)$};

 \draw[->][line width=rule_thickness] (A) to node [swap]{$\circ \times 1$}(B);
 \draw[->][line width=rule_thickness] (A) to node {$1 \times \circ$}(C);
 \draw[->][line width=rule_thickness] (B) to node [swap]{$\circ$}(D);
\draw[->][line width=rule_thickness]  (C) to node {$\circ$}(D);
\end{tikzpicture}
\end{center}

\noindent Next, a diagram that ensures each $\iota_m$ picks out an identity morphism:

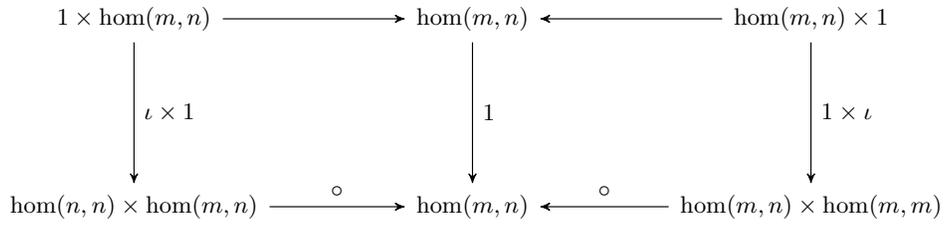
\begin{center}
\begin{tikzpicture}[->,>=stealth',node distance=3.5cm, auto]
 \node (A) {$1 \times \hom(m,n)$};
 \node (B) [right of=A,xshift=1cm] {$\hom(m,n)$};
 \node (C) [right of=B,xshift=1cm] {$\hom(m,n) \times 1$};
 \node (D) [below of=A,yshift=1cm] {$\hom(n,n)\times \hom(m,n)$};
\node (E) [below of=B,yshift=1cm] {$\hom(m,n)$};
\node (F) [below of=C,yshift=1cm] {$\hom(m,n)\times \hom(m,m)$};
 \draw[->][line width=rule_thickness] (A) to node {$$}(B);
 \draw[->][line width=rule_thickness] (C) to node {$$}(B);
 \draw[->][line width=rule_thickness] (A) to node {$\iota \times 1$}(D);
\draw[->][line width=rule_thickness] (B) to node {$1$}(E);
 \draw[->][line width=rule_thickness] (C) to node {$1\times \iota$}(F);
 \draw[->][line width=rule_thickness] (D) to node {$\circ$}(E);
 \draw[->][line width=rule_thickness] (F) to node [swap]{$\circ$}(E);
\end{tikzpicture}
\end{center}

\noindent Next, a diagram that ensures associativity of the tensor product of morphisms:

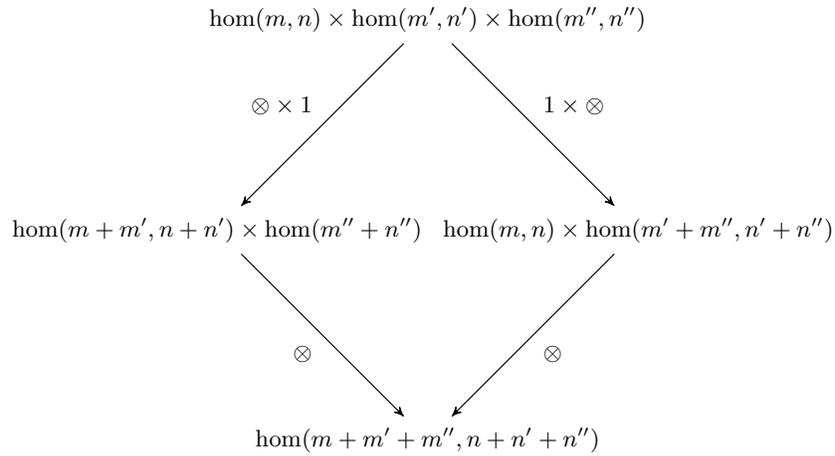
\begin{center}
\begin{tikzpicture}[->,>=stealth',node distance=2.8cm, auto]
 \node (A) {$\hom(m,n)\times \hom(m',n')\times \hom(m'',n'')$};
 \node (B) [below of=A,left of=A] {$\hom(m+m',n+n')\times \hom(m''+n'')$};
 \node (C) [below of=A,right of=A] {$\hom(m,n)\times \hom(m'+m'',n'+n'')$};
 \node (D) [right of=B,below of=B] {$\hom(m+m'+m'',n+n'+n'')$};

 \draw[->][line width=rule_thickness] (A) to node [swap]{$\otimes \times 1$}(B);
 \draw[->][line width=rule_thickness] (A) to node {$1 \times \otimes$}(C);
 \draw[->][line width=rule_thickness] (B) to node [swap]{$\otimes$}(D);
\draw[->][line width=rule_thickness] (C) to node {$\otimes$}(D);
\end{tikzpicture}
\end{center}

\noindent Next, a diagram that ensures that tensoring with the identity morphism on $0$ acts trivially on morphisms:

\begin{center}
\begin{tikzpicture}[->,>=stealth',node distance=3.5cm, auto]
 \node (A) {$1 \times \hom(m,n)$};
 \node (B) [right of=A,xshift=1cm] {$$};
 \node (C) [right of=B,xshift=1cm] {$\hom(m,n) \times 1$};
 \node (D) [below of=A,yshift=1cm] {$\hom(0,0)\times \hom(m,n)$};
\node (E) [below of=B,yshift=1cm] {$\hom(m,n)$};
\node (F) [below of=C,yshift=1cm] {$\hom(m,n)\times \hom(0,0)$};
 \draw[->][line width=rule_thickness] (A) to node {$$}(E);
 \draw[->][line width=rule_thickness] (C) to node {$$}(E);
 \draw[->][line width=rule_thickness] (A) to node [swap]{$\iota \times 1$}(D);
 \draw[->][line width=rule_thickness] (C) to node {$1\times \iota$}(F);
 \draw[->][line width=rule_thickness] (D) to node {$\otimes$}(E);
 \draw[->][line width=rule_thickness] (F) to node [swap]{$\otimes$}(E);
\end{tikzpicture}
\end{center}

\noindent Next, a diagram that ensures that the tensor product preserves composition:

\begin{center}
\begin{tikzpicture}[->,>=stealth',node distance=1.2cm, auto]
 \node (A) {$\hom(m',n')\times \hom(l',m')\times \hom(m,n)\times \hom(l,m)$};
 \node (B) [below of=A] {$\hom(l',n')\times \hom(l,n)$};
 \node (C) [right of=A,xshift=7.25cm] {$\hom(m',n')\times \hom(m,n) \times \hom(l',m') \times \hom(l,m)$};
 \node (D) [below of=C] {$\hom(m+m',n+n')\times \hom(l+l',m+m')$};
 \node (E) [below of=B,right of=B,xshift=2.25cm] {$\hom(l+l',n+n')$};

 \draw[->][line width=rule_thickness] (A) to node [swap]{$\circ \times \circ$}(B);
 \draw[->][line width=rule_thickness] (A) to node {$$}(C);
 \draw[->][line width=rule_thickness] (B) to node [swap]{$\otimes$}(E);
\draw[->][line width=rule_thickness] (C) to node {$\otimes \times \otimes$}(D);
\draw[->][line width=rule_thickness] (D) to node {$\circ$}(E);
\end{tikzpicture}
\end{center}

\noindent Next, a diagram that ensures the naturality of the braiding:

\begin{center}
\begin{tikzpicture}[->,>=stealth',node distance=2.2cm, auto]
 \node (A) {$1 \times \hom(m',n')\times \hom(m,n)$};
 \node (B) [right of=A,xshift=2.5cm] {$\hom(m',n')\times 1 \times \hom(m,n)$};
 \node (C) [right of=B,xshift=2.5cm] {$\hom(m',n') \times \hom(m,n) \times 1$};
 \node (D) [below of=A,yshift=1cm] {$\hom(n+n',n'+n)\times \hom(m+m',n+n')$};
\node (E) [below of=C,yshift=1cm] {$\hom(m+m',n+n')\times \hom(m+m',m'+m)$};
\node (F) [below of=B,yshift=-2cm] {$\hom(m+m',n'+n)$};

 \draw[->][line width=rule_thickness] (B) to node [swap]{$$}(A);
 \draw[->][line width=rule_thickness] (B) to node {$$}(C);
 \draw[->][line width=rule_thickness] (A) to node [swap]{$b\times \otimes$}(D);
\draw[->][line width=rule_thickness] (D) to node [swap]{$\circ$}(F);
 \draw[->][line width=rule_thickness] (C) to node {$\otimes \times b$}(E);
 \draw[->][line width=rule_thickness] (E) to node {$\circ$}(F);
\end{tikzpicture}
\end{center}

Next, a diagram that ensures that the braiding is a symmetry:

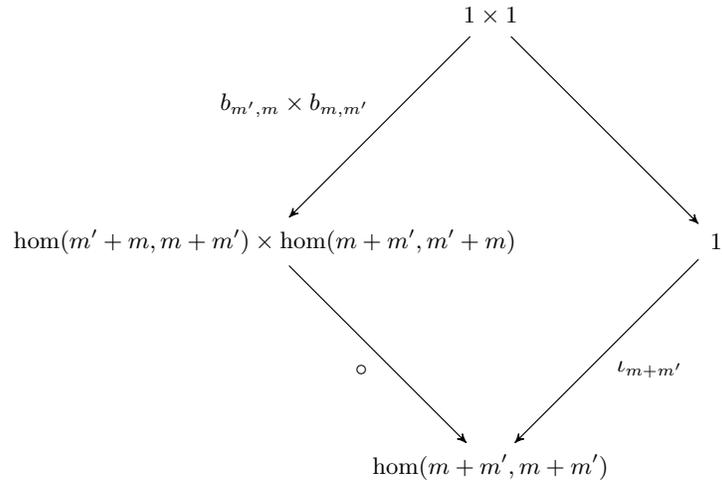
\begin{center}
\begin{tikzpicture}[->,>=stealth',node distance=3cm, auto]
 \node (A) {$1 \times 1$};
 \node (B) [below of=A,left of=A] {$\hom(m'+m,m+m')\times \hom(m+m',m'+m)$};
 \node (C) [below of=A,right of=A] {$1$};
 \node (D) [right of=B,below of=B] {$\hom(m+m',m+m')$};

 \draw[->][line width=rule_thickness] (A) to node [swap]{$b_{m',m}\times b_{m,m'}$}(B);
 \draw[->][line width=rule_thickness] (A) to node {$$}(C);
 \draw[->][line width=rule_thickness] (B) to node [swap]{$\circ$}(D);
\draw[->][line width=rule_thickness] (C) to node {$\iota_{m+m'}$}(D);
\end{tikzpicture}
\end{center}

\noindent
Next, a diagram that ensures that the braidings $b_{0,n}$ are identity morphisms:

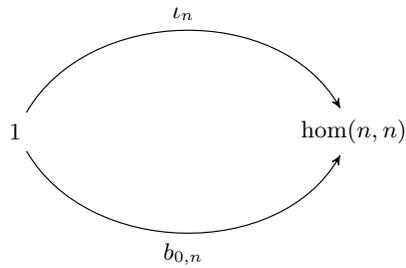
\begin{center}
\begin{tikzpicture}[->,>=stealth',node distance=2cm, auto]
 \node (B) {$1$};
 \node (C) [right of=B,xshift=2.5cm] {$\hom(n,n)$};

 \draw[->][line width=rule_thickness]  [in=120, out=60, looseness=1.0](B) to node {$\iota_n$}(C);
 \draw[->][line width=rule_thickness] [in=-120, out=-60, looseness=1.0] (B) to node [swap]{$b_{0,n}$}(C);
\end{tikzpicture}
\end{center}

\noindent
Finally, we need a diagram to ensure the braiding obeys the hexagon identities. However, since the associators are trivial and the braiding is a symmetry, the two hexagons reduce to a 
single triangle.   To provide for this, we use the following diagram:

\begin{center}
\begin{tikzpicture}[->,>=stealth',node distance=2.5cm, auto]
 \node (A) {$1$};
 \node (B) [right of=A,xshift=7.5cm] {$1 \times 1 \times 1 \times 1$};
 \node (C) [below of=B] {$\hom(m+m'',m''+m)\times \hom(m',m')\times \hom(m'',m'') \times \hom(m+m',m'+m)$};
 \node (D) [below of=C] {$\hom(m'+m+m'',m'+m''+m)\times \hom(m+m'+m'',m'+m+m'')$};
 \node (E) [below of=D,left of=D,xshift=-1.75cm] {$\hom(m+m'+m'',m'+m''+m)$};

 \draw[->][line width=rule_thickness] (B) to node {$$}(A);
 \draw[->][line width=rule_thickness] (B) to node {$b\times \iota \times \iota \times b$}(C);
 \draw[->][line width=rule_thickness] (C) to node {$\otimes \times \otimes$}(D);
\draw[->][line width=rule_thickness] (D) to node {$\circ$}(E);
\draw[->][line width=rule_thickness] (A) to node  [swap]{$b$}(E);

\end{tikzpicture}
\end{center}

This completes the list of commutative diagrams in the sketch for $\Theta_{\PROP}$.  These diagrams simply state the definition of a PROP, so there is a 1-1 correspondence between models of $\Theta_{\PROP}$ in $\Set$ and props.  Similarly, morphisms of models of $\Theta_{\PROP}$ in $\Set$ correspond to morphisms of props.  This gives an isomorphism of
categories $\PROP \cong \Mod(\Theta_{\PROP})$ as desired. This concludes the proof.
\end{proof}

It is worth noting that Proposition \ref{prop:monadic}, and the proof just given, generalize straightforwardly from props to `$C$-colored' props, with $\N$ replaced everywhere by the free commutative monoid on a set $C$, called the set of `colors'.    There is also a version for operads, a version for $C$-colored operads, and a version for $T$-typed Lawvere theories: there is a typed Lawvere theory whose models in $\Set$ are $T$-typed Lawvere theories!  In each case we simply need to write down a sketch that describes the structure under consideration.

Proposition \ref{prop:monadic} has a wealth of consequences; we conclude with two that
Erbele needed in his work on control theory \cite[Prop.\ 6]{E}.  In rough terms, these results 
say that adding generators to a presentation of a prop $P$ gives a new prop $P'$ having
$P$ as a sub-prop, while adding equations gives a new prop $P'$ that is a quotient of $P$.  
We actually prove more general statements.

In all that follows, let $\Theta$ be a $T$-typed Lawvere theory.    Let $U \maps \Mod(\Theta) \to \Set^T$ be the forgetful functor and $F \maps \Set^T \to \Mod(\Theta)$ its left adjoint.  Further, suppose we have two coequalizer diagrams in $\Mod(\Theta)$:
\[
    \xymatrix{
      F(E) \ar@<-.5ex>[r]_{\rho} \ar@<.5ex>[r]^{\lambda} & F(\Sigma) \ar[r]^-{\pi} & P   
      }
\] 
\[
	\xymatrix{     
       F(E') \ar@<-.5ex>[r]_{\rho'} \ar@<.5ex>[r]^{\lambda'} & F(\Sigma') \ar[r]^-{\pi'} & P'
    }
\]
together with morphisms $f \maps E \to E'$, $g \maps \Sigma \to \Sigma'$ such that these squares commute:
\[ 
\xymatrix{
	F(E) \ar[r]^{\lambda} \ar[d]_{F(f)} & F(\Sigma) \ar[d]^{F(g)} \\
	F(E') \ar[r]^{\lambda'}                        & F(\Sigma') 
	}
\qquad \qquad
	\xymatrix{
	F(E) \ar[r]^{\rho} \ar[d]_{F(f)} & F(\Sigma) \ar[d]^{F(g)} \\
	F(E') \ar[r]^{\rho'}                        & F(\Sigma') 
	}
\] 
Then, thanks to the universal property of $P$, there exists a unique morphism $h \maps P \to P'$ making the square at right commute:
\[ 
\xymatrix{
F(E) \ar@<-.5ex>[r]_{\rho} \ar@<.5ex>[r]^{\lambda} \ar[d]_{F(f)} 
& F(\Sigma) \ar[r]^-{\pi} \ar[d]_{F(g)}             & P \ar[d]^{h} \\
F(E') \ar@<-.5ex>[r]_{\rho'} \ar@<.5ex>[r]^{\lambda'} 
&	F(\Sigma') \ar[r]^-{\pi'}                        & P' 
	}
\]
In this situation, adding extra equations makes $P'$ into a quotient object of $P$.  More precisely, and also more generally:

\begin{corollary}
If $g$ is an epimorphism, then $h$ is a regular epimorphism.  
\end{corollary}

\begin{proof} 
Given that $g$ is an epimorphism in $\Set^T$, it is a regular epimorphism.
So is $F(g)$, since left adjoints preserve regular epimorphisms, and so is $\pi'$, by definition.
It follows that $\pi' \circ F(g) = h \circ \pi$ is a regular epimorphism, and thus so is $h$.
\end{proof}
	
One might hope that in the same situation, adding extra generators makes $P$ into a subobject of $P'$.    More precisely, one might hope that if $f$ is an isomorphism and $g$ is a monomorphism, $h$ is a monomorphism.   This is not true in general, but it is when the typed Lawvere theory $\Theta$ is $\Theta_{\PROP}$.

To see why some extra conditions are needed, consider a counterexample provided by Charles Rezk \cite{Rezk}.  There is a typed Lawvere theory with two types whose models consist of:
\begin{itemize}
\item a ring $R$,
\item a set $S$,
\item a function $f \maps S \to R$ with $f(s) = 0$ and $f(s) = 1$ for 
all $s \in S$.
\end{itemize}
Thanks to the peculiar laws imposed on $f$, the only models are pairs $(R,S)$ where $R$ is an arbitrary ring and $S$ is empty, and pairs $(R,S)$ where $R$ is a terminal ring (one with $0 = 1$) and $S$ is an arbitrary set.   The free model on $(\emptyset,\emptyset) \in \Set^2$ is $(\Z,\emptyset)$, while the free model on $(\emptyset, 1)$ is $(\{0\},1)$.   Thus, the monomorphism $(\emptyset,\emptyset) \to (\emptyset,1)$ in $\Set^2$ does not induce a monomorphism between the corresponding free models: the extra generator in the set part of $(R,S)$ causes the ring part to `collapse'.

This problem does not occur for Lawvere theories with just one type, nor does it happen for typed Lawvere theories that arise from typed operads, more commonly known as `colored' operads \cite{BM,Yau}. A typed Lawvere theory arises from a typed operad when it can be presented in terms of operations obeying purely equational laws for which each variable appearing in an equation shows up exactly once on each side.  The laws governing props are of this form: for example, the operations for composition of morphisms obey the associative law
\[            (f \circ g) \circ h = f \circ (g \circ h)  .\]
It follows that $\Theta_{\PROP}$ arises from a typed operad, so the following corollary 
applies to this example:

\begin{corollary}
Suppose that either $\Theta$ is a $T$-typed Lawvere theory with $T = 1$ or $\Theta$ arises from a $T$-typed operad.  If $f$ is an isomorphism and $g$ is a monomorphism, then $h$ is a monomorphism.
\end{corollary}
  	
\begin{proof}
We may assume without loss of generality that $f \maps E \to E'$ is the identity and $g \maps \Sigma \to \Sigma'$ is monic.  Since monomorphisms in $\Set^T$ are just $T$-tuples of injections, we can write $\Sigma' \cong \Sigma + \Delta$ for some signature $\Delta$ in such a way that 
$g \maps \Sigma \to \Delta$ is isomorphic to the coprojection $\Sigma \to \Sigma + \Delta$.   It follows that $F(g)$ is isomorphic to the coprojection $i \maps F(\Sigma) \to F(\Sigma) + F(\Delta)$, and the diagram
\[ 
\xymatrix{
F(E) \ar@<-.5ex>[r]_{\rho} \ar@<.5ex>[r]^{\lambda} \ar[d]_{F(f)} 
& F(\Sigma) \ar[r]^-{\pi} \ar[d]_{F(g)}             & P \ar[d]^{h} \\
F(E') \ar@<-.5ex>[r]_{\rho'} \ar@<.5ex>[r]^{\lambda'} 
&	F(\Sigma') \ar[r]^-{\pi'}                        & P' 
	}
\]
is isomorphic to this diagram:
\[ 
\xymatrix{
F(E) \ar@<-.5ex>[r]_{\rho} \ar@<.5ex>[r]^{\lambda} \ar[d]_{1} 
& F(\Sigma) \ar[r]^-{\pi} \ar[d]_{i}             & P \ar[d]^{j} \\
F(E') \ar@<-.5ex>[r]_-{\rho} \ar@<.5ex>[r]^-{\lambda} 
&	F(\Sigma) + F(\Delta) \ar[r]^-{\pi + 1}                        & P + F(\Delta)
	}
\]
where $j$ is the coprojection from $P$ to $P + F(\Delta)$.  Thus it suffices to prove the following:

\begin{lemma}
Suppose that either $\Theta$ is a $T$-typed Lawvere theory with $T = 1$ or $\Theta$ arises from a $T$-typed operad.  If $P \in \Mod(\Theta)$ and $\Delta \in \Set^T$ then the coprojection
$j \maps P \to P + F(\Delta)$ is a monomorphism.
\end{lemma}

\noindent \textsl{Proof.}  We thank Todd Trimble for this proof.  First suppose $T = 1$.  To show that $j$ is monic it suffices to show that $U(j)$ is injective, since $U \maps \Mod(\Theta) \to \Set$ is faithful and thus it reflects monomorphisms \cite[Prop.\ 11.8]{ARV}.  Either $U(P)$ is empty and the injectivity is trivial, or $U(P)$ is nonempty, in which case we can split the coprojection $j \maps P \to P + F(\Delta)$, since all we need for this is a morphism $F(\Delta) \to P$, or equivalently, a function $\Delta \to U(P)$.

Next suppose that $\Theta$ is a $T$-typed Lawvere theory that comes from a $T$-typed operad $O$.  Here we can use the following construction: given $P \in \Mod(\Theta)$, we can form a model $P^\ast \in \Mod(\Theta)$ that has an extra element for each type $t \in T$.  To do this, we first set
\[  M^\ast(t) = M(t) \sqcup \{x_t\}   \]
for all $t \in T$, where $x_t$ is an arbitrary extra element.  Then, we make $P^*$ into an algebra of $O$ as follows.  Suppose $f \in O(t_1, \dots, t_n; t)$ is any operation of $O$ with inputs of type
$t_1, \dots, t_n \in T$ and output of type $t \in T$.   Since $P$ is an algebra of $O$, $f$ acts on $P$ as some function 
\[ P(f) \maps P(t_1) \times \cdots \times P(t_n) \to P(t)   .\]
Then we let $f$ act on $P^*$ as the function 
\[ P^*(f) \maps P^\ast(t_1) \times \cdots \times P^\ast(t_n) \to P^\ast(t)   .\]
that equals $P(f)$ on $n$-tuples $(p_1, \dots, p_n)$ with $p_i \in P(t_i)$ for all $i$, and otherwise gives $x_t$.   One can readily check that this really defines an algebra of $O$ and thus a model of $\Theta$.   The evident morphism of models $k \maps P \to P^*$ is monic because again $U$ is faithful \cite[Prop.\ 14.8]{ARV} and the underlying morphism of signatures $U(k) \maps U(P) \to U(P^*)$ is monic.

With this construction in hand, we can show that the coprojection $j \maps P \to P + F(\Delta)$ is monic.  We have just constructed a monomorphism $k \maps P \to P^*$. Now extend this to a morphism $\ell \maps P + F(\Delta) \to P^*$: to do this, we just need a morphism $F(\Delta) \to  P^*$, which we can take to be the one corresponding to the map $X \to U(P^*)$ whose component $X(t) \to U(P^*)(t)$ is the function mapping every element of $X(t)$ to $x_t$.   Then, we have $k = \ell \circ j$, and since $k$ is monic, $j$ must be as well.  
\end{proof}

\end{document}